\documentclass[12pt,a4paper,english,reqno]{amsart}
\usepackage[utf8]{inputenc}
\usepackage{mathrsfs}
\usepackage{amsfonts}
\usepackage{amsthm}
\usepackage{amssymb}
\usepackage{amscd}
\usepackage{amstext}
\usepackage[all]{xy}
\xyoption{2cell}
\UseTwocells
\usepackage{graphicx}
\usepackage{comment}
\usepackage{calrsfs}
\usepackage{framed}

\usepackage[colorinlistoftodos]{todonotes}
\usepackage[colorlinks=true, allcolors=blue]{hyperref}

\usepackage{tikz}
\usetikzlibrary{positioning,intersections,cd,matrix, arrows, backgrounds}
\usetikzlibrary{shapes,shapes.geometric,shapes.misc}
\usetikzlibrary{decorations}
\usetikzlibrary{decorations.pathreplacing}
\newcommand\Ar[3]{\ar[from={#1}, to={#2}, #3]}
\usepackage[OT2,T1]{fontenc}

\theoremstyle{plain}
\newtheorem{thm}{Theorem}[section]
\newtheorem{prp}[thm]{Proposition}
\newtheorem{lem}[thm]{Lemma}
\newtheorem{cor}[thm]{Corollary}
\newtheorem{clm}{Claim}

\newtheorem{prb}[thm]{Problem}

\newtheorem*{thm-nn}{Theorem}
\newtheorem*{prp-nn}{Proposition}
\newtheorem*{lem-nn}{Lemma}
\newtheorem*{cor-nn}{Corollary}
\newtheorem*{clm-nn}{Claim}
\newtheorem*{cnj-nn}{Conjecture}
\newtheorem*{prb-nn}{Problem}
\newtheorem*{fct}{Fact}

\theoremstyle{definition}
\newtheorem{dfn}[thm]{Definition}
\newtheorem{exm}[thm]{Example}

\newtheorem*{dfn-nn}{Definition}

\newtheorem{rmk}[thm]{Remark}
\newtheorem{ntn}[thm]{Notation}

\newcommand{\xyR}[1]{%
\xydef@\xymatrixrowsep@{#1}}

\newcommand{\xyC}[1]{%
\xydef@\xymatrixcolsep@{#1}}

\newcommand\al{\alpha}
\newcommand\be{\beta}
\newcommand\ga{\gamma}
\newcommand\de{\delta}
\newcommand\ep{\varepsilon}
\newcommand\ze{\zeta}
\newcommand\et{\eta}
\renewcommand\th{\theta}

\newcommand\la{\lambda}
\newcommand\ro{\rho}
\newcommand\si{\sigma}
\newcommand\Si{\Sigma}

\newcommand\ph{\phi}
\newcommand\ps{\psi}
\newcommand\Ga{\Gamma}
\newcommand\De{\Delta}

\newcommand\Ps{\Psi}

\newcommand\Hom{\operatorname{Hom}}
\newcommand\RHom{\mathbf{R}\!\Cdg}

\newcommand\End{\operatorname{End}}

\renewcommand\mod{\operatorname{mod}}
\newcommand\Mod{\operatorname{Mod}}
\newcommand\dgMod{\calC}
\newcommand\DGMod{\calC_{\mathrm{dg}}}

\newcommand\Cdg{\operatorname{\calC_{\mathrm{dg}}}}

\newcommand\ChMod{\operatorname{\calC}}

\newcommand\Jac{\operatorname{Jac}}
\newcommand\Pot{\operatorname{Pot}}

\newcommand\prj{\operatorname{prj}}

\newcommand\hprj{\operatorname{\calH_{\mathrm{p}}}}

\newcommand\thick{\operatorname{thick}}

\newcommand\perf{\operatorname{per}}

\newcommand\calA{{\mathcal A}}
\newcommand\calB{{\mathcal B}}
\newcommand\calC{{\mathcal C}}
\newcommand\calD{{\mathcal D}}
\newcommand\calE{{\mathcal E}}
\newcommand\calF{{\mathcal F}}

\newcommand\calH{{\mathcal H}}

\newcommand\calP{{\mathcal P}}

\newcommand\calS{{\mathcal S}}
\newcommand\calT{{\mathcal T}}
\newcommand\calU{{\mathcal U}}
\newcommand\calV{{\mathcal V}}

\newcommand\bbN{{\mathbb N}}
\newcommand\bbZ{{\mathbb Z}}
\newcommand\bbP{{\mathbb P}}

\newcommand\bbV{{\mathbb V}}
\newcommand\op{^{\mathrm{op}}} 

\newcommand\inv{^{-1}}

\renewcommand\implies{\text{$\Rightarrow$}\ }

\newcommand\incl{\hookrightarrow}
\newcommand\iso{\cong}
\newcommand\ds{\oplus}
\newcommand\ox{\otimes}

\newcommand\Lox{\overset{\mathbf{L}}{\otimes}}

\newcommand\udl{\underline}
\newcommand\ovl{\overline}
\newcommand\Ds{\bigoplus}

\def\dsm#1,#2..#3{\bigoplus_{{#1}={#2}}^{#3}}
\def\sm#1,#2..#3{\sum_{{#1}={#2}}^{#3}}

\newcommand\id{1\kern-.25em{\text{{\rm l}}}}

\newcommand\isoto{\ \raise.8ex\hbox{$^{\sim}$}\kern-.7em\hbox{$\to$}\ }

\newcommand\ya[1]{\xrightarrow{#1}}
\newcommand\blank{\operatorname{-}}

\newcommand\Ltimes{\overset{\mathbf{L}}{\otimes}}

\newcommand\bg{%
\family{cmr}\size{20}{12pt}\selectfont}

\newcommand\bigzerou{%
\smash{\lower1.7ex\hbox{\bg 0}}}

\def\repr[#1;#2;#3;#4;#5]{
\left(
\begin{matrix}#1\\#2\end{matrix}
#3
\begin{matrix}#4\\#5\end{matrix}
\right)}
\newcommand\bmat[1]{\begin{bmatrix} #1 \end{bmatrix}}
\newcommand\smat[1]{\begin{smallmatrix} #1 \end{smallmatrix}}

\newcommand\kCat{\Bbbk\text{-}\mathbf{Cat}}
\newcommand\VCat{\mathbb{V}\text{-}\mathbf{Cat}}

\newcommand\kAB{\Bbbk\text{-}\mathbf{AB}}

\newcommand\DGkCat{\k\text{-}\mathbf{dgCat}}
\newcommand\kdgCat{\k\text{-}\mathbf{dgCat}}
\newcommand\kdgCAT{\k\text{-}\mathbf{dgCAT}}
\newcommand\kDGCAT{\k\text{-}\mathbf{DGCAT}}

\newcommand\kDGCat{\k\text{-}\mathbf{DGCat}}

\newcommand\sfHom{\mathsf{Hom}}

\newcommand\coop{^{\mathrm{coop}}}

\newcommand{\dcoprod}{\displaystyle\coprod}

\newcommand\bfj{\mathbf{j}}
\newcommand\bfp{\mathbf{p}}
\newcommand\bfB{\mathbf{B}}
\newcommand\bfC{\mathbf{C}}
\newcommand\bfD{\mathbf{D}}
\newcommand\bfL{\mathbf{L}}
\newcommand\bfR{\mathbf{R}}
\newcommand\bfV{\mathbf{V}}

\def\Gr{\textstyle\int}
\renewcommand\k{\Bbbk}
\newcommand\Colax{\operatorname{Colax}}
\newcommand\dom{\operatorname{dom}}

\newcommand\com{\operatorname{com}}
\newcommand\To{\Rightarrow}

\newcommand{\bi}[3]{{}_{#2}{#1}_{#3}}
\newcommand{\hyph}{\text{-}}
\newcommand\dbigoplus{\displaystyle\bigoplus}

\newcommand\tri{\mathrm{tri}}
\newcommand\red{\mathrm{red}}
\newcommand{\Ginz}{\widehat{\Gamma}}

\newcommand\bfF{\mathbf{F}}
\newcommand\bfG{\mathbf{G}}
\newcommand\bfps{\boldsymbol{\ps}}
\newcommand\bfa{\boldsymbol{a}}

\newcommand\Univ{\mathsf{Univ}}
\newcommand\SET{\mathsf{SET}}

\newcommand\frU{\mathfrak{U}}
\newcommand\Cls{\mathbf{Class}}

\newcommand\Cat{\mathbf{Cat}}
\newcommand\CAT{\mathbf{CAT}}

\newcommand\uCAT{\underline{\CAT}}

\newcommand\TRI{\mathbf{TRI}}
\newcommand\kTRI{\Bbbk\text{-}\TRI}
\newcommand\kuTRI{\Bbbk\text{-}\underline{\TRI}}

\newcommand\kFRB{\Bbbk\text{-}\mathbf{FRB}}

\newcommand\st{\mathrm{st}}
\newcommand\qis{\mathrm{qis}}

\newcommand\Loc{\operatorname{Loc}}
\newcommand\stder{\overset{\mathrm{sd}}{\leadsto}}
\newcommand\stderr{\overset{\mathrm{sd}}{\,\reflectbox{$\leadsto$}\,}}

\newcommand{\dereq}{\overset{\text{der}}{\sim}}

\newcommand\Nname[1]{|[alias=#1]|}

\newcommand{\wdt}{\widetilde}

\begin{document}
\title[Standard derived equivalences of
diagrams of dg categories]{Characterizations of standard derived equivalences of
diagrams of dg categories and their gluings}

\author{Hideto Asashiba and Shengyong Pan}
\date{December 2025}

\thanks{Hideto Asashiba is partially supported by Grant-in-Aid for Scientific Research
(B) 25287001 and (C) 18K03207 from JSPS, and by Osaka Central Advanced Mathematical Institute
(MEXT Joint Usage/Research Center on Mathematics and Theoretical Physics JPMXP0619217849); and
Shengyong Pan is supported by Beijing Natural Science Foundation (1262017,1252011) and the Fundamental Research Funds for the Central Universities of Beijing Jiaotong University (2024JBMC001).}
  
\address{Department of Mathematics, Faculty of Science, Shizuoka University,
836 Ohya, Suruga-ku, Shizuoka, 422-8529, Japan;}
\address{
Institute for Advanced Study, KUIAS, Kyoto University,
Yoshida Ushinomiya-cho, Sakyo-ku, Kyoto 606-8501, 
Japan; and}
\address{Osaka Central Advanced Mathematical Institute,
3-3-138 Sugimoto, Sumiyoshi-ku,
Osaka, 558-8585, Japan.}
\email{asashiba.hideto@shizuoka.ac.jp}

\address{School of Mathematics and Statistics, Beijing Jiaotong University,
Beijing, 100044, China.}
\email{shypan@bjtu.edu.cn}

\begin{abstract}
A diagram consisting of differential graded (dg for short) categories and dg functors
is formulated in this paper as a colax functor $X$
from a small category $I$ to the 2-category $\kdgCat$ of small dg categories, dg functors and dg natural transformations over a fixed commutative ring $\k$.
If $I$ is a group regarded as a category with only one object $*$,
then $X$ is nothing but a colax action of the group $I$ on the dg category $X(*)$.
In this sense, this $X$ can be regarded as a generalization of a dg category with a colax action of a group.
We define a notion of standard derived equivalence between such colax functors by generalizing the corresponding notion between dg categories with a group action.
Our first main result gives some characterizations of this notion, one of which is given in terms of generalized versions of a tilting object and a quasi-equivalence.
On the other hand, for such a colax functor $X$,
the dg categories $X(i)$ with $i$ objects of $I$ can be glued together to have a single dg category $\Gr_I X$, called the Grothendieck construction of $X$.
Our second main result asserts that for such colax functors $X$ and $X'$, the Grothendieck construction $\Gr_I X'$ is derived equivalent to $\Gr_I X$ if there exists a standard derived equivalence from $X'$ to $X$. 
These results generalize the first-named author's results to the dg case, respectively.
Even for dg categories with group actions, these results are new.
In particular, the second result gives a new tool to show the derived equivalence between the orbit categories of dg categories with group actions, which will be illustrated in some examples.
\end{abstract}


\maketitle

\tableofcontents

\section{Introduction}
Throughout this paper we fix a commutative ring $\k$, and all linear categories
and all linear functors are considered to be linear over $\k$.
For any category $\calC$, we denote
by $\calC_0$ and $\calC_1$ the collections of all objects of $\calC$ and of all morphisms in $\calC$, respectively.
Let $\calA$ be a linear category.
Then we have canonical embeddings
$\calA \hookrightarrow \Mod \calA \hookrightarrow \calD(\Mod \calA)$,
where $\Mod \calA$ denotes the category of (right) $\calA$-modules, and
$\calD(\Mod \calA)$ stands for the derived category of $\calA$-modules.
Two linear categories $\calA$ and $\calA'$ are said to be
{\em derived equivalent} if $\calD(\Mod \calA)$ and $\calD(\Mod \calA')$
are equivalent as triangulated categories.
If $\calA$ and $\calA'$ are {\em Morita equivalent}, i.e., if $\Mod \calA$
and $\Mod \calA'$ are equivalent as linear categories, then $\calA$ and $\calA'$
are derived equivalent, but the converse is not true in most cases.
Thus, the derived equivalence classification is usually rougher than the Morita equivalence classification.
Brou\'{e}'s abelian defect conjecture in \cite{Br} made this notion of derived equivalences more important.
In this connection, Rickard classified Brauer tree algebras up to derived equivalence in \cite{Rick2},
and the first-named author gave the derived equivalence classification of
representation-finite selfinjective algebras in \cite{Asa99}.
An essential tool for the classifications above was given by
Rickard's Morita type theorem for derived categories of rings in \cite{Rick}, which was
generalized later by Keller in \cite{Ke1} to differential graded (dg for short)
categories with an alternative proof.
Both theorems give very useful criteria to check whether rings or dg categories
are derived equivalent in terms of tilting complexes or tilting subcategories,
which will be also used in this paper as a fundamental tool.

Recall that a dg category is a graded linear category whose morphism spaces are endowed with differentials satisfying suitable compatibility with the grading, and note that a dg category with a single object is nothing but a dg algebra.
Dg categories are used to enhance triangulated categories
by Bondal--Kapranov in \cite{BK}, which was
motivated by the study of exceptional collections of coherent sheaves on projective varieties.
Also, they are efficiently used in \cite{Ke3} by Keller
to compute derived invariants such as
K-theory, Hochschild (co-)homology and cyclic homology associated with a ring or a variety.

Now, we come back to derived equivalences of linear categories.
If $\calA$ and $\calA'$ are derived equivalent linear categories,
then they share invariants under derived equivalences,
such as the center, the Grothendieck group, and those listed in the previous paragraph.
If we have the classification of a class $\calS$ of linear categories under derived equivalences,
then the computation of an invariant under derived equivalences in question for a complete set of representatives
gives the invariant for all linear categories in the class $\calS$.
To obtain such a classification we need a tool that produces a lot of derived equivalent pairs $\calA$ and $\calA'$.
In \cite{Asa-13}, we have considered a diagram of linear categories over a small category $I$,
which is formulated as a colax functor $X$ from $I$ to the $2$-category $\kCat$
of small linear categories, linear functors and natural transformations.
Then for each morphism $a \colon i \to j$ in $I$,
we have a linear functor $X(a) \colon X(i) \to X(j)$ between linear categories.
We take the Grothendieck construction $\Gr_I X$ of $X$ as a gluing of these $X(i)$s along $X(a)$s.
For two such colax functors $X$ and $X'$, suppose that a derived equivalence from $X'(i)$ to $X(i)$
is given for each $i \in I_0$.
Then we have given a way to glue together these derived equivalences from $X'(i)$ to $X(i)$
to have a derived equivalence between the gluings $\Gr_I X'$ and $\Gr_I X$
if the derived equivalences are ``compatible'' with $X'(a)$ and $X(a)$ \ $(a \in I_1)$.
The latter condition was shown to follow if $X'$ and $X$ are derived equivalent in a natural sense.
This also shows us how to produce
a glued linear category $\Gr_I X'$ that is derived equivalent to $\Gr_I X$,
by using $X$ and derived equivalences from $X'(i)$ to $X(i)$ with $i \in I_0$.
The class of colax functors from $I$ to $\kCat$ is naturally extended to a $2$-category $\Colax(I, \kCat)$,
and hence it is possible to define a notion of equivalences between its objects.
A colax functor $X$ is said to be $\k$-\emph{projective} (resp.\ $\k$-\emph{flat}) if the $\k$-modules
$X(i)(x, y)$ are projective (resp.\ flat) for all $i \in I_0$
and for all objects $x,y$ of $X(i)$.
After defining a tilting colax functor for $X$,
the derived equivalence of colax functors are characterized in
the main result in \cite{Asa-a} as follows.

\begin{thm}
Let $X, X' \in \Colax(I, \kCat)_0$.
If $X$ and $X'$ are derived equivalent, then
$X'$ is equivalent in the $2$-category $\Colax(I, \kCat)$
to a tilting colax functor $\calT$ for $X$.
If $X'$ is $\k$-projective, then the converse holds.
\end{thm}

The main result in  \cite{Asa-13} gives a sufficient condition for two colax functors to have derived equivalent Grothendieck constructions
as follows :

\begin{thm}
Let $X, X' \in \Colax(I, \kCat)_0$.
Assume that $X$ is $\k$-flat and that $X'$ is equivalent to a tilting colax functor $\calT$ for $X$ in $\Colax(I, \kCat)$.
Then $\Gr_I X$ and $\Gr_I X'$ are derived equivalent.
\end{thm}

As a special case when $I$ is a group $G$ (regarded as a category with a single object $*$),
$\Gr_I X = \calA/G$ is the orbit category
(also called the skew group category and denoted by $\calA * G$)
of the linear category $\calA:= X(*)$ with the $G$-action $X$, and hence
it tells us when a derived equivalence between linear categories $\calA$ and $\calA'$ with $G$-actions
have derived equivalent orbit categories $\calA/G$ and $\calA'/G$.

To consider derived equivalences of linear categories it is natural
to deal with it in the setting of dg categories as is seen
in \cite{Ke1}.
Therefore, it is natural to investigate the same problem for dg categories.
In this paper, we will do it by considering the 2-category $\kdgCat$
of small dg $\k$-categories instead of $\kCat$.
To avoid a set theoretic difficulty, we fix a Grothendieck universe
$\frU$ once for all and use the hierarchy:
$\Cls^0_0 \subset \Cls^1_0 \subset \Cls^2_0 \subset \cdots \subset \SET$ of the class $\SET$ of all sets given by Levy \cite{Le18},
where $\Cls^0_0:= \frU$, $\Cls^1_0:= \calP(\frU)$, the power set of $\frU$.
A 2-category $\bfC$ is called $k$-{\em moderate} ($k \ge 0$)
if $\bfC_0, \bfC_1, \bfC_2 \subseteq \Cls^k_0$, where
$\bfC_i$ is the set of all $i$-cells in $\bfC$ for all $i = 0,1,2$.
For a colax functor $X$ from the small category $I$
to $\kdgCat$, the ``derived colax functor'' $\calD(X)$ can be defined
in a natural way (see Corollary \ref{cor:colax-colax}), which is a
colax functor from $I$ to the 2-category $\kuTRI^2$ of 2-moderate
triangulated categories (see Notation \ref{ntn:2-cats}(4)).
For two colax functors $X$ and $X'$ from $I$
to $\kdgCat$, they are said to be \emph{derived equivalent} if
$\calD(X)$ and $\calD(X')$ are equivalent in the 2-category $\kuTRI^2$.
Here we consider a special type of derived equivalence, called
a ``standard derived equivalence''
and (i) characterize it by using the notions of quasi-equivalence
bimodules over colax functors
and tilting colax functors; and (ii) in that case we show that their gluings $\Gr_I X$ and $\Gr_I X'$ are derived equivalent.
More precisely, the first result is stated as follows.

\begin{thm}[Theorem \ref{thm:characterization-1} in the text]
\label{first-result}
Let $X,X'\in \Colax(I,\kdgCat)_0$. 
Then among the statements below, we have implications $(1) \Rightarrow (2) \Rightarrow (3) \Rightarrow (4)$. If $X$ and $X'$ are $\k$-flat, then we also have the implication $(4) \Rightarrow (1)$, namely, all four conditions are equivalent.
\begin{enumerate}
\item
There exists an $X'$-$X$-bimodule $Z$ such that
$\blank\Lox_{X'}Z \colon \calD(X') \to \calD(X)$ is an equivalence in
$\Colax(I, \kuTRI^2)$.

\item
There exists a $1$-morphism $(\bfF, \bfps)\colon \DGMod(X') \to \DGMod(X)$ in the $2$-category \newline
$\Colax(I,\kdgCAT)$ such that
$\bfL(\bfF, \bfps)\colon$ $\calD(X') \to \calD(X)$ is an equivalence in $\Colax(I, \kuTRI^2)$.
\item
There exists a tilting colax functor $T$ for $X$, and 
there exists a quasi-equivalence $X'$-$T$-bimodule $E$.

\item
There exists a tilting colax functor $T$ for $X$, and
there exist $1$-morphisms
\[
(\bfG, \bfps')\colon \DGMod(X') \to \DGMod(T)\ \text{and } (\bfF, \bfps)\colon \DGMod(T) \to \DGMod(X)
\]
in the $2$-category
$\Colax(I,\kdgCAT)$ such that 
\[\bfL(\bfG, \bfps')\colon\calD(X') \to \calD(T)\ \text{and}\ 
\bfL(\bfF, \bfps)\colon\calD(T) \to \calD(X)
\]
are equivalences in $\Colax(I, \kuTRI^2)$.

\end{enumerate}

\end{thm}

The equivalence of the form $\blank\Lox_{X'}Z \colon \calD(X') \to \calD(X)$
in (1) above is called a {\em standard derived equivalence} from $X'$ to $X$,
and we say that $X'$ {\em is standardly derived equivalent to} $X$
if one of the conditions in the theorem above holds.
We denote this situation by $X' \stder X$ or $X \stderr X'$.
The second result is stated as follows.

\begin{thm}[Theorem \ref{mainthm2-bimod} in the text]
\label{second-result}
Let $X, X' \in \Colax(I, \kdgCat)_0$, and
assume that $X$ is $\k$-flat.
If $X'$ is standardly derived equivalent to $X$,
or equivalently, if
there exists a quasi-equivalence $X'$-$T$-bimodule for some
tilting colax functor $\calT$ for $X$,
then $\Gr_I X'$ is derived equivalent to $\Gr_I X$.
\end{thm}

We remark that for any $X, X' \in \Colax(I, \kdgCat)$,
we do not know whether the relation 
$X' \stder X$ is a symmetric relation.
Therefore, when we say that ``$X$ and $X'$ are standardly derived equivalent'', this means that there exists a zigzag chain
form $X$ to $X'$ of the form
$X = X_0 \stder X_1 \stderr X_2 \stder \cdots \stderr X_{2n} = X'$
for some $n \ge 1$ and $X_1, X_2, \dots, X_{2n} \in \Colax(I, \kdgCat)$.
In contrast, for dg categories $\calA$ and $\calA'$
the relation that ``$\calA'$ is standardly derived equivalent to $\calA$''
is symmetric (\cite[Lemma 6.2 (b)]{Ke1}), which allows us to express it by saying that ``$\calA$ and $\calA'$ are standardly derived equivalent''.

We also remark the following two special cases.
The first one is the case where $I$ is a category with only one object $*$ and one morphism $\id_{*}$.
In this case, $X$ is identified with a dg category $X(*)$, and hence
Theorem \ref{first-result} generalizes Keller's theorem \cite[Theorem 8.1]{Ke1}.
The second one is the case when $I = G$ is a group.
In this case, $\Gr_I X = \calA/G$ is the orbit dg category 
of a dg category $\calA:= X(*)$ with a $G$-action, and hence
Theorem \ref{second-result} gives us a sufficient condition for a derived equivalence
between dg categories $\calA$ and $\calA'$ with $G$-actions
to have derived equivalent dg orbit categories $\calA/G$ and $\calA'/G$.
We will apply this to the complete Ginzburg dg algebras of quivers with potentials having a $G$-action.
Recall that a quiver with potentials was introduced by Derksen, Weyman and Zelevinsky in \cite{DWZ} to study the theory of cluster algebras.
From a quiver with potentials $(Q, W)$, the Jacobian algebra
$J(Q, W)$ and the completed Ginzburg dg algebra $\Ginz(Q, W)$ are defined,
which are related as $H^0(\Ginz(Q,W)) = J(Q,W)$.
Therefore, $\Ginz(Q,W)$ is regarded as an extension of Jacobian algebra to a dg algebra.

The orbit category (the skew group algebra) $J(Q,W)/G$ was computed
up to Morita equivalence
as the form $J(Q_G, W_G)$ for some quiver with potentials $(Q_G, W_G)$
by Paquette--Schiffler in \cite{PS} in the case that $G$ is a finite subgroup of
the automorphism group of $J(Q_G, W_G)$ acting freely on vertices.
On the other hand,
the orbit dg category (the skew group dg algebra) $\Ginz(Q,W)/G$ was computed
up to Morita equivalence
as the form $\Ginz(Q_G, W_G)$ for some quiver with potentials $(Q_G, W_G)$
by Le Meur in \cite{LM} in the case that $G$ is a finite group (see also Amiot--Plamondon \cite{AP} for the case that $G = \bbZ/2\bbZ$,
Giovannini and Pasquali \cite{GP} for the cyclic case, and Giovannini, Pasquali and Plamondon \cite{GPP} for the finite abelian case).
We remark that for both $J(Q,W)$ and $\Ginz(Q,W)$,
the quiver $Q_G$ can be computed by using a result by Demonet in \cite{De} on the
computation of the skew group algebra of the path algebra of a quiver with an action of a finite group, and in the arbitrary group case, $Q_G$ can be computed from a non-admissible presentation
given in \cite{Asa-Kim} by giving it as an admissible presentation.

By Keller--Yang \cite{KeY},
if $(Q',W')$ is obtained as a mutation of $(Q, W)$, then
the dg algebras $\Ginz(Q,W)$ and $\Ginz(Q',W')$ are derived equivalent.
Using Theorem \ref{first-result} above, we can show that
this derived equivalence sometimes induces
a derived equivalence between $\Ginz(Q_G, W_G)$ and $\Ginz(Q'_G, W'_G)$,
where even if $(Q_G, W_G)$ and $(Q'_G, W'_G)$ do not need to be obtained
by a mutation from each other.
For this phenomenon, an example will be given at the end of the paper
(see Example \ref{exm:Mizuno}).

The paper is organized as follows. In Section 2, we shall fix notations and prepare some basic facts for
our proofs. In Section 3, we collect basic facts about enriched categories that will be needed later.
In section 4, we will introduce the notion of $I$-coverings. This is a generalization of
that of $G$-coverings for a group $G$ introduced in \cite{Asa11}, which was
obtained by generalizing the notion of Galois coverings
introduced by Gabriel in \cite{Gab}. This will be used in the proof of Theorem \ref{second-result}.
In Section 5, we define a 2-functor $\Gr_I\colon \Colax(I, \DGkCat) \to \DGkCat$ whose
correspondence on objects is a dg version of (the opposite version of)
the original Grothendieck construction. In Section 6, we will show that the Grothendieck construction is a strict left adjoint to the diagonal 2-functor, and that $I$-coverings are essentially given
by the unit of the adjunction.
In Section 7, we will give the definition of derived colax functors
together with necessary pseudofunctors.
In Section 8, we define necessary terminologies such as $2$-quasi-isomorphisms for $2$-morphisms,
quasi-equivalences for $1$-morphisms, and the derived $1$-morphism
$\bfL (F, \ps) \colon \calD(X') \to \calD(X)$ of a $1$-morphism
$(F, \ps) \colon X' \to X$ between colax functors,
and show the fact that the derived $1$-morphism of a quasi-equivalence  1-morphism between colax functors $X$, $X'$
turns out to be an equivalence between derived dg module colax functors of $X$, $X'$.
Also, we give definitions of tilting colax functors and of derived equivalences.
In Section 9, we define standard derived equivalences of colax functors from $I$ to $\kdgCat$ and quasi-equivalence bimodules, and prove our Theorem \ref{first-result}, and in Section 10, we prove our Theorem \ref{second-result}.
Two examples are given in Section 11 that illustrate our second main theorem in the group action case.
In Appendix, we review the notions of quasi-equivalences and derived equivalences for dg categories for convenience of the reader.

\section*{Acknowledgements}
The basic part of this work was done during visits of the second author to Shizuoka University
from January to February in 2018 and in 2020. Most of the paper was written through discussions by Zoom afterword. 
The second author would like to thank the first author
for his nice hospitality and useful discussions. He is also very grateful to Xiao-Wu Chen, Bernhard Keller, Henning Krause and Changchang Xi for their encouragement and helpful discussions.

\section{Preliminaries}

In this section, after making some set theoretic remarks,
we recall the definition of the 2-category of colax functors
from a small category $I$ to a 2-category from \cite{Asa-a} (see also Tamaki \cite{Tam}).

First of all, we make set theoretic remarks
(see \cite{Le18} or \cite[Appendix A]{Asa-book} for details).
In this paper, we adopt ZFC (Zermelo-Fraenkel set-theory (ZF) with the axiom of choice (C)) as axioms of set theory,
and we do not assume the existence of urelements.
In addition, we assume {\em the axiom of universe}, stating that
any set is an element of some (Grothendieck) universe.
The class of all sets 
is denoted by
$\SET$
.
The power set of a set $A$ is denoted by $\calP A$, and the set of all non-negative
integers by $\bbN$.
Recall the following (see e.g., \cite{Wi}):
\begin{fct}
The class $\Univ$ of all universes are well-ordered.
\end{fct}
We fix a universe $\frU$ with $\bbN \in \frU$ once for all.
A set $S$ is called a $\frU$-{\em small} set (resp.\ a $\frU$-{\em class}) if
$S \in \frU$ (resp.\ $S \subseteq \frU$).

\begin{dfn}\label{dfn:Psi}
Let $A$ be a set, and assume that $\frU \subseteq A$.
By axiom of universe and Fact above,
there exists the smallest universe $\frU'$ such that
$A \in \frU'$.
We define $\Ps A$ to be the smallest set $X \in \frU'$ satisfying the condition that
\[
X \supseteq A \cup (X \times X) \cup (\bigcup_{I \in A}X^I).
\]
Therefore in particular, since $X = \frU$ satisfies this condition for $A = \frU$,
we have $\Ps \frU = \frU$.
\end{dfn}

\begin{rmk}
The existence of $\Ps A$ is proved in \cite[Proposition A.2.2]{Asa-book}, and it satisfies
\[
\Ps A =  A \cup (\Ps A \times \Ps A) \cup \left(\bigcup_{I \in A} (\Ps A)^I\right).
\]
\end{rmk}

\begin{dfn}\label{dfn:class}
Let $k \in \bbN$.
\begin{enumerate}
\item
An element of $(\calP\Ps)^k\frU$ is called a $k$-\emph{class},
and an element of $((\calP\Ps)^k\frU) \setminus ((\calP\Ps)^{k-1}\frU)$ is called
a \emph{proper $k$-class}.
\item
The category of the $k$-classes (and the maps between them) is denoted by $\Cls^k$.
Therefore we have $\Cls^k_0 = (\calP\Ps)^k\frU$.
\end{enumerate}
\end{dfn}

Following Levy \cite{Le18}, we use the hierarchy of $\SET$:
$\Cls^0_0 \subset \Cls^1_0 \subset \Cls^2_0 \subset \cdots \subset \SET$.

\begin{rmk}\label{rmk:k-class-imi}
The following are immediate from the definition above:
\begin{enumerate}
\item
A 0-class is nothing but a $\frU$-small set.
\item
A 1-class is nothing but a $\frU$-class (indeed, $\Ps \frU = \frU$ shows that $(\calP\Ps)\frU = \calP\frU$).
\item
A $k$-class is nothing but a subset of $\Ps\Cls^{k-1}_0$ for all $k \ge 1$.
\end{enumerate}
In the following, we call $\frU$-small sets and $\frU$-classes simply
small sets and 1-classes, respectively.
\end{rmk}

In this paper, all categories $\calC$ are assumed to be ``small'' categories
in the sense that $\calC_0, \calC(x,y) \in \SET$ for all $x, y \in \calC_0$,
where $\calC_0$ denotes the set of all objects in $\calC$.
For each object $x$ of $\calC$, the identity morphism at $x$ is denoted by
$\id_x$.
We now define $\frU$-small categories, which are simply called small categories below.

\begin{dfn}\label{dfn:small-light-moderate}
Let $\calC$ be a category, and $k \in \bbN$.
\begin{enumerate}
\item
$\calC$ is called a \emph{small category}
if 
$\calC_0$ and all local morphism sets $\calC(x,y)\ (x,y \in \calC_0)$ are small.
\item
$\calC$ is called a \emph{light category}
if
$\calC_0$ is a $1$-class, and all local morphism sets are small sets.
\item
$\calC$ is called a \emph{moderate category}
if 
$\calC_0$ and all local morphism sets are $1$-classes.
\item
More generally, $\calC$ is called a \emph{$k$-moderate category}
if
$\calC_0$ and all local morphism sets are $k$-classes.
$\calC$ is called a \emph{properly} $k$-moderate category
if it is $k$-moderate but not $(k-1)$-moderate.
\end{enumerate}
\end{dfn}

\begin{rmk}
A $0$-moderate category is nothing but a small category.
A $1$-moderate category is just a moderate category.
A small category is a light category, and a light category is a $1$-moderate category.
\end{rmk}


We next summarize essential facts on $2$-categories that will be used later.

\begin{dfn} A $2$-category $\bfC$ is a sequence of the following data:
\begin{itemize}
\item A set $\bfC_0\ (\in \SET)$ of objects, 

\item A family of categories $(\bfC(x,y))_{x,y\in\bfC_0}$,

\item A family of functors $\circ:=(\circ_{x,y,z}:\bfC(y,z)\circ\bfC(x,y)\to\bfC(x,z))_{x,y,z\in\bfC_0}$,

\item A family of functors $(\mu_x:\{*\} \to \bfC(x,x))_{x\in\bfC_0}$,
where $\{*\}$ is the category with only one object $*$ and only one morphism
$\id_*$, and we set $\id_x:= \mu_x(*)$.
\end{itemize}
These data are required to satisfy the following axioms:
\begin{itemize}
\item (Associativity) The following diagram is commutative for all $x,y,z\in\bfC_0$
$$
\xymatrix@C=4pc{
\bfC(z,w)\circ\bfC(y,z)\circ\bfC(x,y)&\bfC(y,w)\circ\bfC(x,y)\\
\bfC(z,w)\circ\bfC(x,z) & \bfC(x,w).
\ar^{\circ\times \id} "1,1"; "1,2"
\ar^{\circ} "2,1"; "2,2"
\ar_{\id\times\circ } "1,1"; "2,1"
\ar^{\circ} "1,2"; "2,2"
}
$$

\item (Unitality) The following diagram is commutative for all  $x,y\in\bfC_0$
$$
\xymatrix@C=4pc{
\{*\}\times\bfC(x,y)& & \bfC(x,y)\times \{*\}\\
& \bfC(x,y)\\
\bfC(y,y)\times\bfC(x,y)& & \bfC(x,y)\times\bfC(x,x).
\ar^{\prj_1} "1,1"; "2,2"
\ar_{\mu_y\times\bfC(x,y)} "1,1"; "3,1"
\ar_{\prj_2} "1,3"; "2,2"
\ar^{\bfC(x,y)\times\mu_x} "1,3"; "3,3"
\ar^{\circ} "3,1"; "2,2"
\ar^{\circ} "3,3"; "2,2"
}
$$

\end{itemize}
\end{dfn}

\begin{rmk}
Elements of $\bfC_0$ are called objects of $\bfC$, and
objects (resp.\ morphisms, compositions) of the category $\bfC(x,y)$
are called 1-morphisms (resp.\  2-morphisms, vertical compositions)
of $\bfC$ for all $x,y\in\bfC_0$.
We sometimes abbreviate $x \in \bfC$ for $x \in \bfC_0$ if there seems to be no risk of confusion, and do the same even when $\bfC$ is a usual category.
To distinguish the vertical composition from the horizontal composition,
we use the notation $\bullet$ for the former, and $\circ$ for the latter.
Sometimes $\circ$ is omitted.
\end{rmk}


\begin{dfn}\label{dfn:2cat-small-light-moderate}
Let $k \in \bbN$.
\begin{enumerate}
\item
The $2$-category of all small categories is denoted by $\Cat$.
\item
The $2$-category of all light categories is denoted by $\CAT$.
\item
The $2$-category of all moderate categories is denoted by $\uCAT$.
\item
The $2$-category of all $k$-moderate categories is denoted by $\uCAT^k$.
\end{enumerate}
\end{dfn}

When $\calC$ is a $2$-category,
if $x, y \in \calC_0$ and $f, g \in \calC(x,y)_0$, then
$\calC(x, y)_0$ (resp.\ $\calC(x, y)(f, g)$) is called
a \emph{local $1$-morphism set} (resp.\ \emph{local $2$-morphism set})
of $\calC$.

\begin{dfn}
Let $\calC$ be a $2$-category.
\begin{enumerate}
\item
$\calC$ is said to be \emph{small} if
$\calC_0$ and all of its local $r$-morphism sets are small for each $r = 1,2$.
\item
$\calC$ is said to be \emph{light} if
$\calC_0$ is a class, and all of its local $r$-morphism sets
are small for each $r = 1,2$.
\item
$\calC$ is said to be \emph{$k$-moderate} if
$\calC_0$ and all of its local $r$-morphism sets are
$k$-classes for each $r = 1,2$, namely if
$\calC_0$ is a $k$-class, and categories $\calC(x, y)$ are
$k$-moderate for all $x, y \in \calC_0$.
\end{enumerate}
\end{dfn}

We cite the following from \cite{Le18} (see \cite[Proposition A.4.2]{Asa-book} for the proof).

\begin{prp}
\label{prp:2-Cat}
The following hold.
\begin{enumerate}
\item
The $2$-category $\Cat$ is light;
\item
The $2$-category $\CAT$ is $2$-moderate; and
\item
The $2$-category $\uCAT^k$ is $(k+1)$-moderate for all $1 \le k \in \bbN$.
\end{enumerate}
\end{prp}


\begin{dfn}
\label{dfn:colax-fun}
Let $I$ be a small category and $\bfC$ a 2-category.
A {\em colax functor}
(or an {\em oplax} functor)
from $I$ to $\bfC$ is a triple
$(X, X_i, X_{b,a})$ of data:
\begin{itemize}
\item
a quiver morphism $X\colon I \to \bfC$, where $I$ and $\bfC$ are regarded as quivers
by forgetting additional data such as 2-morphisms or compositions;
\item
a family $(X_i)_{i\in I_0}$ of 2-morphisms $X_i\colon X(\id_i) \Rightarrow \id_{X(i)}$ in $\bfC$
indexed by $i\in I_0$; and
\item
a family $(X_{b,a})_{(b,a)}$ of 2-morphisms
$X_{b,a} \colon X(ba) \Rightarrow X(b)X(a)$
in $\bfC$ indexed by $(b,a) \in \com(I):=
\{(b,a)\in I_1 \times I_1 \mid ba \text{ is defined}\}$
\end{itemize}
satisfying the axioms:

\begin{enumerate}
\item[(a)]
For each $a\colon i \to j$ in $I$ the following are commutative:
$$
\vcenter{
\xymatrix{
X(a\id_i) \ar@{=>}[r]^(.43){X_{a,\id_i}} \ar@{=}[rd]& X(a)X(\id_i)
\ar@{=>}[d]^{X(a)X_i}\\
& X(a)\id_{X(i)}
}}
\qquad\text{and}\qquad
\vcenter{
\xymatrix{
X(\id_j a) \ar@{=>}[r]^(.43){X_{\id_j,a}} \ar@{=}[rd]& X(\id_j)X(a)
\ar@{=>}[d]^{X_jX(a)}\\
& \id_{X(j)}X(a)
}}\quad;\text{ and}
$$
\item[(b)]
For each $i \ya{a}j \ya{b} k \ya{c} l$ in $I$ the following is commutative:
$$
\xymatrix@C=3em{
X(cba) \ar@{=>}[r]^(.43){X_{c,ba}} \ar@{=>}[d]_{X_{cb,a}}& X(c)X(ba)
\ar@{=>}[d]^{X(c)X_{b,a}}\\
X(cb)X(a) \ar@{=>}[r]_(.45){X_{c,b}X(a)}& X(c)X(b)X(a).
}
$$
\end{enumerate}
\end{dfn}

\begin{dfn}\label{dfn:1-mor-Colax(I,C)}
Let $\bfC$ be a 2-category and $X = (X, X_i, X_{b,a})$, $X'= (X', X'_i, X'_{b,a})$
be colax functors from $I$ to $\bfC$.
A {\em $1$-morphism} (called a {\em left transformation}) from $X$ to $X'$
is a pair $(F, \ps)$ of data
\begin{itemize}
\item
a family $F:=(F(i))_{i\in I_0}$ of 1-morphisms $F(i)\colon X(i) \to X'(i)$
in $\bfC$
; and
\item
a family $\ps:=(\ps(a))_{a\in I_1}$ of 2-morphisms
$\ps(a)\colon X'(a)F(i) \Rightarrow F(j)X(a)$
$$
\xymatrix{
X(i) & X'(i)\\
X(j) & X'(j)
\ar_{X(a)} "1,1"; "2,1"
\ar^{X'(a)} "1,2"; "2,2"
\ar^{F(i)} "1,1"; "1,2"
\ar_{F(j)} "2,1"; "2,2"
\ar@{=>}_{\ps(a)} "1,2"; "2,1"
}
$$
in $\bfC$ indexed by $a\colon i \to j$ in $I_1$
\end{itemize}
satisfying the axioms
\begin{enumerate}
\item[(a)]
For each $i \in I_0$ the following is commutative:
$$
\vcenter{
\xymatrix{
X'(\id_i)F(i) & F(i)X(\id_i)\\
\id_{X'(i)}F(i) & F(i)\id_{X(i)}
\ar@{=>}^{\ps(\id_i)} "1,1"; "1,2"
\ar@{=} "2,1"; "2,2"
\ar@{=>}_{X'_iF(i)} "1,1"; "2,1"
\ar@{=>}^{F(i)X_i} "1,2"; "2,2"
}}\quad;\text{ and}
$$
\item[(b)]
For each $i \ya{a} j \ya{b} k$ in $I$ the following is commutative:
$$
\xymatrix@C=4pc{
X'(ba)F(i) & X'(b)X'(a)F(i) & X'(b)F(j)X(a)\\
F(k)X(ba) & & F(k)X(b)X(a).
\ar@{=>}^{X'_{b,a}F(i)} "1,1"; "1,2"
\ar@{=>}^{X'(b)\ps(a)} "1,2"; "1,3"
\ar@{=>}_{F(k)\,X_{b,a}} "2,1"; "2,3"
\ar@{=>}_{\ps(ba)} "1,1"; "2,1"
\ar@{=>}^{\ps(b)X(a)} "1,3"; "2,3"
}
$$
\end{enumerate}
A $1$-morphism $(F, \ps) \colon X \to X'$ is said to be
$I$-{\em equivariant} if $\ps(a)$ is a 2-isomorphism in $\bfC$
for all $a \in I_1$.
\end{dfn}

\begin{dfn}
\label{dfn:2-mor-Colax(I,C)}
Let $\bfC$ be a 2-category, $X = (X, X_i, X_{b,a})$, $X'= (X', X'_i, X'_{b,a})$
be colax functors from $I$ to $\bfC$, and
$(F, \ps)$, $(F', \ps')$ 1-morphisms from $X$ to $X'$.
A {\em $2$-morphism} from $(F, \ps)$ to $(F', \ps')$ is a
family $\ze= (\ze(i))_{i\in I_0}$ of 2-morphisms
$\ze(i)\colon F(i) \Rightarrow F'(i)$ in $\bfC$
indexed by $i \in I_0$
such that the following is commutative for all $a\colon i \to j$ in $I$:
$$
\xymatrix@C=4pc{
X'(a)F(i) & X'(a)F'(i)\\
F(j)X(a) & F'(j)X(a).
\ar@{=>}^{X'(a)\ze(i)} "1,1"; "1,2"
\ar@{=>}^{\ze(j)X(a)} "2,1"; "2,2"
\ar@{=>}_{\ps(a)} "1,1"; "2,1"
\ar@{=>}^{\ps'(a)} "1,2"; "2,2"
}
$$
\end{dfn}

\begin{dfn}
\label{dfn:comp-1-mors}
Let $\bfC$ be a 2-category, $X = (X, X_i, X_{b,a}), X'= (X', X'_i, X'_{b,a})$
and $X''= (X'', X_i'', X_{b,a}'')$ colax functors from $I$ to $\bfC$, and
let $(F, \ps)\colon X \to X'$, $(F', \ps')\colon X' \to X''$
be 1-morphisms.
Then the {\em composite} $(F', \ps')(F, \ps)$ of $(F, \ps)$ and
$(F', \ps')$ is a 1-morphism from $X$ to $X''$ defined by
$$
(F', \ps')(F, \ps):= (F'F, \ps'\circ\ps),
$$
where $F'F:=(F'(i)F(i))_{i\in I_0}$ and for each $a\colon i \to j$ in $I$,
$
(\ps'\circ\ps)(a):= F'(j)\ps(a)\bullet \ps'(a)F(i)
$
is the pasting of the diagram
$$
\xymatrix@C=4pc{
X(i) & X'(i) & X''(i)\\
X(j) & X'(j) & X''(j).
\ar_{X(a)} "1,1"; "2,1"
\ar_{X'(a)} "1,2"; "2,2"
\ar^{F(i)} "1,1"; "1,2"
\ar_{F(j)} "2,1"; "2,2"
\ar@{=>}_{\ps(a)} "1,2"; "2,1"
\ar^{X''(a)} "1,3"; "2,3"
\ar^{F'(i)} "1,2"; "1,3"
\ar_{F'(j)} "2,2"; "2,3"
\ar@{=>}_{\ps'(a)} "1,3"; "2,2"
}
$$
\end{dfn}

The following is straightforward to verify.

\begin{prp}
Let $\bfC$ be a $2$-category.
Then colax functors $I \to \bfC$,
$1$-morphisms between them, and $2$-morphisms between
$1$-morphisms $($defined above$)$ define a $2$-category,
which we denote by $\Colax(I, \bfC)$.
\end{prp}

\begin{ntn}\label{ntn:co-op}
Let $\bfC$ be a 2-category.
Then we denote by $\bfC^{\text{op}}$
(resp.\ $\bfC^{\text{co}}$) the 2-category
obtained from $\bfC$ by reversing the 1-morphisms
(resp.\ the 2-morphisms), and we set
$\bfC^{\text{coop}}:=(\bfC^{\text{co}})^{\text{op}}=(\bfC^{\text{op}})^{\text{co}}$.
\end{ntn}

\section{Enriched categories}

In this section we collect basic facts about enriched categories which will be needed later. Throughout this section, we fix a symmetric monoidal category $\mathbb{V}$ and work over $\mathbb{V}$.
Before starting our discussion we recall the definition of symmetric monoidal categories.

\begin{dfn}
\label{dfn:monoidal-cat}
(1) A {\em monoidal category} is a sequence of the data
\begin{itemize}
\item
a category $\bbV$,
\item
an object 1 of $\bbV$,
\item
a functor $\ox \colon \bbV \times \bbV \to \bbV$,
\item
a natural isomorphism
$a = (a_{A,B,C} \colon (A\ox B) \ox C \to A \ox (B \ox C))_{A,B,C \in \bbV_0}$
called the {\em associator},
\item
a natural isomorphism $\ell = (\ell_A \colon 1 \ox A \to A)_{A \in \bbV_0}$
called the {\em left unit isomorphism},
\item
a natural isomorphism $r = (r_A \colon A \ox 1 \to A)_{A \in \bbV_0}$
called the {\em right unit isomorphism}
\end{itemize}
that satisfies the following axioms:
\begin{enumerate}
\item[(a)]
For any $A, B, C, D \in \bbV_0$, the following is commutative:
\[
\begin{tikzcd}
\Nname{L2}((A \ox B) \ox C) \ox D&& \Nname{R2}(A \ox (B \ox C))\ox D \\
 \Nname{L1}(A\ox B) \ox (C \ox D) && \Nname{R1}A\ox ((B\ox C) \ox D)\\
&\Nname{C}A\ox (B \ox (C\ox D))
\Ar{L1}{C}{"a"'}
\Ar{R1}{C}{"1\ox a"}
\Ar{L2}{L1}{"a"'}
\Ar{R2}{R1}{"a"}
\Ar{L2}{R2}{"a\ox 1"}
\end{tikzcd};
\]
\item[(b)]
For any $A, B \in \bbV_0$, the following is commutative:
\[
\begin{tikzcd}
A \ox (B \ox 1) & A\ox B\\
(A \ox B) \ox 1
\Ar{2-1}{1-1}{"a"'}
\Ar{1-1}{1-2}{"\id_A\ox r"}
\Ar{2-1}{1-2}{"r"'}
\end{tikzcd}; \text{and}
\]
\item[(c)]
$\ell_1 = r_1 \colon 1\ox 1 \to 1$.
\end{enumerate}
According to \cite{Kel64}, it is known that both of the following diagrams automatically turn out to be commutative for all objects $A, B$ in a monoidal category $\bbV$:
\[
\begin{tikzcd}
1 \ox (A\ox B) & A\ox B\\
(1\ox A) \ox B
\Ar{1-1}{1-2}{"\ell"}
\Ar{2-1}{1-1}{"a"'}
\Ar{2-1}{1-2}{"\ell \ox \id_B"'}
\end{tikzcd}\ \text{and}\ 
\begin{tikzcd}
A \ox (1\ox B) & A\ox B\\
(A\ox 1) \ox B
\Ar{1-1}{1-2}{"\id_A \ox \ell"}
\Ar{2-1}{1-1}{"a"'}
\Ar{2-1}{1-2}{"r \ox \id_B"'}
\end{tikzcd}.
\]

(2) A {\em switching operation} on $\bbV$ is a natural isomorphism
$t= (t_{A,B} \colon A \ox B \to B \ox A)_{A,B \in \bbV_0}$
such that the following is commutative:
\[
\begin{tikzcd}
A\ox B & B \ox A\\
C\ox D & D\ox C
\Ar{1-1}{1-2}{"t_{A,B}"}
\Ar{1-1}{2-1}{"f\ox g"'}
\Ar{1-2}{2-2}{"g\ox f"}
\Ar{2-1}{2-2}{"t_{C,D}"'}
\end{tikzcd}
\]
for all morphisms $f\colon A \to C$ and $g\colon B \to D$ in $\bbV$.

(3) A monoidal category $\bbV$ with a switching operation $t$ is called a {\em symmetric monoidal category} if 
the following hold:
\begin{enumerate}
\item[(a)]
$t_{A,B} \circ t_{B,A} = \id_{B \ox A}$ for all $A, B \in \bbV_0$; and
\item[(b)]
For any $A, B, C \in \bbV_0$, the following is commutative:
\[
\begin{tikzcd}
& A \ox (B\ox C)\\
(B\ox C)\ox A && (A \ox B) \ox C\\
B\ox (C \ox A) && (B \ox A) \ox C\\
& B\ox (A\ox C)
\Ar{1-2}{2-1}{"t_{A, B\ox C}"'}
\Ar{2-3}{1-2}{"a_{A,B,C}"}
\Ar{2-1}{3-1}{"a_{B,C,A}"}
\Ar{2-3}{3-3}{"t_{A, B}\ox \id_C"}
\Ar{4-2}{3-1}{"\id_B\ox t_{A,C}"}
\Ar{3-3}{4-2}{"a_{B,A,C}"}
\end{tikzcd}.
\]
\end{enumerate}

\end{dfn}

\begin{exm}
\label{exm:s-mon-cat}
The following give examples of symmetric monoidal categories:
\begin{enumerate}
\item
$\bbV:= {\bf Cat}$, the category of small caetegories.
Here, $1$ is given by the category $\{*\}$ with only one object and one morphism, $\ox$ is given by the direct product of small categories and
$a, \ell, r, t$ are given as the canonical isomorphisms.
\item
$\bbV:= \Mod\k$, the category of $\k$-modules.
In this case, $1$ is given by $\k$, $\ox$ is given by
the tensor product over $\k$, and
$a, \ell, r, t$ are also given as the canonical isomorphisms.
\item
$\bbV:= \ChMod(\k)$, category of the (unbounded) chain complexes
(here we use cocomplexes) of $\k$-modules
and chain morphisms, i.e., degree-preserving morphisms commuting with the differentials.
In this case, $1$ is given by the complex $\k$ concentrated in degree 0,
for $A, B \in \bbV_0$, $A \ox B$ is given as the tensor chain complex over $\k$, 
and also $a, \ell, r, t$ are given as the canonical isomorphisms.
Note that for each $A \in \bbV_0$, the ``underlying set''
$\ChMod(\k)(\k, A)$ is the set of 0-cocycles $Z^0(A)$ of $A$.
\end{enumerate}
\end{exm}

\begin{dfn} A category $\calA$ {\em enriched over} $\mathbb{V}$, or simply a $\mathbb{V}$-{\em category}, consists of the following data:
\begin{itemize}
\item
a class of objects $\calA_0$;
\item
for two objects $x, y$ in $\calA$, an object $\calA(x, y)$ in $\mathbb{V}$;
\item
for three objects $x, y, z$ in $\calA$, a morphism
 $$
 \circ: \calA(y,z)\otimes \calA(x,y)\to  \calA(x,z)
 $$
in $\mathbb{V}$; and
\item
for an object $x$ in $\calA$, a morphism in $\mathbb{V}$
$$
\id_x: 1 \to \calA(x,x)
$$
\end{itemize}
satisfying the following conditions:
\begin{itemize}
\item[(1)] For any objects $x, y, z, w$, the following diagram is commutative:
$$\footnotesize
\begin{tikzcd}[column sep=10pt]
\Nname{L}(\calA(z,w)\otimes\calA(y,z))\otimes \calA(x,y)&&
\Nname{R}\calA(z,w)\otimes(\calA(y,z)\otimes \calA(x,y))\\
 \Nname{y}\calA(y,w)\otimes\calA(x,y)&&\Nname{z} \calA(z,w)\otimes\calA(x,z)\\
&\Nname{A}\calA(x,w)
\Ar{L}{R}{"a"}
\Ar{L}{y}{"\circ\times \id_{\calA(x,y)}"'}
\Ar{R}{z}{"\id_{\calA(z,w)} \times\circ"}
\Ar{y}{A}{"\circ"}
\Ar{z}{A}{"\circ"'}
\end{tikzcd}
; \text{and}
$$
\item[(2)] For any objects $x, y$, the following diagram is commutative:
$$
\begin{tikzcd}
\calA(y,y)\otimes\calA(x,y) &  \calA(x,y) & \calA(x,y)\otimes \calA(x,x)\\
1\otimes\calA(x,y)& &  \calA(x,y)\otimes 1
\Ar{1-1}{1-2}{"\circ"}
\Ar{1-3}{1-2}{"\circ"'}
\Ar{2-1}{1-2}{"{\ell_{\calA(x,y)}}"'}
\Ar{2-3}{1-2}{"{r_{\calA(x,y)}}"}
\Ar{2-1}{1-1}{"{\id_y \ox \id_{\calA(x,y)}}"}
\Ar{2-3}{1-3}{"{\id_{\calA(x,y)}\ox \id_x}"'}
\end{tikzcd}.
$$
\end{itemize}
\end{dfn}

\begin{dfn} Given $\mathbb{V}$-categories $\calA,\calB$, a {\em $\mathbb{V}$-functor} or an {\em enriched functor} $F:\calA\to\calB$
consists of the following data:
\begin{itemize}
\item
for each $x \in \calA_0$, an object $F(x)$ of $\calB$;

\item
for any $x, y\in \calA_0$, a morphism in $\mathbb{V}$,
$$
F_{x,y}: \calA(x,y)\to \calB(F(x),F(y))
$$
\end{itemize}
that satisfies the following axioms:
\begin{itemize}
 \item[(1)] For any $x,y,z \in \calA_0$, the following diagram is commutative:
$$
\vcenter{
\xymatrix@C=4pc{
\calA(y,z)\otimes \calA(x,y) &&    \calA(x,z)\\
 \calB(F(y),F(z))\otimes\calB(F(x),F(y))&&  \calB(F(x),F(z))
\ar^{\circ}"1,1";"1,3"
\ar^{\circ}"2,1";"2,3"
\ar_{F_{y,z}\times F_{x,y}}"1,1";"2,1"
\ar^{F_{x,z} }"1,3";"2,3"
}}; \text{and}
$$

\item[(2)] For each $x \in \calA_0$, the following diagram is commutative:
$$
\vcenter{
\xymatrix@C=4pc{
1 & \calA(x,x)\\
 &  \calA(F(x),F(x))
\ar^{1_x}"1,1";"1,2"
\ar^{F_{x,x}}"1,2";"2,2"
\ar_{1_{F(x)}} "1,1";"2,2"
}}.
$$
\end{itemize}
\end{dfn}

\begin{dfn}
Let $F, G \colon \calA \to \calB$ be $\bbV$-functors between $\bbV$-categories.
A {\em $\bbV$-natural transformation} $\al$ from $F$ to $G$, denoted by
$\al \colon F \To G$,
is a family $\al = (\al(x))_{x \in \calA_0}$ of morphisms $\al(x) \colon 1 \to \calB(F(x), G(x))$ in $\bbV$ making the following diagram commutative
for all $x,y \in \calA_0$:
\begin{equation}
\label{eq:V-nat}\footnotesize
\begin{tikzcd}
&\calA(x,y)\\
\calA(x,y) \ox 1 && 1\ox \calA(x,y)\\
\calB(G(x),G(y))\ox \calB(F(x),G(x)) &&\calB(F(y),G(y))\ox \calB(F(x), F(y))\\
&\calB(F(x),G(y))
\Ar{1-2}{2-1}{"r\inv"'}
\Ar{1-2}{2-3}{"\ell\inv"}
\Ar{2-1}{3-1}{"G\ox \al(x)"'}
\Ar{2-3}{3-3}{"\al(y)\ox F"}
\Ar{3-1}{4-2}{"\circ"'}
\Ar{3-3}{4-2}{"\circ"}
\end{tikzcd}.
\end{equation}
The composition of $\bbV$-natural transformations is defined in an obvious way.
\end{dfn}

\begin{dfn}
The $2$-category of {\em small} $\mathbb{V}$-categories, $\mathbb{V}$-functors, and $\bbV$-natural transformations is denoted by $\VCat$.
\end{dfn}

\begin{exm}\label{exm-dg}
The following are examples of $\bbV$-categories.
\begin{enumerate}
\item
In the case where $\mathbb{V} = {\bf Cat}$,
$\bbV$-categories are nothing but (strict) 2-categories. $\bbV$-functors are called 2-functors.
\item
In the case where $\mathbb{V} = \Mod \k$, $\mathbb{V}$-categories are nothing but 
$\k$-linear categories.
In this case, $\VCat$ is denoted by $\kCat$.
\item
In the case where $\mathbb{V} = \ChMod(\k)$, 
$\mathbb{V}$-categories are called dg (differential graded)
categories over $\k$.
In this case, $\VCat$ is denoted by $\DGkCat$.
In most cases we only deal with small dg categories, therefore
we sometimes omit the word ``small'' if there seems to be no
confusion.
\end{enumerate}
\end{exm}

\begin{dfn}
A dg category $\DGMod(\k)$ is defined as follows.
Objects are chain complexes of $\k$-modules, thus the same as $\ChMod(\k)$.
Let $X, Y$ be objects of $\DGMod(\k)$.
Then $\DGMod(\k)(X, Y)$ is a complex of $\k$-modules defined by
\[
\begin{aligned}
\DGMod(\k)(X, Y)&:= \Ds_{n \in \bbZ} \DGMod(\k)^n(X, Y), \text{where}\\
\DGMod(\k)^n(X, Y)&:= \prod_{p\in \bbZ}(\Mod \k)(X^p, Y^{p+n}),
\end{aligned}
\]
and with a differential
$d = (d^n \colon \DGMod(\k)^n(X, Y) \to \DGMod(\k)^{n+1}(X, Y))_{n\in \bbZ}$
defined by
\[
d^n(f):= (d_{Y}^{p+n}f^p - (-1)^{n} f^{p+1}d_X^p)_{p\in \bbZ}
\]
for all $f = (f^p)_{p\in \bbZ} \in \DGMod(\k)^n(X, Y)$.
\end{dfn}

\begin{dfn}
\label{dfn:0-cocyc}
Let $\calC$ be a dg category, $x, y \in \calC_0$, and take
$f = (f^i)_{i\in \bbZ} \in \calC(x,y) = \Ds_{i\in \bbZ}\calC^i(x,y)$.
If $f^i = 0$ for all $i \ne 0$ and $d_{\calC}(f) = 0$, then we call $f$ a {\em $0$-cocycle morphism}.
We identify each $0$-cocycle element $g \in Z^0(\calC(x,y))$ of $\calC(x,y)$ with the $0$-cocycle morphism $f \in \calC(x,y)$ defined by
$f^0:= g$ and $f^i = 0$ for all $i \ne 0$.
\end{dfn}

\begin{rmk}
We here recall the explicit form of compositions in a dg category.
Let $\calC$ be a dg category, $x, y, z \in \calC$,
and $f  = (f^i)_{i\in \bbZ} \in \calC(x,y) = \Ds_{i\in \bbZ}\calC^i(x,y)$,
$g = (g^j)_{j\in \bbZ} \in \calC(y, z) = \Ds_{j\in \bbZ}\calC^j(y,z)$.
Then we have the formula
\begin{equation}\label{eq:comp-dg}
g \circ f := \left(\sum_{i\in\bbZ} g^{n-i}\circ f^i \right)_{n \in \bbZ}.
\end{equation}
On the other hand, in the opposite category $\calC\op$ of
$\calC$ having the composition $*$, we have
$f \in \calC\op(y,x), g \in \calC\op(z,y)$, and
\begin{equation}\label{eq:opp-comp}
f * g = \left(\sum_{i\in\bbZ} (-1)^{(n-i)i}\, g^{n-i}\circ f^i \right)_{n \in \bbZ}.
\end{equation}
Note that the representable functor $\calC(\blank, z) = \calC\op(z, \blank)$
is a functor $\calC\op \to \DGMod(\k)$, and hence
$\calC(f, z) \colon \calC(y,z) \to \calC(x,z)$ is defined
as $\calC\op(z, f) \colon \calC\op(z,y) \to \calC\op(z,x)$ by
\[
\calC(f, z)(g):= \calC\op(z,f)(g):= f*g = \left(\sum_{i\in\bbZ} (-1)^{(n-i)i}\, g^{n-i}\circ f^i \right)_{n \in \bbZ}.
\]
\end{rmk}

\begin{rmk}
\label{rmk:dg-nat}
Consider the case that $\bbV= \ChMod(\k)$, and let
$F, G \colon \calA \to \calB$ be dg functors between dg categories.
Then a $\bbV$-natural transformation is called a {\em dg natural transformation}.
By definition, a dg natural transformation $\al \colon F \To G$ is a family
$\al = (\al(x))_{x \in \calA_0}$ of morphisms $\al(x) \colon \k \to \calB(F(x), G(x))$ in $\ChMod(\k)$ making the diagram \eqref{eq:V-nat} commutative.
We set $\al_x:= \al(x)(1_{\k})$, where $1_{\k}$ is the identitiy of $\k$,
and make the identification $\al = (\al_x)_{x\in \calA_0}$.
As in Exmaple \ref{exm:s-mon-cat} (3), $\al_x \in Z^0(\calB(F(x), G(x)))$
for all $x \in \calA_0$, and
the commutativity of \eqref{eq:V-nat} is equivalent to saying that the following is
commutative in $\calB$ for all morphisms $f \colon x \to y$ in $\calA$:
\[
\begin{tikzcd}
F(x) & F(y)\\
G(x) & G(y)
\Ar{1-1}{1-2}{"F(f)"}
\Ar{2-1}{2-2}{"G(f)"'}
\Ar{1-1}{2-1}{"\al_x"'}
\Ar{1-2}{2-2}{"\al_y"}
\end{tikzcd}.
\]
Here we have to remark that both $\al_x$ and $\al_y$ are $0$-cocycles in $\calB(F(x), G(x))$
and in $\calB(F(y), G(y))$, respectively.
Thus, we can set $F(f) = (F(f)^n)_{n\in \bbZ}$,
$G(f) = (G(f)^n)_{n\in \bbZ}$,
$\al_x = (\al_x)^0$, and $\al_y = (\al_y)^0$,
and the commutativity of the diagram above is
equivalent to
the equality
\[
\al_y F(f)^n = G(f)^n \al_x
\]
for all $n \in \bbZ$.
In particular, this is used in the case where $\calB = \DGMod(\k)$, the dg category of dg $\k$-modules, later.
In this case the 0-cocycles are the chain morphisms.
\end{rmk}

We do not use the following notion explicitly, but we introduce it here to make clear the relationship between our setting and other general settings.

\begin{dfn}
\label{dfn:fun-dg-cat}
Let $\calA$ and $\calB$ be dg categories.
Then the {\em functor dg category} $\sfHom(\calA, \calB)$ is defined as follows.
Objects are the dg functors $\calA \to \calB$.
Let $F, G \colon \calA \to \calB$ be two dg functors, and $n \in \bbZ$.
A  {\em derived transformation}
$\al^n \colon F \To G$ {\em of degree $n$} from $F$ to $G$ is a family
$\al^n = (\al^n_{x})_{x \in \calA_0}$ of morphisms $\al^n_{x} \in \calB(F(x), G(x))^n$ such that for any morphism $f \in \calA(x,y)^m$, $x,y \in \calA_0$, we have 
\begin{equation}
\label{eq:derived-natural}
\al^n_{y}F(f)=(-1)^{mn}G(f)\al^n_{x}.
\end{equation}
Then we denote by $\sfHom(\calA, \calB)^n(F, G)$
the set of all derived transformations of degree $n$ from $F$ to $G$, and set
\[
\sfHom(\calA, \calB)(F, G):= \Ds_{n \in \bbZ} \sfHom(\calA, \calB)^n(F, G),
\]
elements of which are called {\em derived transformations} from $F$ to $G$.
The differential $d$ is given by
$d(\al^n_x):= d_{\calB}(\al^n_x)$ for all
$\al \in \sfHom(\calA, \calB)(F,G)$, $n \in \bbZ$,
and $ x \in \calA_0$.
\end{dfn}

\begin{rmk}
\label{rmk:moderate-fun-dg-cat}
In Definition \ref{dfn:fun-dg-cat}, the category
$\sfHom(\calA, \calB)$ is \emph{small} if both $\calA$ and $\calB$ are small;
\emph{light} if $\calA$ is small and $\calB$ is light; and 
\emph{$k$-moderate} if $\calA$ is $(k-1)$-moderate and $\calB$ is $k$-moderate
for all $k \ge 1$.
\end{rmk}

\begin{dfn}
We denote by $\kDGCat$ the $2$-category whose objects are the small dg categories,
whose $1$-morphisms are the dg-functors between these objects,
and whose $2$-morphisms are the derived transformations bvetween these dg functors.
The vertical composition of $2$-morphisms is defined in an obvious way
by a formula similar to \eqref{eq:comp-dg}, and the horizontal composition of $2$-morphisms
is defined by
\[
\be^n \circ \al^m:= (\be^n \circ F) \bullet (E' \circ \al^m)
= (-1)^{mn} (F' \circ \al^m) \bullet (\be^n \circ E)
\]
for all $2$-morphisms $\al = (\al^n)_{n \in \bbZ}$ and $\be = (\be^n)_{n \in \bbZ}$ in a diagram
\[
\begin{tikzcd}
\calA & \calB & \calC
\Ar{1-1}{1-2}{"E", ""'{name=a}, bend left}
\Ar{1-1}{1-2}{"F"', ""{name=b}, bend right}
\Ar{a}{b}{"\al", Rightarrow}
\Ar{1-2}{1-3}{"E'", ""'{name=c}, bend left}
\Ar{1-2}{1-3}{"F'"', ""{name=e}, bend right}
\Ar{c}{e}{"\be", Rightarrow}
\end{tikzcd}
\]
in $\kDGCat$, where
$(E' \circ \al^m)_x:= E'(\al^m_x)$ and $(\be^n \circ E)_y:= \be^n_{E(y)}$
for all $x \in \calA_0, y \in \calB_0$.
Note the difference with $\DGkCat$.
Objects and $1$-morphisms are the same, but $2$-morphisms are different:
they are dg natural transformations in $\DGkCat$, and
derived transformations in $\kDGCat$.
Therefore, for left part of the diagram above,
we have
\begin{equation}
\label{eq:dg-DG}
\DGkCat(\calA, \calB)(E, F) = Z^0(\kDGCat(\calA, \calB)(E,F)).
\end{equation}
This means that for any derived transformation $\al \colon E \to F$,
$\al$ becomes a dg natural transformation if and only if
$\al$ is a 0-cocycle morphism.
We note that both $\DGkCat$ and $\kDGCat$ are light $2$-categories.
\end{dfn}

\begin{dfn}
By replacing small dg categories with light dg categories in the definitions of
$\kdgCat$ and $\kDGCat$, we define
the $2$-categories $\kdgCAT$ and $\kDGCAT$, respectively.
We remark that these are $2$-moderate $2$-categories.
\end{dfn}

\section{$I$-coverings}

In this section we introduce the notion of $I$-coverings that is a generalization of
that of $G$-coverings for a group $G$ introduced in \cite{Asa11}, which was
obtained by generalizing the notion of Galois coverings
introduced by Gabriel in \cite{Gab}.
This will be used in the proof of our main theorem.

In the following, we will consider $I$-coverings 
in $\DGkCat$, i.e., in the case that $\mathbb{V} = \ChMod(\k)$.
The precise form in this case is described as follows.

\begin{dfn}
We define a 2-functor $\De\colon \DGkCat \to \Colax(I, \DGkCat)$ as follows,
which is called the {\em diagonal} 2-functor:
\begin{itemize}
\item
Let $\calC \in \DGkCat$. Then $\De(\calC)$ is defined to be a functor
sending each morphism $a\colon i \to j$ in $I$ to
$\id_{\calC}\colon \calC \to \calC$.
\item
Let $E \colon \calC \to \calC'$ be a $1$-morphism in $\DGkCat$.
Then $\De(E)\colon \De(\calC) \to \De(\calC')$ is a $1$-morphism
$(F,\ps)$ in $\Colax(I, \DGkCat)$
defined by $F(i):=E$ and $\ps(a):= \id_E$ for all $i \in I_0$ and all $a \in I_1$:
$$
\xymatrix{
\calC & \calC'\\
\calC & \calC'.
\ar^E "1,1"; "1,2"
\ar^E "2,1"; "2,2"
\ar_{\id_{\calC}} "1,1"; "2,1"
\ar^{\id_{\calC'}} "1,2"; "2,2"
\ar@{=>}_{\id_E}"1,2";"2,1"
}
$$
\item
Let $E, E'\colon \calC \to \calC'$ be 1-morphisms (that is, dg functors) in $\DGkCat$, and
$\al \colon  E \To E'$ a 2-morphism in $\DGkCat$.
Then $\De(\al)\colon \De(E) \To \De(E')$ is a 2-morphism in
$\Colax(I, \DGkCat)$ defined by $\De(\al):= (\al)_{i\in I_0}$.
\end{itemize}
\end{dfn}

\begin{rmk}\label{rmk:1-mor-to-De}
Let $\bfC$ be a 2-category,
$X=(X, X_i, X_{b,a}) \in \Colax(I, \bfC)_0$, and $C \in \bfC_0$.
Further let
\begin{itemize}
\item $F$ be a family of  1-morphisms $F(i)\colon X(i) \to C$ in $\bfC$
indexed by $i\in I_0$; and
\item $\ps$ be a family of 2-morphisms $\ps(a)\colon F(i) \Rightarrow F(j)X(a)$
indexed by $a\colon i \to j$ in $I$:
$$
\xymatrix{
X(i) & C\\
X(j) & C
\ar^{F(i)} "1,1"; "1,2"
\ar_{F(j)} "2,1"; "2,2"
\ar_{X(a)} "1,1";"2,1"
\ar@{=} "1,2";"2,2"
\ar@{=>}_{\ps(a)} "1,2";"2,1"
}
$$
\end{itemize}
Then $(F, \ps)$ is in $\Colax(I, \bfC)(X, \De(C))$
if and only if the following hold.
\begin{enumerate}
\item[(a)]
For each $i \in I_0$ the following is commutative:
$$
\vcenter{
\xymatrix{
F(i) & F(i)X(\id_i)\\
 & F(i)\id_{X(i)}
\ar@{=>}^{\ps(\id_i)} "1,1"; "1,2"
\ar@{=} "1,1"; "2,2"
\ar@{=>}^{F(i)X_i} "1,2"; "2,2"
}}\quad;\text{ and}
$$
\item[(b)]
For each $i \ya{a} j \ya{b} k$ in $I$ the following is commutative:
$$
\xymatrix@C=4pc{
F(i) &  F(j)X(a)\\
F(k)X(ba) &  F(k)X(b)X(a).
\ar@{=>}^{\ps(a)} "1,1"; "1,2"
\ar@{=>}_{F(k)\,X_{b,a}} "2,1"; "2,2"
\ar@{=>}_{\ps(ba)} "1,1"; "2,1"
\ar@{=>}^{\ps(b)X(a)} "1,2"; "2,2"
}
$$
\end{enumerate}
\end{rmk}

\begin{dfn}
Let $\calC \in \DGkCat$ and $(F, \ps) \colon X \to \De(\calC) $ be in $\Colax(I, \DGkCat)$.
Then
\begin{enumerate}
\item
$(F, \ps)$ is called an $I$-{\em precovering} (of $\calC$) if 
for any $i,j \in I_0$, $x \in X(i), y \in X(j)$,
the morphism
$$
(F,\ps)_{x,y}^{(1)}\colon \Ds_{a\in I(i,j)}X(j)(X(a)x, y) \to \calC(F(i)x, F(j)y)
$$
of $\k$-complexes 
defined by the following is an isomorphism:
$$\begin{aligned}
\Ds_{a\in I(i,j)}X(j)(X(a)x, y)
&\ya{\Ds_{a\in I(i,j)}F(j)} \Ds_{a\in I(i,j)}\calC(F(j)X(a)x, F(j)y)
\\
&\ya{\Ds_{a\in I(i,j)}\calC(\ps(a)_x, F(j)y)}\Ds_{a\in I(i,j)}
\calC(F(i)x, F(j)y)\\
&\ya{\text{summation}}
\calC(F(i)x, F(j)y),
\end{aligned}
$$
the precise form of which is given as follows:
\begin{equation}
\label{eq:1-covering}
\begin{aligned}
(F,\ps)^{(1)}_{x,y}(((f^n_a)_{n\in \mathbb{Z}})_{a\in I(i,j)})
&=
\sum_{a\in I(i,j)} \ps(a)_x * F(j)(f_a)\\
&=
\left(\sum_{a\in I(i,j)}\sum_{r \in \bbZ} (-1)^{(n-r)r} F(j)(f_a)^{n-r} \circ \ps(a)_x^r\right)_{n \in \bbZ}\\
&=\left(\sum_{a\in I(i,j)}F(j)(f_a)^{n} \circ \ps(a)_x\right)_{n \in \bbZ},
\end{aligned}
\end{equation}
where the second term is computed by using \eqref{eq:opp-comp}, and
the last term uses the fact that $\ps(a)_x$ is concentrated in degree 0 (Remark \ref{rmk:dg-nat}).
\item
$(F, \ps)$ is called an $I$-{\em covering} if it is an $I$-precovering and is {\em dense},
i.e., for each $c \in \calC_0$ there exists an $i \in I_0$ and $x \in X(i)_0$
such that $F(i)(x)$ is isomorphic to $c$ in $\calC$.
\end{enumerate}
\end{dfn}

\section{Grothendieck constructions}

In this section we define a 2-functor $\Gr_I\colon \Colax(I, \VCat) \to \VCat$ whose
correspondence on objects is a $\mathbb{V}$-enriched version of (the opposite version of)
the original Grothendieck construction (cf. \cite{Tam}).
In particular, we deal with the case of $\DGkCat$ later.

\begin{dfn}
We define a $2$-functor $\Gr_I\colon \Colax(I, \VCat) \to \VCat$,
which is called the {\em Grothendieck construction}.

{\bf On objects.}  Let $X=(X(i), X_i, X_{b,a}) \in \Colax(I, \VCat)_0$.
Then $\Gr_I X \in \VCat_0$ is defined as follows.
\begin{itemize}
\item  $(\Gr_I X)_0:= \bigcup_{i\in I_0} \{ i \} \times X(i)_0
= \{{}_ix:= (i,x) \mid i \in I_0, x \in X(i)_0\}$.
\item  For each ${}_ix, {}_jy \in (\Gr_I X)_0$, we set
$$
(\Gr_I X)({}_ix, {}_jy) := \bigoplus_{a\in I(i,j)} X(j)(X(a)x, y).
$$
\item  For any ${}_ix, {}_jy, {}_kz \in (\Gr_I X)_0$ and
each $f=(f_a)_{a\in I(i,j)}\in (\Gr_I X)({}_ix, {}_jy)$,
$g=(g_b)_{b\in I(j,k)}\in (\Gr_I X)({}_jy, {}_kz)$, we set
$$
g\circ f:= \left(\sum_{\begin{smallmatrix}a\, \in\, I(i,j)\\b\, \in\, I(j,k)\\c\, =\, ba\end{smallmatrix}}
g_b\circ X(b)f_a
\circ X_{b,a}x \right)_{c\,\in\, I(i,k)},
$$
which is the composite of the following:
\begin{equation}\footnotesize
\label{eq:comp-Gr}
\begin{tikzcd}
(\Gr_I X)({}_jy, {}_kz) \times (\Gr_I X)({}_ix, {}_jy) & (\Gr_I X)({}_ix, {}_kz)\\
\Ds_{b\in I(j,k)}X(k)(X(b)y, z) \times \Ds_{a\in I(i,j)}X(j)(X(a)x, y)&
\Ds_{c\in I(i,k)}X(k)(X(c)x, z)\\
\Ds_{b,a}\{X(k)(X(b)y, z) \times X(j)(X(a)x, y)\}&
\Ds_{b,a}X(k)(X(ba)x, z)\\
\Ds_{b,a}\{X(k)(X(b)y, z) \times X(j)(X(b)X(a)x, X(b)y)\}&
\Ds_{b,a}X(k)(X(b)X(a)x, z),
\Ar{1-1}{1-2}{dashed}
\Ar{1-1}{2-1}{equal}
\Ar{2-1}{3-1}{equal}
\Ar{3-1}{4-1}{"\Ds_{b,a}(\id \times X(b))"'}
\Ar{4-1}{4-2}{}
\Ar{4-2}{3-2}{"{\Ds_{b,a}X(k)(X_{b,a}x,\, z)}"'}
\Ar{3-2}{2-2}{"\text{summation}"'}
\Ar{2-2}{1-2}{equal}
\end{tikzcd}
\end{equation}
where elements are mapped as follows:
\[
\begin{tikzcd}
((g_b)_b, (f_a)_a) & \left(\sum_{c=ba} g_b \circ X(b)f_a \circ X_{b,a}x\right)_{c}\\
(g_b, f_a)_{b,a} & (g_b \circ X(b)f_a \circ X_{b,a}x)_{b,a}\\
(g_b, X(b)f_a)_{b,a} & (g_b \circ X(b)f_a)_{b,a}.
\Ar{1-1}{1-2}{mapsto, dashed}
\Ar{1-1}{2-1}{mapsto}
\Ar{2-1}{3-1}{mapsto}
\Ar{3-1}{3-2}{mapsto}
\Ar{3-2}{2-2}{mapsto}
\Ar{2-2}{1-2}{mapsto}
\end{tikzcd}
\]

Note here that the composition with $X_{b,a}x$ is ``contravariant'',
which is used in \eqref{eq:comp-dg-Gr}.

\item For each ${}_ix \in (\Gr_I X)_0$ the identity $\id_{{}_ix}$ is given by
$$
\id_{{}_ix} = (\de_{a,\id_i}X_i\,x)_{a\in I(i,i)} \in \Ds_{a\in I(i,i)}X(i)(X(a)x,x),
$$
where $\de$ is the Kronecker delta\footnote{
This is used to mean that the $a$-th component is $\et_i\,x$ if $a=\id_i$,
or 0 otherwise.
}.
\end{itemize}

{\bf On 1-morphisms.}
Let $X=(X, X_i, X_{b,a})$ and $X'=(X', X_i', X_{b,a}')$ be objects of $\Colax(I, \VCat)$, and let
$(F, \ps)\colon X \to X'$ be a 1-morphism in $\Colax(I, \VCat)$.
Then a 1-morphism
$$
\Gr_I(F, \ps) \colon \Gr_I X \to \Gr_I X'
$$
in $\VCat$ is defined as follows.
\begin{itemize}
\item For each ${}_ix \in (\Gr_I X)_0$, $\Gr_I(F, \ps)({}_ix):={}_i(F(i)x)$.
\item
Let ${}_ix, {}_jy \in (\Gr_I X)_0$.
Then we define
\[
\Gr_I(F,\ps) \colon (\Gr_I X)({}_ix, {}_jy) \to (\Gr_I X')({}_i(F(i)x), {}_j(F(j)y))
\]
as the composite
\begin{equation}
\label{eq:Gr-mor}
\begin{aligned}
\Ds_{a\in I(i,j)}X(j)(X(a)x, y)
&\ya{\Ds_{a\in I(i,j)}F(j)} \Ds_{a\in I(i,j)}X'(j)(F(j)X(a)x, F(j)y)
\\
&\ya{\Ds_{a\in I(i,j)}X'(j)(\ps(a)_x, F(j)y)}\Ds_{a\in I(i,j)}
X'(j)(X'(a)F(i)x, F(j)y).
\end{aligned}
\end{equation}
Namely, for each $f=(f_a)_{a\in I(i,j)} \in (\Gr_I X)({}_ix, {}_jy)$,
we set
$$
\Gr_I(F,\ps)(f):= (F(j)f_a\circ \ps(a)x)_{a\in I(i,j)}.
$$
\end{itemize}

{\bf On 2-morphisms.}
Let $X=(X, X_i, X_{b,a})$ and $X'=(X', X_i', X_{b,a}')$ be objects of $\Colax(I, \VCat)$,
$(F, \ps)\colon X \to X'$ and $(F', \ps')\colon X' \to X''$  1-morphisms in $\Colax(I, \VCat)$,
and let $\ze\colon (F,\ps) \Rightarrow (F', \ps')$ be a 2-morphism in $\Colax(I, \VCat)$.
Then a 2-morphism
$$
\Gr_I \ze \colon \Gr_I(F,\ps) \Rightarrow \Gr_I(F', \ps')
$$
in $\VCat$ is defined by
$$
(\Gr_I \ze){}_ix :=
\begin{cases}
\ze(i)_x \circ X'_i(F(i)x)& \text{\ if\ } a=\id_i \\
0& \text{\ if\ } a\neq \id_i
\end{cases}
$$
in $\Gr_I X'$ for each ${}_ix \in (\Gr_I X)_0$.
\end{dfn}

In particular, in the case that $\mathbb{V} = \dgMod(\k)$, i.e., that
$\VCat=\kdgCat$,
the precise form of the
Grothendieck construction
$$\Gr_I\colon \Colax(I, \kdgCat) \to \kdgCat
$$ 
is described as follows.

{\bf On objects.}  Let $X=(X, X_i, X_{b,a}) \in \Colax(I, \DGkCat)_0$.
Then $\Gr_I X \in \DGkCat_0$ is defined as follows.
\begin{itemize}
\item  $(\Gr_I X)_0:= \bigcup_{i\in I_0} \{ i \} \times X(i)_0
= \{{}_ix:= (i,x) \mid i \in I_0, x \in X(i)_0\}$.
\item  For each ${}_ix, {}_jy \in (\Gr_I X)_0$, we set
$$
(\Gr_I X)({}_ix, {}_jy) := \bigoplus_{a\in I(i,j)} X(j)(X(a)x, y)= \bigoplus_{a\in I(i,j)} \bigoplus_{n\in \mathbb{Z}}X(j)^n(X(a)x, y),
$$ where note that $X(j)(X(a)x, y)$ is a dg $\k$-module.

\item  For any ${}_ix, {}_jy, {}_kz \in (\Gr_I X)_0$ and
each $f=(f^p_a)_{a\in I(i,j),p\in \mathbb{Z}}\in (\Gr_I X)({}_ix, {}_jy)$,
$g=(g^q_b)_{b\in I(j,k),q\in \mathbb{Z}}\in (\Gr_I X)({}_jy, {}_kz)$, it turns out that
\begin{equation}\label{eq:comp-dg-Gr}
\begin{aligned}
g\circ f&=
\left(\sum_{\begin{smallmatrix}a\, \in\, I(i,j)\\b\, \in\, I(j,k)\\c\, =\, ba\end{smallmatrix}}
X_{b,a}x * (g_b\circ (X(b)f_a))\right)_{c\,\in\, I(i,k),n\in\bbZ}\\
&=
\left(\sum_{\begin{smallmatrix}a\, \in\, I(i,j)\\b\, \in\, I(j,k)\\c\, =\, ba\end{smallmatrix}}
\sum_{p,r\in\mathbb{Z}}
(-1)^{(n-r)r}g^{n-r-p}_b\circ (X(b)f_a)^{p}
\circ (X_{b,a}x)^r \right)_{c\,\in\, I(i,k),n\in\mathbb{Z}}
\end{aligned}
\end{equation}
because of the contravariant part in \eqref{eq:comp-Gr}.

\item For each ${}_ix \in (\Gr_I X)_0$ the identity $\id_{{}_ix}$ is given by
$$
\id_{{}_ix} = (\de_{a,\id_i}X_i\,x)_{a\in I(i,i)} \in \Ds_{a\in I(i,i)}X(i)(X(a)x,x)=\bigoplus_{a\in I(i,j)} \bigoplus_{p\in \mathbb{Z}}X(i)^p(X(a)x, x).
$$
\end{itemize}

{\bf On 1-morphisms.}
Let $X=(X, X_i, X_{b,a})$ and $X'=(X', X_i', X_{b,a}')$ be objects of $\Colax(I, \kdgCat)$, and let
$(F, \ps)\colon X \to X'$ be a 1-morphism in $\Colax(I, \kdgCat)$.
Then a 1-morphism
$$
\Gr_I(F, \ps) \colon \Gr_I X \to \Gr_I X'
$$
in $\DGkCat$ is defined as follows.
\begin{itemize}
\item
For each ${}_ix \in (\Gr_I X)_0$, $\Gr_I(F, \ps)({}_ix):={}_i(F(i)x)$.
\item
Let ${}_ix, {}_jy \in (\Gr_I X)_0$.
Then we define
\[
\Gr_I(F,\ps) \colon (\Gr_I X)({}_ix, {}_jy) \to (\Gr_I X')({}_i(F(i)x), {}_j(F(j)y))
\]
as in \eqref{eq:Gr-mor}.
Namely, for each $f=((f^n_a)_{n\in \mathbb{Z}})_{a\in I(i,j)} \in (\Gr_I X)({}_ix, {}_jy)=\Ds_{a\in I(i,j)}X(j)\linebreak[3] (X(a)x, y)$, we have
\[
\begin{aligned}
((f^n_a)_{n\in \mathbb{Z}})_{a\in I(i,j)} &\mapsto ((F(j)(f^n_a))_{n\in \mathbb{Z}})_{a\in I(i,j)}\\
&\mapsto 
\ps(a)_x * ((F(j))(f^n_a))_{n\in \mathbb{Z}})_{a\in I(i,j)}
\\
&=
(( F(j)(f_a)^{n} \circ\ps(a)_x)_{n\in\mathbb{Z}})_{a\in I(i,j)}
\quad \text{(cf.\ \eqref{eq:opp-comp})}
\end{aligned}
\]
Thus we have
\begin{equation}
\label{eq:comp-Gr1}
\Gr_I(F,\ps)(f) =
(( F(j)(f_a)^{n} \circ\ps(a)_x)_{n\in\mathbb{Z}})_{a\in I(i,j)}.
\end{equation}
\end{itemize}

{\bf On 2-morphisms.}
Let $X=(X, X_i, X_{b,a})$ and $X'=(X', X_i', X_{b,a}')$ be objects of $\Colax(I, \kdgCat)$,
$(F, \ps)\colon X \to X'$ and $(F', \ps')\colon X \to X'$ 1-morphisms in $\Colax(I, \kdgCat)$,
and let $\ze\colon (F,\ps) \Rightarrow (F', \ps')$ a 2-morphism in $\Colax(I, \kdgCat)$.
Then a 2-morphism
$$
\Gr_I \ze \colon \Gr_I(F,\ps) \Rightarrow \Gr_I(F', \ps')
$$
in $\kdgCat$ is defined by
$$
(\Gr_I \ze)({}_ix) =
\begin{cases}
\ze(i)_x \circ X'_i(F(i)x)
= (\sum_{r \in \bbZ} 
\ze(i)^{n-r}_x \circ X'_i(F(i)x)^r)_{n\in \bbZ}
& \text{\ if\ } a=\id_i \\
0& \text{\ if\ } a\neq \id_i
\end{cases}
$$
in $\Gr_I X'$ for each ${}_ix \in (\Gr_I X)_0$.

\begin{exm}\label{exm:Gr}
Let $A$ be a dg $\k$-algebra with the differential $d_A$
regarded as a dg $\k$-category with a single object.
Then $A \in \kdgCat_0$.
Consider the functor $X:= \De(A) \colon I \to \kdgCat$.
Then it is straightforward to verify the following.
\begin{enumerate}
\item If $I$ is a free category defined by the quiver $1 \to 2$,
then $\Gr_I X$ is isomorphic to the triangular dg algebra
$\bmat{A&0\\A&A}$.

\item More generally, if $I$ is a free category $\bbP Q$ defined by a quiver $Q$,
then $\Gr_I X$ is isomorphic to the dg path-category $AQ$ of $Q$ over $A$
defined as follows:
\begin{itemize}
\item $(AQ)_0:= Q_0$.

\item For any $i,j \in Q_0$, 
\[
AQ(i,j):=\Ds_{\mu\in \mathbb{P}Q(i,j)}A\mu=\left\{\sum_{\mu\in \mathbb{P}Q(i,j)}a_{\mu}\mu\, \right|\, \left.(a_{\mu})_{\mu \in \bbP Q(i,j)}\in \Ds_{\mu \in \bbP Q(i,j)}A \right\}.
\]

\item For any $i,j,k \in Q_0$, the composition 
$AQ(j,k) \times AQ(i,j)\to AQ(i,k)$ is given by 
$$\sum_{\nu\in  \bbP Q(j,k)}b_{\nu}\nu \times
\sum_{\mu\in  \bbP Q(i,j)}a_{\mu}\mu \mapsto \sum_{\smat{\mu\in\bbP Q(i,j),\\\nu\in  \bbP Q(j,k)}}b_{\nu}a_{\mu}\nu\mu=\sum_{\la\in  \bbP Q(i,k)}\left(\sum_{\la=\nu\mu}b_{\nu}a_{\mu}\right)\la.
$$
\item
For any $i,j \in Q_0$ and any $n \in \bbZ$, we set $(AQ)^n(i,j)=\Ds_{\mu\in  \bbP Q(i,j)}A^n\mu$.

\item For any $i,j \in Q_0$ and any $n \in \bbZ$,
the differential $d:(AQ)^n(i,j)\to (AQ)^{n+1}(i,j)$ is given by 
\[d\left(\sum_{\mu\in  \bbP Q(i,j)}a_{\mu}\mu\right)=\sum_{\mu\in  \bbP Q(i,j)}d_A(a_\mu)\mu,\]
which automatically satisfies the graded Leibniz rule.
\end{itemize}
Indeed, we can define an isomorphism $\ph \colon AQ \to \Gr_I X$ as follows:
We regard $A$ as a category with a single object $*$.
Then for each $i \in Q_0$, we have $X(i)_0 = \{*\}$ and $X(i)_1 = A$.
Then $(\Gr_I X)_0 = \bigsqcup_{i \in Q_0}X(i)_0 = \bigcup_{i \in Q_0}\{{}_i*\} = \{{}_i* \mid i \in Q_0\}$.
Therefore, we define a bijection $\ph_0 \colon (AQ)_0 \to (\Gr_I X)_0$ by $i \mapsto {}_i*$.
For any $i,j \in Q_0$,
since we have $(AQ)(i,j) = \bigoplus_{\mu\in \bbP Q(i,j)} A\mu$,
and
$$
(\Gr_I X)({}_i*, {}_j*) := \bigoplus_{\mu\in I(i,j)} X(j)(X(\mu)*, *)
=\bigoplus_{\mu\in I(i,j)} X(j)_1 
= \bigoplus_{\mu\in \bbP Q(i,j)} A,
$$
we define a bijection $\ph_1 \colon (AQ)(i,j) \to (\Gr_I X)({}_i*, {}_j*)$ by
$\sum_{\mu\in \bbP Q} a_{\mu}\mu \mapsto (a_{\mu})_{\mu\in \bbP Q}$.
Then $\ph:= (\ph_0, \ph_1) \colon AQ \to \Gr_I X$ turns out to be an isomorphism.

\item If $I$ is a poset $S$,
then $\Gr_I X$ is isomorphic to the incidence dg category $AS$ of $S$ over $A$
defined as follows:
\begin{itemize}
\item
$(AS)_0:= S$ as a set.

\item
For any $i,j \in S$, $(AS)(i,j):= \begin{cases}
A & \text{if $i \le j$},\\
0 & \text{otherwise}.
\end{cases}$

\item
For any $i,j,k \in S$, the composition $AS(j,k) \times AS(i,j) \to AS(i,k)$ 
is given by the multiplication of $A$ for the case that $i \le j \le k$,
and as zero otherwise.

\item
For any $i,j \in Q_0$ and any $n \in \bbZ$, we set
$(AS)^n(i,j):= \begin{cases}
A^n & \text{if $i \le j$},\\
0 & \text{otherwise}.
\end{cases}$

\item
For any $i,j \in Q_0$ and any $n \in \bbZ$,
the differential $d:(AS)^n(i,j)\to (AS)^{n+1}(i,j)$ is given by 
$d_A \colon A^n \to A^{n+1}$ if $i \le j$, and as zero otherwise,
which automatically satisfies the graded Leibniz rule.
\end{itemize}
Indeed, we can define an isomorphism $\ph \colon AS \to \Gr_I X$ as follows:
We regard $A$ as a category with a single object $*$.
Then for each $i \in S$, we have $X(i)_0 = \{*\}$ and $X(i)_1 = A$.
Then $(\Gr_I X)_0 = \bigsqcup_{i \in I_0}X(i)_0 = \bigcup_{i \in I_0}\{{}_i*\} = \{{}_i* \mid i \in S\}$.
Therefore, we define a bijection $\ph_0 \colon (AS)_0 \to (\Gr_I X)_0$ by $i \mapsto {}_i*$.
For any $i,j \in S$,
since we have $(AS)(i,j) = \begin{cases}
A & \text{if $i \le j$},\\
0 & \text{otherwise}
\end{cases}$,
and
$$
\begin{aligned}
(\Gr_I X)({}_i*, {}_j*) &:= \bigoplus_{\mu\in S(i,j)} X(j)(X(\mu)*, *)
=\bigoplus_{\mu\in S(i,j)} X(j)_1 \\
&= \bigoplus_{\mu\in S(i,j)} A=A, \text{if $i \le j$},
\end{aligned}
$$
we define a bijection $\ph_1 \colon (AS)(i,j) \to (\Gr_I X)({}_i*, {}_j*)$ by
$\sum_{\mu\in S} a_{\mu}\mu \mapsto (a_{\mu})_{\mu\in S}$.
Then $\ph:= (\ph_0, \ph_1) \colon AQ \to \Gr_I X$ turns out to be an isomorphism.

\item If $I$ is a monoid $G$,
then $\Gr_I X$ is isomorphic to the monoid dg algebra\footnote{
Since $AG$ has the identity $1_A1_G$,
this is regarded as a category with a single object.
} $AG$ of $G$ over $A$ defined as follows:
\begin{itemize}
\item
$AG:= \Ds_{g\in G}Ag$.

\item
The multiplication $AG \times AG \to AG$ is defined by
\[
\left(\sum_{g\in G} a_g g\right)\cdot \left(\sum_{h\in G} b_h h\right):=
\sum_{g,h \in G} (a_g b_h) gh = \sum_{f \in G}\left(\sum_{gh=f} a_g b_h \right)f.
\]

\item
For each $n \in \bbZ$, $(AG)^n:= \Ds_{g\in G}A^n g$.

\item
The differential $d:(AG)^n\to (AG)^{n+1}$ is given by 
$d\left(\sum_{g\in G} a_g g\right):= \sum_{g\in G} d_A(a_g) g$,
which automatically satisfies the graded Leibniz rule.

\end{itemize}
\end{enumerate}

In (3) above, $AS$ is defined to be the factor category
of the dg path-category $AQ$ modulo the ideal
generated by the full commutativity relations in $Q$,
where $Q$ is the Hasse diagram of $S$ regarded as a quiver by
drawing an arrow $x \to y$ if $x \le y$ in $Q$.
If $S$ is a finite poset, then $AS$ is identified with the usual incidence dg algebra.

See \cite{Asa-Kim} for further examples of the Grothendieck constructions
of functors, further examples of the Grothendieck constructions
of a functor $X \colon I \to \kdgCat$ will be done in a forthcoming paper
\cite{Asa-P3}.
\end{exm}

\begin{dfn}
Let $X \in \Colax(I, \VCat)$.
We define a left transformation $(P_X, \ph_X):= (P, \ph)\colon X \to \De(\Gr_I X)$
(called the {\em canonical morphism}) as follows.
\begin{itemize}
\item For each $i \in I_0$, the functor $P(i)\colon X(i) \to \Gr X$ is  defined by
$$
\left\{
\begin{aligned}
P(i)x&:= {}_ix\\
P(i)f &:=(\de_{a,\id_i} f\circ (X_i\,x))_{a\in I(i,i)}\colon {}_ix \to {}_iy\text{\   in\ $\Gr_I X$}
\end{aligned}
\right.
$$
for all $f\colon x \to y$ in $X(i)$.
\item For each $a \colon i \to j$ in $I$, the natural transformation
$\ph(a)\colon P(i) \Rightarrow P(j)X(a)$
$$\xymatrix{
X(i) & \Gr_I X\\
X(j) & \Gr_I X
\ar^{P(i)} "1,1";"1,2"
\ar_{P(j)} "2,1";"2,2"
\ar_{X(a)} "1,1";"2,1"
\ar@{=} "1,2";"2,2"
\ar@{=>}_{\ph(a)}"1,2";"2,1"
}
$$
is defined by $\ph(a)x:= (\de_{b,a} \id_{X(a)x})_{b \in I(i,j)}$ for all $x \in X(i)_0$.
\end{itemize}
\end{dfn}

Now let  $X \in \Colax(I, \kdgCat)$.
The left transformation $(P_X, \ph_X):= (P, \ph)\colon X \to \De(\Gr_I X)$ is as follows.
\begin{itemize}
\item For each $i \in I_0$, the dg functor $P(i)\colon X(i) \to \Gr_I X$ is  defined by $P(i)x:= {}_ix$ for all $x \in X(i)_0$, and
by setting $P(i)f \colon {}_ix \to {}_iy$ as
\begin{equation}\label{eq:can-cov}
\begin{aligned}
P(i)f:&=(\de_{a,\id_i} (X_i\,x) * f)_{a\in I(i,i)}\\
&= \left(\left(\de_{a,\id_i} \sum_{r\in \bbZ}(-1)^{(n-r)r} f^{n-r} \circ (X_i\,x)^r\right)_{n\in \bbZ}\right)_{a\in I(i,i)}
\end{aligned}
\end{equation}
for all $f\colon x \to y$ in $X(i)$.
Note here that the map
$\calC(X_i(x), y) \colon \calC(x,y) \to \calC(X(\id_i)x, y)$, $f \mapsto f \circ X_ix$
is given by the contravariant functor $\calC(\blank, y)$ at $X_ix$.

\item For each $a \colon i \to j$ in $I$, the dg natural transformation
$\ph(a)\colon P(i) \Rightarrow P(j)X(a)$
$$\xymatrix{
X(i) & \Gr_I X\\
X(j) & \Gr_I X
\ar^{P(i)} "1,1";"1,2"
\ar_{P(j)} "2,1";"2,2"
\ar_{X(a)} "1,1";"2,1"
\ar@{=} "1,2";"2,2"
\ar@{=>}_{\ph(a)}"1,2";"2,1"
}
$$
is defined by $\ph(a)x:= (\de_{b,a} \id_{X(a)x})_{b \in I(i,j)}$ for all $x \in X(i)_0$.
\end{itemize}

\begin{lem}
The $(P, \ph)$ defined above is a $1$-morphism in $\Colax(I, \kdgCat)$.
\end{lem}

\begin{proof}
This is straightforward by using Remark \ref{rmk:1-mor-to-De}.
\end{proof}

Consider the cases that 
$X$ in $\Colax(I, \kDGCat)$ and that $X$ in $\Colax(I, \kdgCat)$.
In both cases, we have the canonical $I$-covering
$(P, \ph)\colon X \to \De(\Gr_I X)$ as shown below.

\begin{prp}\label{prp:can-covering-DGCat}
Let $X \in \Colax(I, \kDGCat)_0$.
Then the canonical morphism $(P, \ph)\colon X \to \De(\Gr_I X)$
is an $I$-covering.
More precisely, the morphism
$$(P, \ph)_{x,y}^{(1)}\colon \Ds_{a\in I(i,j)}X(j)(X(a)x, y) \to (\Gr_I X)(P(i)x, P(j)y)$$
is the identity
for all $i, j \in I_0$ and all $x \in X(i)_0$, $y \in X(j)_0$.
\end{prp}

\begin{proof}
By the definitions of $(\Gr_I X)_0$ and of $P$ it is obvious
that $(P, \ph)$ is dense.
Let $i, j \in I_0$ and $x \in X(i)$, $y \in X(j)$.
We only have to show that
$$
(P, \ph)_{x,y}^{(1)}\colon \Ds_{a\in I(i,j)}X(j)(X(a)x, y) \to (\Gr_I X)(P(i)x, P(j)y)
$$
is the identity.
Let $f = (f_a)_{a\in I(i,j)}\in \Ds_{a\in I(i,j)}X(j)(X(a)x, y)$.
Then by noting the form of $f_a \colon X(a)x \to y$ in $X(j)$,
we have the following equalities for each $n \in \bbZ$:
\[\footnotesize
\begin{aligned}
&(P, \ph)_{x,y}^{(1)}(f)^n
=
\sum_{a\in I(i,j)}  P(j)(f_a)^{n} \circ \ph(a)_x\quad (\text{by }\eqref{eq:1-covering})\\
&= \sum_{a\in I(i,j)}  \left(\de_{b, \id_j} \sum_{s \in \bbZ}(-1)^{(n-s)s}f_a^{n-s} \circ X_j(X(a)x)^s \right)_{b\in I(j,j)} \circ \ph(a)_x
\quad (\text{by } \eqref{eq:can-cov})\\
&= \sum_{a\in I(i,j)} \left(\de_{b, \id_j} \sum_{s \in \bbZ}(-1)^{(n-s)s}f_a^{n-s} \circ X_j(X(a)x)^s \right)_{b\in I(j,j)} \circ (\de_{c,a} \id_{X(a)x})_{c \in I(i,j)}\quad (\text{by } \eqref{eq:comp-dg-Gr})\\
&= \sum_{a\in I(i,j)} \hspace{-5pt}
\left(\hspace{-5pt}\sum_{\smat{b\in I(j,j)\\c\in I(i,j)\\d=bc}}  \hspace{-10pt}\de_{b, \id_j} \hspace{-5pt}\sum_{r,s \in \bbZ}(-1)^{(n-r)r}(-1)^{(n-r-s)s}f_a^{n-r-s} \circ X_j(X(a)x)^s \circ X(b)(\de_{c,a} \id_{X(a)x})^0 \circ (X_{b,c}x)^r \hspace{-5pt} \right)_{\hspace{-5pt}{d\in I(i,j)}}\\
&= \sum_{a\in I(i,j)} 
\left(\de_{d,a}\sum_{r,s \in \bbZ}(-1)^{(n-r)r}(-1)^{(n-r-s)s}f_a^{n-r-s} \circ X_j(X(a)x)^s \circ X(\id_j)( \id_{X(a)x})^0 \circ (X_{\id_j,a}x)^r\right)_{d\in I(i,j)}\\
&= \sum_{a\in I(i,j)} 
\left(\de_{d,a}\sum_{r,s \in \bbZ}(-1)^{(n-r)r}(-1)^{(n-r-s)s}f_a^{n-r-s} \circ X_j(X(a)x)^s \circ  (X_{\id_j,a}x)^r\right)_{d\in I(i,j)}\\
&= \sum_{a\in I(i,j)} 
\left(\de_{d,a}\sum_{\smat{r,s,t \in \bbZ\\n=r+s+t}}(-1)^{rs+rt+st}f_a^{t} \circ X_j(X(a)x)^s \circ  (X_{\id_j,a}x)^r\right)_{d\in I(i,j)}\quad (m:=r+s)\\
\end{aligned}
\]
\[
\begin{aligned}
&= \sum_{a\in I(i,j)} 
\left(\de_{d,a}\sum_{\smat{r,m,t \in \bbZ\\n=m+t}}(-1)^{r(m-r)+mt}f_a^{t} \circ X_j(X(a)x)^{(m-r)} \circ  (X_{\id_j,a}x)^r\right)_{d\in I(i,j)}\\
&= \sum_{a\in I(i,j)} 
\left(\de_{d,a}\sum_{\smat{m,t \in \bbZ\\n=m+t}}(-1)^{mt}f_a^{t} \circ \sum_{r\in \bbZ}(-1)^{(m-r)r}X_j(X(a)x)^{(m-r)} \circ  (X_{\id_j,a}x)^r\right)_{d\in I(i,j)}\\
&=\sum_{a\in I(i,j)} 
\left(\de_{d,a} ((X_{\id_j,a}x * X_j(X(a)x) ))*
f_a)^n \right)_{d\in I(i,j)}\hspace{15.5em}\\
&\overset{*}{=}\sum_{a\in I(i,j)} 
\left(\de_{d,a} (\id_{X(a)x}*
f_a)^n \right)_{d\in I(i,j)}
= f^n.
\end{aligned}
\]
In the above, the equality $\overset{*}{=}$ holds.
Indeed, let $(\blank)\op \colon X(j) \to X(j)\op$ be the
canonical contravariant functor defined by
$u\op:= u$ for all $u \in X(j)_0 \cup X(j)_1$, and $(h \circ g)\op = g * h$ for all morphisms $g \colon u \to v, h\colon v \to w$ in $X(j)$.
If we have an equality $h \circ g = \id_{u}$ in $X(j)$,
then we have $g * h = (h \circ g)\op = \id_u\op = \id_u$.
By applying this fact to the case where
$g = X_{\id_j,a}x, h = X_j(X(a)x), u = X(a)x$, we have
$X_{\id_j,a}x * X_j(X(a)x) = \id_{X(a)x}$.
\end{proof}

\begin{prp}\label{prp:can-covering}
Let $X \in \Colax(I, \kdgCat)_0$.
Then the canonical morphism $(P, \ph)\colon X \to \De(\Gr_I X)$
is an $I$-covering.
More precisely, the morphism
$$(P, \ph)_{x,y}^{(1)}\colon \Ds_{a\in I(i,j)}X(j)(X(a)x, y) \to (\Gr_I X)(P(i)x, P(j)y)$$
is the identity
for all $i, j \in I_0$ and all $x \in X(i)_0$, $y \in X(j)_0$.
\end{prp}

\begin{proof}
The proof is almost the same.
The difference is  that
$X_j$ and $X_{b,c}$ are dg natural transformations, and thus their degrees are $0$.
This makes the long computation above simpler as follows.

\[\small
\begin{aligned}
&(P, \ph)_{x,y}^{(1)}(f)^n
=
\sum_{a\in I(i,j)}  P(j)(f_a)^{n} \circ \ph(a)_x\\
&= \sum_{a\in I(i,j)}  \left(\de_{b, \id_j} \sum_{s \in \bbZ}(-1)^{(n-s)s}f_a^{n-s} \circ X_j(X(a)x)^s \right)_{b\in I(j,j)} \circ \ph(a)_x
\\
&= \sum_{a\in I(i,j)} \left(\de_{b, \id_j} f_a^n \circ X_j(X(a)x) \right)_{b\in I(j,j)} \circ (\de_{c,a} \id_{X(a)x})_{c \in I(i,j)}\\
&= \sum_{a\in I(i,j)} \hspace{-5pt}
\left(\hspace{-5pt}\sum_{\smat{b\in I(j,j)\\c\in I(i,j)\\d=bc}}  \hspace{-10pt}\de_{b, \id_j} \hspace{-5pt}\sum_{r\in \bbZ}(-1)^{(n-r)r}f_a^{n-r} \circ X_j(X(a)x) \circ X(b)(\de_{c,a} \id_{X(a)x})^0 \circ (X_{b,c}x)^r \hspace{-5pt} \right)_{\hspace{-5pt}{d\in I(i,j)}}\\
&= \sum_{a\in I(i,j)} 
\left(\de_{d,a}f_a^n \circ X_j(X(a)x) \circ X(\id_j)( \id_{X(a)x})^0 \circ (X_{\id_j,a}x)^0\right)_{d\in I(i,j)}\\
&= \sum_{a\in I(i,j)} 
\left(\de_{d,a}f_a^{n} \circ X_j(X(a)x) \circ  (X_{\id_j,a}x)\right)_{d\in I(i,j)}
\\
&=\sum_{a\in I(i,j)} 
\left(\de_{d,a}f_a^n\circ (X_j(X(a)x)\circ X_{\id_j,a}x))
 \right)_{d\in I(i,j)}\hspace{15.5em}\\
&\overset{*}{=}\sum_{a\in I(i,j)} 
\left(\de_{d,a}f_a^n\circ \id_{X(a)x}\right)_{d\in I(i,j)}
= f^n.
\end{aligned}
\]
The equality $\overset{*}{=}$ holds since
$X_j(X(a)x)\circ X_{\id_j,a}x= \id_{X(a)x}$.

\end{proof}

\begin{lem}\label{covering-equivalence}
Let $X \in \Colax(I, \kdgCat)_0$ and
$H\colon \Gr_I X \to \calC$ be in $\DGkCat$ and consider the composite $1$-morphism
$(F, \ps) \colon X \ya{(P, \ph)} \De(\Gr_I X) \ya{\De(H)} \De(\calC)$.
Then $(F, \ps)$ is an $I$-covering if and only if $H$ is an equivalence.
\end{lem}

\begin{proof}
Obviously $(F, \ps)$ is dense if and only if $H$ is dense.
Further, for each $i, j \in I_0$, $x \in X(i)$ and $y \in X(j)$,
$(F, \ps)^{(1)}_{x,y}$ is an isomorphism if and only if $H_{{}_ix, {}_jy}$ is an isomorphism,
because
we have a commutative diagram
$$
\xymatrix@C=40pt{
\Ds_{a\in I(i,j)}X(j)(X(a)x, y) & \calC(F(i)x, F(j)y)\\
(\Gr_I X)({}_ix, {}_jy)
\ar^(.58){(F, \ps)^{(1)}_{x,y}}"1,1";"1,2"
\ar@{=}_{(P,\ph)^{(1)}_{x,y}}"1,1";"2,1"
\ar_{H_{{}_ix, {}_jy}}"2,1";"1,2"
}
$$
by Proposition \ref{prp:can-covering}.
\end{proof}

\section{Adjoints}

In this section we will show that the Grothendieck construction is a strict left adjoint to
the diagonal 2-functor, and that $I$-coverings are essentially given
by the unit of the adjunction.

\begin{dfn}
Let $\calC \in \VCat$.
We define a functor $Q_{\calC} \colon \Gr_I\De(\calC) \to \calC$ by
\begin{itemize}
\item $Q_{\calC}({}_ix) := x$ for all ${}_ix \in (\Gr_I\De(\calC))_0$; and
\item $Q_{\calC}((f_a)_{a \in I(i,j)}):= \sum_{a\in I(i,j)}f_a$
for all $(f_a)_{a \in I(i,j)} \in (\Gr_I\De(\calC))({}_ix, {}_jy)$
and for all ${}_ix, {}_jy \in (\Gr_I\De(\calC))_0$.
\end{itemize}
It is easy to verify that $Q_{\calC}$ is a $\mathbb{V}$-functor.
\end{dfn}

\begin{thm}\label{Gr-De-adjoint}
The $2$-functor $\Gr_I\colon \Colax(I, \VCat) \to \VCat$ is a strict left $2$-adjoint to
the $2$-functor $\De\colon \VCat \to \Colax(I, \VCat)$.
The unit is given by the family
of canonical morphisms $(P_X, \ph_X) \colon X \to \De(\Gr_I X)$
indexed by $X \in \Colax(I, \VCat)$, and the counit is given by the family of
$Q_{\calC} \colon \Gr_I\De(\calC) \to \calC$
defined as above indexed by $\calC \in \VCat$.

In particular, $(P_X, \ph_X)$ has a strict universality
in the comma category $(X \!\!\downarrow\!\! \De)$, 
i.e., for each $(F, \ps) \colon X \to \De(\calC)$ in
$\Colax(I, \VCat)$ with $\calC \in \VCat$,
there exists a unique $H \colon \Gr_I X) \to \calC $ in $\VCat$ such that
the following is a commutative diagram in $\Colax(I, \VCat)$:
$$
\xymatrix{
X & \De(\calC).\\
\De(\Gr_I X)
\ar^{(F,\ps)}"1,1";"1,2"
\ar_{(P_X,\ph_X)} "1,1"; "2,1"
\ar@{-->}_{\De(H)} "2,1";"1,2"
}
$$
\end{thm}

\begin{proof}
For simplicity set $\et:=((P_X, \ph_X))_{X \in \Colax(I, \VCat)_0}$ and
$\ep:= (Q_{\calC})_{\calC \in \VCat_0}$.

\begin{clm}
$\De \ep \cdot \et \De = \id_{\De}$.
\end{clm}
Indeed,
let $\calC \in \VCat$.
It is enough to show that
$\De(Q_{\calC}) \cdot (P_{\De(\calC)}, \ph_{\De(\calC)}) = \id_{\De(\calC)}$.
Now
$$
\begin{aligned}
\mathrm{LHS}
 &=\left((Q_{\calC}P_{\De(\calC)}(i))_{i\in I_0}, (Q_{\calC}\ph_{\De(\calC)}(a))_{a\in I_1}\right)
 , \text{ and}\\
\mathrm{RHS}
 &=\left((\id_{\calC})_{i\in I_0}, (\id_{\id_{\calC}})_{a\in I_1}\right).
\end{aligned}
 $$
{\it First entry$\colon$}Let $i \in I$.
Then $Q_{\calC}P_{\De(\calC)}(i) = \id_{\calC}$
because for each $x, y \in \calC_0$ and each $f\in \calC(x, y)$ we have
$(Q_{\calC}P_{\De(\calC)}(i))(x) = Q_{\calC}({}_ix) = x$; and
$(Q_{\calC}P_{\De(\calC)}(i))(f) = (\de_{a, \id_i}f \cdot ((\et_{\De(\calC)})_i\, x))_{a\in I_1} =
\sum_{a \in I(i, i)}\de_{a, \id_i}f = f$.

{\it Second entry$\colon$} Let $a \colon i \to j$ in $I$.
Then $Q_{\calC}\ph_{\De(\calC)}(a) = \id_{\id_{\calC}}$ because
for each $x \in \calC_0$ we have
$Q_{\calC}\left(\ph_{\De(\calC)}(a)x\right)
= Q_{\calC}\left((\de_{b,a}\id_{\De(\calC)(a)x})_{b\in I(i,j)}\right)
=\sum_{b\in I(i,j)}\de_{b,a}\id_x = \id_x = \id_{\id_{\calC}x}$.
This shows that $\mathrm{LHS} = \mathrm{RHS}$.

\begin{clm}
$\ep \Gr_I \cdot \Gr_I \et = \id_{\Gr_I}$.
\end{clm}
Indeed, let $X \in \Colax(I, \VCat)$.
It is enough to show that
$Q_{\Gr_I X}\cdot \Gr_I(P_X, \ph_X) = \id_{\Gr_I X}$.

{\it On objects$\colon$} Let ${}_ix \in (\Gr_I X)_0$.
Then $Q_{\Gr_I X}\left(\Gr_I(P_X, \ph_X)(x)\right)
= Q_{\Gr_I X}({}_i(P_X(i)x))
= {}_ix$.

{\it On morphisms$\colon$} Let $f = (f_a)_{a\in I(i,j)} \colon {}_ix \to {}_jy$ be in $\Gr_I X$.
Then we have 
\[\begin{aligned}
Q_{\Gr_I X}\Gr_I(P_X, \ph_X)(f)
= Q_{\Gr_I X}((P_X(j)(f_a)\circ\ph_X(a)x)_{a\in I(i,j)})\\
= \sum_{a\in I(i,j)}P_X(j)(f_a)\circ\ph_X(a)x
= (P_X, \ph_X)^{(1)}_{x,y}(f) = f.
\end{aligned}
\]
Thus the claim holds.
The two claims above prove the assertion.
\end{proof}

\begin{cor}\label{covering-Gr}
Let $(F, \ps) \colon X \to \De(\calC)$ be in $\Colax(I, \VCat)$.
Then the following are equivalent.
\begin{enumerate}
\item $(F, \ps)$ is an $I$-covering;
\item There exists an equivalence $H \colon \Gr_I X \to \calC$ such that
the diagram
$$
\xymatrix{
X & \De(\calC)\\
\De(\Gr_I X)
\ar^{(F, \ps)}"1,1";"1,2"
\ar_{(P_X, \ph_X)}"1,1";"2,1"
\ar_{\De(H)}"2,1";"1,2"
}
$$
is strictly commutative.
\end{enumerate}
\end{cor}

\begin{proof}
This immediately follows by Theorem \ref{Gr-De-adjoint} and Lemma \ref{covering-equivalence}. More precisely, 

\begin{equation}
\begin{aligned}
(F,\ps)^{(1)}_{x,y}(((f^n_a)_{n\in \mathbb{Z}})_{a\in I(i,j)})
&=
\sum_{a\in I(i,j)} \ps(a)_x * F(j)(f_a)\\
&=\sum_{a\in I(i,j)} H\phi(a)_x * HP(j)(f_a)\\
&=H(\sum_{a\in I(i,j)} \phi(a)_x * P(j)(f_a))\\
&=H(P, \ph)_{x,y}^{(1)}(f).
\end{aligned}
\end{equation}
\end{proof}

In particular, in the case that $\mathbb{V} = \dgMod(\k)$, i.e., 
$\VCat=\kdgCat$, we have the following.

The $2$-functor $\Gr_I\colon \Colax(I, \kdgCat) \to \kdgCat$ is a strict left $2$-adjoint to
the $2$-functor $\De\colon \kdgCat \to \Colax(I, \kdgCat)$.
The unit is given by the family
of canonical morphisms $(P_X, \ph_X) \colon X \to \De(\Gr_I X)$
indexed by $X \in \Colax(I, \kdgCat)$, and the counit is given by the family of
$Q_{\calC} \colon \Gr_I\De(\calC) \to \calC$
defined as above indexed by $\calC \in \kdgCat$.

In particular, $(P_X, \ph_X)$ has a strict universality in the comma category
$(X \!\!\downarrow\!\! \De)$, i.e., for each $(F, \ps) \colon X \to \De(\calC)$ in
$\Colax(I, \kdgCat)$ with $\calC \in \kdgCat$,
there exists a unique $H \colon \Gr_I X \to \calC $ in $\kdgCat$ such that
the following is a commutative diagram in $\Colax(I, \kdgCat)$:
$$
\xymatrix{
X & \De(\calC).\\
\De(\Gr_I X)
\ar^{(F,\ps)}"1,1";"1,2"
\ar_{(P_X,\ph_X)} "1,1"; "2,1"
\ar@{-->}_{\De(H)} "2,1";"1,2"
}
$$

\section{The derived colax functors}

Let $X\colon I \to \kdgCat$ be a colax functor.
In this section we formulate
the definition of
the ``derived category $\calD(X)$'' of $X$
as a colax functor from $I$ to a $2$-category of triangulated categories
by modifying the definition given in the previous paper
\cite{Asa-a}.
We first recall the definition of colax functors between $2$-categories.

\begin{dfn}
\label{dfn:colax-fun-2cat}
Let $\bfB$ and $\bfC$ be 2-categories.

(1) A {\em colax functor} from $\bfB$ to $\bfC$ is a triple
$(X, \et, \th)$ of data:
\begin{itemize}
\item
a triple $X=(X_0, X_1, X_2)$ of maps $X_i\colon \bfB_i \to \bfC_i$ ($\bfB_i$ denotes the
collection of $i$-morphisms of $\bfB$ for each $i=0,1,2$) preserving domains and codomains of all 1-morphisms and 2-morphisms
(i.e.\ $X_1(\bfB_1(i,j)) \subseteq \bfC_1(X_0i, X_0j)$
for all $i, j \in \bfB_0$ and $X_2(\bfB_2(a,b)) \subseteq \bfC_2(X_1a, X_1b)$
for all $a, b \in \bfB_1$ (we omit the subscripts of $X$ below));
\item
a family $\et:=(\et_i)_{i\in \bfB_0}$ of 2-morphisms $\et_i\colon X(\id_i) \Rightarrow \id_{X(i)}$ in $\bfC$
indexed by $i\in \bfB_0$; and
\item
a family $\th:=(\th_{b,a})_{(b,a)}$ of 2-morphisms
$\th_{b,a} \colon X(ba) \Rightarrow X(b)X(a)$
in $\bfC$ indexed by $(b,a) \in \com(\bfB):=
\{(b,a)\in \bfB_1 \times \bfB_1 \mid ba \text{ is defined}\}$
\end{itemize}
satisfying the axioms:

\begin{enumerate}
\item[(i)]
$(X_1, X_2) \colon \bfB(i,j) \to \bfC(X_0i,X_0j)$ is a functor
for all $i, j \in \bfB_0$;
\item[(ii)]
For each $a\colon i \to j$ in $\bfB_1$ the following are commutative:
$$
\vcenter{
\xymatrix{
X(a\id_i) \ar@{=>}[r]^(.43){\th_{a,\id_i}} \ar@{=}[rd]& X(a)X(\id_i)
\ar@{=>}[d]^{X(a)\et_i}\\
& X(a)\id_{X(i)}
}}
\qquad\text{and}\qquad
\vcenter{
\xymatrix{
X(\id_j a) \ar@{=>}[r]^(.43){\th_{\id_j,a}} \ar@{=}[rd]& X(\id_j)X(a)
\ar@{=>}[d]^{\et_jX(a)}\\
& \id_{X(j)}X(a)
}}\quad;
$$
\item[(iii)]
For each $i \ya{a}j \ya{b} k \ya{c} l$ in $\bfB_1$ the following is commutative:
$$
\vcenter{
\xymatrix@C=3em{
X(cba) \ar@{=>}[r]^(.43){\th_{c,ba}} \ar@{=>}[d]_{\th_{cb,a}}& X(c)X(ba)
\ar@{=>}[d]^{X(c)\th_{b,a}}\\
X(cb)X(a) \ar@{=>}[r]_(.45){\th_{c,b}X(a)}& X(c)X(b)X(a)
}}\quad;\text{ and}
$$
\item[(iv)]
For each $a, a' \colon i \to j$ and $b, b' \colon j \to k$ in $\bfB_1$
and each $\al \colon a \to a'$, $\be \colon b \to b'$ in $\bfB_2$
the following is commutative:
$$
\xymatrix{
X(ba) & X(b)X(a)\\
X(b'a') & X(b')X(a').
\ar@{=>}^{\th_{b,a}}"1,1";"1,2"
\ar@{=>}^{\th_{b',a'}}"2,1";"2,2"
\ar@{=>}_{X(\be*\al)}"1,1";"2,1"
\ar@{=>}^{X(\be)*X(\al)}"1,2";"2,2"
}
$$
\end{enumerate}

(2) A {\em lax functor} from $\bfB$ to $\bfC$ is a colax functor
from $\bfB$ to $\bfC^{\text{co}}$ (see Notation \ref{ntn:co-op}).

(3) A {\em pseudofunctor} from $\bfB$ to $\bfC$ is a colax functor with
all $\et_i$ and $\th_{b,a}$ 2-isomorphisms.

(4) We define a 2-category $\Colax(\bfB, \bfC)$ having all the colax functors
$\bfB \to \bfC$ as the objects as follows.

{\bf 1-morphisms.}
Let $X = (X, \et, \th)$, $X'= (X', \et', \th')$
be colax functors from $\bfB$ to $\bfC$.
A {\em $1$-morphism} (called a {\em left transformation}) from $X$ to $X'$
is a pair $(F, \ps)$ of data
\begin{itemize}
\item
a family $F:=(F(i))_{i\in \bfB_0}$ of 1-morphisms $F(i)\colon X(i) \to X'(i)$
in $\bfC$
; and
\item
a family $\ps:=(\ps(a))_{a\in \bfB_1}$ of 2-morphisms
$\ps(a)\colon X'(a)F(i) \Rightarrow F(j)X(a)$
$$\vcenter{\xymatrix{
X(i) & X'(i)\\
X(j) & X'(j)
\ar_{X(a)} "1,1"; "2,1"
\ar^{X'(a)} "1,2"; "2,2"
\ar^{F(i)} "1,1"; "1,2"
\ar_{F(j)} "2,1"; "2,2"
\ar@{=>}_{\ps(a)} "1,2"; "2,1"
}}
$$
in $\bfC$  indexed by $a\colon i \to j \text{ in }\bfB_1$
that satisfies the following three conditions:
\item[(0)]
for each $\al \colon a \To b$ in $\bfB(i,j)$ the following
is commutative:
\begin{equation}\label{eq:naturality-psi}
\vcenter{\xymatrix@C=10ex{
X'(a)F(i) & X'(b)F(i)\\
F(j)X(a) & F(j)X(b),
\ar@{=>}^{X'(\al)F(i)}"1,1";"1,2"
\ar@{=>}_{F(j)X(\al)}"2,1";"2,2"
\ar@{=>}_{\ps(a)}"1,1";"2,1"
\ar@{=>}^{\ps(b)}"1,2";"2,2"
}}
\end{equation}
thus $\ps$ gives a family of natural transformations of functors:
$$
\vcenter{\xymatrix@C=15ex{
\bfB(i,j) & \bfC(X'(i), X'(j))\\
\bfC(X(i), X(j)) & \bfC(X(i), X'(j))
\ar"1,1";"1,2"^{X'}
\ar"1,1";"2,1"_{X}
\ar"1,2";"2,2"^{\bfC(F(i), X'(j))}
\ar"2,1";"2,2"_{\bfC(X(i), F(j))}
\ar@{=>}"1,2";"2,1"_{\ps_{ij}}
}}\quad(i,j\in \bfB_0),
$$
\end{itemize}
\begin{enumerate}
\item[(a)]
For each $i \in \bfB_0$ the following is commutative:
$$
\vcenter{
\xymatrix{
X'(\id_i)F(i) & F(i)X(\id_i)\\
\id_{X'(i)}F(i) & F(i)\id_{X(i)}
\ar@{=>}^{\ps(\id_i)} "1,1"; "1,2"
\ar@{=} "2,1"; "2,2"
\ar@{=>}_{\et'_iF(i)} "1,1"; "2,1"
\ar@{=>}^{F(i)\et_i} "1,2"; "2,2"
}}\quad;\text{ and}
$$
\item[(b)]
For each $i \ya{a} j \ya{b} k$ in $\bfB_1$ the following is commutative:
$$
\xymatrix@C=4pc{
X'(ba)F(i) & X'(b)X'(a)F(i) & X'(b)F(j)X(a)\\
F(k)X(ba) & & F(k)X(b)X(a).
\ar@{=>}^{\th'_{b,a}F(i)} "1,1"; "1,2"
\ar@{=>}^{X'(b)\ps(a)} "1,2"; "1,3"
\ar@{=>}_{F(k)\,\th_{b,a}} "2,1"; "2,3"
\ar@{=>}_{\ps(ba)} "1,1"; "2,1"
\ar@{=>}^{\ps(b)X(a)} "1,3"; "2,3"
}
$$
\end{enumerate}

{\bf 2-morphisms.}
Let $X = (X, \et, \th)$, $X'= (X', \et', \th')$
be colax functors from $\bfB$ to $\bfC$, and
$(F, \ps)$, $(F', \ps')$ 1-morphisms from $X$ to $X'$.
A {\em $2$-morphism} from $(F, \ps)$ to $(F', \ps')$ is a
family $\ze= (\ze(i))_{i\in \bfB_0}$ of 2-morphisms
$\ze(i)\colon F(i) \Rightarrow F'(i)$ in $\bfC$
indexed by $i \in \bfB_0$
such that the following is commutative for all $a\colon i \to j$ in $\bfB_1$:
$$
\xymatrix@C=4pc{
X'(a)F(i) & X'(a)F'(i)\\
F(j)X(a) & F'(j)X(a).
\ar@{=>}^{X'(a)\ze(i)} "1,1"; "1,2"
\ar@{=>}^{\ze(j)X(a)} "2,1"; "2,2"
\ar@{=>}_{\ps(a)} "1,1"; "2,1"
\ar@{=>}^{\ps'(a)} "1,2"; "2,2"
}
$$

{\bf Composition of 1-morphisms.}
Let $X = (X, \et, \th)$, $X'= (X', \et', \th')$ and $X''= (X'', \et'', \th'')$
be colax functors from $\bfB$ to $\bfC$, and
let $(F, \ps)\colon X \to X'$, $(F', \ps')\colon X' \to X''$
be 1-morphisms.
Then the composite $(F', \ps')(F, \ps)$ of $(F, \ps)$ and
$(F', \ps')$ is a 1-morphism from $X$ to $X''$ defined by
$$
(F', \ps')(F, \ps):= (F'F, \ps'\circ\ps),
$$
where $F'F:=((F'(i)F(i))_{i\in \bfB_0}$ and for each $a\colon i \to j$ in $\bfB$,
$
(\ps'\circ\ps)(a):= F'(j)\ps(a)\circ \ps'(a)F(i)
$
is the pasting of the diagram
$$
\xymatrix@C=4pc{
X(i) & X'(i) & X''(i)\\
X(j) & X'(j) & X''(j).
\ar_{X(a)} "1,1"; "2,1"
\ar_{X'(a)} "1,2"; "2,2"
\ar^{F(i)} "1,1"; "1,2"
\ar_{F(j)} "2,1"; "2,2"
\ar@{=>}_{\ps(a)} "1,2"; "2,1"
\ar^{X''(a)} "1,3"; "2,3"
\ar^{F'(i)} "1,2"; "1,3"
\ar_{F'(j)} "2,2"; "2,3"
\ar@{=>}_{\ps'(a)} "1,3"; "2,2"
}
$$
\end{dfn}

\begin{rmk}
We make the following remarks.
\begin{enumerate}
\item
Note that a (strict) 2-functor from $\bfB$ to $\bfC$ is a pseudofunctor with
all $\et_i$ and $\th_{b,a}$ identities.
\item
By regarding the category $I$ as a 2-category with all 2-morphisms identities,
the definition (1) of colax functors above coincides
with Definition \ref{dfn:colax-fun}.
\item
When $\bfB = I$, the definition (4) of $\Colax(\bfB, \bfC)$
above coincides with that of $\Colax(I, \bfC)$ given before.
\item
It is well-known that the composite of pseudofunctors turns out to be a pseudofunctor \cite[Lemma 9.2]{Asa-13} or \cite[Lemma 4.1.27(2)]{JY}.
\end{enumerate}

\end{rmk}

\begin{ntn}
\label{ntn:2-cats}
We introduce the following $2$-categories.
\begin{enumerate}
\item
$\kAB$ denotes the $2$-category of light abelian $\k$-categories,
$\k$-functors between them,
and natural transformations between these $\k$-functors.
\item
$\kFRB$ denotes the $2$-category of light Frobenius $\k$-categories,
$\k$-functors between them,
and natural transformations between these $\k$-functors.
\item
$\kTRI$ denotes the $2$-category of light triangulated $\k$-categories,
triangle $\k$-functors between them,
and natural transformations between these triangle $\k$-functors.
\item
$\kuTRI^2$ denotes the $2$-category of $2$-moderate triangulated $\k$-categories,
triangle $\k$-functors between them,
and natural transformations between these triangle $\k$-functors.
\end{enumerate}
\end{ntn}

\begin{dfn}
Let $\calA \in \DGkCat_0 = \kDGCat_0$.
A dg functor $\calA\op \to \DGMod(\k)$ is called a {\em right dg $\calA$-module}.
We set
\[
\begin{aligned}
\ChMod(\calA)&:= \DGkCat(\calA\op, \DGMod(\k)), \text{and}\\
\DGMod(\calA)&:= \kDGCat(\calA\op, \DGMod(\k)).
\end{aligned}
\]
$\ChMod(\calA)$ is called the {\em category of $($right$)$ dg $\calA$-modules}, and
$\DGMod(\calA)$ is called the {\em dg category of $($right$)$ dg $\calA$-modules}.
Thus in particular, we have
\[
\ChMod(\calA)_0 = \DGMod(\calA)_0,
\]
which consists of the right dg $\calA$-modules.
Note that $\ChMod(\calA)$ is in $\kAB$, or more presicely, in $\kFRB$,
whereas $\DGMod(\calA)$ is in $\kdgCAT$ and in $\kDGCAT$.
By \eqref{eq:dg-DG},
they have the following relation for all objects $X, Y$:
\[
\ChMod(\calA)(X, Y) = Z^0(\DGMod(\calA)(X,Y)).
\]
Thus a morphism $X \to Y$ in $\ChMod(\calA)$ is given as a dg natural transformation, but in $\DGMod(\calA)$ it is given as a derived transformation.
\end{dfn}

\begin{dfn}
\label{dfn:bimodule}
Let $\calA$ and $\calB$ be small dg categories.
\begin{enumerate}
\item
A dg functor $\calB \to \DGMod(\k)$ is called a {\em left dg $\calB$-module}.
A $\calB$-$\calA$-bimodule is a dg functor $M \colon \calA\op \ox_\k \calB \to \DGMod(\k)$.
\item
For each $B \in \calB_0$ and $A \in \calA_0$, we set ${}_BM:= M(\blank, B) \colon \calA\op \to \DGMod(\k)$
and $M_A:= M(A, \blank) \colon \calB \to \DGMod(\k)$, and ${}_BM_A:= M(A, B)$.
Note that ${}_BM$ is a right dg $\calA$-module, $M_A$ is a left dg $\calB$-module,
and ${}_BM_A$ is a dg $\k$-module.
\item
If $f \colon A' \to A$ is a morphism in $\calA$, and $g \colon B' \to B$ is a morphism in $\calB$,
then we set ${}_gM:= M(\blank, g)$ and $M_f:= M(f,\blank)$.
Note that $M_f \colon M_{A} \to M_{A'}$ is a derived transformation between 
covariant dg functors,
while
${}_gM \colon {}_{B'} M \to {}_{B}M$ is a derived transformation between
contravariant dg functors.
\item
To emphasise that $M$ is a $\calB$-$\calA$-bimodule, we sometimes use the notation
${}_\calB M_\calA$.
\item
The dg category $\calA$ defines an $\calA$-$\calA$-bimodule 
$\calA(\blank, ?) \colon \calA\op \ox_\k \calA \to \DGMod(\k)$
by $(x,y) \mapsto \calA(x, y)$.
We denote this bimodule by ${}_\calA \calA_\calA$,
and we use the same convention that
${}_x\calA:= \calA(\blank, x)$, $\calA_y:= \calA(y, \blank)$, and
${}_x\calA_y:= \calA(y,x)$ for all $x, y \in \calA_0$; and
for any morpisms $f \colon x' \to x$, $g \colon y' \to y$ in $\calA$,
we write
${}_g \calA \colon {}_{y'}\calA \to {}_{y}\calA$, and
$\calA_f \colon \calA_x \to \calA_{x'}$.
Sometimes we abbreviate them as $g^\wedge \colon {y'}^\wedge \to y^\wedge$ and
${}^\wedge f \colon {}^\wedge x \to {}^\wedge x'$, respectively.
\end{enumerate}
\end{dfn}

\begin{rmk}
\label{rmk:0-cocycle}
In Definition \ref{dfn:bimodule} (3),
note that $0$-cocycle morphisms are preserved by the correspondences
$f \mapsto M_{f}$ and $g \mapsto {}_{g}M$, namely,
if $f \in Z^0(\calA(A', A))$ (resp.\ $g \in Z^0(\calB(B',B))$),
then $M_f \in Z^0(\DGMod(\calB\op)(M_A, M_{A'})) = \calC(\calB\op)(M_A, M_{A'})$
(resp.\ ${}_gM \in Z^0(\DGMod(\calA)$ $({}_{B'}M, {}_BM)) = \calC(\calA)({}_{B'}M, {}_BM)$).

Indeed, 
since $M_A$ is a left dg $\calB$-module, 
the dg functor $M \colon \calA\op \ox_\k \calB \to \DGMod(\k)$
induces a dg functor $M:\calA\op  \to \DGMod(\calB\op)$.
Therefore, for any $A, A'\in\calA_0$, by noting that $\calA\op(A,A') = \calA(A', A)$,
it induces a chain map
$$
M_{A,A'}:\calA(A',A)  \to \DGMod(\calB\op)( M_{A}, M_{A'}).
$$
Hence if $f \in Z^0(\calA(A', A))$, then
$M_f \in Z^0(\DGMod(\calB\op)( M_{A}, M_{A'}))$.
Thus in this case,
$M_f \colon M_{A} \to M_{A'}$
is a dg natural transformation between dg functors.
Similar argument works for the remaining case.

\end{rmk}

\begin{ntn}
\label{ntn:bimodule}
Let $\calA, \calB, \calA', \calB'$ be small dg categories,
$E \colon \calA' \to \calA$, $F \colon \calB' \to \calB$ dg functors, and
$M$ an $\calA$-$\calB$-bimodule.
\begin{enumerate}
\item
We denote by ${}_E M$, $M_F$ and ${}_E M_F$
the $\calA'$-$\calB$-bimodule, $\calA$-$\calB'$-bimodule, and $\calA'$-$\calB'$-bimodule
defined as follows, respectively:
\[\footnotesize
\begin{aligned}
{}_E M&:= M(\blank, E(?)) = M \circ (\id_{\calB\op}\ox_k E) \colon \calB\op \ox_\k \calA' \ya{\id_{\calB\op}\ox_k E} \calB\op \ox_\k \calA \ya{M} \DGMod(\k),\\
M_F&:= M(F(\blank), ?) = M \circ (F \ox_k \id_{\calA}) \colon {\calB'}\op \ox_\k \calA \ya{F \ox_k \id_{\calA}} \calB\op \ox_\k \calA \ya{M} \DGMod(\k), \text{and}\\
{}_E M_F&:= M(F(\blank), E(?)) = M \circ (F\ox_k E) \colon {\calB'}\op \ox_\k \calA' \ya{F\ox_k E} \calB\op \ox_\k \calA \ya{M} \DGMod(\k).
\end{aligned}
\]
\item
Moreover, if $E' \colon \calA' \to \calA$ and $F' \colon \calB' \to \calB$ are dg functors,
and $\al \colon E \To E'$, $\be \colon F \To F'$ are derived transformations,
then $M$ defines morphisms of bimodules as follows:
\[
\begin{aligned}
{}_{\al} M&:= M(\blank, \al_{(?)}) = M \circ  (\id_{\calB\op}\ox_k \al) \colon {}_E M \to {}_{E'}M\\
M_{\be}&:= M(\be_{(\blank)}, ?) = M \circ (\be \ox_k \id_{\calA}) \colon M_{F'} \to M_F, \text{and}\\
{}_{\al} M_{\be}&:= M(\be_{(\blank)}, \al_{(?)}) = M \circ (\be \ox_\k \al) \colon {}_E M_{F'} \to {}_{E'}M_F.
\end{aligned}
\]
\item
We often abbreviate the morphism of $\calA'$-$\calA$-bimodules
(induced from the bimodule ${}_{\calA}\calA_\calA$)
\[
{}_{\al} \calA \colon {}_E \calA \to {}_{E'}\calA \quad \text{as}\quad 
\ovl{\al} \colon \ovl{E} \to \ovl{E'},
\]
and the morphism of $\calB$-$\calB'$-bimodules (induced from the bimodule ${}_{\calB}\calB_\calB$)
\[
\calB_{\be} \colon \calB_{F'} \to \calB_F \quad \text{as}\quad
\ovl{\be}^* \colon \ovl{F'}^* \to \ovl{F}^*.
\]
\end{enumerate}
\end{ntn}

\begin{dfn}
\label{dfn:associator}
Let $\calA, \calB, \calC$ and $\calD$ be small dg categories, and
${}_\calD W_\calC, {}_\calC V_\calB$, ${}_\calB U_\calA$ bimodules.
Then the canonical isomorphism
\[
\bfa = \bfa_{W,V,U} \colon (W\ox_\calC V) \ox_\calB U \to W \ox_\calC (V \ox_\calB U)
\]
that represent the associativity of the tensor products
is called the {\em associator} of tensor products (cf.\ Definition \ref{dfn:monoidal-cat}).
\end{dfn}

\begin{dfn}
Since both $\kdgCat$ and $\kDGCat$ are 2-categories,
the assignments $\calA \mapsto \DGMod(\calA)$
and $\calA \mapsto \ChMod(\calA)$
are extended to representable 2-functors
\[
\begin{aligned}
\DGMod':= \kDGCat((\blank)\op, \DGMod(\k))) &\colon \kDGCat \to \kDGCAT\coop \text{ and}\\
\dgMod':= \kdgCat((\blank)\op, \DGMod(\k))) &\colon \kdgCat \to \kFRB\coop,
\end{aligned}
\]
respectively.
By modifying these, we define pseudofunctors
\[
\begin{aligned}
\DGMod &\colon \kDGCat \to \kDGCAT \text{ and}\\
\dgMod &\colon \kdgCat \to \kFRB
\end{aligned}
\]
as follows.
For any diagram
\[
\begin{tikzcd}
\calA & \calB
\Ar{1-1}{1-2}{"E", ""'{name=a}, bend left}
\Ar{1-1}{1-2}{"F"', ""{name=b}, bend right}
\Ar{a}{b}{"\al", Rightarrow}
\end{tikzcd}
\]
in $\kDGCat$ (resp.\ in $\kdgCat$),
we define $\DGMod(E):= \blank \ox_\calA {}_E\calB$ and
$\DGMod(\al):= \blank \ox_\calA {}_{\al}\calB$
(resp.\ $\dgMod(E):= \blank \ox_\calA {}_E\calB$ and
$\dgMod(\al):= \blank \ox_\calA {}_{\al}\calB$) (see Notation \ref{ntn:bimodule}).
Note that
$\dgMod(\al):= \blank \ox_\calA {}_{\al}\calB$ is defined by
regarding $\al$ in $\kdgCat$ as a $2$-morphism in $\kDGCat$ concentrated in degree 0.

We define the structures of pseudofunctors for $\DGMod$ and $\dgMod$ as follows.
\begin{itemize}
\item For each $\calA \in \kDGCat_0 = \kdgCat_0$, we define
$\et_\calA \colon \DGMod(\id_\calA) \To \id_{\DGMod(\calA)}$
(resp.\ $\et_\calA \colon \dgMod(\id_\calA) \To \id_{\dgMod(\calA)}$)
by setting 
$$
\et_\calA M \colon M \otimes_\calA \calA(?,\blank) \to M
$$
to be the canonical isomorphisms for all $M \in \DGMod(\calA)_0 = \dgMod(\calA)_0$.

\item For each pair of dg functors $\calA \ya{F} \calA' \ya{G} \calA''$
in $\kDGCat_1 = \kdgCat_1$,
we define 
\[
\begin{aligned}
\th_{G, F} &\colon \DGMod(GF) \To \DGMod(G) \circ \DGMod(F)\\
(\text{resp.}\ \th_{G, F} &\colon \dgMod(GF) \To \dgMod(G) \circ \dgMod(F))
\end{aligned}
\]
as the canonical isomorphism
$$
 \blank\otimes_\calA {}_{GF}\calA''\To (\blank\otimes_\calA {}_F\calA')\otimes_{\calA'} {}_G\calA''
$$
\end{itemize}
given as the composite of the canonical isomorphisms
(see Definition \ref{dfn:associator})
\[
\begin{aligned}
\blank\otimes_\calA {}_{GF}\calA''
&\iso \blank\otimes_\calA (\calA'(?,F(\blank))\otimes_{\calA'}\calA''(?, G(\blank)))\\
&\ya{\bfa_{(\blank),\calA'(?,F(\blank)),\calA''(?, G(\blank))}\inv} (\blank\otimes_\calA \calA'(?,F(\blank)))\otimes_{\calA'} \calA''(?, G(\blank)).
\end{aligned}
\]
It is straightforward to check that this defines pseudofunctors $\DGMod$ and $\dgMod$.
\end{dfn}

\begin{rmk}
\label{rmk:C'-exact}
In the definition above, note that $\dgMod(E)$ is a left adjoint to $\dgMod'(E)$ 
and that $\dgMod'(E)$ has also a right adjoint.
Therefore, $\dgMod'(E)$ is an exact functor.
\end{rmk}

\begin{rmk}
\label{rmk:Yoneda-emb}
Using the notation in Definition \ref{dfn:bimodule} (5),
the Yoneda embedding $Y_\calA \colon \calA \to \DGMod(\calA)$
can be defined 
as the functor sending a morphism $f \colon x \to y$ in $\calA$ to
the morphism $f^\wedge \colon x^\wedge \to y^\wedge$ in $\DGMod(\calA)$.
\end{rmk}

\begin{dfn}
As is easily seen,
the stable category construction $\calF \mapsto \udl{\calF}$ can be
extended to a $2$-functor $\st \colon \kFRB \to \kTRI$
in an obvious way.
Then we set
\[
\calH:= \st \circ \dgMod \colon \kdgCat \to \kTRI,
\]
which turns out to be a pseudofunctor
as a composite of a pseudofunctor and a 2-functor.
For each $\calA \in \kdgCat$,
$\calH(\calA)$ is called the {\em homotopy category} of $\calA$.

For each $\calA \in \kDGCat_0$, we set
\[
\calD(\calA):= \calH(\calA)[\qis\inv]
\]
to be the quotient category of the homotopy category of $\calA$ with respect to quasi-isomorphisms, and call it the {\em derived category} of $\calA$.
\end{dfn}

\begin{rmk}
We note that for each $\calA \in \kDGCat$,
the homotopy category $\calH(\calA)$ is a light triangulated category, but
the derived category $\calD(\calA)$ is a properly 2-moderate triangulated category.
This is in spite of the fact that we have the isomorphisms \eqref{eq:Hom-dercat} below.
\end{rmk}

\begin{dfn}
\label{dfn:dgMod-D}
Let $\calA \in \kDGCat_0$.
\begin{enumerate}
\item
We denote by $\hprj(\calA)$
the full subcategory of the homotopy category $\calH(\calA)$ of $\calA$ consisting of
the {\em homotopically projective} objects $M$, i.e., objects $M$ such that $\calH(\calA)(M, A) =0$
for all acyclic objects $A$.

\item
Let $\hprj(\calA) \ya{\si_{\calA}} \calH(\calA) \ya{Q_{\calA}} \calD(\calA)$
be the inclusion functor and the quotient functor, respectively,
and set $\bfj_{\calA}:= Q_{\calA} \circ \si_{\calA}$.
Then there exists a functor $\bfp_{\calA} \colon \calD(\calA) \to \hprj(\calA)$
giving a left adjoint $\si_\calA \circ \bfp_\calA$ to $Q_\calA$:
\[
\begin{tikzcd}
&\hprj(\calA)\\
\calH(\calA) && \calD(\calA)
\Ar{2-3}{1-2}{"\bfp_{\calA}"'}
\Ar{1-2}{2-1}{"\si_{\calA}"'}
\Ar{2-1}{2-3}{"Q_{\calA}"', ""{name=b}, bend right=10pt}
\Ar{1-2}{b}{"\perp"description,phantom}
\end{tikzcd}
\]
such that
\begin{equation}
\label{eq:pj=1}
\bfp_{\calA}\bfj_{\calA} = \id_{\hprj(\calA)}
\end{equation}
is satisfied\footnote{
This can be done by taking  $\bfp_\calA(P):= P$ for all $P \in \hprj(\calA)$.}
and the counit
$\ep_{\calA} \colon (\si_\calA \circ \bfp_{\calA}) \circ Q_\calA \To \id_{\calH(\calA)}$
consists of quasi-isomorphisms
$\ep_{\calA,M}\colon \bfp_{\calA}M \to M$ for all $M \in \calH(\calA)_0$.
In particular, both $\bfp_{\calA}$ and $\bfj_{\calA}$ are equivalences and
quasi-inverses to each other, and by the adjoint above we have a canonical
isomorphism
\begin{equation}
\label{eq:Hom-dercat}
\calD(\calA)(L, M) \iso \calH(\calA)(\bfp_\calA(L), M)
\end{equation}
for all $L, M \in \calD(\calA)_0 = \calH(\calA)_0$.
Here, note that the right hand sides are always small, but the left hand sides are not.
\end{enumerate}
\end{dfn}

\begin{dfn}
We define two $2$-categories $\calC(\kdgCat)$ and $\calH(\kdgCat)$ as follows:
\begin{itemize}
\item
$\calC(\kdgCat)_0:= \{\calC(\calA) \mid \calA \in \kdgCat_0\}$.
\item
For any objects $\dgMod(\calA), \dgMod(\calB)$ of $\calC(\kdgCat)$,
$1$-morphisms from $\dgMod(\calA)$ to $\dgMod(\calB)$ are the $\k$-functors
$F \colon \dgMod(\calA) \to \dgMod(\calB)$
satisfying the condition
\begin{equation}
\label{eq:Hp-preserved}
F(\hprj(\calA)_0) \subseteq \hprj(\calB)_0
\end{equation}
When this is the case, we say that $F$ {\em preserves homotomically projectives}.
\item
$\calH(\kdgCat)_0:= \{\calH(\calA) \mid \calA \in \kdgCat_0\}$.
\item
For any objects $\calH(\calA), \calH(\calB)$ of $\calH(\kdgCat)$),
$1$-morphisms from $\calH(\calA)$ to $\calH(\calB)$ are
the $\k$-functors of the form $\udl{F}\, (:= \mathrm{st}(F))$
for some $\k$-functors $F \colon \dgMod(\calA) \to \dgMod(\calB)$
satisfying the condition \eqref{eq:Hp-preserved}.
\item
In both $2$-categories, the $2$-morphisms are the natural transformations between those $1$-morphisms.
\end{itemize}
\end{dfn}

\begin{rmk}
(1) By definition, the $2$-functor $\st \colon \kFRB \to \kTRI$
restricts to a $2$-functor
\[
\st \colon \calC(\kdgCat) \to \calH(\kdgCat).
\]

(2) In the definition above, unlike objects,
note that we defined the $1$-morphisms in $\calC(\kdgCat)$
not as the ``image'' of $1$-morphisms in $\kdgCat$ under $\calC$.

Nevertheless, pseudofunctors
\[
\calC \colon \kdgCat \to \kFRB \ \text{and}\ 
\calH \colon \kdgCat \to \kTRI
\]
restrict to pseudofunctors
\[
\calC \colon \kdgCat \to \calC(\kdgCat) \ \text{and}\ 
\calH \colon \kdgCat \to \calH(\kdgCat).
\]
Indeed,
let $F \colon \calA \to \calB$ be a $1$-morphism in $\kdgCat$.
It is enough to show that 
$\calC(F)$ is a $1$-morphism in $\calC(\kdgCat)$, i.e. that
$\calC(F)(\hprj(\calA)_0) \subseteq \hprj(\calB)_0$.
Let $P \in \hprj(\calA)_0$.
Then since $\calC(F)$ is a left adjoint to $\calC'(F)$ , we have
\[
\calH(\calB)(\calC(F)(P), A) \iso \calH(\calA)(P, \calC'(F)(A))
\]
for all acyclic objects $A$ of $\calC(\calB)$.
The right hand side is zero because $\calC'(F)$ is exact by Remark \ref{rmk:C'-exact}, and hence $\calC'(F)(A)$ is acyclic in $\calC(\calA)$.
This shows that $\calC(F)(P) \in \hprj(\calB)_0$.
Noting that $\calC(F) = \blank \ox_{\calA} ({}_F\calB)$ and
that ${}_A({}_F\calB) = \calB(\blank, F(A))$ is a projective $\calB$-module
for all $A \in \calA_0$,
this argument is generalized in Lemma \ref{lem:right-htp} below.

Since we would like to make the domain of a pseudofunctor $\bfL$ larger,
we adapted this definition of $1$-morphisms.
\end{rmk}

\begin{dfn}
Let $\calA$ and $\calB$ be small dg categories.
A $\calB$-$\calA$-bimodule $U$ is said to be {\em right homotopically projective}
if ${}_BU$ is a homotopically projective right dg $\calA$-module for all $B \in \calB_0$.
\end{dfn}

\begin{lem}
\label{lem:right-htp}
Let $U$ be a $\calB$-$\calA$-bimodule for dg categories $\calA$ and $\calB$.
Then the following are equivalent:
\begin{enumerate}
\item
$U$ is right homotopically projective.
\item
The dg functor $\blank \ox_{\calB}U \colon \DGMod(\calB) \to \DGMod(\calA)$
preserves homotopically projective objects.
\end{enumerate}
\end{lem}

\begin{proof}
(1) \implies (2).
For any $P \in \hprj(\calB)_0$ and any acyclic object $A \in \calH(\calA)_0$,
applying $H^0$ to the isomorphism
$\DGMod(\calA)(P\ox_{\calB}U, A) \iso
\DGMod(\calB)(P, \DGMod(\calA)(U,A))$, we have
\[
\calH(\calA)(P\ox_{\calB}U, A) \iso
\calH(\calB)(P, \DGMod(\calA)(U,A)).
\]
The right hand side is $0$
because for each $B \in \calB_0$, $(\DGMod(\calA)(U,A))(B)$ becomes acyclic, which is shown by
\[
H^i((\DGMod(\calA)(U,A))(B)) = H^i(\DGMod(\calA)({}_BU,A)) = \calH(\calA)({}_BU,A[i]) = 0
\]
for all $i \in \bbZ$.  Hence $P \ox_{\calB}U \in \hprj(\calA)$.

(2) \implies (1).
Let $B \in \calB_0$.
Then by (2), ${}_BU \iso {}_B\calB \ox_{\calB}U$ is homotopically projective
because ${}_B\calB$ is.
\end{proof}

\begin{dfn}
\label{dfn:L}
We further define pseudofunctors $\bfL$ and $\calD$ in the diagram
\[
\begin{tikzcd}
&\calC(\kdgCat) & \kFRB\\
\kdgCat &\calH(\kdgCat) & \kTRI\\
&\kuTRI^2
\Ar{2-1}{1-2}{"\dgMod"}
\Ar{1-2}{1-3}{hookrightarrow}
\Ar{2-1}{2-2}{"\calH"}
\Ar{2-2}{2-3}{hookrightarrow}
\Ar{2-1}{3-2}{"\calD"'}
\Ar{1-2}{2-2}{"\mathrm{st}"}
\Ar{1-3}{2-3}{"\mathrm{st}"}
\Ar{2-2}{3-2}{"\bfL"}
\end{tikzcd}.
\]
To define $\bfL$, consider a diagram
\[
\begin{tikzcd}
\calH(\calA) & \calH(\calB)
\Ar{1-1}{1-2}{"E", ""'{name=a}, bend left}
\Ar{1-1}{1-2}{"F"', ""{name=b}, bend right}
\Ar{a}{b}{"\al", Rightarrow}
\end{tikzcd}
\]
in $\calH(\kdgCat)$.
We set $\bfL(\calH(\calA)):= \calD(\calA)$.
Since $E(\hprj(\calA)_0) \subseteq \hprj(\calB)_0$ by definition of $\calH(\kdgCat)$,
$E$ restricts to a functor $E| \colon \hprj(\calA) \to \hprj(\calB)$,
and we can define $\bfL(E)$ as the
composite $\bfL(E):= \bfj_\calB \circ E| \circ \bfp_\calA$ as in the diagram
\[
\begin{tikzcd}
\hprj(\calA) & \hprj(\calB)\\
\calD(\calA) & \calD(\calB)
\Ar{2-1}{1-1}{"\bfp_\calA"}
\Ar{1-1}{1-2}{"E|"}
\Ar{1-2}{2-2}{"\bfj_\calB"}
\Ar{2-1}{2-2}{"\bfL(E)"', dashed}
\end{tikzcd}.
\]
Moreover, using the restriction $\al| \colon E| \To F|$ of $\al$,
we define $\bfL(\al)$ by setting
$\bfL(\al):= \bfj_\calB \circ \al| \circ \bfp_\calA$.

Then for any functors $\calH(\calA) \ya{E} \calH(\calB) \ya{E'} \calH(\calC)$
in $\calH(\kdgCat)$, we have
\begin{equation}
\label{eq:L-comp}
\bfL(E')\circ \bfL(E) = \bfL(E' \circ E).
\end{equation}
Indeed, since $\bfp_\calB \circ \bfj_\calB = \id_{\hprj(\calB)}$ (see \eqref{eq:pj=1}), we have
\[
\begin{aligned}
\bfL(E')\circ \bfL(E)
&= j_\calC \circ E'| \circ \bfp_\calB \circ \bfj_\calB \circ E| \circ \bfp_\calA\\
&= j_\calC \circ E'| \circ E| \circ \bfp_\calA = \bfL(E' \circ E).
\end{aligned}
\]
Also, for each $\calH(\calA) \in \calH(\kdgCat)$, we have a natural isomorphism
\[
\et_\calA  \colon  \id_{\calD(\calA)} \Longrightarrow 
Q_\calA \circ \si_\calA \circ\bfp_\calA
= Q_\calA \circ \id_{\calH(\calA)} \circ \si_\calA \circ \bfp_\calA
= \bfL(\id_{\calH(\calA)})
\]
given by the unit of the adjoint $\si_\calA \circ \bfp_\calA \dashv Q_\calA$.
These natural isomorphisms define a structure of a pseudofunctor for $\bfL$,
and it is easy to check that $\bfL$ is in fact a pseudofunctor.
Finally, we define $\calD$
as the composite
\[
\calD:= \bfL \circ \calH \colon \kdgCat \to \kuTRI^2,
\]
which is a pseudofunctor being a composite of pseudofunctors.
\end{dfn}

Since $\bfL$ is a pseudofunctor satisfying
\eqref{eq:L-comp}, we have the following.
This will be used in the proof of Theorem \ref{thm:characterization-1}.

\begin{lem}
\label{lem:L-2mor}
The pseudofunctor $\bfL \colon \calH(\kdgCat) \to \kuTRI^2$ strictly preserves
the vertical and the horizontal compositions of $2$-morphisms.
More precisely, let $E,F,G: \calH(\calA) \to \calH(\calB)$,
$E', F' \colon \calH(\calB) \to \calH(\calC)$,
and $\al \colon E \To F$, $\be \colon F \To G, \ga \colon E' \To F'$
be in $\calH(\kdgCat)$.
Then we have
\[
\begin{aligned}
\bfL(\be \bullet \al) &= \bfL(\be) \bullet \bfL(\al), \text{and}\\
\bfL(\ga \circ \al) &= \bfL(\ga) \circ \bfL(\al).
\end{aligned}
\]
\end{lem}

\begin{proof}
The first equality follows from the following computation:
$$
\begin{aligned}
\bfL(\be \bullet \al )&= \bfj_\calB \circ (\be \bullet \al)| \circ \bfp_\calA\\
&= \bfj_\calB \circ (\be| \bullet \al|) \circ \bfp_\calA\\
&= (\bfj_\calB \circ \be|\circ \bfp_\calA) \bullet
(\bfj_\calB \circ \al|\circ \bfp_\calA) = \bfL(\be) \bullet \bfL(\al).
\end{aligned}
$$
The second equlity holds by \eqref{eq:pj=1} as follows.
$$
\begin{aligned}
\bfL(\be) \circ \bfL(\al)
&= (\bfj_\calC \circ \ga|\circ \bfp_\calA) \circ
   (\bfj_\calC \circ \al|\circ \bfp_\calA)\\
&= \bfj_\calC \circ (\ga| \circ \al |) \circ \bfp_\calA)\\
&= \bfj_\calC \circ (\ga \circ \al)| \circ \bfp_\calA
= \bfL(\ga \circ \al).
\end{aligned}
$$
\end{proof}

\begin{dfn}
\label{dfn:L-dgfun}
(1) Define a $2$-subcategory $\DGMod(\kdgCat)$ of $\kdgCAT$ as follows:
Objects are the dg categories of the form $\DGMod(\calA)$ for
some $\calA \in \kdgCat_0$, $1$-morphisms are the dg functors 
$F \colon \DGMod(\calA) \to \DGMod(\calB)$ preserving homotopically projective objects
with $\calA, \calB \in \kdgCat_0$,
and $2$-morphisms are the dg natural transformations between these $1$-morphisms.

(2) By noting the fact that
for any $\calA \in \kdgCat_0$, we have
\[
\begin{aligned}
\ChMod(\calA)(X, Y) &= Z^0(\DGMod(\calA)(X,Y))\\
\calH(\calA)(X, Y) &= H^0(\DGMod(\calA)(X,Y))
\end{aligned}
\]
for all objects $X, Y \in \DGMod(\calA)_0 = \calC(\calA)_0 = \calH(\calA)_0$, we define 
$2$-functors $Z^0$ and \ $H^0$ in the 
the following diagram so that it turns out to be
strictly commutative:
\[
\begin{tikzcd}[column sep=20pt]
&\DGMod(\kdgCat)\\
\kdgCat &\dgMod(\kdgCat)\\
& \calH(\kdgCat)
\Ar{2-1}{1-2}{"\DGMod"}
\Ar{2-1}{2-2}{"\dgMod"}
\Ar{2-1}{3-2}{"\calH"'}
\Ar{1-2}{2-2}{"Z^0"}
\Ar{2-2}{3-2}{"\mathrm{st}"}
\Ar{1-2}{3-2}{"H^0", bend left, xshift=30pt}
\end{tikzcd}.
\]
For each $\DGMod(\calA) \in \DGMod(\kdgCat)_0$,
$Z^0(\DGMod(\calA)):= \calC(\calA)$ and $H^0(\DGMod(\calA))$ $:= \calH(\calA)$.
Next, let $E, F \colon$ $\DGMod(\calA) \to \DGMod(\calB)$ be dg functors in $\DGMod(\kdgCat)$
and $\al \colon E \To F$ a dg natural transformation.
We define $Z^0(E), Z^0(\al)$ and $H^0(E), H^0(\al)$ as follows.
We set
\[
(Z^0(E))(M):= E(M), \text{and } (H^0(E))(M):= E(M)
\]
for all $M \in \DGMod(\calA)_0 = \calC(\calA)_0 = \calH(\calA)_0$.
For each $M, N \in \DGMod(\calA)_0$, the dg functor
$E$ induces a chain map
\[
E(M, N) \colon \DGMod(\calA)(M,N) \to \DGMod(\calB)(E(M), E(N)).
\]
Using this we set
\[
\begin{aligned}
(Z^0(E))(M,N)&:= Z^0(E(M,N)) \colon \calC(\calA)(M,N) \to \calC(\calB)(E(M),E(N)), \text{and}\\
(H^0(E))(M,N)&:= H^0(E(M,N)) \colon \calH(\calA)(M,N) \to \calH(\calB)(E(M),E(N)).
\end{aligned}
\]
Finally, by noting that $\al_M \in \calC(\calB)(E(M), F(M))$, we set
\[
\begin{aligned}
(Z^0(\al))_M&:= \al_M\in \calC(E(M),F(M)), \text{and}\\
(H^0(\al))_M&:= H^0(\al_M)\in \calH(E(M),F(M))
\end{aligned}
\]
Note here that if  $\al$ above were just a general derived transformation, then neither $Z^0(\al)$ nor $H^0(\al)$ is defined.
This is a reason why we consider colax functors $X \colon I \to \kdgCat$ rather than
$X \colon I \to \kDGCat$ below.
\end{dfn}

\begin{rmk}
\label{rmk:L-dgfun}
Consider a dg functor $F \colon \DGMod(\calA) \to \DGMod(\calB)$
with $\calA, \calB \in \kdgCat_0$.
Here, we do not assume the condition $F(\hprj(\calA)_0) \subseteq \hprj(\calB)_0$.
Even in this case, $F$ induces a triangle functor
$H^0(F): \calH(\calA)\to \calH(\calB)$, and it is possible to define
\[
\bfL(H^0(F)) \colon \calD(\calA) \to \calD(\calB)
\]
by
$\bfL(H^0(F)):= Q_\calB \circ H^0(F) \circ \si_\calA \circ \bfp_\calA$ as in the left part of the diagram
\[
\begin{tikzcd}
\calH(\calA) & \calH(\calB) & \calH(\calC)\\
\hprj(\calA) & \hprj(\calB) & \hprj(\calC)\\
\calD(\calA) & \calD(\calB) & \calD(\calC)
\Ar{1-1}{1-2}{"H^0(F)"}
\Ar{1-2}{1-3}{"H^0(F')"}
\Ar{2-1}{1-1}{"\si_\calA"}
\Ar{3-1}{2-1}{"\bfp_\calA"}
\Ar{2-2}{1-2}{"\si_\calB"}
\Ar{3-2}{2-2}{"\bfp_\calB"}
\Ar{1-2}{3-2}{"Q_\calB", bend left, xshift=10pt}
\Ar{2-3}{1-3}{"\si_\calC"}
\Ar{3-3}{2-3}{"\bfp_\calC"}
\Ar{1-3}{3-3}{"Q_\calC", bend left, xshift=10pt}
\end{tikzcd},
\]
although $H^0(F)$ may be not in the domain of the pseudofunctor $\bfL$.

Of course, if $F$ preserves homotopically projective objects, then these two definitions of $\bfL(H^0(F))$
coincide because
we have the restriction $H^0F| \colon \hprj(\calA) \to \hprj(\calB)$ satisfying the commutativity
$H^0(F)\circ \si_\calA = \si_\calB \circ H^0(F)|$,
and $\bfj_\calB = Q_\calB \circ \si_\calB$.
In the following, we simply set $\bfL(F):= \bfL(H^0(F))$ (e.g., see Theorem \ref{thm:LH-eq}).

Let
$\DGMod(\calA) \ya{F} \DGMod(\calB) \ya{F'} \DGMod(\calC)$ be dg functors with
$\calA, \calB$ and $\calC$ small dg categories.
Then we can definie a natural transformation
$\bfL_{F',F} \colon \bfL(F') \circ \bfL(F) \To \bfL(F' \circ F)$ by
\begin{equation}
\label{eq:lax-comp-L}
\begin{aligned}
\bfL_{F',F}=(Q_{\calC}\circ H^0F') \circ \ep_{\calB} \circ (H^0F \circ \si_{\calA} \circ \bfp_{\calA}):\\
\overbrace{Q_{\calC} \circ H^0F' \circ \si_{\calB}\circ\bfp_{\calB}}^{\bfL(F')}\circ \overbrace{Q_{\calB}\circ H^0F \circ \si_{\calA} \circ \bfp_{\calA}}^{\bfL(F)}\\
\Longrightarrow
\overbrace{Q_{\calC} \circ (H^0F' \circ \id_{\calH(\calB)} \circ  H^0F) \circ \si_{\calA} \circ \bfp_{\calA}}^{\bfL(F'\circ F)},
\end{aligned}
\end{equation}
where $\ep_{\calB} \colon (\si_\calB \circ \bfp_{\calB}) \circ Q_\calB\To \id_{\calH(\calB)}$ is the the counit.
Then in general, the equality $\eqref{eq:L-comp}$ does not need to hold.
There are following two special cases.

Case 1. In the case that $F$ preserves homotopically projective objects.
In this case, the equality $\eqref{eq:L-comp}$ holds,
Indeed,
\[
\begin{aligned}
\bfL(F')\circ \bfL(F) &=
Q_{\calC} \circ H^0F' \circ \si_{\calB}\circ\bfp_{\calB} \circ
Q_{\calB}\circ H^0F \circ \si_{\calA} \circ \bfp_{\calA}\\
&=
Q_{\calC} \circ H^0F' \circ \si_{\calB}\circ\bfp_{\calB} \circ
Q_{\calB}\circ \si_{\calB} \circ (H^0F|) \circ \bfp_{\calA}\\
&=
Q_{\calC} \circ H^0F' \circ \si_{\calB}\circ\bfp_{\calB} \circ
\bfj_{\calB} \circ (H^0F|) \circ \bfp_{\calA}\\
&=
Q_{\calC} \circ H^0F' \circ \si_{\calB}\circ (H^0F|) \circ \bfp_{\calA}\\
&=
Q_{\calC} \circ H^0F'\circ H^0F \circ \si_{\calA} \circ \bfp_{\calA}
= \bfL(F'\circ F).
\end{aligned}
\]

Case 2. In the case that $F'$ preserves acyclic objects (hence preserves quasi-isomorphisms).
In this case, 
$\bfL_{F',F} \colon \bfL(F') \circ \bfL(F) \overset{\sim}{\To} \bfL(F' \circ F)$ is an isomorphism
because for each object $M \in \calD(\calA)$,
we have
\[
(\bfL_{F',F})_M = Q_{\calC}(H^0F'(\ep_{\calB,H^0F(\si_{\calA}(\bfp_{\calA}(M)))}),
\]
which turns out to be an isomorphism.
\end{rmk}

To make the remark above more precisely, we introduce the following three 2-categories
$\tilde{\DGMod}(\kdgCat)$, $\tilde{\calC}(\kdgCat)$ and $\tilde{\calH}(\kdgCat)$
that have more 1-morphisms than $\DGMod(\kdgCat)$, $\calC(\kdgCat)$ and $\calH(\kdgCat)$ have, resplectively.

\begin{dfn}
We define $2$-categories $\tilde{\DGMod}(\kdgCat)$, $\tilde{\calC}(\kdgCat)$
and \linebreak[4] $\tilde{\calH}(\kdgCat)$ as follows:
\begin{itemize}
\item
$\tilde{\DGMod}(\kdgCat)_0:= \{\DGMod(\calA) \mid \calA \in \kdgCat_0\}$.
\item
For any objects $\DGMod(\calA), \DGMod(\calB)$ of $\tilde{\DGMod}(\kdgCat)$,
$1$-morphisms from $\DGMod(\calA)$ to $\DGMod(\calB)$ are the dg functors 
$F \colon \DGMod(\calA) \to \DGMod(\calB)$ with $\calA, \calB \in \kdgCat_0$
(here we do not assume that they satisfy \eqref{eq:Hp-preserved}).
\item
$\tilde{\calC}(\kdgCat)_0:= \{\calC(\calA) \mid \calA \in \kdgCat_0\}$.
\item
For any objects $\dgMod(\calA), \dgMod(\calB)$ of $\tilde{\calC}(\kdgCat)$,
$1$-morphisms from $\dgMod(\calA)$ to $\dgMod(\calB)$ are the $\k$-functors
$F \colon \dgMod(\calA) \to \dgMod(\calB)$.
\item
$\tilde{\calH}(\kdgCat)_0:= \{\calH(\calA) \mid \calA \in \kdgCat_0\}$.
\item
For any objects $\calH(\calA), \calH(\calB)$ of $\tilde{\calH}(\kdgCat)$),
$1$-morphisms from $\calH(\calA)$ to $\calH(\calB)$ are
the $\k$-functors of the form $\udl{F}\, (:= \mathrm{st}(F))$
for some $\k$-functors $F \colon \dgMod(\calA) \to \dgMod(\calB)$.
\item
In each $2$-category, the $2$-morphisms are the natural transformations between those $1$-morphisms.
\end{itemize}
\end{dfn}

The following expands upon Remark \ref{rmk:L-dgfun}.

\begin{prp}
\label{prp:L-lax}
The $2$-functor $H^0 \colon \Cdg(\kdgCat) \to \calH(\kdgCat)$ is extended to
a $2$-functor $H^0 \colon \tilde{\Cdg}(\kdgCat) \to \tilde{\calH}(\kdgCat)$, and
the pseudofunctor $\bfL \colon \calH(\kdgCat) \to \kuTRI^2$ is extended to
a lax functor $\bfL \colon \tilde{\calH}(\kdgCat) \to \kuTRI^2$.
\end{prp}

\begin{proof}
The first assertion is straightforward.
For the second assertion, we summarize the structures of the lax functor $\bfL$.
For dg functors
$\DGMod(\calA) \ya{F} \DGMod(\calB) \ya{F'} \DGMod(\calC)$ with
$\calA, \calB$ and $\calC$ small dg categories.
The structure morphism $\bfL_{F',F} \colon \bfL(F') \circ \bfL(F) \to \bfL(F' \circ F)$ is given in
\eqref{eq:lax-comp-L}.
Also, for each $\calH(\calA) \in \tilde{\calH}(\kdgCat)$, we have a natural isomorphism
\[
\et_\calA  \colon  \id_{\calD(\calA)} \Longrightarrow 
Q_\calA \circ \si_\calA \circ\bfp_\calA
= Q_\calA \circ H^0(\id_{\DGMod(A)}) \circ \si_\calA \circ \bfp_\calA
= \bfL(\id_{\DGMod(\calA)})
\]
given by the unit of the adjoint $\si_\calA \circ \bfp_\calA \dashv Q_\calA$, which defines the remaining structure of the lax functor $\bfL$.
It is easy to check that $\bfL$ is in fact a lax functor which satisfies the axioms:

\begin{enumerate}
\item[(i)]
For each $a\colon i \to j$ in $I$ the following are commutative:
$$
\vcenter{
\xymatrix{
\bfL F\id_{\calD(\calA)}\ar@{=>}[r]^(.53){\bfL F\circ \eta_{\calA}}  \ar@{=}[rd]& \bfL F\bfL(\id)
\ar@{=>}[d]^{\bfL_{F,\id}}\\
& \bfL(F\id)
}}
\qquad\text{and}\qquad
\vcenter{
\xymatrix{
\id_{\calD(\calA)} \bfL F\ar@{=>}[r]^(.53){\eta_{\calA}\circ \bfL F} \ar@{=}[rd]& \bfL(\id)\bfL F
\ar@{=>}[d]^{\bfL_{\id,F}}\\
& \bfL(\id F)
}}\quad;\text{ and}
$$

\[
\begin{aligned}
\bfL_{F,\id}=(Q_{\calC}\circ H^0F) \circ \ep_{\calB} \circ (H^0\id \circ \si_{\calA} \circ \bfp_{\calA}):\\
\overbrace{Q_{\calC} \circ H^0F \circ \si_{\calB}\circ\bfp_{\calB}}^{\bfL(F)}\circ \overbrace{Q_{\calB}\circ H^0\id \circ \si_{\calA} \circ \bfp_{\calA}}^{\bfL(\id)}\\
\Longrightarrow
\overbrace{Q_{\calC} \circ (H^0F \circ \id_{\calH(\calB)} \circ  H^0\id) \circ \si_{\calA} \circ \bfp_{\calA}}^{\bfL(F\id)},
\end{aligned}
\]

\item[(ii)]

For dg functors
$\DGMod(\calA) \ya{F} \DGMod(\calB) \ya{F'} \DGMod(\calC)\ya{F''} \DGMod(\calE)$ with
$\calA, \calB,\calC$ and $\calE$ small dg categories. The following is commutative:
$$
\xymatrix@C=3em{
\bfL(F'')\circ\bfL(F') \circ \bfL(F)  \ar@{=>}[r]^(.53){\bfL(F'')\circ\bfL_{F',F}} \ar@{=>}[d]_{\bfL_{F'',F'}\circ\bfL F}& \bfL F''\circ \bfL(F' \circ F)
\ar@{=>}[d]^{\bfL_{F'',(F',F)}}\\
\bfL (F''\circ F') \circ \bfL F \ar@{=>}[r]_(.45){\bfL_{(F'',F'),F}}& \bfL(F''\circ F' \circ F).
}
$$
\end{enumerate}

\end{proof}

\begin{dfn}
\label{dfn:prj-Kb}
(1) The pseudofunctor $\calH \colon \kdgCat \to \kTRI$
restricts to a pseudofunctor $\hprj\colon \kdgCat \to \kTRI$
sending $\calA$ to $\hprj(\calA)$ for all $\calA \in \kdgCat_0$.

(2) 
For each $\calA \in \kdgCat$, we define 
$\perf(\calA)$  to be the smallest full triangulated subcategory of $\calD(\calA)$ closed under direct summands (and isomorphisms), and containing the representable functors $\calA(\blank,M)$ for all $M\in\calA_0$.
Hence note that the class of objects of $\perf(\calA)$ coincides with the class of compact objects in $\calD(\calA)$ by a theorem explained in \cite[Sect.\ 5]{Ke1}.
The category $\perf(\calA)$ is called the {\em perfect derived category} of $\calA$.
Then the pseudofunctor $\calD\colon \kdgCat \to \kuTRI^2$ restricts to
a pseudofunctor $\perf \colon \kdgCat \to \kuTRI^2$ sending $\calA$ to $\perf(\calA)$ for all $\calA \in \kdgCat_0$.
\end{dfn}

We cite the following theorem from \cite{Asa-13},
which is a useful tool to define new colax functors from an old one by composing with pseudofunctors.

\begin{thm}\label{comp-pseudofun}
Let $\bfB, \bfC$ and $\bfD$ be $2$-categories and $V \colon \bfC \to \bfD$ a
pseudofunctor.
Then the obvious correspondence
$$
\Colax(\bfB, V) \colon \Colax(\bfB, \bfC) \to \Colax(\bfB, \bfD)
$$
turns out to be a pseudofunctor.
\end{thm}

\begin{cor}\label{cor:colax-colax}
Let $X \colon I \to \kdgCat$
be a colax functor.
Then the following are colax functors again
\[
\begin{aligned}
&\text{The dg colax functor of }X:\ \DGMod(X) := \DGMod \circ X \colon I \to \kDGCAT,\\
&\text{The complex colax functor of }X:\ \ChMod(X):=\ChMod\circ X\colon I \to \kFRB,\\
&\text{The homotopy colax functor of }X:\ \calH(X):=\calH\circ X \colon I \to \kTRI,\\
&\text{The homotopically projecitve colax functor of }X:\ \hprj(X):=\hprj\circ X \colon I \to \kTRI,\\
&\text{The derived colax functor of }X:\ \calD(X):= \calD  \circ X \colon I \to \kuTRI^2, \text{ and}\\
&\text{The perfect derived colax functor of }X:\ \perf(X) := \perf \circ X \colon I \to \kuTRI^2.
\end{aligned}
\]
\end{cor}

\begin{rmk}
Let $X = (X, X_i, X_{b,a}) \in \Colax(I, \kdgCat)$.

(1) An explicit description of the complex colax functor
\[
\ChMod(X):=\ChMod\circ X = (\ChMod(X), \ChMod(X)_i, \ChMod(X)_{b,a}) \colon I \to \kFRB
\]
of $X$ is given as follows.

\begin{itemize}
\item
for each $i \in I_0$,
$\ChMod(X)(i) = \ChMod(X(i))$; and
\item
for each $a \colon i \to j$ in $I$, the functor
$\ChMod(X)(a) \colon \ChMod(X)(i) \to \ChMod(X)(j)$
is given by $\ChMod(X)(a) = \text{-}\ox_{X(i)}\ovl{X(a)}$,
where $\ovl{X(a)}$ is the $X(i)$-$X(j)$-bimodule 
$\ovl{X(a)}:= {}_{X(a)}X(j)$ (Notation \ref{ntn:bimodule} (3)).
\end{itemize}

(2) An explicit description of the derived colax functor
$\calD(X) \colon I \to  \kuTRI^2$ of $X$ is as follows. 

\begin{itemize}
\item
for each $i \in I_0$,
$\calD(X)(i) = \calD(X(i))$; and
\item For each $a\colon i \to j$ in $I$,
$\calD(X)(a)\colon \calD(X)(i) \to \calD(X)(j)$
is given by
$$\text{-}\Lox_{X(i)}\ovl{X(a)}:= \bfL(\text{-}\ox_{X(i)}\ovl{X(a)})\colon \calD(X(i)) \to \calD(X(j)).$$
\end{itemize}
Note that by the remark in Definition \ref{dfn:prj-Kb} (2), $\perf(X)$ is
a colax subfunctor of $\calD(X)$.
\end{rmk}

\begin{rmk} Recall that $\De$ is a diagonal functor.
Let $\calC \in \kdgCat_0$.
Then it is obvious by definitions that
$$
\De(\hprj(\calC)) = \hprj(\De(\calC)) \quad \text{and}\quad \De(\perf(\calC)) = \perf(\De(\calC)).
$$
\end{rmk}

\begin{prp}
\label{prp:precovering-preserved}
The pseudofunctor $\perf$ preserves $I$-precoverings, that is,
if $(F, \ps) \colon X \to \De(\calC)$ is an $I$-precovering in $\Colax(I, \kdgCat)$
with $\calC \in \kdgCat_0$, then
so is 
$$
\perf(F, \ps) \colon \perf(X) \to \De(\perf(\calC))
$$ in $\Colax(I, \kuTRI^2).$
\end{prp}

\begin{proof}

Let $i,j \in I_0$ and $M \in (\perf X(i))_0, N \in (\perf X(j))_0$.
It suffices to show that $\perf(F,\ps)$ induces an isomorphism
\[
\perf(F, \ps)^{(1)}_{M,N} \colon
\coprod_{a \in I(i,j)}\perf X(j)(M \Lox_{X(i)} \ovl{X(a)},N)
\to \perf\calC(M \Lox_{X(i)} \ovl{F(i)}, N \Lox_{X(j)} \ovl{F(j)}).
\]
By assumption, $(F, \ps)$ induces an isomorphism
$(F,\ps)^{(1)}_{x,y} \colon \coprod_{a\in I(i,j)}X(j)(X(a)x,y)$ $\to \calC(F(i)x, F(j)y)$
for all $x \in X(i)_0, y \in X(j)_0$.
Namely, 
\[
(F, \ps)^{(1)} \colon \coprod_{a\in I(i,j)}X(j)_{X(a)} \to {}_{F(j)}\calC_{F(i)}
\]
is a morphism of $X(j)$-$X(i)$-bimodules.
We first show the following.
\begin{clm-nn}
There exists an isomorphism
\[
\bfR\Hom_{\calC}(\ovl{F(i)}, N \Lox_{X(j)}\ovl{F(j)}) \to 
\coprod_{a\in I(i,j)}\bfR\Hom_{X(j)}(\ovl{X(a)}, N)).
\]
\end{clm-nn}
Indeed,
this is given by the composite of the following isomorphisms:
\[
\begin{aligned}
\bfR\Hom_{\calC}(\ovl{F(i)}, N \Lox_{X(j)}\ovl{F(j)}) &=
\bfR\Hom_{\calC}({}_{F(i)}\calC, N \Lox_{X(j)} {}_{F(j)}\calC)\\
&\overset{(\rm a)}{\to}
N \Lox_{X(j)} {}_{F(j)}\calC_{F(i)}\\
&\overset{(\rm b)}{\to}
N \Lox_{X(j)} \dcoprod_{a \in I(i,j)} X(j)_{X(a)}\\
&\overset{(\rm c)}{\to}
\dcoprod_{a \in I(i,j)} N \Lox_{X(j)} X(j)_{X(a)}\\
&\overset{(\rm d)}{\to}
\dcoprod_{a \in I(i,j)} N_{X(a)}\\
&\overset{(\rm e)}{\to}
\dcoprod_{a \in I(i,j)} \bfR\Hom_{X(j)}({}_{X(a)}X(j),N)\\
&=
\dcoprod_{a \in I(i,j)} \bfR\Hom_{X(j)}(\ovl{X(a)},N),
\end{aligned}
\]
where (a) is obtained by the Yoneda lemma, (b) is an isomorphism,
induced from  $((F,\ps)^{(1)})\inv$, 
(c) is the natural isomorphism induced by the cocontinuity of the tensor product,
(d) comes from the property of the tensor product, and
(e) is given by the Yoneda lemma.
Now, it is not hard to verify the commutativity of the following diagram:
\[\footnotesize
\xymatrix{
\dcoprod_{a \in I(i,j)}\perf X(j)(M \Lox_{X(i)} \ovl{X(a)},N) &
 \perf\calC(M \Lox_{X(i)} \ovl{F(i)}, N \Lox_{X(j)} \ovl{F(j)})\\
\dcoprod_{a \in I(i,j)}\perf X(i)(M ,\bfR\Hom_{X(j)}(\ovl{X(a)}, N)) &
 \perf X(i)(M, \bfR\Hom_{\calC}(\ovl{F(i)}, N \Lox_{X(j)} \ovl{F(j)})\\
\perf X(i)(M , \dcoprod_{a \in I(i,j)} \bfR\Hom_{X(j)}(\ovl{X(a)}, N)),
\ar"1,1";"1,2"^{\perf(F, \ps)^{(1)}_{M,N}}
\ar"1,1";"2,1"^{\simeq}_{(\rm i)}
\ar"1,2";"2,2"^{\simeq}_{(\rm ii)}
\ar"2,1";"3,1"^{\simeq}_{(\rm iii)}
\ar"2,2";"3,1"_{\simeq}^{(\rm iv)}
}
\]
where the isomorphisms (i) and (ii) are given by adjoints,
and (iii) is the natural morphism, which is an isomorphism because $M$ is compact,
and (iv) is an isomorphism given by the claim above.
Hence $\perf(F, \ps)^{(1)}_{M,N}$ is an isomorphism.
\end{proof}

\begin{prp}
\label{prp:hprj}
The homotopically projecitve colax functor $\hprj$ induce a colax funtor
$$
\begin{aligned}
\hprj(P, \ps) \colon \hprj(X) \to \De(\hprj(\Gr_I X))
\end{aligned}
$$ in $\Colax(I, \kuTRI^2).$
\end{prp}

\begin{proof}
For each $i\in I_0$, if $U$ is in $\hprj(X(i))$,
then 
$$
\hprj(P(i))(U)\iso U\ox_{X(i)}\ovl{P(i)}\\
=U \ox_{X(i)}(\Gr_I X)(\blank, P(i)(?))
$$
Therefore, if $Y$ is an acyclic $\Gr_I X$-module, 
$$
\begin{aligned}
\calH(\Gr_I X)(\hprj(P(i))(U), Y)\iso \calH(\Gr_I X)(U \ox_{X(i)}(\Gr_I X)(\blank, P(i)(?)),Y)\\\iso 
\calH(X(i))(U, \calH(\Gr_I X)((\Gr_I X)(\blank, P(i)(?)),Y))=0.
\end{aligned}
$$
Then $
\hprj(P(i))(U)$ is in $\hprj(\Gr_I X)$. 
For each $a \colon i \to j$ in $I_1$, the dg natural transformation
$\ph(a)\colon P(i) \Rightarrow P(j)X(a)$
$$\xymatrix{
X(i) & \Gr_I X\\
X(j) & \Gr_I X
\ar^{P(i)} "1,1";"1,2"
\ar_{P(j)} "2,1";"2,2"
\ar_{X(a)} "1,1";"2,1"
\ar@{=} "1,2";"2,2"
\ar@{=>}_{\ph(a)}"1,2";"2,1"
}
$$
is defined by $\ph(a)x:= (\de_{b,a} \id_{X(a)x})_{b \in I(i,j)}$ for all $x \in X(i)_0$, then we have the following diagram
$$\xymatrix{
\hprj(X(i)) & \hprj(\Gr_I X)\\
\hprj(X(j)) & \hprj(\Gr_I X).
\ar^{\hprj(P(i))} "1,1";"1,2"
\ar_{\hprj(P(j))} "2,1";"2,2"
\ar_{\hprj(X(a))} "1,1";"2,1"
\ar@{=} "1,2";"2,2"
\ar@{=>}_{\hprj(\ph(a))}"1,2";"2,1"
}
$$
This completes the proof.
\end{proof}

\begin{dfn}[Quasi-equivalences \cite{Ke3}]
\label{dfn:q-eq}
Let $\calA, \calB$ be small dg categories and
$E \colon \calA \to \calB$ a dg functor.
Then $E$ is called a {\it quasi-equivalence} if
\begin{enumerate}
\item
The restriction
$E_{X,Y}\colon \calA(X,Y) \to \calB(E(X), E(Y))$  of $E$ to $\calA(X,Y)$ is a quasi-isomorphism for all $X, Y \in \calA_0$; and
\item
The induced functor $H^0(E): H^0(\calA)\to H^0(\calB)$ is an equivalence.
\end{enumerate}
We say that $\calA$ is \emph{quasi-equivalent to} $\calB$ if
there exists a quasi-equivalence $E \colon \calA \to \calB$.
\end{dfn}

\begin{rmk}
\label{rmk:q-eq-rel}
The relation for small dg categories $\calA$ and be
that $\calA$ is quasi-equivalent to $\calB$ 
is clearly reflexive and transitive,
but we do not know whether this relation is symmetric or not.
\end{rmk}

\begin{dfn}\label{dfn:dg tilting}
For a triangulated category $\calU$ and a class of objects $\calV$ in $\calU$,
we denote by $\thick_\calU \calV$ (resp.\ $\Loc_\calU \calV$)
the smallest full triangulated subcategory of $\calU$
closed under direct summands (resp.\ infinite direct sums) that contains $\calV$.

Let $\calA$ be a small dg category, and $\calT$ a full dg subcategory of {$\DGMod(\calA)$}.
Then $\calT$ is called a {\em tilting dg subcategory} for $\calA$, if
\begin{enumerate}
\item
$\calT_0\subseteq \perf(\calA)_0\bigcap\hprj(\calA)_0$,
\ (hence every $T \in \perf(\calA)_0$ is a compact object in $\calD(\calA)$\footnote{Recall that an object $T$ of $\calD(\calA)$ is in $\perf(\calA)$ if and only if $T$ is a compact object in $\calD(\calA)$ by Neeman's theorem \cite[Theorem 5.3]{Ke1}.}); and
\item
$\thick_{\calD(\calA)}(\calT_0)=\perf(\calA)$ (equivalently $\Loc_{\calD(\calA)}(\calT)= \calD(\calA)$).
\end{enumerate}
Thus, in Keller's terminology in \cite{Ke1}, $\calT$ is tilting
if and only if $\calT_0$ forms a set of small (= compact) generators for $\calD(\calA)$.
\end{dfn}

\section{Derived equivalences of colax functors}

In this section, we define necessary terminologies such as $2$-quasi-isomorphisms for $2$-morphisms,
quasi-equivalences for $1$-morphisms, and the derived $1$-morphism
$\bfL (\check{F}, \check{\ps}) \colon$ $\calD(X) \to \calD(X')$ of a $1$-morphism
$(F, \ps) \colon X \to X'$ between colax functors,
and show the fact that the derived $1$-morphism of a quasi-equivalence between colax functors $X$, $X'$
turns out to be an equivalence between derived dg module colax functors of $X$, $X'$.
Finally, we give definitions of tilting subfunctors and of derived equivalences.

\begin{dfn}
Let $\bfC$ be a $2$-category and $(F, \ps) \colon X \to X'$
a $1$-morphism in the $2$-category $\Colax(I, \bfC)$.
Then $(F, \ps)$ is called $I$-{\em equivariant}
if for each $a \in I_1$, $\ps(a)$ is a $2$-isomorphism in $\bfC$.
\end{dfn}

We cite the following without a proof.
\begin{lem}[\cite{Asa-a}]
\label{colax-eq}
Let $\bfC$ be a $2$-category and $(F, \ps) \colon X \to X'$
a $1$-morphism in the $2$-category $\Colax(I, \bfC)$.
Then $(F, \ps)$ is an equivalence in $\Colax(I, \bfC)$
if and only if
\begin{enumerate}
\item
For each $i \in I_0$, $F(i)$
is an equivalence in $\bfC$; and
\item
For each $a \in I_1$, $\ps(a)$ is a $2$-isomorphism in $\bfC$
$($namely, $(F,\ps)$ is $I$-equivariant$)$.
\end{enumerate}
\end{lem}

To define the notion of $2$-quasi-isomorphisms in $\kdgCat$, we need the following statement.

\begin{lem}
\label{lem:char-q-eq}
Consider a $2$-morphism $\al$ in the $2$-category $\kdgCat$ as in
\[
\begin{tikzcd}
\calA & \calB
\Ar{1-1}{1-2}{"E", ""'{name=a}, bend left}
\Ar{1-1}{1-2}{"F"', ""{name=b}, bend right}
\Ar{a}{b}{"\al", Rightarrow}
\end{tikzcd}.
\]
We adapt Notation \ref{ntn:bimodule} (3), e.g., $\ovl{E}:= {}_E\calB$, $\ovl{\al}:= {}_{\al}\calB$
and $\ovl{E}^*:= \calB_E, \ovl{\al}^*:= \calB_{\al}$.
Note that since $\al$ is a dg natural transformation,
both $\ovl{\al}$ and $\ovl{\al}^*$ are $0$-cocyle morphisms
by Remark \ref{rmk:0-cocycle}, and
hence it is possible to define $\blank\Lox_{\calA}\ovl{\al}$, $H^0{}_{\al_x}\calB$ and so on.
Then the following are equivalent.

\begin{enumerate}
\item
The $2$-morphism $\text{-}\Lox_{\calA} \ovl{\al}$ in the diagram
\[
\begin{tikzcd}[column sep=60pt]
\calD(\calA) & \calD(\calB)
\Ar{1-1}{1-2}{"\text{-}\Lox_{\calA} \ovl{E}", ""'{name=a}, bend left}
\Ar{1-1}{1-2}{"\text{-}\Lox_{\calA} \ovl{F}"', ""{name=b}, bend right}
\Ar{a}{b}{"\text{-}\Lox_{\calA} \ovl{\al}", Rightarrow}
\end{tikzcd}
\]
is a $2$-isomorphism in $\kuTRI^2$.

\item
$H^0{}_{\al_x}\calB \colon {}_{E(x)}\calB \to {}_{F(x)}\calB$ is a quasi-isomorphism in $\calH(\calB)$
for all $x \in \calA_0$.

\item
The $2$-morphism $\ovl{\al}^* \Lox_{\calA} \text{-}$ in the diagram
\[
\begin{tikzcd}[column sep=60pt]
\calD(\calA\op) & \calD({\calB}\op)
\Ar{1-1}{1-2}{"\ovl{F}^* \Lox_{\calA} \text{-}", ""'{name=a}, bend left}
\Ar{1-1}{1-2}{"\ovl{E}^* \Lox_{\calA} \text{-}"', ""{name=b}, bend right}
\Ar{a}{b}{"\ovl{\al}^* \Lox_{\calA} \text{-}", Rightarrow}
\end{tikzcd}
\]
is a $2$-isomorphism in $\kuTRI^2$.

\item
$H^0\calB_{\al_x} \colon \calB_{E(x)} \to \calB_{F(x)}$ is a quasi-isomorphism in $\calH({\calB}\op)$.
\end{enumerate}
\end{lem}

\begin{proof}
(1) $\Rightarrow$ (2).
Let $x \in \calA_0$.
Note that we have
${}_x\calA\Lox_{\calA} \ovl{\al} \iso Q_{\calB}H^0 {}_{\al_x}\calB$, 
which is an isomorphism in $\calD(\calB)$ if and only if
$H^0{}_{\al_x}\calB$ is a quasi-isomorphism in $\calH(\calB)$.
Hence (2) follows from (1) by applying (1) to the representable functor ${}_x\calA$.

(2) $\Rightarrow$ (1).
Let $\calU$ be the full subcategory of $\calD(\calA)$ consisting of objects $M$ satisfying the condition
that $M \Lox_{\calA}\ovl{\al} \colon M \Lox_{\calA}\ovl{E} \to M \Lox_{\calA}\ovl{F}$ is an isomorphism.
Then by (2) we have ${}_x\calA \in \calU$ for all $x \in \calA_0$.
Here, it is easy to show that
$\calU$ is a triangulated subcategory of $\calD(\calA)$ and that
$\calU$ is closed under isomorphisms and direct sums with
small index sets.
Therefore we have $\calU = \calD(\calA)$,
which means that (1) holds. 

(2) $\Rightarrow$ (4).
Assume that
$H^0{}_{\al_x}\calB \colon {}_{E(x)}\calB \to {}_{F(x)}\calB$ 
is a quasi-isomorphism in $\calH(\calB)$.
Then $Q_{\calB}H^0{}_{\al_x}\calB$ is an isomorphism in $\calD(\calB)$.
We set $\Hom_{\calB}(\cdot, \blank):= \DGMod(\calB)(\cdot, \blank)$.
Then the functor
\[
\RHom_{\calB}(\cdot, {}_{\calB}\calB_{\calB}):= Q_{\calB}H^0\DGMod(\calB)(\bfp_{\calB}(\cdot), {}_{\calB}\calB_{\calB}) \colon
\calD(\calB) \to \calD({\calB}\op)
\]
sends the isomorphism $Q_{\calB}H^0{}_{\al_x}\calB$ to an isomorphism
$$
\begin{aligned}
\RHom_{\calB}({}_{\al_x}\calB,{}_{\calB}\calB_{\calB}) &\colon
\RHom_{\calB}({}_{F(x)}\calB,{}_{\calB}\calB_{\calB})\\
&\qquad\to \RHom_{\calB}({}_{E(x)}\calB,{}_{\calB}\calB_{\calB}),
\end{aligned}
$$
in $\calD({\calB}\op)$,
which is given by
$$
\begin{aligned}
Q_{\calB\op}H^0\Hom_{\calB}({}_{\al_x}\calB,{}_{\calB}\calB_{\calB}) &\colon
\Hom_{\calB}({}_{F(x)}\calB,{}_{\calB}\calB_{\calB})\\
&\qquad\to \Hom_{\calB}({}_{E(x)}\calB,{}_{\calB}\calB_{\calB}),
\end{aligned}
$$
in $\calD(\calB\op)$ (see Remark \ref{dfn:L-dgfun} for $H^0$).
By the Yoneda lemma, it is isomorphic to
$$
Q_{\calB\op}H^0\calB_{\al_x} \colon \calB_{F(x)} \to \calB_{E(x)}
$$
and is an isomorphism in $\calD({\calB}\op)$.
As a consequence, $H^0\calB_{\al_x}$ is a quasi-isomorphism in $\calH({\calB}\op)$.


(4) $\Rightarrow$ (2).
This is proved in the same way as in the converse direction.

(3) $\Leftrightarrow$ (4).
The same proof for the equivalence (1) $\Leftrightarrow$ (2) works
also for this case.
\end{proof}

\begin{dfn}
\label{q-eq}
Let $E, F \colon  \calA \to \calB$ be $1$-morphisms and
$\al \colon E\To F$
a $2$-morphism in the $2$-category $\kdgCat$.
Then $\al$ is called a {\em $2$-quasi-isomorphism} in
$\kdgCat$
if one of the statements (1), \ldots, (4) in Lemma \ref{lem:char-q-eq} holds.
\end{dfn}

\begin{rmk}
We can use the condition (2) above to check whether 
$\al$ is a $2$-quasi-equivalence.
Once it is checked, we can use the property (1).
\end{rmk}

\begin{dfn}
\label{dfn:der-eq-criterion}
Let $(F, \ps) \colon X \to X'$ be a $1$-morphism in
$\Colax(I, \kdgCat)$.
Then $(F, \ps)$ is called a {\em quasi-equivalence $1$-morphism} if
\begin{enumerate}
\item
For each $i \in I_0$, $F(i): X(i) \to  X'(i)$
is a quasi-equivalence; and
\item
For each $a \in I_1$, $\ps(a)$ is a 2-quasi-isomorphism (see Definition \ref{q-eq}).
\end{enumerate}
See the diagram below to understand the situation:
$$
\xymatrix@R=30pt@C=40pt{
X(i)&X'(i)\\
X(j)&X'(j).
\ar"1,1";"2,1"_{X(a)}
\ar"1,2";"2,2"^{X'(a)}
\ar"1,1";"1,2"^{F(i)}_{\text{\rm q-eq}}
\ar"2,1";"2,2"_{F(j)}^{\text{\rm q-eq}}
\ar@{=>}"1,2";"2,1"^{\text{\rm 2-qis}}_{\ps(a)}
}
$$
\end{dfn}

\begin{rmk}
In the above, consider the condition
\begin{enumerate}
\item[$(2')$]
For each $a \in I_1$, $\ps(a)$ is a $2$-isomorphiusm.
\end{enumerate}
Then obviously $(2')$ implies $(2)$.
Therefore, a $1$-morphism $(F, \ps)$ satisfying (1) and $(2')$ can be called
an {\em $I$-equivariant quasi-equivalence $1$-morphism}.
\end{rmk}

\begin{dfn}
\label{dfn:std-der-eq-2}
Let $X, X' \in \Colax(I, \kdgCat)$, and
$(\bfF, \bfps) \colon \Cdg(X)\to \Cdg(X')$
be in $\Colax(I, \kdgCAT)$\footnote{
We distinguish 1-morphisms between $X$ and $X'$ in $\Colax(I, \kdgCat)$
and 1-morphisms between $\Cdg(X)$ and $\Cdg(X')$ in $\Colax(I, \kdgCAT)$
by using bold face fonts for the latter.}.
Then we define a $1$-morphism
\[
\bfL(\bfF, \bfps):= (\bfL\bfF, \bfL\bfps) \colon \calD(X') \to \calD(X)
\]
in $\kuTRI^2$ as follows.
For each $a \colon i \to j$ in $I$, $(\bfL \bfF)(i):= \bfL(\bfF(i))$, and $(\bfL\bfps)(a)$ is
defined by the following commutative diagram:
\[
\begin{tikzcd}
(\blank\Lox_{X'(i)}\ovl{X'(a)}) \circ \bfL \bfF(i) & \bfL \bfF(j) \circ (\blank\Lox_{X'(i)}\ovl{X(a)})\\
\bfL((\blank\ox_{X(i)}\ovl{X'(a)} \circ \bfF(i)) & \bfL(\bfF(j) \circ (\blank\ox_{X(i)}\ovl{X(a)})),
\Ar{1-1}{1-2}{"(\bfL \bfps)(a)", Rightarrow}
\Ar{2-1}{2-2}{"\bfL(\bfps(a))", Rightarrow}
\Ar{1-1}{2-1}{Rightarrow}
\Ar{1-2}{2-2}{equal}
\end{tikzcd}
\]
where the vertical arrow on the right is the identity
(see Case 1 in Remark \ref{rmk:L-dgfun}), and that
on the left is
$\bfL_{(\blank\ox_{X(i)}\ovl{X(a)}), \bfF(i)}$, the structure morphism of the lax functor $\bfL$ (see Remark \ref{rmk:L-dgfun}),
which turns out to be the identity if $\bfF(i)$ preserves homotopically projective objects.
\end{dfn}

\begin{dfn}
Let $(F, \ps) \colon X \to X'$ be 
a 1-morhism
in $\Colax(I, \kdgCat)$.
This yields a $1$-morphism
\[
(\check{F}, \check{\ps}):= ((\check{F}(i))_{i\in I_0}, (\check{\ps}(a))_{a\in I_1}) \colon \DGMod(X) \to \DGMod(X'),
\]
in $\Colax(I, \kdgCAT)$, which defines  a 1-morphism
\[
\bfL(\check{F}, \check{\ps}) = \calD((F, \ps)) \colon \calD(X) \to \calD(X')
\]
in $\Colax(I, \kuTRI^2)$.
The explicit forms of $(\check{F}, \check{\ps})$ and $\bfL(\check{F}, \check{\ps}):= (\bfL \check{F}, \bfL\check{\ps})$ are given as follows.
For each $i \in I_0$, using the dg $X(i)$-$X'(i)$-bimodule $\ovl{F(i)}:={}_{F(i)}X'(i)$, we define
a dg functor
\[
\check{F}(i):= \blank\ox_{X'(i)}\ovl{F(i)} \colon \DGMod(X(i)) \to \DGMod(X'(i)).
\]
Note here that the bimodule $\ovl{F(i)}$ is right homotopically projective.
Hence by Lemma \ref{lem:right-htp}, $\check{F}(i)$ preserves homotopically projective objects.
This defines a triangle functor
\[
(\bfL \check{F})(i) := \bfL(\check{F}(i)) = \blank\Lox_{X(i)} \ovl{F(i)} \colon \calD(X(i)) \to \calD(X'(i)).
\]
Next let $a \colon i \to j$ be a morphism in $I$.
Then $\ps(a) \colon X'(a) F(i) \To F(j)X(a)$
induces a morphism of $X'(j)$-$X(i)$-bimodules
\[
\ovl{\ps(a)} \colon \ovl{X'(a) F(i)} \to \ovl{F(j)X(a)},
\]
where we adapt Notation \ref{ntn:bimodule} (3),
e.g., $\ovl{X'(a) F(i)}:={}_{X'(a) F(i)}X'(j)$,
which induces the diagram
\begin{equation}\label{eq:ovl-psi}
\vcenter{
\xymatrix{
(\blank\otimes_{X(i)}\ovl{F(i)})\otimes_{X'(i)} \ovl{X'(a)}& (\blank\otimes_{X(i)}\ovl{{X(a)}})\otimes_{X(j)}\ovl{F(j)}\\
 \blank\otimes_{X(i)}\ovl{X'(a) F(i)}& \blank\otimes_{X(i)}\ovl{F(j)X(a)}
\ar@{==>}"1,1";"1,2"^{\check{\ps}(a)}
\ar@{=>}"2,1";"2,2"_{\blank\otimes_{X(i)}\ovl{\ps(a)}}
\ar@{=>}"1,1";"2,1"_{\sim}
\ar@{=>}"1,2";"2,2"^{\sim}
}}
\end{equation}
of 2-morphisms in $\DGMod(\kdgCat)$, where the vertical morphisms are natural isomorphisms.
Then $\check{\ps}(a)$ is defined as the 
unique 2-morphism making this diagram commutative, which is usually identified with
$\blank\otimes_{X(i)}\ovl{\ps(a)}$.
This gives us the diagram
\[
\xymatrix@C=50pt{
\DGMod(X(i)) & \DGMod(X'(i))\\
\DGMod(X(j)) & \DGMod(X'(j)).
\ar"1,1";"1,2"^{\check{F}(i)}
\ar"2,1";"2,2"_{\check{F}(j)}
\ar"1,1";"2,1"_{\blank\otimes_{X(i)}\ovl{X(a)}=\DGMod(X(a))}
\ar"1,2";"2,2"^{\blank\otimes_{X'(i)}\ovl{X'(a)}=\DGMod(X'(a))}
\ar@{=>}"1,2";"2,1"_{\check{\ps}(a)}
}
\]
and the $1$-morphism $(\check{F}, \check{\ps}) \colon \DGMod(X) \to \DGMod(X')$.
By Lemma \ref{lem:L-2mor},
the pseudofunctor $\bfL \circ H^0$ sends the diagram \eqref{eq:ovl-psi} to the commutative diagram
\begin{equation}
\label{eq:L-dot-Lox}
\vcenter{
\xymatrix@C=50pt{
(\blank\Lox_{X(i)}\ovl{F(i)})\Lox_{X'(i)}\ovl{X'(a)} & (\blank\Lox_{X(i)}\ovl{{X(a)}})\Lox_{X(j)}\ovl{F(j)}\\
\blank\Lox_{X(i)}\ovl{X'(a) F(i)}& \blank\Lox_{X(i)}\ovl{F(j)X(a)}
\ar@{=>}"1,1";"1,2"^{\bfL(\check{\ps}(a))}
\ar@{=>}"2,1";"2,2"_{\blank\Lox_{X(i)}\ovl{\ps(a)}}
\ar@{=>}"1,1";"2,1"_{\sim}
\ar@{=>}"1,2";"2,2"^{\sim}
}}
\end{equation}
in $\kuTRI^2$.
Using this we set
\[
(\bfL \check{\ps})(a) := \bfL (\check{\ps}(a)) 
\colon (\blank\Lox_{X'(i)}\ovl{X'(a)}) \circ \bfL \check{F}(i) \To \bfL \check{F}(j) \circ (\blank\Lox_{X(i)}\ovl{X(a)}),
\]
which gives us the diagram
\[
\xymatrix@C=50pt{
\calD(X(i)) & \calD(X'(i))\\
\calD(X(j)) & \calD(X'(j)).
\ar"1,1";"1,2"^{(\bfL \check{F})(i)}
\ar"2,1";"2,2"_{(\bfL \check{F})(j)}
\ar"1,1";"2,1"_{\blank\Lox_{X(i)}\ovl{X(a)}=\calD(X(a))}
\ar"1,2";"2,2"^{\blank\Lox_{X'(i)}\ovl{X'(a)}=\calD(X'(a))}
\ar@{=>}"1,2";"2,1"_{(\bfL\check{\ps})(a)}
}
\]
and the $1$-morphism $(\bfL \check{F}, \bfL \check{\ps}) \colon \calD(X) \to \calD(X')$.
\end{dfn}

The following says that a quasi-equivalence between colax functors
induces a derived equivlence between them, which will be important for our main result.

\begin{prp}
\label{prp:der-eq-criterion}
Let $(F, \ps) \colon X \to X'$ be a quasi-equivalence
1-morphism in $\Colax(I,$ $\kdgCat)$.
Then
$\bfL (\check{F}, \check{\ps})
\colon \calD(X) \to \calD(X')$
is an equivalence in $\Colax(I, \kuTRI^2)$.
\end{prp}

\begin{proof}
Let $i \in I_0$.
Then since $F(i) \colon X(i) \to X'(i)$ is a quasi-equivalence,
we have
$$
(\bfL \check{F})(i) := \blank\Lox_{X(i)} \ovl{F(i)} \colon \calD(X(i)) \to \calD(X'(i))
$$ 
is an equivalence of triangulated categories by Theorem \ref{thm:q-eq-der-eq}.

Let $a \colon i \to j$ be a morphism in $I$.
Then since
\[
\ps(a) \colon X'(a) F(i) \To F(j)X(a)
\] 
is a 2-quasi-isomorphism, $\blank\Lox_{X(i)}\ovl{\ps(a)}$ is a $2$-isomorphism by definition.
Hence by the commutative diagram \eqref{eq:L-dot-Lox},
\[
(\bfL \check{\ps})(a)
\colon (\blank\Lox_{X'(i)}\ovl{X'(a)}) \circ \bfL \check{F}(i) \To \bfL \check{F}(j) \circ (\blank\Lox_{X(i)}\ovl{X(a)}).
\] 
is a 2-isomorphism.
Therefore, by Lemma \ref{colax-eq}, $\bfL (\check{F}, \check{\ps})$ is an equivalence in $\Colax(I, \kuTRI^2)$.
\end{proof}

A dg $\k$-category $\calA$ is called $\k$-{\em projective}
(resp.\ $\k$-{\em flat}) if
$\calA(x,y)$ are dg projective (resp.\ flat) $\k$-modules if for any $x,y\in\calA_0$, the complex $\calA(x,y)$ is homotopically projective (resp.\ flat) in $\calH(\k)$, that is,
$\Hom_\k(\calA(x,y),\blank)$ (resp.$\blank\ox_{\k}\calA(x,y)$) sends acyclic complexes to acyclic ones.

\begin{dfn}\label{dfn:tilting-colax 1}
Let $X\colon I \to \kdgCat$ be a colax functor.
\begin{enumerate}
\item
$X$ is called $\k$-{\em projective} (resp.\ $\k$-{\em flat})
if $X(i)$ are $\k$-projective (resp.\ $\k$-flat) for all $i \in I_0$.

\item
Let $Y, Y' \colon I \to \kdgCAT$ be colax functors.
Then $Y'$ is called a {\em colax subfunctor} of $Y$ if
there exists a $I$-{\em equivariant inclusion}
$1$-morphism $Y' \to Y$,
namely, a $1$-morphism $(\si, \ro) \colon Y' \to Y$
such that $\si(i) \colon Y'(i) \to Y(i)$ is the inclusion for each $i \in I_0$, and $\ro(a) \colon Y(a) \si(i) \To \si(j) Y'(a)$ is
an $2$-isomorphism (i.e., a dg natural isomorphism)
for each morphism $a \colon i \to j$ in $I$.

\item
A colax subfunctor $\calT$ of $\DGMod(X)$
is called a {\em tilting colax functor} for $X$
if for each $i \in I_0$,
$\calT(i)\subseteq \DGMod(X(i))$ is a tilting dg subcategory for $X(i)$ (see Definition \ref{dfn:dg tilting}).
See the diagram below for $(\si, \ro)$:
$$
\xymatrix{
\calT(i) &\DGMod(X(i)) \\
\calT(j) &\DGMod(X(j)).
\ar"1,1";"2,1"_{\calT(a)}
\ar@{^{(}->}"1,1";"1,2"^{\si(i)}
\ar@{^{(}->}"2,1";"2,2"_{\si(j)}
\ar"1,2";"2,2"^{\DGMod(X(a))}
\ar@{=>}"1,2";"2,1"^{\ro(a)}_{\sim}
}
$$
\end{enumerate}
\end{dfn}

\begin{dfn}
\label{dfn:der-eq}
Let $X, X' \in \Colax(I, \kdgCat)$.
Then $X$ and $X'$ are said to be {\em derived equivalent} 
(we denoted this fact by $X \dereq X$)
if
$\calD(X)$ and $\calD(X')$ are equivalent
in the 2-category $\Colax(I, \kuTRI^2)$.
Note by Lemma \ref{colax-eq} that this is the case if and only if
there exists a $1$-morphism
$(F, \ps): \calD(X)\to \calD(X')$ in $\Colax(I, \kuTRI^2)$
such that
\begin{enumerate}
\item
For each $i \in I_0$, $F(i):\calD(X(i)) \to \calD(X'(i))$ is a triangle equivalence in $\kuTRI^2$; and
\item
For each $a \in I_1$, $\ps(a)$ is a 2-isomorphism in $\kuTRI^2$
$($i.e., $(F, \ps)$ is $I$-equivariant$)$.
\end{enumerate}
\end{dfn}

In the next section, we will characterize
a derived equivalence between colax functors in $\Colax(I, \kdgCat)$
given by a left derived functor between dg module categories
or given by the left derived tensor functor of a bimodule.

\section{Characterizations of standard derived equivalences of colax functors}

In this section,
we define standard derived equivalences between colax functors
from $I$ to $\kdgCat$, and give its characterizations as
our first main result in this paper.
We will fully use notations in Definition \ref{dfn:bimodule}.

The following lemma given by Keller characterizes dg bimdoules which induced an equivalences of derived categories.

\begin{lem}
\cite[Lemma 3.10]{Ke3}
\label{lem:eq-bi-der-eq}
Let $\calA$ and $\calB$ be dg categories and $E$ an $\calA$-$\calB$ bimodule.
 Then
$\blank\Lox_{\calA}E\colon \calD( \calA) \to 
\calD(\calB)$
is an equivalence of triangulated categories
if and only if
\begin{enumerate}
\item the dg $\calB$-module ${}_AE$ is perfect for all $A\in\calA$,
\item the morphism 
\[
\bi{\calA}{A'}{A}\to \bi{(\bfR\Cdg(\calB))}{\bi{E}{A'}{}}{\bi{E}{A}{}}
\]
is a quasi-isomorphism for all $A,A'\in\calA$ and
\item the dg $\calB$-modules ${}_AE$, $A\in\calA$, form a tilting dg subcategory for $\calB$
{\em (Definition \ref{dfn:dg tilting})}.
\end{enumerate}
\end{lem}

\begin{dfn}
\label{dfn:qeq-dg-bimod}
Let $\calA$ and $\calB$ be dg categories and $E$ an $\calA$-$\calB$ bimodule.
We denote by $\udl{\calA}$ the full subcategory of $\calD(\calA)$
with 
\[
\udl{\calA}_0 = \{D \in \calD(\calA) \mid D \iso {}_C\calA
\text{ for some }C \in \calA_0\}.
\]
Then $E$ is called a \emph{quasi-equivalence bimodule}
if
\begin{enumerate}
\item[(a)] 
$\blank\Lox_{\calA}E\colon \calD(\calA) \to 
\calD(\calB)$
is an equivalence of triangulated categories,
and 
\item[(b)] it gives rise to an equivalence
$\udl{\calA} \to \udl{\calB}$, that is,
$\udl{\calA}\Lox_{\calA} E \subseteq \udl{\calB}$.
\end{enumerate}
\end{dfn}

We first cite the following from \cite[Theorem 8.1]{Ke1} without a proof.

\begin{thm}
\label{thm:LH-eq}
Let $\calA$ and $\calC$ be small dg categories.
Consider the following conditions.
\begin{enumerate}
\item
There is a dg functor $H: \DGMod(\calC) \to \DGMod(\calA)$ such that $\bfL
H: \calD(\calC)\to \calD(\calA)$ is an equivalence $($see Remark $\ref{rmk:L-dgfun})$.
\item
There exists a quasi-equivalence $\calC$-$\calT$-bimodule
for some tilting dg subcategory $\calT$ for $\calA$.
\item
There exists a dg category $\calB$ and dg functors
$$\DGMod(\calC) \xrightarrow{G} \DGMod(\calB) \xrightarrow{F} \DGMod(\calA)$$
such that $\bfL G$ and $\bfL F$ are equivalences.
\end{enumerate}
Then
\begin{enumerate}
\item[(a)]
$(1)$ implies $(2)$.
\item[(b)]
$(2)$ implies $(3)$.
\end{enumerate}
\end{thm}

Next, we cite the statement \cite[Theorem 8.2]{Ke1} in the $\k$-flat case.

\begin{thm}[Keller]
\label{thm:Keller}
Let $\calA$ and $\calB$ be small dg $\k$-categories and assume that
$\calA$ is $\k$-flat.
Then the following are equivalent.
\begin{enumerate}
\item
There exists a $\calB$-$\calA$-bimodue $Y$ such that
$\text{-}\Lox_\calB Y: \calD(\calB) \to \calD(\calA)$
is a triangle equivalence.
\item
There is a dg functor $H: \DGMod(\calC) \to \DGMod(\calA)$ such that $\bfL H: \calD(\calC)\to \calD(\calA)$ is an equivalence $($see Remark $\ref{rmk:L-dgfun})$.
\item
$\calB$ is quasi-equivalent to a tilting dg subcategory for $\calA$.
\end{enumerate}
\end{thm}

\begin{dfn}
\label{dfn:dg-std-der-eq}
The derived equivalence of the form 
$\bfL H\colon \calD(\calB)\to \calD(\calA)$ or
$\text{-}\Lox_\calB Y\colon \calD(\calB)\to \calD(\calA)$ above
is called a {\em standard derived equivalence} from $\calB$ to $\calA$, and if the statements of Theorem \ref{thm:Keller} hold,
then we say that $\calB$ is {\em standardly derived equivalent} to $\calA$.
Here it seems as if this relation between $\calA$ and $\calB$ would not be symmetric,
but actually it is symmetric.
Indeed, in that case, the $\calA$-$\calB$-bimodule
$Y^T:= \DGMod(\calA)(\bi{Y}{\calB}{\calA}, \bi{\calA}{\calA}{\calA})$ induces a triangle equivalence
$\text{-}\Lox_\calA Y^T: \calD(\calA) \to \calD(\calB)$.
Thus this relation is symmetric for $\calA$ and $\calB$.
\end{dfn}

We cite the following from \cite[Proposition 3.14]{Gen17} and \cite[Proposition 2.16]{Im}.

\begin{prp}[Co-Yoneda lemma]
\label{prp:coYoneda}
Let $\calC$ be a small dg category, and $T$ a right dg $\calC$-module.
Then we have
$T(y) \iso T \ox_\calC \calC_y
\iso \int^{x \in \calC}T(x) \ox_\k {}_x\calC_y$
natural in $y \in \calC$.
Hence we have
$$
T \iso \int^{x\in \calC} T(x) \ox_\k {}_x\calC.
\qed
$$
\end{prp}

Note that by definition, the coend on the right hand side is a kind of
a dg colimit, and hence commutes with the dg tensor product functors.

The following lemma can be seen as a preservation of dg natural transformations by the tensor product.

\begin{lem}
\label{lem:tens-pres-dg-nat}
Let $\calA, \calB$ be small dg categories, and
$\bi{M}{\calB}{\calA}, \bi{N}{\calB}{\calA}$ bimodules.
If
$$
f \colon M \to N
$$
is a $0$-cocycle morphism of $\calB$-$\calA$-bimodules, then
$$
\blank\ox_{\calB} f \colon
\blank \ox_{\calB}M \to \blank \ox_{\calB}N
$$
is a dg natural transformation between dg functors
$$
\blank \ox_{\calB}M \text{ and } \blank \ox_{\calB}N
\colon \Cdg(\calB) \to \Cdg(\calA).
$$
\end{lem}

\begin{proof}
(1) Note first that $\blank\ox_{\calB} f$ is homogeneous of degree zero because for any $X\in\Cdg(\calB)$, we have $(\blank\ox_{\calB} f)_X =\id_X\ox_{\calB} f \colon X \ox_{\calB}M \to X\ox_{\calB}N$, and both
$\id_X$ and $f$ are homogeneous of degree 0.

(2) $d(\blank \ox_{\calB} f)=0$, namely, the following diagram is commutative for any $X\in\Cdg(\calB)$:
$$
\begin{tikzcd}[column sep=50pt]
X \ox_{\calB} M  &
              X\ox_{\calB}N\\
X\ox_{\calB} M &
             X \ox_{\calB}N
\Ar{1-1}{1-2}{"X \ox_{\calB} f"}
\Ar{2-1}{2-2}{"X\ox_{\calB} f"}
\Ar{1-1}{2-1}{"d"'}
\Ar{1-2}{2-2}{"d"}
\end{tikzcd}.
$$

Indeed, for any homogeneous element $x \in X$ and any $m \in M$,
we have
$$
\begin{aligned}
(d(X\ox_\calB f))(x \ox m) &= dx\ox f(m)+(-1)^{|x|}x\ox df(m), \text{ and}\\
((X\ox_\calB f)d)(x \ox m) &= (X\ox_\calB f)(dx\ox m +(-1)^{|x|}x \ox dm)\\
&= dx\ox f(m)+(-1)^{|x|}x\ox f(dm)\\
&= dx\ox f(m)+(-1)^{|x|}x\ox df(m),
\end{aligned}
$$
which shows the commutativity of the diagram.

(3)
It remains to show that for any $\al:X\to Y$ in $\Cdg(\calB)$, the following diagram is commutative:
$$
\begin{tikzcd}[column sep=50pt]
X \ox_{\calB} M  &
              Y\ox_{\calB}M\\
X\ox_{\calB} N &
             Y \ox_{\calB}N
\Ar{1-1}{1-2}{"\al \ox_{\calB} \id_M"}
\Ar{2-1}{2-2}{"\al\ox_{\calB} \id_N"}
\Ar{1-1}{2-1}{"\id_X \ox_{\calB} f"'}
\Ar{1-2}{2-2}{"\id_Y \ox_{\calB} f"}
\end{tikzcd},
$$
that is, for any $x\in\calA_0$, we have to show
the commutativity of the diagram
\begin{equation}
\label{eq:alpha-nat}
\begin{tikzcd}[column sep=50pt]
X \ox_{\calB} M_x  &
              Y\ox_{\calB}M_x\\
X\ox_{\calB} N_x &
             Y \ox_{\calB}N_x
\Ar{1-1}{1-2}{"\al \ox_{\calB} \id_{M_x}"}
\Ar{2-1}{2-2}{"\al\ox_{\calB} \id_{N_x}"}
\Ar{1-1}{2-1}{"\id_X \ox_{\calB} f_x"'}
\Ar{1-2}{2-2}{"\id_Y \ox_{\calB} f_x"}
\end{tikzcd}.
\end{equation}
Now the diagram
$$
\begin{tikzcd}[column sep=50pt]
X_y \ox_{\k} {}_y M_x  &
              Y_y\ox_{\k}{}_y M_x\\
X_y\ox_{\k} {}_y N_x &
             Y_y \ox_{\k}{}_y N_x
\Ar{1-1}{1-2}{"\al_y \ox_{\k} \id_{{}_y M_x}"}
\Ar{2-1}{2-2}{"\al_y\ox_{\k} \id_{{}_y N_x}"}
\Ar{1-1}{2-1}{"\id_X \ox_{\k} {}_y f_x"'}
\Ar{1-2}{2-2}{"\id_Y \ox_{\k} {}_y f_x"}
\end{tikzcd}
$$
is commutative because both the clockwise composite and
the counter-clockwise composite coincide with $\al \ox_\k {}_yf_x$.
Therefore, since $X \ox_{\calB} M_x =\int^{y\in \calB} X_y\ox_{\k} {_y}M_x$, we obtain the commutative diagram \eqref{eq:alpha-nat}
by applying $\int^{y\in \calB}$ to this diagram.
\end{proof}

The following is immediate from Lemma \ref{lem:tens-pres-dg-nat}.

\begin{lem}
\label{lem:tens-pres-dg-nat-2}
Let $\calA, \calA', \calB, \calB'$ be small dg categories, and
$\bi{M}{\calA'}{\calA}, \bi{M'}{\calB'}{\calB,}$,
$\bi{L}{\calA}{\calB}$, and $\bi{L'}{\calA'}{\calB'}$ bimodules.
If
$$
f \colon M \ox_\calA L \to L'\ox_{\calB'}M'
$$
is a $0$-cocycle morphism of $\calA'$-$\calB$-bimodules, then
$$
\blank\ox_{\calA'} f \colon
\blank \ox_{\calA'}(M \ox_\calA L) \to \blank \ox_{\calA'}(L'\ox_{\calB'}M')
$$
turns out to be a dg natural transformation
between functors
$$
\blank \ox_{\calA'}(M \ox_\calA L)\text{ and } \blank \ox_{\calA'}(L'\ox_{\calB'}M')
\colon \Cdg(\calA') \to \Cdg(\calB).
$$
Hence up to associators, we have the following diagram
in $\kdgCAT$:
\[\begin{tikzcd}
	{\Cdg(\calA')} & {\Cdg(\calA)} \\
	{\Cdg(\calB')} & {\Cdg(\calB)}
	\arrow["{\blank\ox_{\calA'}M}", from=1-1, to=1-2]
	\arrow["{\blank\ox_{\calA'}L'}"', from=1-1, to=2-1]
	\arrow["{\blank\ox_{\calA'}f}"', Rightarrow, dashed, from=1-2, to=2-1]
	\arrow["{\blank\ox_{\calA}L}", from=1-2, to=2-2]
	\arrow["{\blank\ox_{\calB'}M'}"', from=2-1, to=2-2]
\end{tikzcd}.\]
\end{lem}

\begin{dfn}
\label{df:bimod-colax}
Let $X,X'\in \Colax(I,\kdgCat)$.
\begin{enumerate}
\item
An $X'$-$X$-{\em bimodule} is a pair $Z =((Z(i))_{i\in I_0}, (Z(a))_{a\in I_1})$ as in the diagram
\[\begin{tikzcd}[column sep = 60pt, row sep = 50pt]
	{\Cdg(X'(i))} & {\Cdg(X(i))} \\
	{\Cdg(X'(j))} & {\Cdg(X(j))}
	\arrow["{\blank\ox_{X'(i)}Z(i)}", from=1-1, to=1-2]
	\arrow["{\blank\ox_{X'(i)}\ovl{X'(a)}}"', from=1-1, to=2-1]
	\arrow["{\blank\ox_{X'(i)}Z(a)}"', Rightarrow, dashed, from=1-2, to=2-1]
	\arrow["{\blank\ox_{X(i)}\ovl{X(a)}}", from=1-2, to=2-2]
	\arrow["{\blank\ox_{X'(j)}Z(j)}"', from=2-1, to=2-2]
\end{tikzcd},\]
where
$Z(i)$ is an $X'(i)$-$X(i)$-bimodule for all $i \in I_0$, and
$$
Z(a) \colon Z(i) \ox_{X(i)}\ovl{X(a)} \to \ovl{X'(a)} \ox_{X'(j)}Z(j)
$$
is a 0-cocycle morphism of $X'(i)$-$X(j)$-bimodules
for all morphisms $a \colon i \to j$ in $I$,
such that
\[
\blank\ox_{X'}Z:= ((\blank\ox_{X'(i)}Z(i))_{i\in I_0}, (\blank\ox_{X'(\dom(a))}Z(a))_{a\in I_1})
\colon \DGMod(X') \to \DGMod(X)
\]
is a $1$-morphism in $\Colax(I, \kdgCAT)$, where $\dom(a) = i$ is the domain of $a \in I_1$, and
$\blank\ox_{X'(i)}Z(a)$ is given up to associators (see Definition \ref{dfn:associator}), 
i.e., it is identified with the composite
with associators as in the diagram
\[
\begin{tikzcd}[column sep=60pt]
(\blank\ox_{X'(i)} Z(i)) \ox_{X(i)}\ovl{X(a)} & (\blank\ox_{X'(i)} \ovl{X'(a)} )\ox_{X'(j)}Z(j)\\
\blank\ox_{X'(i)} (Z(i) \ox_{X(i)}\ovl{X(a)}) & \blank\ox_{X'(i)} (\ovl{X'(a)} \ox_{X'((j))}Z(j))
\Ar{1-1}{2-1}{"{\bfa_{(\blank),Z(i),\ovl{X(a)}}}"}
\Ar{2-2}{1-2}{"{\bfa_{(\blank),\ovl{X'(a)},Z(j)}\inv}"'}
\Ar{2-1}{2-2}{"\blank\ox_{X'(i)}Z(a)"}
\Ar{1-1}{1-2}{dashed}
\end{tikzcd}
\]
(Note that $\blank\ox_{X'(i)}Z(a)$ is a dg natural transformation by
Lemma \ref{lem:tens-pres-dg-nat-2}).

\item
If ${}_BZ(i)$ is a homotopically projective dg $X(i)$-module for all $B \in X'(i)_0$
and $i \in I_0$, then $Z$ is said to be {\em right homotopically projective}.

\item
A $1$-morphism $(F, \ps) \colon X' \to X$ in $\Colax(I, \kdgCat)$
is said to {\em preserve homotopically projective objects} if $F(i)$ does for all $i \in I_0$.
\end{enumerate}

\end{dfn}

Recall that there is another definition of quasi-equivalences between dg categories by using ``quasi-functor'' given by Keller.
We will give a similar definition in our setting.

\begin{dfn}
\label{dfn:qe-bimod}
Let $X, X' \in \Colax(I, \kdgCat)$, and
$Z$ be a $X'$-$X$-bimodule.
\begin{enumerate}
\item
We denote by $\udl{X}$ the colax subfunctor of $\calD(X)$
such that $\udl{X}(i)$ is the full subcategory of $\calD(X(i))$
with $\udl{X}(i):= \udl{X(i)}$.
The structure of $\udl{X}$ is given as follows:
$\udl{X} = ((\udl{X}(i))_{i\in I_0}, (\udl{X}(a))_{a\in I_1})$, where for each $a \colon i \to j$ in $I$, we have a
strict commutative diagram
\[
\begin{tikzcd}[column sep=90pt]
\calD(X(i)) & \calD(X(j))\\
\udl{X}(i) & \udl{X}(j)
\Ar{1-1}{1-2}{"\calD(X(a)) =\, \blank\Lox_{X(i)}\ovl{X(a)}"}
\Ar{2-1}{2-2}{"\udl{X}(a)"'}
\Ar{2-1}{1-1}{"\si(i)", hookrightarrow}
\Ar{2-2}{1-2}{"\si(j)"', hookrightarrow}
\end{tikzcd}.
\]
Namely, it is required that 
$$
(\udl{X}(i)_0) \Lox_{X(i)}\ovl{X(a)} \subseteq \udl{X}(j)_0.
$$
\item
We denote by $\bbZ\udl{X}$ the colax subfunctor of $\calD(X)$
such that $\bbZ \udl{X}(i)$ is the full subcategory of $\calD(X(i))$
with
\[
\bbZ \udl{X}(i)_0 = \{D \in \calD(X(i)) \mid D \iso {}_CX(i)[n]
\text{ for some }C \in X(i)_0, n\in\bbZ\}.
\]

\item
The $X'$-$X$-bimodule $Z$ is called a \emph{quasi-functor} if 
$\blank\Lox_{X'}Z \colon \calD(X') \to \calD(X)$ gives rise to a $1$-morphism $\udl{X'}\to \udl{X}$ in the sense that the following two statements
(a) and (b) hold:

(a) for each $i \in I_0$,
we have a strictly commutative diagram
\[
\begin{tikzcd}
\calD(X'(i)) & \calD(X(i))\\
\udl{X'}(i) & \udl{X}(i)
\Ar{1-1}{1-2}{"\blank\Lox_{X'(i)}Z(i)"}
\Ar{2-1}{2-2}{"\blank\Lox_{X'(i)}Z(i)"'}
\Ar{2-1}{1-1}{"\si'(i)", hookrightarrow}
\Ar{2-2}{1-2}{"\si(i)"', hookrightarrow}
\end{tikzcd}.
\]
Namely, it is required that for any $C \in X'(i)_0$,
there exists some $D \in X(i)_0$ such that
${}_CZ(i) \iso {}_C{X'(i)} \Lox_{X'(i)}Z(i) \iso {}_D{X(i)}$.

(b) For each $a:i\to j$ in $I$, the 2-morphism
$$
(\blank\Lox_{X'(i)}Z)(a): \calD X(a) \circ (\blank \Lox_{X'(i)}Z(i))
\To (\blank \Lox_{X'(j)}Z(j)) \circ \calD X'(a)
$$
induces a 2-morphism
$\blank\Lox_{X'(i)}Z(a)$ between 1-morphisms from $\udl{X'}(i)$
to $\udl{X}(j)$ in the following diagram: 
\begin{equation}
\label{eq:cube}
\begin{tikzcd}[row sep=scriptsize, column sep=scriptsize]
& \calD(X'(i)) \arrow[dl,"\calD X'(a)"'] \arrow[rr,"\blank\Lox_{X'(i)}Z(i)"]  & & \calD(X(i)) \arrow[dl,"\calD X(a)"]\\
\calD(X'(j))\arrow[rr, crossing over
]  & & \calD(X(j)) \\
&\udl{X'}(i) \arrow[dl] \arrow[rr] \arrow["\si'(i)" near start, uu,hookrightarrow] & & \udl{X}(i) \arrow[dl]\arrow[uu,hookrightarrow]\\
\udl{X'}(j) \arrow[rr]\arrow[uu,hookrightarrow] & & \udl{X}(j)
\Ar{1-4}{2-1}{"(\blank\Lox_{X'(i)}Z)(a)"' description, Rightarrow,crossing over}
\Ar{3-4}{4-1}{"(\blank\Lox_{X'(i)}Z)(a)"' description, Rightarrow}
\Ar{4-3}{2-3}{"\si(j)"' near end, crossing over,hookrightarrow}
\end{tikzcd}.
\end{equation}

\item
The $X'$-$X$-bimodule  $Z$ is called a \emph{quasi-equivalence bimodule} if 
\begin{enumerate}
\item[(a)]
$\blank\Lox_{X'}Z \colon \calD(X') \to \calD(X)$ is an equivalence in
$\Colax(I, \kuTRI^2)$, and 
\item[(b)] it gives rise to an equivalence $\udl{X'}\to \udl{X}$.
\end{enumerate}
Here in the diagram \eqref{eq:cube},
the condition (b) is equivalent to saying that
 $\blank\Lox_{X'(i)}Z(i) \colon \udl{X'}(i) \to \udl{X}(i)$
is an equivalence for all $i \in I_0$, and
that $(\blank\Lox_{X'(i)}Z)(a)$ in the bottom square
is a 2-isomorphism for all $a\colon i \to j$ in $I_1$.
Note that it also required that
$((\blank\Lox_{X'(i)}Z)(a)) \circ \si'(i) = 
\si(j) \circ ((\blank\Lox_{X'(i)}Z)(a))$ for all morphisms $a \colon i \to j$ in $I$
(note that both hand sides are horizontal composites), 
but this is automatically satisfied%
\footnote{
The left/right faces, and the front/back faces are strictly commutative by (3) and  (1), respectively.}.
\end{enumerate}
\end{dfn}

\begin{rmk}
\label{rmk:q-eq-1-mor-bimod}
Note that in Definition \ref{dfn:der-eq-criterion} 
another notion of \emph{quasi-equivalence} is defined
for a $1$-morphsim, which we have to distinguish.
Any quasi-equivalence $1$-morphism $X' \to X$ induces
a quasi-equivalence $X'$-$X$-bimdoule by Proposition \ref{prp:der-eq-criterion}.
Indeed, in this proposition,
the $X'$-$X$-bimodule
$E:=(\ovl{F(i)}, \ovl{\ps(a)})_{i\in I_0, a\in I_1}$ is equal to
$\bfL (\check{F}, \check{\ps})$ and
is a quasi-equivalence bimodule.
\end{rmk}

\begin{lem} 
\label{lem:equiv-bimodule}
In the same setting as in Definition \ref{dfn:qe-bimod},
the following are equivalent
\begin{enumerate}
\item $\blank\Lox_{X'}Z \colon \calD(X') \to \calD(X)$ is an equivalence in
$\Colax(I, \kuTRI^2)$ giving rise to an equivalence $\udl{X'}\to \udl{X}$.

\item $\blank\Lox_{X'}Z$ gives rise to equivalences  $\bbZ\udl{X'}\to \bbZ\udl{X}$ and $\udl{X'}\to \udl{X}$.
\end{enumerate}
In this case,
$Z$ turns out to be  a quasi-equivalence $X'$-$X$-bimodule.
\end{lem}
\begin{proof}
(1) \implies (2). This is trivial by the definition of quasi-equivalence.

(2) \implies (1).
Assume that the statement (2) holds.
To show that $\blank\Lox_{X'}Z \colon \calD(X') \to \calD(X)$ is an equivalence in
$\Colax(I, \kuTRI^2)$. It follows from \cite[Lemma 7.2]{Ke1} that $\blank\Lox_{X'(i)}Z(i) \colon \calD(X'(i)) \to \calD(X(i))$ is an equivalence for each $i\in I_0$.
Hene by Lemma \ref{colax-eq},
it remains to show that $(\blank\Lox_{X'}Z)(a) =( \blank\Lox_{X'(i)} Z(a))\bullet \bfL_{(\blank\ox_{X(i)}\ovl{X(a)}),(\blank\ox_{X'(i)}Z(i))}$ is a $2$-natural isomorphism for all $a \colon i \to j$ in $I$. 
To show this, it is enough to show that
$(\blank\Lox_{X'}Z)(a)_{C^\wedge} = \blank\Lox_{X'(i)}Z(a))_{x'^\wedge} \circ\bfV_{x'^\wedge}$ is invertible for all $C\in X'(i)$, where
$\bfV:= \bfL_{(\blank\ox_{V(i)}\ovl{X(a)}),(\blank\ox_{X'(i)}Z(i))}$.
We have a commutative diagram
{\footnotesize
$$
\begin{tikzcd}[column sep=70pt]
(x'^\wedge\Lox_{X'(i)}Z(i))\Lox_{V(i)}\ovl{X(a)}& (x'^\wedge\ox_{X'(i)}Z(i))\Lox_{V(i)}\ovl{X(a)}\\
(x'^\wedge\ox_{X'(i)}\ovl{X'(a)})\ox_{X'(j)}Z(j))
& (x'^\wedge\ox_{X'(i)}Z(i))\ox_{V(i)}\ovl{X(a)}\\
\Ar{1-1}{1-2}{"\iso"}
\Ar{1-2}{2-2}{"\iso", "\text{(a)}"'}
\Ar{1-1}{2-1}{"(\blank\Lox_{X'} Z)(a)_{x'^\wedge}"'}
\Ar{2-2}{2-1}{"(\blank\Lox_{X'(i)}Z(a))_{x'^\wedge}"}
\Ar{1-1}{2-2}{"\bfV_{x'^\wedge}"}
\end{tikzcd},
$$
}
by which
$\bfV_{x'^\wedge}$ is clearly an isomorphism, and
$\blank\Lox_{X'(i)}Z(a))_{x'^\wedge}$ is an isomorphism by the second assumption.
\end{proof}

\begin{lem}\label{2-morph}
Let $X, X' \in \Colax(I, \kdgCat)$, and
$Z$ a $X'$-$X$-bimodule.
Then for any $a \colon i \to j$ in $I$,
$Z$ naturally induces a $2$-morphism
\[
\widehat{Z(a)} \colon
(\blank \ox_{X'(i)}\ovl{X'(a)}) \circ \Cdg(X(i))(Z(i),\blank)
\To
\Cdg(X(j))(Z(j), \blank) \circ (\blank \ox_{X(i)}\ovl{X(a)}).
\]
\end{lem}

\begin{proof}
We first define a 2-morphism
\[
\begin{aligned}
\Cdg(X(i))(Z(i),\blank) &\circ \Cdg(X(j))(\ovl{X(a)},\blank)\\
&\To 
\Cdg(X'(j))(\ovl{X'(a)},\blank) \circ \Cdg(X(j))(Z(j),\blank).
\end{aligned}
\]
This is defined as a composite of the following 2-morphisms
\[
\begin{aligned}
& \Cdg(X(i))(Z(i),\blank) \circ \Cdg(X(j))(\ovl{X(a)},\blank)\\
=& \Cdg(X(i))(Z(i), \Cdg(X(j))(\ovl{X(a)},\blank))\\
\overset{\iso}{\To}& \Cdg(X(j))(Z(i)\ox_{X(i)}\ovl{X(a)},\blank)\\
\overset{(*)}{\To}& 
\Cdg(X(j))(\ovl{X'(a)}\ox_{X'(j)}Z(j),\blank)\\
\overset{\iso}{\To}& 
\Cdg(X'(j))(\ovl{X'(a)},\Cdg(X(j))(Z(j),\blank))\\
=& \Cdg(X'(j))(\ovl{X'(a)},\blank) \circ \Cdg(X(j))(Z(j),\blank),
\end{aligned}
\]
where $(*)$ stands for $\Cdg(X(j))(Z(a),\blank)$, and
the isomorphisms are given by adjunctions.

Then we can define $\widehat{Z(a)}$ as the composite of the following
(the second and the third ones are obtained by
applying the functors $(\blank \ox_{X'(i)}\ovl{X'(a)})$ from the left, and
$(\blank \ox_{X(i)}\ovl{X(a)})$ from the right to the above.)
\[
\begin{aligned}
& (\blank \ox_{X'(i)}\ovl{X'(a)}) \circ \Cdg(X(i))(Z(i),\blank) \circ \id\\
\overset{\mathrm{(a)}}{\To}& 
(\blank \ox_{X'(i)}\ovl{X'(a)}) \circ \Cdg(X(i))(Z(i),\blank) \circ \Cdg(X(j))(\ovl{X(a)},\blank) \circ (\blank \ox_{X(i)}\ovl{X(a)})\\
\overset{\mathrm{(b)}}{\To}& 
(\blank \ox_{X'(i)}\ovl{X'(a)}) \circ \Cdg(X'(j))(\ovl{X'(a)},\blank) \circ \Cdg(X(j))(Z(j),\blank)
\circ (\blank \ox_{X(i)}\ovl{X(a)})\\
\overset{\mathrm{(c)}}{\To}& 
\id \circ \Cdg(X(j))(Z(j), \blank) \circ (\blank \ox_{X(i)}\ovl{X(a)}),
\end{aligned}
\]
where (a) is given by the unit, (b) is given by the 2-morphism
defined above, and (c) is given by the counit.
\end{proof}

\begin{dfn}
\label{dfn:lift}
Let $X\in \Colax(I, \kdgCat)$,
and let $W$ be a colax subfunctor of $\calD(X)$
such that $W(i)$ is a full subcategory of $\calD(X(i))$ for all $i \in I_0$.
\begin{enumerate}
\item
We denote by $\bbZ W$ the colax subfunctor of $\calD(X)$
such that $\bbZ W(i)$ is the full subcategory of $\calD(X(i))$
with $\bbZ W(i)_0 = \{U[n] \mid U \in W(i)_0, n\in\bbZ\}$.

\item
A \emph{lift} of $W$ is a pair $(X', {}_{X'}Z_{X})$ of a colax functor
$X'\in \Colax(I, \kdgCat)$ and an $X'$-$X$-bimodule $Z$ such that
$\blank\Lox_{X'}Z \colon \calD(X') \to \calD(X)$ gives rise to equivalences
$\bbZ\udl{X'}\to \bbZ W$ and $\udl{X'}\to W$. Note that we have the following commutative diagram
\[
\begin{tikzcd}[row sep=scriptsize, column sep=scriptsize]
& \calD(X'(i)) \arrow[dl,"\calD(X'(a))"'] \arrow[rr,"\blank\Lox_{X'(i)}Z(i)"]  & & \calD(X(i)) \arrow[dl,"\calD(V(a))"near end]  \\
\calD(X'(j)) 
& & \calD(X(j)) \\
&\bbZ\udl{X'}(i) \arrow[dl] \arrow[rr]& & \bbZ W(i)  \arrow[dl]\arrow[uu]\\
\bbZ\udl{X'}(j) \arrow[rr]\arrow[uu] & & \bbZ W(j)
\Ar{3-2}{1-2}{}
\Ar{2-1}{2-3}{"\blank\Lox_{X'(j)}Z(j)"'near end,crossing over}
\Ar{1-4}{2-1}{"(\blank\Lox_{X'(i)}Z)(a)"' description, Rightarrow,crossing over}
\Ar{3-4}{4-1}{"(\blank\Lox_{X'(i)}Z)(a)"' description, Rightarrow}
\Ar{4-3}{2-3}{crossing over}
\end{tikzcd}.
\]

\item A \emph{standard lift} of $W$ is a lift $(V, {}_VM_X)$ of $W$ constructed as follows:
Take $V$ as the colax subfunctor of $\DGMod(X)$ such that
for each $i\in I$,  $V(i)$ is the full subcategory of $\DGMod(X(i))$
with $V(i)_0:= \{\bfp_{X(i)}U \mid U \in W(i)_0\},V(a):= \DGMod(X(a))|_{V(i)}$ for all $a\in I_1$,
and $M$ as a $V$-$X$-bimodule definied by
${}_BM(i)_A:= B(A) \iso \DGMod(X(i))(A^\wedge, B)$ for all $B \in V_0, A \in X(i)_0$.
Note here that ${}_BM(i) = B$ is homotopically projective for all $B \in V(i)$.
Moreover, we define
a $0$-cocycle morphism
\[
M(a) \colon M(i) \ox_{X(i)}\ovl{X(a)} \to \ovl{V(a)} \ox_{V(j)}M(j)
\]
of $V(i)$-$X(j)$-bimodules by
\[
\begin{aligned}
M(a)_{C^\wedge}\colon C^\wedge\ox_{V(i)} M(i)\ox_{X(i)}\ovl{X(a)}\iso M(i)(\blank,C)\ox_{X(i)}\ovl{X(a)}\\
\iso C\ox_{X(i)}\ovl{X(a)}\iso
V(a)(C)\iso
C^\wedge\ox\ovl{V(a)}\ox M(j)
\end{aligned}
\]
for all $a\in I_1$ and $C\in V(i)$.
\end{enumerate}
\end{dfn}

\begin{dfn}
For each small dg category $\calC$,
we denote by $\hprj^b(\calC)$ the smallest full subcategory of $\hprj(\calC)$ that is closed under isomorphisms and contains $A^\wedge$ for all $A\in \calC$.

For each $X \in \Colax(I, \kdgCat)$,
we define $\hprj^b(X)$ to be a colax subfuncor of $\hprj(X)$ such that $(\hprj^b(X))(i):=\hprj^b(X(i))$ for all $i \in I_0$.
\end{dfn}

\begin{rmk}
Consider the same setting as in Definition \ref{dfn:lift}.
If $(X', {}_{X'}Z_{X})$ is a lift of $W$, then $\blank\Lox_{X'}Z$ induces an equivalence from  $\hprj^b(X')$ onto the colax subfunctor of $\calD(X)$ generated by $W$.
If $Z(i)_B$ is homotopically projective for all $B\in X'(i)$, 
then a quasi-inverse of $\blank\Lox_{X'}Z$ is given by $\RHom(X(i))(Z(i),\blank)$. 
\end{rmk}

\begin{dfn}\label{lift:inverse}
Consider the same setting as in Definition \ref{dfn:lift}, and
assume that $(X', {}_{X'}Z_{X})$ is a lift of $W$.
Then we define a 2-isomorphism
\[
\wdt{Z(a)} \colon
(\blank \Lox_{X'(i)}\ovl{X'(a)}) \circ \RHom(X(i))(Z(i),\blank)
\To
\RHom(X(j))(Z(j),\blank) \circ W(a)
\]
as the composite of the following three 2-isomorphisms
\[
\begin{aligned}
&(\blank \Lox_{X'(i)}\ovl{X'(a)})\circ \RHom(X(i))(Z(i), \blank)\\
&\overset{\mathrm{(a)}}{\To}
\RHom(X(j)(Z(j),\blank)\circ(\blank \Lox_{X'(j)}Z(j))\circ (\blank \Lox_{X'(i)}\ovl{X'(a)})\circ \RHom(X(i))(Z(i), \blank)\\
&\overset{\mathrm{(b)}}{\To}
\RHom(X(j))(Z(j),\blank)\circ W(a) \circ (\blank\Lox_{X'(i)}Z(i))\circ \RHom(X(i))(Z(i), \blank)
\\
&\overset{\mathrm{(c)}}{\To}
\RHom(X(j))(Z(j),\blank)\circ W(a),
\end{aligned}
\]
where (a) and (c) are induced by $\id \To  \RHom(X(j))(Z(j), \blank)\circ (\blank \Lox_{X'(j)}Z(j))$ and
$(\blank\Lox_{X'(i)}Z(i))\circ \RHom(X(i))(Z(i), \blank)\To \id$,
respectively, and
(b) is given by 
\[
\RHom(X(j)(Z(j),\blank)) \circ ((\blank\Lox_{X'(i)} (Z(a)^{\inv}) \circ \RHom(X(i))(Z(i), \blank).
\]
The situation above is visualized in the following diagram:
   \[
\begin{tikzcd}[column sep=70pt]
\hprj^b(X'(i))&W(i)&\hprj^b(X'(i)) \\
\hprj^b(X'(j))&W(j) &\hprj^b(X'(j))
\Ar{1-1}{2-1}{"\hprj^b{X'(a)}"'}
\Ar{1-2}{2-2}{"W(a)"}
\Ar{1-1}{1-2}{"\blank\Lox_{X'(i)}Z(i)"}
\Ar{2-1}{2-2}{"\blank\Lox_{X'(j)}Z(j)"'}
\Ar{1-2}{2-1}{"(\blank\Lox_{X'(i)} Z)(a)"', Rightarrow}
\Ar{1-2}{1-3}{"{\RHom(X(i))(Z(i),-)}"}
\Ar{2-2}{2-3}{"{\RHom(X(j))(Z(j),-)}"'}
\Ar{1-3}{2-3}{"\hprj^b{X'(a)}"}
\Ar{1-3}{2-2}{"\wdt{Z(a)}"', Rightarrow}
\end{tikzcd}.
\]
\end{dfn}

The following is used in our proof of the implication (4) $\implies$(1) in
Theorem \ref{thm:characterization-1}.

\begin{lem}
\label{lem:bimodule}
Let $(\bfF, \bfps)\colon \DGMod(X') \to \DGMod(X)$ be a $1$-morphism in the $2$-category $\Colax(I,\kdgCAT)$. Then
\begin{enumerate}
\item We can define an $X'$-$X$-bimodule $N$ by setting
$$
N(i)(D,\linebreak[3] C)=\bfF(i)(C^\wedge)(D)
$$
for all $C\in X'(i)$ and $D\in X(i)$. 

\item We can define a canonical $2$-morphism
$$
\ze= (\ze(i))_{i \in I_0} \colon  \blank\Lox_{X'}N\To \bfL(\bfF, \bfps)
$$
in the $2$-category $\Colax(I,\kdgCAT)$, where it has
the following properties:
\begin{enumerate}
\item 
For any $i \in I_0$ and $U \in \hprj^b(X'(i))$,
$\ze(i)_{U} \colon U \Lox_{X'(i)}N(i)\to  \bfL \bfF(i)(U)$
is an isomorphism, and
\item 
$\ze$ turns out to be a $2$-isomorphism if and only if $\bfL\bfF(i)$ commutes with direct sums for all $i\in I_0$.
\end{enumerate}
\end{enumerate}
\end{lem}

\begin{proof}
(1) Let $N(i)$ be the bimodule $N(i)(D,C)=\bfF(i)(C^\wedge)(D)$ for $C\in X'(i)$ and $D\in X(i)$. For any $C\in X'(i)$, we define a 2-morphism
\[N(a)_{C^\wedge} \colon C^\wedge\ox_{X'(i)} N(i) \ox_{X(i)}\ovl{X(a)}
\to C^\wedge\ox_{X'(i)} \ovl{X'(a)} \ox_{X'(j)}N(j)
\]
as the composite
of the following 2-morphisms:
\[
\begin{aligned}
C^\wedge\ox_{X'(i)} N(i) \ox_{X(i)}\ovl{X(a)} &\iso N(i)(\blank,C) \ox_{X(i)}\ovl{X(a)}\\
&=\bfF(i)(C^\wedge)\ox_{X(i)}\ovl{X(a)}\\
&\hspace{-1ex}\ya{{\ps(a)_{C^\wedge}}}
\bfF(j)(C^\wedge\ox_{X'(i)}\ovl{X'(a)})\\
&\iso \bfF(j)((X'(a)C)^\wedge)\\
&=N(j)(\blank,X'(a)C)\\
&\iso X'(j)(\blank,X'(a)C)\ox_{X'(j)} N(j)\\
&\iso C^\wedge\ox_{X'(i)}\ovl{X'(a)} \ox_{X'(j)}N(j)
\end{aligned}
\]
Let $T$ be any dg $X'(i)$-module.
By Proposition \ref{prp:coYoneda}, we have an isomorphism
$$
\xi_T \colon \int^{C\in X'(i)} T(C) \ox_\k C^{\wedge}
\ \overset{\iso}{\longrightarrow}\  T.
$$
Then we can define $N(a)_T$ by the following commutative diagram:
\[
\tiny
\begin{tikzcd}[column sep=70pt]
\Nname{cL} \int^{C}(T(C) \ox_\k C^{\wedge}
   \ox_{X'(i)}N(i) \ox_{X(i)}\ovl{X(a)})
   &\Nname{cR} \int^{C} (T(C) \ox_\k C^{\wedge}
      \ox_{X'(i)}\ovl{X'(a)}\ox_{X'(j)}N(j))\\
\Nname{dL} \int^{C}(T(C) \ox_\k C^{\wedge})
   \ox_{X'(i)}N(i) \ox_{X(i)}\ovl{X(a)}
   &\Nname{dR} \int^{C} (T(C) \ox_\k C^{\wedge})
      \ox_{X'(i)}\ovl{X'(a)}\ox_{X'(j)}N(j)\\
\Nname{TL}T \ox_{X'(i)}N(i) \ox_{X(i)}\ovl{X(a)} &
   \Nname{TR} T \ox_{X'(i)}\ovl{X'(a)} \ox_{X'(j)}N(j),
\Ar{cL}{cR}{"\int^C(T(C)\ox_\k N(a)_{C^\wedge})"}
\Ar{TL}{TR}{"N(a)_{T}"}
\Ar{cL}{dL}{"\iso"'}
\Ar{cR}{dR}{"\iso"}
\Ar{dL}{TL}{"\xi_T\ox_{X'(i)}N(i) \ox_{X(i)}\ovl{X(a)}", "\iso" '}
\Ar{dR}{TR}{"\xi_T\ox_{X'(i)}\ovl{X'(a)}\ox_{X'(j)}N(j)"', "\iso"}
\end{tikzcd}
\]
and we obtain a 2-morphism
\[
N(a):= (N(a)_T)_{T \in \Cdg(X'(i))} \colon
\blank \ox_{X'(i)}N(i) \ox_{X(i)}\ovl{X(a)} \To 
\blank \ox_{X'(i)}\ovl{X'(a)} \ox_{X'(j)}N(j).
\]
By Lemma \ref{lem:tens-pres-dg-nat-2},
the pair $((N(i))_{i \in I_0}, (N(a))_{a \in I_a})$ gives us the diagram
\[\begin{tikzcd}[column sep = 60pt, row sep = 50pt]
	{\Cdg(X'(i))} & {\Cdg(X(i))} \\
	{\Cdg(X'(j))} & {\Cdg(X(j))}
	\arrow["{\blank\ox_{X'(i)}N(i)}", from=1-1, to=1-2]
	\arrow["{\blank\ox_{X'(i)}\ovl{X'(a)}}"', from=1-1, to=2-1]
	\arrow["{\blank\ox_{X'(i)}N(a)}"', Rightarrow, dashed, from=1-2, to=2-1]
	\arrow["{\blank\ox_{X(i)}\ovl{X(a)}}", from=1-2, to=2-2]
	\arrow["{\blank\ox_{X'(j)}N(j)}"', from=2-1, to=2-2]
\end{tikzcd},\]
up to associator, and it is easy to verify that it turns out to be
an $X'$-$X$-bimodule.

(2) (a)
Let $i\in I_0$. Then
by the same way as explained in \cite[6.4]{Ke1}, we can
define a natural transformation
$\ze(i) \colon \bfL\bfF(i) \To \blank\Lox_{X'(i)}N(i)$ as follows.
Let $U\in\calD(X'(i))$.  Then for any $C\in X'(i)$, there exists
a canonical map $\xi(i)_{U,C} \colon U(C) \to \Cdg(X(i), \bfF(i)(U))(C)$
defined as the composite of the following:
\[
\begin{aligned}
U(C)\iso \Cdg(X'(i))(C^\wedge, U)\ya{\bfF(i)_{(C^\wedge,U)}} \Cdg(X(i))(\bfF(i)(C^\wedge), \bfF(i)(U))\\
= \Cdg(X(i))(N(i)(\blank,C), \bfF(i)(U)) = \Cdg(X(i))(N(i), \bfF(i)(U))(C)
\end{aligned},
\]
which defines a natural morphism
$$
\xi(i)_U:= (\xi(i)_{U,C})_{C \in X'(i)} \colon U\to \Cdg(X(i))(N(i), \bfF(i)(U)).
$$
By adjunction, we obtain
$U\ox_{X'(i)} N(i) \to \bfF(i)(U)$. The induced morphism
$\ze(i)_{U} \colon U\Lox_{X'(i)}N(i)\to \bfL\bfF(i)(U)$
is invertible for $U = C^\wedge[n]$ for $C\in X'(i), n\in\bbZ$.

Since the smallest triangulated subcategory of $\hprj(X'(i))$ closed under direct summands including $\udl{X'(i)}$
is equal to $\hprj^b(X'(i))$, 
we see that $\ze(i)_{U}$
is an isomorphism for all $U \in \hprj^b(X'(i))$, and thus
Statement (a) holds.

(b)
Let $\calU$ be the triangulated subcategory of $\calD(X'(i))$ consisting
of objects $M$ such that $\ze(i)_M$ is an isomorphism.
Then $\calU$ is closed under small coproducts if and only if
$\bfL\bfF(i)$ commutes with small coproducts.
Note that we already know that $\udl{X'(i)}$ is included in $\calU$.

Assume that $\bfL\bfF(i)$ commutes with small coproducts.
Then 
since the smallest triangulated subcategory of $\calD(X'(i))$ closed under
small coproducts including $\udl{X'(i)}$
is equal to $\calD(X'(i))$, we have $\calU = \calD(X'(i))$.
Thus $\ze(i)_U$ is an isomorphism for all $U \in \calD(X'(i))$.
Conversely, assume that $\ze(i)_U$ is an isomorphism for all $U \in \calD(X'(i))$.
Then $\calU = \calD(X'(i))$, and it is closed under small coproducts,
and hence $\bfL\bfF(i)$ commutes with small coproducts.
Therefore, for each $i \in I_0$, 
$\ze(i) \colon \blank\Lox_{X'(i)}N(i)\to \bfL\bfF(i)$
is a natural isomorphism if and only if $\bfL\bfF(i)$ commutes with
small coproducts.

It remains to show the commutativity of the following diagram
for all $a\colon i \to j$ in $I_1$:
$$
\xymatrix@C=8pc{
(\blank\Lox_{X(i)}\ovl{X(a)})\circ(\blank\Lox_{X(i)} N(i)) & (\blank\Lox_{X(i)}\ovl{X(a)})\circ\bfL\bfF(i)\\
(\blank\Lox_{X'(j)} N(j))\circ(\blank\Lox_{X'(i)}\ovl{X'(a)}) &
\bfL\bfF(j)\circ(\blank\Lox_{X'(i)}\ovl{X'(a)}).
\ar@{=>}^{(\blank\Lox_{X(i)}\ovl{X(a)})\circ\ze(i)} "1,1"; "1,2"
\ar@{=>}^{\ze(j)\circ(\blank\Lox_{X'(i)}\ovl{X'(a)})} "2,1"; "2,2"
\ar@{=>}_{(\blank\Lox_{X} N)(a)} "1,1"; "2,1"
\ar@{=>}^{(\bfL\ps)(a)} "1,2"; "2,2"
}
$$
To this end, we first show the commutativity of the following
diagram for all $C\in X'(i)$:
\[
\begin{tikzcd}[column sep = 70pt, row sep = 20pt]
 (C^\wedge\Lox_{X'(i)} N(i) )\Lox_{X(i)}\ovl{X(a)}&\bfL\bfF(i)(C^\wedge)\Lox_{X(i)}\ovl{X(a)}\\
(C^\wedge\Lox_{X'(i)}\ovl{X'(a)}) \Lox_{X'(j)}N(j)&\bfL\bfF(j)(C^\wedge\Lox_{X'(i)}\ovl{X'(a)})\\
\Ar{1-1}{1-2}{"\ze(i)_{C^\wedge}\Lox_{X(i)}\ovl{X(a)}"}
\Ar{2-1}{2-2}{"\ze(j)_{\left(C^\wedge\Lox_{X'(i)}\ovl{X'(a)}\right)}"}
\Ar{1-1}{2-1}{"(\blank\Lox_{X'} N)(a)_{C^\wedge}"}
\Ar{1-2}{2-2}{"(\bfL\ps)(a)_{C^\wedge}"}
\end{tikzcd}
\]
By decomposing the vertical morphisms,
this diagram is completed in the following (see \eqref{eq:lax-comp-L}):
$$
\tiny
\begin{tikzcd}[column sep = 80pt, row sep = 40pt]
 (C^\wedge\Lox_{X'(i)} N(i) )\Lox_{X(i)}\ovl{X(a)}&\bfL\bfF(i)(C^\wedge)\Lox_{X(i)}\ovl{X(a)}\\
C^\wedge\ox (N(i) \ox\ovl{X(a)})& \bfF(i)(C^\wedge)\ox_{X(i)}\ovl{X(a)}\\
(C^\wedge\Lox\ovl{X'(a)}) \Lox N(j)& \bfL\bfF(j)(C^\wedge\Lox_{X'(i)}\ovl{X'(a)})
\Ar{1-1}{1-2}{"\ze(i)_{C^\wedge}\Lox_{X(i)}\ovl{X(a)}"}
\Ar{3-1}{3-2}{"\ze(j)_{\left(C^\wedge\Lox_{X'(i)}\ovl{X'(a)}\right)}"}
\Ar{1-1}{3-1}{"(\blank\Lox_{X'} N)(a)_{C^\wedge}"', bend right=70pt}
\Ar{1-2}{3-2}{"(\bfL\ps)(a)_{C^\wedge}",bend left=70pt}
\Ar{2-2}{3-2}{"\bfL\ps(a)_{C^\wedge}"'}
\Ar{2-1}{3-1}{"(\blank\Lox_{X'(i)}N(a))_{C^\wedge}"}
\Ar{1-1}{2-1}{"\bfL_{(\blank\ox_{X(i)}\ovl{X(a)}),(\blank\ox N(i)) }"}
\Ar{1-2}{2-2}{"\bfL_{(\blank\ox_{X(i)}\ovl{X(a)}),\bfF(i)}"'}
\Ar{2-1}{2-2}{"\ze(i)_{C^\wedge}\ox_{X(i)}\ovl{X(a)}"}
\end{tikzcd}
$$
The commutativity of the upper square is easy to show.
That of the lower square follows by considering the diagram below:
\[
\begin{tikzcd}
C^\wedge\Lox_{X'(i)} (N(i) \ox_{X(i)}\ovl{X(a)}) & C^\wedge\ox_{X'(i)} (N(i) \ox_{X(i)}\ovl{X(a)})\\
&=\bfF(i)(C^\wedge)\ox_{X(i)}\ovl{X(a)}\\
&\bfL\bfF(j)(C^\wedge\Lox_{X'(i)}\ovl{X'(a)})\\
C^\wedge\Lox_{X'(i)}(\ovl{X'(a)} \ox_{X'(j)}N(j)) & \bfF(j)(C^\wedge\ox_{X'(i)}\ovl{X'(a)})\\
\Ar{1-1}{4-1}{"(\blank\Lox_{X'(i)}N(a))_{C^\wedge}"}
\Ar{1-1}{1-2}{"\iso"}
\Ar{1-2}{2-2}{"\iso"}
\Ar{2-2}{3-2}{"\bfL\ps(a)_{C^\wedge}"}
\Ar{3-2}{4-2}{"\iso"}
\Ar{4-1}{4-2}{"\iso"}
\end{tikzcd}.
\qedhere
\]
\end{proof}

We now prepare the following for the next lemma.
Take any $X, X'\in \Colax(I,$\linebreak[2]$\kdgCat)$.
Let $W$ be a colax subfunctor of $\calD(X)$
such that $W(i)$ is a full subcategory of $\calD(X(i))$ for all $i \in I_0$,
and let $(V, {}_VM_X)$ be a lift (Definition \ref{dfn:lift}(3)) of $W$ such that ${}_BM(i)$ is homotopically projective for each $B\in V(i)$
(for example, this is satisfied by the standard lift of $W$).
Take any $1$-morphism
$(\bfF, \bfps)\colon \DGMod(X') \to \DGMod(X)$ in the $2$-category
$\Colax(I,$\linebreak[2]$\kdgCAT)$
such that $\bfL(\bfF, \bfps)\colon$ $\calD(X') \to \calD(X)$
restricts to a $1$-morphism $\udl{X'}\to W$.
Then we have the following diagram:
\begin{equation}
\label{lift}
\begin{tikzcd}[row sep=10pt, column sep=10pt]
&&& & \calD(X'(i))\arrow[dl,"\bfL \bfF(i)"swap] \arrow[dd,"\calD X'(a)"near start] && \udl{X'}(i)\arrow[ll]\arrow[dl]\arrow[dd] \\
& \calD(V(i)) \arrow[rr,"\blank\Lox_{V(i)}{M(i)}"] \arrow[dd] & & \calD(X(i))\arrow[dd,"\calD X(a)"]& &W(i)\arrow[ll, crossing over]\arrow[dd] \\
\udl{V}(i) \arrow[rr, crossing over] \arrow[ur]\arrow[dd] & & W(i) \arrow[ur] & & \calD(X'(j))\arrow[dl,"\bfL \bfF(j)"] & & \udl{X'}(j)\arrow[ll]\arrow[dl]\\
& \calD(V(j))  \arrow[rr] & & \calD(X(j))  & & W(j)\arrow[ll]\\
\udl{V}(j) \arrow[ur] \arrow[rr] & & W(j)  \arrow[ur] \arrow[from=uu, crossing over]\\
\end{tikzcd}
\end{equation}

Keeping the notations above,
we define an $X'$-$V$-bimodule $Y$ as follows.
First for each $i \in I_0$, $Y(i)$ is defined by
\begin{equation}
\label{eq:Y(i)}
Y(i)(B,C)=\Cdg(X(i))({}_BM(i), \bfF(i)(C^\wedge))
\end{equation}
for all $B\in V(i), C\in X'(i)$.
To define $Y(a)$ for each $a \in I_1$, we notice the following.
By Lemma \ref{2-morph}, we have 
\[
\widehat{M(a)}\colon \blank (\ox_{V(i)}\ovl{V(a)}) \circ \Cdg(V(i))(M(i),\blank)\To 
 \Cdg(V(j))(M(j),\blank) \circ (\blank \ox_{X(i)}\ovl{X(a)}).
 \]
Then we have the following
$$
\begin{aligned}
& (\blank \ox_{V(i)}\ovl{V(a)}) \circ \Cdg(V(i))(M(i),\blank) \circ \bfF(i)\\
\overset{\mathrm{(a)}}{\To}& 
 \Cdg(V(j))(M(j),\blank) \circ (\blank \ox_{X(i)}\ovl{X(a)})\circ\bfF(i)\\
\overset{\mathrm{(b)}}{\To}& 
 \Cdg(V(j))(M(j),\blank)
\circ \bfF(j)\circ (\blank \ox_{X'(i)}\ovl{X'(a)}).
\end{aligned}
$$
In the above, (a) is given by $\widehat{M(a)}\circ\bfF(i)$, and
(b) is given by 
$ \Cdg(V(j))(M(j),\blank)\circ\bfps(a)$,
where $\bfps(a)$ is a 2-morphism
$(\blank \ox_{X(i)}\ovl{X(a)})\circ\bfF(i)\To
\bfF(j)\circ (\blank \ox_{X'(i)}\ovl{X'(a)})$.

Now for $a\in I_1$, and any $C\in X'(i)$, we have
\begin{equation}
\label{eq:(c)}
\begin{aligned}
 C^\wedge \ox_{X'(i)}Y(i) \ox_{V(i)}\ovl{V(a)}
&\iso Y(i)(\blank, C)\ox_{V(i)}\ovl{V(a)}\\
&= \Cdg(X(i))(M(i), \bfF(i)(C^\wedge))\ox_{V(i)}\ovl{V(a)}\\
&= [(\blank \ox_{V(i)}\ovl{V(a)}) \circ \Cdg(V(i))(M(i),\blank) \circ \bfF(i)](C^\wedge)\\
& \overset{\mathrm{(c)}}{\To}
[\Cdg(V(j))(M(j),\blank)
\circ \bfF(j)\circ (\blank \ox_{X'(i)}\ovl{X'(a)})](C^\wedge)\\
&= \Cdg(V(j))(M(j),\bfF(j)(C^\wedge \ox_{X'(i)}\ovl{X'(a)}))\\
&\iso
C^\wedge \ox_{X'(i)}\ovl{X'(a)} \ox_{X'(j)}Y(j)
\end{aligned}
\end{equation}
where (c) is given as the composite of (a) and (b) above.
By the same argument used in the proof of Lemma \ref{lem:bimodule},
we can construct a bimodule morphism 
\begin{equation}
\label{eq:Y(a)}
 Y(a)\colon Y(i) \ox_{V(i)}\ovl{V(a)}
\to \ovl{X'(a)} \ox_{X'(j)}Y(j).
\end{equation}

\begin{lem}
\label{quasi-equiv-bim}
Take any $X, X'\in \Colax(I, $\linebreak[2]$\kdgCat)$.
Let $W$ be a colax subfunctor of $\calD(X)$
such that $W(i)$ is a full subcategory of $\calD(X(i))$ for all $i \in I_0$,
and let $(V, {}_VM_X)$ be a lift of $W$ such that ${}_BM(i)$ is homotopically projective for each $B\in V(i)$.
Take any $1$-morphism
$(\bfF, \bfps)\colon \DGMod(X') \to \DGMod(X)$ in the $2$-category
$\Colax(I,$\linebreak[2]$\kdgCAT)$
such that $\bfL(\bfF, \bfps)\colon$ $\calD(X') \to \calD(X)$
restricts to a $1$-morphism $\udl{X'}\to W$
in the $2$-category $\Colax(I, \kuTRI)$.
Let $Y$ be a $X'$-$V$-bimodule defined by \eqref{eq:Y(i)} and \eqref{eq:Y(a)}.
Then the following statements hold.
\begin{enumerate}
\item
\begin{enumerate}
\item $Y$ is a quasi-functor, and hence a $1$-morphism 
$\blank\Lox_{X'}Y \colon \udl{X'}\to \udl{V}$ is induced; and 

\item $Y$ is a quasi-equivalence bimodule if $\bfL(\bfF, \bfps)$ induces an equivalence $\bbZ\udl{X'}\to \bbZ W$.
\end{enumerate}
\item There is a $2$-morphism
$$
\ze \colon (\blank\Lox_{X'}Y)\Lox_{V}M = (\blank\Lox_{V}M)\circ(\blank\Lox_{X'}Y)\To \bfL \bfF
$$
having the following properties:
\begin{enumerate}
\item
For all $i \in I_0$ and $Q\in\hprj^b(X'(i))$,
$$
\ze(i) \colon (Q\Lox_{X'(i)}Y(i))\Lox_{V(i)}M(i)\isoto \bfL \bfF(i)(Q)
$$
is an isomorphism.
\item
For all $i\in I_0$, $\ze(i)$ is an isomorphism for all $Q\in\calD(X(i))$ if and only if $\bfL \bfF(i)$ commutes with direct sums (namely cocontinuous).
\end{enumerate}

\item If $(X',Z)$ is a lift of $W$ and $\bfF=\blank\otimes Z$, then $Y$ is a quasi-equivalence $X'$-$V$-bimodule and we have 
$(\blank\Lox_{X'}Y)\Lox_{V}M
= (\blank \Lox_{V}M) \circ (\blank \Lox_{X'}Y) \iso \blank \Lox_{X'}Z$. 
\end{enumerate}
\end{lem}

\begin{proof}
(1)(a)
To show that $Y$ is a quasi-functor,
it suffices to show the following statements (i) and (ii) by Definition \ref{dfn:qe-bimod}:

(i)  For each $i\in I$, 
$\blank\Lox_{X'(i)}Y(i)$ induces a functor $\udl{X'}(i)\to \udl{V}(i)$; and

(ii) In the diagram
\[
\begin{tikzcd}[row sep=scriptsize, column sep=scriptsize]
& \calD(X'(i)) \arrow[dl,"\calD(X'(a))"'] \arrow[rr,"\blank\Lox_{X'(i)}Y(i)"]  & & \calD(V(i)) \arrow[dl,"\blank\Lox\ovl{V(a)}"near end]  \\
\calD(X'(j)) 
& & \calD(V(j)) \\
&\udl{X'}(i) \arrow[dl] \arrow[rr]& & \udl{V}(i)  \arrow[dl]\arrow[uu]\\
\udl{X'}(j) \arrow[rr]\arrow[uu] & & \udl{V}(j)
\Ar{3-2}{1-2}{}
\Ar{2-1}{2-3}{"\blank\Lox_{X'(j)}Y(j)"'near end,crossing over}
\Ar{1-4}{2-1}{"\blank\Lox_{X'(i)}Y(a)"' description, Rightarrow,crossing over}
\Ar{3-4}{4-1}{"\blank\Lox_{X'(i)}Y(a)"' description, Rightarrow}
\Ar{4-3}{2-3}{crossing over}
\end{tikzcd},
\]
the upper natural transformation $\blank\Lox_{X'(i)}Y(a)$
induces the lower natural transformation $\blank\Lox_{X'(i)}Y(a)$
up to associator.

For each $i\in I_0$, we can follow the idea of the proof of \cite[Lemma 7.3]{Ke1} (see Lemma \ref{lem:dg-lift}). We give the details here. Consider the dg functor $\bfG(i)=\Cdg(X(i))(M(i), \blank)\circ \bfF(i)\colon \Cdg(X'(i))\to\Cdg(V(i))$, then we have $\bfL \bfG(i)=\RHom(X(i))(M(i), \blank)\circ \bfL F(i)$. 
Indeed, for each $N\in X'(i)$, we have 
\[
\begin{aligned}
\bfL \bfG(i)(N) &=\bfG(i)(\bfp N)=\Cdg(X(i))(M(i), \bfF(i)(\bfp N))\\
&\overset{*}{=} \RHom(X(i))(M(i), \bfF(i)(\bfp N))\\
&=\RHom(X(i))(M(i), \bfL\bfF(i)(N)),
\end{aligned}
\]
where the equality $\overset{*}{=}$ above follows from the assumption that ${}_BM(i)$ is homotopically projective right $X(i)$-module for all $B \in V(i)$.
Hence, $\bfL \bfG(i)$ induces a dg functor $\udl{X'(i)}\to \udl{V(i)}$. By the argument in \cite[Section 6.4]{Ke1}, we have a dg natural transformation
$$
\th(i) \colon \blank\ox_{X'(i)}Y(i)\To \bfG(i)
$$
such that
$\bfL\th(i) \colon \blank\Lox_{X'(i)}Y(i)\To \bfL \bfG(i)$
is an isomorphism between triangle functors $\calD(X'(i)) \to \calD(V(i))$ for all $i\in I$ . 
Therefore, $\blank\Lox_{X'(i)}Y(i)$ induces a functor $\udl{X'(i)}\to \udl{V(i)}$. It follows from  \eqref{eq:Y(a)} that we have the following diagram
\[
\tiny
\begin{tikzcd}[column sep=200pt, row sep=150pt]
\udl{X'(i)}& \udl{V(i)} \\
\udl{X'(j)}& \udl{V(j)} 
\Ar{1-1}{2-1}{"\blank\Lox_{X'(i)}\ovl{X'(a)}"'}
\Ar{1-2}{2-2}{"\blank\Lox\ovl{V(a)}"}
\Ar{1-1}{1-2}{"\blank\Lox_{X'(i)}Y(i)"}
\Ar{2-1}{2-2}{"\blank\Lox_{X'(j)}Y(j)"'}
\Ar{1-1}{2-2}{"\blank\Lox_{X(i)}(Y(i)\ox_{V(i)}\ovl{V(a)})"{name=A},
bend left=20pt, ""'{name=a}}
\Ar{1-1}{2-2}{"\blank\Lox_{X'(i)}(\ovl{X'(a)}\ox_{X'(j)}Y(j))"'{name=B}, bend right=20pt, ""{name=b}}
\Ar{a}{b}{"\blank\Lox_{X'(i)}Y(a)"' description, Rightarrow}
\Ar{1-2}{A}{"\mathrm{associator}" description, Rightarrow}
\Ar{2-1}{B}{"\mathrm{associator}" description,Rightarrow}
\end{tikzcd},
\]
where the associators above are usual ones and not derived associators
because we can change all $\Lox$ to $\ox$.
This completes the proofs of (i) and (ii).

(b) We have the following natural transformation
{\footnotesize
\begin{equation}
\label{eq:ze(i)}
\begin{aligned}
\ze(i)\colon(\blank\Lox_{X'(i)}Y(i))\Lox_{V(i)}M(i)
&\ya{\bfL\th(i) \Lox_{V(i)}M(i)} (\blank\Lox_{V(i)}M(i))\circ\bfL \bfG(i)
\\&=(\blank\Lox_{V(i)}M(i))\circ\RHom(X(i))(M(i), \blank)\circ \bfL \bfF(i)\\
&\ya{\de(i) \circ \bfL\bfF(i)} \bfL \bfF(i),
\end{aligned}
\end{equation}
}
where $\de(i)$ is a counit of the adjoint pair given by $M(i)$.
Here note that $\ze(i)$ is invertible on $\hprj^b(X'(i))$ for each $i\in I$.
Indeed, since $\blank \Lox_{V(i)}M(i)$ restricts to an equivalence
$\udl{V(i)} \to W(i)$, the counit $\de(i)$ is invertible on $W(i)$.
Therefore, Since
$\bfL\bfF(i) \colon \calD(X'(i)) \to \calD(X(i))$ restricts to an equivalence
$\udl{X'(i)} \to W(i)$, it follows that
$\de(i) \circ \bfL\bfF(i)$ is invertible on $\udl{X'(i)}$.
Hence $\ze(i)$ is invertible on $\udl{X'(i)}$, and hence on $\hprj^b(X'(i))$
as desired.
Consider the following diagram on the left:
\[\begin{tikzcd}[column sep=15pt]
	& {\calD(V(i))} \\
	{\calD(X'(i))} && {\calD(X(i))} \\
	& {\bbZ\underline{V}(i)} \\
	{\bbZ\underline{X'}(i)} && {\bbZ W(i)}
	\arrow["{\blank\Lox_{V(i)}M(i)}", from=1-2, to=2-3]
	\arrow["{\blank\Lox_{X'(i)}Y(i)}"{pos=0.3}, from=2-1, to=1-2]
	\arrow[""{name=0, anchor=center, inner sep=0}, "{\bfL\bfF(i)}"{pos=0.2}, from=2-1, to=2-3]
	\arrow[hook, from=3-2, to=1-2]
	\arrow["\be", from=3-2, to=4-3]
	\arrow[hook, from=4-1, to=2-1]
	\arrow["\al", from=4-1, to=3-2]
	\arrow[""{name=1, anchor=center, inner sep=0}, "\ga"', from=4-1, to=4-3]
	\arrow[hook, from=4-3, to=2-3]
	\arrow[shift left=3, shorten >=3pt, Rightarrow, from=1-2, to=0]
	\arrow["\ep", shorten >=3pt, Rightarrow, from=3-2, to=1]
\end{tikzcd}
\quad
\begin{tikzcd}[column sep=15pt]
	{\bbZ\underline{V}(i)(\al x, \al y)} & {\bbZ W(i)(\be\al x, \be\al y)} \\
	{\bbZ\underline{X'}(i)(x,y)} & {\bbZ W(i)(\ga x, \ga y)}
	\arrow["\be", tail, two heads, from=1-1, to=1-2]
	\arrow["{\ep_y(\blank)\ep_x\inv}", tail, two heads, from=1-2, to=2-2]
	\arrow["\al", from=2-1, to=1-1]
	\arrow["\ga"', from=2-1, to=2-2]
\end{tikzcd},
\]
where $\al, \be, \ga, \ep$ in the lower triangle are induced from
the corresponding ones in the upper triangle. 
Then $\ep \colon \be \circ \al \To \ga$ is a natural isomorphism.
Hence $\blank\Lox_{X'}Y(i)$ induces an equivalence $\al \colon \bbZ\udl{X'}(i)\to \bbZ \udl{V}(i)$ if and only if $\bfL\bfF(i)$ induces an equivalence
$\ga \colon \bbZ\udl{X'}(i)\to \bbZ W(i)$
because $\be \colon \bbZ\udl{V}(i) \to \bbZ W(i)$ is an equivalence
by the assumption that $(V, M)$ is a lift of $W$
(see the diagram above on the right for full faithfulness part
to check that $\al$ is bijective if and only if so is $\ga$).

It remains to show that for each $a\colon i\to j$  in $I$,
the 2-morphism
$$(\blank\Lox_{X'(i)}Y)(a) \colon (\blank\Lox_{X'(i)} Y(i))  \Lox_{V(i)}\ovl{V(a)})
\To (\blank\Lox_{X'(i)} \ovl{X'(a)}) \Lox_{X'(j)}Y(j))
$$
in the diagram below is an isomorphism on $\udl{X'(i)}$:
{\footnotesize
\[
\begin{tikzcd}[column sep = 80pt, row sep = 60pt]
	\udl{X'(i)} & \udl{V(i)} \\
	\udl{X'(j)} & \udl{V(j)}
	\arrow["{\blank\Lox_{X'(i)}Y(i)}", from=1-1, to=1-2]
	\arrow["{\blank\Lox_{X'(i)}\ovl{X'(a)}}"', from=1-1, to=2-1]
	\arrow["{(\blank\Lox_{X'(i)}Y)(a)}", Rightarrow, dashed, from=1-2, to=2-1]
	\arrow["{\blank\Lox_{X(i)}\ovl{V(a)}}", from=1-2, to=2-2]
	\arrow["{\blank\Lox_{X'(j)}Y(j)}"', from=2-1, to=2-2]
\end{tikzcd}.\]
}
To show this, we use the following commutative diagram: 
{\footnotesize
\begin{equation}
\label{eq:L_,_}
\begin{tikzcd}[column sep = 50pt, row sep = 60pt]
(\blank\Lox_{V(i)}\ovl{V(a)}) \circ (\blank\Lox_{X'(i)}Y(i)) & (\blank\Lox_{X'(j)}Y(j) \circ (\blank\Lox_{X'(i)}\ovl{X'(a)})\\
\bfL((\blank\ox_{V(i)}\ovl{V(a)}) \circ (\blank\ox_{X'(i)}Y(i))) & \bfL((\blank\ox_{X'(j)}Y(j))\circ(\blank\ox_{X'(i)}\ovl{X'(a)}))
\Ar{1-1}{1-2}{"(\blank\Lox_{X'} Y)(a)", Rightarrow}
\Ar{2-1}{2-2}{"\blank\Lox_{X'(i)} Y(a)", Rightarrow}
\Ar{1-1}{2-1}{"\bfL_{(\blank\ox_{V(i)}\ovl{V(a)}),(\blank\ox_{X'(i)}Y(i))}=:\bfV",Rightarrow}
\Ar{1-2}{2-2}{equal}
\end{tikzcd}
\end{equation}
}
(see Definition \ref{dfn:std-der-eq-2} for the vertical morphism on the left).
Now let $x'\in X'(i)$, and consider the following diagram:
{\footnotesize
$$
\begin{tikzcd}[column sep=70pt]
(x'^\wedge\Lox_{X'(i)}Y(i))\Lox_{V(i)}\ovl{V(a)}& (x'^\wedge\ox_{X'(i)}Y(i))\Lox_{V(i)}\ovl{V(a)}\\
(x'^\wedge\ox_{X'(i)}\ovl{X'(a)})\ox_{X'(j)}Y(j))
& (x'^\wedge\ox_{X'(i)}Y(i))\ox_{V(i)}\ovl{V(a)}\\
\Ar{1-1}{1-2}{"\iso"}
\Ar{1-2}{2-2}{"\iso", "\text{(a)}"'}
\Ar{1-1}{2-1}{"(\blank\Lox_{X'} Y)(a)_{x'^\wedge}"'}
\Ar{2-2}{2-1}{"(\blank\Lox_{X'(i)}Y(a))_{x'^\wedge}"}
\Ar{1-1}{2-2}{"\bfV_{x'^\wedge}"}
\end{tikzcd},
$$
}
where (a) above becomes isomorphism since $x'^\wedge\Lox_{X'(i)}Y(i)$ is a representable functor by the
definition of quasi-functors, and the left triangle above is commutative by the
commutative diagram \eqref{eq:L_,_}.
Moreover, it is easy to check the commutativity of the right triangle,
which shows that $\bfV_{x'^\wedge}$ is an isomorphism.
Therefore, it remains to show that $x'^\wedge\Lox_{X'(i)} Y(a)$ is an isomorphism. 
This follows from the fact that
$x'^\wedge\Lox_{X'(i)} Y(a)$ is obtained by composing the following
isomorphsms:
\[
\tiny
\begin{tikzcd}
x'^\wedge\Lox_{X'(i)}(Y(i)\ox_{V(i)}\ovl{V(a)}) & x'^\wedge\ox_{X'(i)}(Y(i)\ox_{V(i)}\ovl{V(a)})\\
& (x'^\wedge\ox_{X'(i)}Y(i))\Lox_{V(i)}\ovl{V(a)}\\
&Y(i)(\blank,x')\Lox_{V(i)}\ovl{V(a)}\\
&\Cdg(X(i))(M(i), \bfF(i)(x'^\wedge))\Lox_{V(i)}\ovl{V(a)}\\
&\RHom(X(i))(M(i),\bfF(i)(x'^\wedge))\Lox_{V(i)}\ovl{V(a)}\\
&\RHom(X(j))(M(j),(\bfL\bfF(i)(x'^\wedge))\Lox_{X(i)}\ovl{X(a)})\\
& \RHom(X(j))(M(j),\bfL\bfF(j)((x'^\wedge)\Lox_{X'(i)}\ovl{X'(a)})\\
&\RHom(X(j))(M(j),\bfL\bfF(j)((X'(a)x')^\wedge))\\
&Y(j)(\blank,X'(a)x')\\
(x'^\wedge)\Lox_{X'(i)}(\ovl{X'(a)}\ox_{X'(j)}Y(j))& ((x'^\wedge)\ox_{X'(i)}\ovl{X'(a)})\ox_{X'(j)}Y(j)\\
\Ar{1-1}{10-1}{"x'^\wedge\Lox_{X'(i)} Y(a)"'}
\Ar{1-1}{1-2}{"\iso"}
\Ar{1-2}{2-2}{"\iso"}
\Ar{2-2}{3-2}{"\iso"}
\Ar{3-2}{4-2}{"\iso"}
\Ar{4-2}{5-2}{"\iso"}
\Ar{5-2}{6-2}{"{\rm Definition\ \ref{lift:inverse}}"}
\Ar{6-2}{7-2}{"{\RHom(X(j))(M(j),(\bfL\bfps)(a))}"}
\Ar{7-2}{8-2}{"\iso"}
\Ar{8-2}{9-2}{"\iso"}
\Ar{9-2}{10-2}{"\iso"}
\Ar{10-2}{10-1}{"\iso"}
\end{tikzcd},
\]
where $\RHom(X(j))(M(j),(\bfL\bfps)(a))$ is an isomorphism because
$(\bfL\bfps)(a)\colon$ \linebreak[4]$\bfL\bfF(i)(x'^\wedge)\Lox_{X(i)}\ovl{X(a)}\to\bfL\bfF(j)((x'^\wedge)\Lox_{X'(i)}\ovl{X'(a)})$
is a 2-isomorphism.
This proves statement (b) by Lemma \ref{lem:equiv-bimodule}.

(2)
We show that the family $\ze:= (\ze(i))_{i\in I_0}$ turns out to be
a 2-morphism.
For each $a\colon i\to j \in I_1$, the bimodules $Y(i), M(i)$ and the functor $\bfF(i)$ give rise to the following diagram 
\[
\begin{tikzcd}[column sep=50pt]
\calD(X'(i))&\calD(V(i))&\calD(X(i)) \\
\calD(X'(j))&\calD(V(j)) &\calD(X(j))
\Ar{1-1}{2-1}{"\calD(X'(a))"'}
\Ar{1-2}{2-2}{"\calD(V(a))"}
\Ar{1-1}{1-2}{"\blank\Lox_{X'(i)}Y(i)"}
\Ar{2-1}{2-2}{"\blank\Lox_{X'(j)}Y(j)"'}
\Ar{1-2}{2-1}{"(\bfL Y)(a)"', Rightarrow}
\Ar{1-2}{1-3}{"\blank\Lox_{V(i)}M(i)"}
\Ar{2-2}{2-3}{"\blank\Lox_{V(j)}M(j)"'}
\Ar{1-3}{2-3}{"\calD(X(a))"}
\Ar{1-3}{2-2}{"(\bfL M)(a)", Rightarrow}
\Ar{1-1}{1-3}{ bend left=30, "\bfL\bfF(i)"}
\Ar{2-1}{2-3}{ bend right=30, "\bfL\bfF(j)"}
\end{tikzcd},
\]
By Definition \ref{dfn:comp-1-mors}, the composite of the 1-morphisms $(\blank\Lox M(i),\blank\Lox M(a))$ and $(\blank\Lox Y(i),\blank\Lox Y(a))$ is defined by
$$
(\blank\Lox M(i),\blank\Lox_M(a))\circ(\blank\Lox_Y(i),\blank\Lox Y(a))
:= ((\blank\Lox_{X'(i)}Y(i))\Lox_{V(i)}M(i), (\bfL M\circ\bfL Y)(a)),
$$
where 
$
(\bfL M\circ\bfL Y)(a):= (\blank\Lox_{V(j)}M(j)\circ(\bfL Y)(a))\bullet ((\bfL M)(a)\circ(\blank\Lox_{X'(i)}Y(i))).
$
The 1-morphism $\ze(i)\colon (\blank\Lox_{X'(i)}Y(i))\Lox_{V(i)}M(i)\To \bfL \bfF(i)$ yields the following commutative diagram
for any $x'\in X'(i)$ and $a\colon i\to j \in I_1$:
\[
\tiny
\begin{tikzcd}[column sep=60pt]
((x'^\wedge)\Lox_{X'(i)} Y(i)\Lox_{V(i)} M(i))\Lox_{X(i)}\ovl{X(a)} &\bfL\bfF(i)((x'^\wedge))\Lox_{X(i)}\ovl{X(a)}\\
 (Y(i)(\blank, x')\Lox_{V(i)} M(i))\Lox_{X(i)}\ovl{X(a)}& \\
Y(i)(\blank, x')\Lox_{V(i)} (M(i)\Lox_{X(i)}\ovl{X(a)})& \\
 Y(i)(\blank, x') \Lox_{V(i)}(\ovl{V(a)}\Lox_{V(j)} M(j))& \\
 (Y(i)(\blank, x') \Lox_{V(i)}\ovl{V(a)})\Lox_{V(j)} M(j)& \\
((x'^\wedge)\Lox_{X(i)}Y(i) \Lox_{V(i)}\ovl{V(a)})\Lox_{V(j)} M(j)& \\
 ((x'^\wedge)\Lox_{X(i)}\ovl{X'(a)}\Lox_{X(j)} Y(j))\Lox_{V(j)} M(j)&
 \bfL\bfF(j)((x'^\wedge)\Lox_{X(i)}\ovl{X'(a)})
\Ar{1-1}{1-2}{"\ze(i)(x'^\wedge)\Lox\ovl{X(a)}"}
\Ar{1-1}{2-1}{"\iso"}
\Ar{2-1}{3-1}{"\iso"}
\Ar{3-1}{4-1}{"\id\Lox_{X(i)}(\bfL M)(a)"}
\Ar{4-1}{5-1}{"\iso"}
\Ar{5-1}{6-1}{"\iso"}
\Ar{1-2}{7-2}{"(\bfL\bfps)(a)(x'^\wedge)"}
\Ar{6-1}{7-1}{"((x'^\wedge)\Lox_{X(i)}(\bfL Y)(a))\Lox_{V(j)}\id"}
\Ar{7-1}{7-2}{"\ze(j)(x'^\wedge\Lox\ovl{X'(a)})"}
\Ar{1-1}{7-1}{"(\bfL M\circ\bfL Y)(a)(x'^\wedge)"', bend right=25pt,start anchor={[xshift=-8em]}, end anchor={[xshift=-8em]}}
\end{tikzcd}.
\]
Thus $\bfL\bfps(a)(x'^\wedge) \circ (\ze(i)(x'^\wedge)\Lox_{X(i)}\ovl{X(a)})
= (\bfL M\circ\bfL Y)(a)(x'^\wedge) \circ \ze(j)(x'^\wedge\Lox\ovl{X'(a)})$. Consequently, $\ze \colon (\blank\Lox_{X'}Y)\Lox_{V}M \To \bfL\bfF$ is a 2-morphism
because the smallest localizing subcategory of $\calD(X'(i))$ containing $x'^\wedge$ for all $x'\in X'(i)$ coincides with $\calD(X'(i))$.

On the other hand,
by Lemma \cite[Lemma 7.3(2)]{Ke1} (see Lemma \ref{lem:dg-lift}),
we have the following two statements:

(i) For each $i\in I_0$ and $Q\in\calD(X(i))$, 
$
\ze(i) \colon Q\Lox_{X'(i)}Y(i)\Lox_{V(i)}M(i)\to \bfL \bfF(i)(Q)
$
is invertible if $Q\in\hprj^b(X'(i))$, and

(ii) $\ze(i)$ is invertible for all $Q\in\calD(X(i))$ if and only if $\bfL F(i)$ commutes with direct sums.

Hence statement (2) holds.

(3) This is immediate from (1) and (2).  
\end{proof}

Recall that the Morita theory for dg categories was given by Keller.
We will give a similar theorem in our setting.

\begin{thm}
\label{thm:characterization-1}
Let $X,X'\in \Colax(I,\kdgCat)$. 
Then we have implications $(1) \Rightarrow (2) \Rightarrow (3) \Rightarrow (4)$. If $X$ and $X'$ are $\k$-flat, then we also have the implication $(4) \Rightarrow (1)$, namely, all four conditions are equivalent.
\begin{enumerate}
\item
There exists an $X'$-$X$-bimodule $Z$ such that
$\blank\Lox_{X'}Z \colon \calD(X') \to \calD(X)$ is an equivalence in
$\Colax(I, \kuTRI^2)$.

\item
There exists a $1$-morphism $(\bfF, \bfps)\colon \DGMod(X') \to \DGMod(X)$ in the $2$-category \newline $\Colax(I,\kdgCAT)$ such that
$\bfL(\bfF, \bfps)\colon$ $\calD(X') \to \calD(X)$ is an equivalence in $\Colax(I, \kuTRI^2)$.
\item
There exists a tilting colax functor $T$ for $X$, and 
there exists a quasi-equivalence $X'$-$T$-bimodule $E$.

\item
There exists a tilting colax functor $T$ for $X$, and
there exist $1$-morphisms
\[
(\bfG, \bfps')\colon \DGMod(X') \to \DGMod(T)\ \text{and } (\bfF, \bfps)\colon \DGMod(T) \to \DGMod(X)
\]
in the $2$-category
$\Colax(I,\kdgCAT)$ such that 
\[\bfL(\bfG, \bfps')\colon\calD(X') \to \calD(T)\ \text{and}\ 
\bfL(\bfF, \bfps)\colon\calD(T) \to \calD(X)
\]
are equivalences in $\Colax(I, \kuTRI^2)$.

\end{enumerate}
\end{thm}

\begin{proof}

(1) \implies (2).
We can take $(\bfF, \bfps):= \blank\ox_{X'}Z$.

(2) \implies (3).
Assume that the statement (2) holds. 
Then for each $a \colon i \to j$ in $I$, we have a diagram
\begin{equation}
\label{eq:11-mor-Cdg}
\vcenter{
\xymatrix{
\DGMod(X'(i)) & \DGMod(X(i))\\
\DGMod(X'(j)) & \DGMod(X(j))
\ar"1,1";"1,2"^{\bfF(i)}
\ar"2,1";"2,2"_{\bfF(j)} 
\ar"1,1";"2,1"_{\DGMod(X'(a))}
\ar"1,2";"2,2"^{\DGMod(X(a))}
\ar@{=>}"1,2";"2,1"_{\bfps(a)}
}
},
\end{equation}
where we note that $\bfps(a)$ is a dg natural transformation by assumption.
For each $i \in I_0$, we set $W(i)$ to be the
full dg subcategory of $\DGMod(X(i))$
with
\[
\begin{aligned}
W(i)_0 = \{\bfL \bfF(i)(C^\wedge) \mid C \in X'(i)_0\}
\end{aligned}.
\]
Let $(T, M)$ be the standard lift of $W$
(see Definition \ref{dfn:lift}(3)):
\[
\begin{aligned}
T(i)_0&:= \{\bfp_{X(i)}U \mid U \in W(i)_0\},\\
M(i)(A, \bfp_{X(i)}U)&:=\bfp_{X(i)}U(A)
\end{aligned}
\]
for all $i \in I_0$.
By Lemma \ref{quasi-equiv-bim}(3), we define a quasi-equivalence $X'$-$T$-bimodule $E$ by setting
\[
E(i)(B,C)=\Cdg(X(i))(_{B}M(i),\bfF(i)(C^\wedge)).
\]
for all $i \in I_0$, $B \in T(i), C \in X'(i)$, and a bimodule morphism 
\[
 E(a)\colon \blank\ox_{X'(i)}E(i) \ox_{V(i)}\ovl{V(a)}
 \\
\to 
  \blank\ox_{X'(i)}\ovl{X'(a)} \ox_{X'(j)}E(j).
\]
It follows from \cite[Lemma 7.3]{Ke1} that $E(i)$ is a quasi-equivalence $X'(i)\to T(i)$, and it follows from Lemma \ref{quasi-equiv-bim} that we have the following diagram such that $\blank\Lox_{X'(i)}E(a)$ is an 2-isomorphism

\begin{tikzcd}[row sep=scriptsize, column sep=scriptsize]
& \calD(X'(i)) \arrow[dl,"\calD(X'(a))"'] \arrow[rr,"\blank\Lox_{X'(i)}E(i)"]  & & \calD(T(i)) \arrow[dl,"\calD(T(a))"near end]  \\
\calD(X'(j))\arrow[rr,"\blank\Lox_{X'(j)}Y(j)"'near end]  & & \calD(T(j)) \\
&\udl{X'}(i) \arrow[dl] \arrow[rr] \arrow[uu,crossing over] & & \udl{T}(i)  \arrow[dl]\arrow[uu]\\
\udl{X'}(j) \arrow[rr]\arrow[uu] & & \udl{T}(j)\arrow[uu,crossing over].
\Ar{1-4}{2-1}{"(\blank\Lox_{X'(i)}E)(a)"' description, Rightarrow,crossing over}
\Ar{3-4}{4-1}{"(\blank\Lox_{X'(i)}E)(a)"' description, Rightarrow}
\end{tikzcd}

(3) \implies (4).
Assume that the statement (3) holds.
Then there exists a quasi-equivalence $X'$-$T$-bimodule $E$ 
for $X$, which gives us the necessary 1-morphism
$(\bfG, \bfps'):= \blank\ox_{X'} E \colon \Cdg(X') \to \Cdg(T)$
in $\Colax(I, \kdgCAT)$ and 
an equivalence $\bfL(\bfG, \bfps') =\blank\Lox_{X'} E \colon \calD(X') \to \calD(T)$
in $\Colax(I, \kuTRI^2)$.

We define a $T$-$X$-bimodule $M$ as follows.
For each $i \in I_0$, we set
\[
{}_BM(i)_A:= B(A) \iso \DGMod(X(i))(A^\wedge, B)
\]
for all $B \in T(i)_0, A \in X(i)_0$.
For each $a \colon i \to j$ in $I$, we define
a $0$-cocycle morphism
\[
M(a) \colon M(i) \ox_{X(i)}\ovl{X(a)} \to \ovl{T(a)} \ox_{T(j)}M(j)
\]
of $T(i)$-$X(j)$-bimodules by
\[
\begin{aligned}
M(a)_{C^\wedge}\colon C^\wedge\ox_{T(i)} M(i)\ox_{X(i)}\ovl{X(a)}\iso M(i)(\blank,C)\ox_{X(i)}\ovl{X(a)}\\
\iso C\ox_{X(i)}\ovl{X(a)}\iso
T(a)(C)\iso
C^\wedge\ox\ovl{T(a)}\ox M(j)
\end{aligned}
\]
for all $a\in I_1$ and $C\in T(i)$.

Let $(\si, \ro) \colon T \to \DGMod(X)$ be an $I$-equivariant inclusion,
and $i \in I_0$.
We define a dg functor
\[
F(i) \colon \DGMod(T(i)) \to \DGMod(X(i))
\]
by setting $F(i):= \blank \ox_{T(i)}M(i)$.
Then since for each $B \in T(i)_0$, the right dg $X(i)$-module
${}_BM(i) = B\ (\in T(i)_0 \subseteq \hprj(X(i))_0)$
is homotopically projective,
$F(i)$ preserves homotopically projectives by Lemma \ref{lem:right-htp}.
Now for any $B, C \in T(i)_0$, the bimodule $M(i)$ defines a morphism
${}_?M(i) \colon T(i)(B, C) \to \DGMod(X(i))({}_BM(i), {}_CM(i))$ in $\DGMod(\k)$ by sending each $f \colon B \to C$ to ${}_fU(i) \colon {}_BM(i) \to {}_CM(i)$.
Here since we have ${}_BM(i) = B$, ${}_fM(i)= f$ by definition,  ${}_?M(i)$ is the identity of $T(i)(B,C)$.
Hence it induces an isomorphism in homology.  Moreover, $\{{}_BM(i) \mid B \in T(i)_0\} = T(i)_0$ and $T(i)$ is a tilting dg category for $X(i)$.
Hence by \cite[Lemma 6.1(a)]{Ke1},
$\bfL F(i) = \blank \Lox_{T(i)}M(i) \colon \calD(T(i)) \to \calD(X(i))$ is a triangle equivalence. Next for each $a \colon i \to j$ in $I$, we construct a $2$-morphism $\bfps(a)$ in the diagram
\[
\vcenter{
\xymatrix{
\DGMod(T(i)) & \DGMod(X(i))\\
\DGMod(T(j)) & \DGMod(X(j))
\ar"1,1";"1,2"^{\bfF(i)}
\ar"2,1";"2,2"_{\bfF(j)}
\ar"1,1";"2,1"_{\DGMod(T(a))}
\ar"1,2";"2,2"^{\DGMod(X(a))}
\ar@{=>}"1,2";"2,1"_{\bfps(a)}
}}.
\]
where \[
\begin{aligned}
\bfps(a):=\blank\ox_{T(i)} M(a):\\\blank \ox_{T(i)}M(i) \ox_{X(i)}\ovl{X(a)}\To \blank \ox_{T(i)}\ovl{T(a)} \ox_{T(j)}M(j).
\end{aligned}
\]
Then we have the following diagram such that 
$\blank\Lox_{X'(i)} E(i)$ and $\bfL F(i)$ are triangle equivalences and that $\bfL E(a)$ and $\bfL \bfps(a)$ are $2$-isomorphisms
\[
\begin{tikzcd}[column sep=70pt]
\calD(X'(i))&\calD(T(i))&\calD(X(i)) \\
\calD(X'(j))&\calD(T(j)) &\calD(X(j))
\Ar{1-1}{2-1}{"\calD(X'(a))"'}
\Ar{1-2}{2-2}{"\calD(T(a))"}
\Ar{1-1}{1-2}{"\blank\Lox_{X'(i)} E(i)"}
\Ar{2-1}{2-2}{"\blank\Lox_{X'(j)} E(j)"'}
\Ar{1-2}{1-3}{"\bfL F(i)"}
\Ar{2-2}{2-3}{"\bfL F(i)"'}
\Ar{1-3}{2-3}{"\calD(X(a))"}
\Ar{1-2}{2-1}{"\bfL E(a)"', Rightarrow}
\Ar{1-3}{2-2}{"\bfL \bfps(a)"', Rightarrow}
\end{tikzcd}.
\]

(4) \implies (1).
Assume that the statement (4) holds.

There is a  diagram
\[
\begin{tikzcd}[column sep=70pt]
\DGMod(X'(i))&\DGMod(T(i))&\DGMod(X(i)) \\
\DGMod(X'(j))&\DGMod(T(j)) &\DGMod(X(j))
\Ar{1-1}{2-1}{"\DGMod(X'(a))"'}
\Ar{1-2}{2-2}{"\DGMod(T(a))"}
\Ar{1-1}{1-2}{"G(i)"}
\Ar{2-1}{2-2}{"G(j)"'}
\Ar{1-2}{2-1}{"\ph(a)"', Rightarrow}
\Ar{1-2}{1-3}{"F(i)"}
\Ar{2-2}{2-3}{"F(j)"'}
\Ar{1-3}{2-3}{"\DGMod(X(a))"}
\Ar{1-3}{2-2}{"\ps(a)", Rightarrow}
\end{tikzcd},
\]
such that we have the following diagram

\[
\begin{tikzcd}[column sep=70pt]
\calD(X'(i))&\calD(T(i))&\calD(X(i)) \\
\calD(X'(j))&\calD(T(j)) &\calD(X(j))
\Ar{1-1}{2-1}{"\calD(X'(a))"'}
\Ar{1-2}{2-2}{"\calD(T(a))"}
\Ar{1-1}{1-2}{"\bfL G(i)"}
\Ar{2-1}{2-2}{"\bfL G(j)"'}
\Ar{1-2}{2-1}{"(\bfL\ph)(a)"', Rightarrow}
\Ar{1-2}{1-3}{"\bfL F(i)"}
\Ar{2-2}{2-3}{"\bfL F(j)"'}
\Ar{1-3}{2-3}{"\calD(X(a))"}
\Ar{1-3}{2-2}{"(\bfL\ps)(a)", Rightarrow}
\end{tikzcd}.
\]
where $\bfL G(i)$ and $\bfL F(i)$ are triangle equivalences and 
\[
\begin{aligned}
(\bfL\ph)(a)\colon\calD(T(a))\circ\bfL G(i)\iso \bfL G(j)\circ\calD(X'(a))\\
(\bfL\ps)(a)\colon\calD(X(a))\circ\bfL F(i)\iso \bfL F(j)\circ\calD(T(a))\\
\end{aligned}
\]
are $2$-isomorphisms.

Let $M(i)$ be the bimodule $M(i)(B,A)=G(i)(A^\wedge)(B)$ for $A\in X'(i)$ and $B\in T(i)$, and let $N(i)$ be the bimodule $N(i)(D,C)=F(i)(C^\wedge)(D)$ for $C\in T(i)$ and $D\in X(i)$. By Lemma \ref{lem:bimodule}, we have 2-morphisms
\[
N(a):= (N(a)_T)_{T \in \Cdg(T(i))} \colon
\blank \ox_{T(i)}N(i) \ox_{X(i)}\ovl{X(a)} \To 
\blank \ox_{T(i)}\ovl{T(a)} \ox_{T(j)}N(j)
\]
and 
\[
M(a):= (M(a)_{T'})_{T' \in \Cdg(X'(i))}\colon \blank \ox_{X'(i)}M(i) \ox_{T(i)}\ovl{T(a)}\To \blank\ox_{X'(i)}\ovl{X'(a)} \ox_{X(j)}M(j).
\]
Since $X$ is $\k$-flat, by \cite[Lemma 6.3(b)]{Ke1}
(see Lemma \ref{lem:flat}),
$\bfL F(i)\circ\bfL G(i)
=(\blank \Lox N(i))\circ (\blank\Lox M(i))
\iso \blank\Lox( M(i)\ox \bfp N(i))$,
where we set
$\bfp N(i):=\bfp_{T(i)\mbox{-}X(i)} N(i)$ for short,
which is a homotopically projective resoultion of $N(i)$ as $X'(i)$-$T(i)$-bimodule
which is a triangle equivalence being the composite of triangle equivalences.

We define an
$X'$-$X$-bimodule $Z$ up to associators.
It is given as follows:
\[
\begin{aligned}
Z(i):&=  M(i)\ox_{T(i)}\bfp N(i)\text{ for all } i \in I_0,\\
Z(a):&=\ {\bfa_{M(i), _\bfp N(i),\ovl{X(a)}}} \circ (M(i) \ox_{T(i)}\bfp N(a)) \circ 
\bfa_{ M(i),\ovl{T(a)},\bfp N(j)}\inv\\
&\ \circ (M(a) \ox_{T(j)} \bfp N(j))
\circ \bfa_{\ovl{X'(a)},M(j),\bfp N(j)}\\
&\phantom{:=\ovl{X'(a)}, M(j), \bfp N(j)}\text{for all $a \colon i \to j$ in $I$},
\end{aligned}
\]
where $\bfa_{M(i), _\bfp N(i),\ovl{X(a)}}$ and $\bfa_{ M(i),\ovl{T(a)},\bfp N(j)}$ are associators, the $X'(i)$-$X(i)$-bimodule morphism $Z(a)$ is defined by the commutative diagram
\begin{equation}
\footnotesize
\label{eq:Z(a)}
\begin{tikzcd}[column sep=50pt]
Z(i)\ox_{X(i)}\ovl{X(a)} & \ovl{X'(a)} \ox_{X'(j)}Z(j)\\
( M(i)\ox_{T(i)}\bfp N(i))\ox_{X(i)}\ovl{X(a)} & \ovl{X'(a)} \ox_{X'(j)}(  M(j) \ox_{T(j)}\bfp N(j))\\
 M(i) \ox_{T(i)}( \bfp N(i)\ox_{X(i)}\ovl{X(a)}) & (\ovl{X'(a)} \ox_{X'(j)} M(j)) \ox_{T(j)}\bfp N(j)\\
  M(i) \ox_{T(i)}(\ovl{T(a)} \ox_{T(j)}\bfp N(j))& ( M(i) \ox_{T(i)}\ovl{T(a)}) \ox_{T(j)} \bfp N(j)
\Ar{1-1}{1-2}{"Z(a)", dashed}
\Ar{1-1}{2-1}{equal}
\Ar{1-2}{2-2}{equal}
\Ar{2-1}{3-1}{"{\bfa_{M(i), _\bfp N(i),\ovl{X(a)}}}"}
\Ar{3-1}{4-1}{"{ M(i) \ox_{T(i)}\bfp N(a)}", dashed}
\Ar{4-1}{4-2}{"{\bfa_{ M(i),\ovl{T(a)},\bfp N(j)}\inv}"'}
\Ar{4-2}{3-2}{" M(a) \ox_{T(j)} \bfp N(j)"}
\Ar{3-2}{2-2}{"{\bfa_{\ovl{X'(a)},M(j),\bfp N(j)}}"}
\end{tikzcd},
\end{equation}
where the morphism $\bfp N(a)$ is defined as follows:
we consider the triangle
\[
\bfp N(i) \ya{\ep_{N(i)}} N(i) \ya{\de_{N(i)}} \bfa N(i) \to 
\]
of $T(i)\mbox{-}X(i)$-bimodules, where $\bfp N(i)$ is homotopically projective and $\bfa N(i)$ is acyclic.
By tensoring with
$\ovl{X(a)}$ from the right, and with $\ovl{T(a)}$ from the left,
this induces two triangles
\begin{equation}
\label{eq:2-triangles}
\begin{tikzcd}[column sep=12pt]
\bfp N(i) \ox_{X(i)}\ovl{X(a)}& N(i) \ox_{X(i)}\ovl{X(a)}& \bfa N(i) \ox_{X(i)}\ovl{X(a)} & {}\\
\ovl{T(a)} \ox_{T(j)}\bfp N(j)&\ovl{T(a)} \ox_{T(j)}N(j)&\ovl{T(a)}\ox_{T(j)} {\bfa N(j)} &
\Ar{1-1}{2-1}{"\bfp N(a)"', dashed}
\Ar{1-2}{2-2}{"N(a)"}
\Ar{1-1}{1-2}{"\ep"}
\Ar{2-1}{2-2}{"\ep"'}
\Ar{1-2}{1-3}{"\de"}
\Ar{2-2}{2-3}{"\de"'}
\Ar{1-3}{2-3}{"\bfa N(a)", dashed}
\Ar{1-3}{1-4}{}
\Ar{2-3}{2-4}{}
\end{tikzcd}
\end{equation}
of $T(i)\mbox{-}X(j)$-bimodules.
We now show that
\[
\Cdg(T(i)\op \ox_\k X(j))(\bfp N(i) \ox_{X(i)}\ovl{X(a)},
\ovl{T(a)}\ox_{T(j)} {\bfa N(j)}) = 0.
\]
For any $t\in T(i)_0$, we have isomorphisms
\[
t^\wedge\ox\ovl{T(a)}\ox_{T(j)} {\bfa N(j)}
\iso 
T(j)(\blank,T(a)t)\ox_{T(j)} {\bfa N(j)}
\iso
\bfa N(j)(\blank,T(a)t),
\]
where the last term is an acyclic right dg $X(j)$-module.
Hence we have
\[
\begin{aligned}
&\Cdg(X(j))(t^\wedge\ox \bfp N(i) \ox_{X(i)}\ovl{X(a)},t^\wedge\ox\ovl{T(a)}\ox_{T(j)} {\bfa N(j)})\\
\iso&\Cdg(X(j))(t^\wedge\ox \bfp N(i) \ox_{X(i)}\ovl{X(a)},\bfa N(j)(\blank,T(a)t))\\
\iso &
\Cdg(X(i)(t^\wedge\ox \bfp N(i), \Cdg(X(j))(\ovl{X(a)}, \bfa N(j)(\blank,T(a)t)).
\end{aligned}
\]
Here we have $\Cdg(X(j))(\ovl{X(a)}, \bfa N(j)(\blank,T(a)t)) = 0$
because $\ovl{X(a)}$ is a homotopocally projective right dg $X(j)$-module.
As a consequence, we have
\begin{equation}
\label{eq:hat-t}
\Cdg(X(j))(t^\wedge\ox \bfp N(i) \ox_{X(i)}\ovl{X(a)},t^\wedge\ox\ovl{T(a)}\ox_{T(j)} {\bfa N(j)})=0.
\end{equation}
Take any $\al \in \Cdg(T(i)\op\ox_\k X(j))(\bfp N(i) \ox_{X(i)}\ovl{X(a)},\ovl{T(a)}\ox_{T(j)} {\bfa N(j)})$.
Then the equality  \eqref{eq:hat-t} means that
for any $t \in T(i)_0$,
we have
\[
0 = \al(?,t)\colon
\Cdg(X(j))((\bfp N(i) \ox_{X(i)}\ovl{X(a)})(?,t),(\ovl{T(a)}\ox_{T(j)} {\bfa N(j)})(?,t)).
\]
Thus $\al = 0$, and hence we have
\[
\Cdg(T(j)\op\ox_\k X(j))(\bfp N(i) \ox_{X(i)}\ovl{X(a)},\ovl{T(a)}\ox_{T(j)} {\bfa N(j)}) = 0.
\]

Then by the diagaram
\eqref{eq:2-triangles}, we see that
there is a unique morphism $\linebreak[3]\bfp N(a)\colon  \bfp N(i) \ox_{X(i)}\ovl{X(a)}\to \ovl{T(a)} \ox_{T(j)}\bfp N(j)$
such that the diagram \eqref{eq:2-triangles} commutes.
In this way $\bfp N(a)$ is defined.

Then by the diagram \eqref{eq:Z(a)}, the $X'$-$X$-bimodule  $Z$ is defined, and we have a $1$-morphism $\blank\ox_{X'}Z \colon \DGMod(X') \to \DGMod(X)$
with the following diagram
\[
\begin{tikzcd}[column sep=70pt]
\DGMod(X'(i))&\DGMod(X(i)) \\
\DGMod(X'(j)) &\DGMod(X(j))
\Ar{1-1}{2-1}{"\DGMod(X'(a))"'}
\Ar{1-2}{2-2}{"\DGMod(X(a))"}
\Ar{1-1}{1-2}{"\blank\ox Z(i)"}
\Ar{2-1}{2-2}{"\blank\ox Z(j)"'}
\Ar{1-2}{2-1}{"\blank\ox Z(a)", Rightarrow}
\end{tikzcd}.
\]
By Definition \ref{dfn:comp-1-mors}, the composite $(\bfL F, \bfL\ps) \circ (\bfL G, \bfL\ph)$
of $(\bfL F, \bfL\ps)$ and
$(\bfL G, \bfL\ph)$
is a 1-morphism from $\calD(X')$ to $\calD(X)$ defined by
$$
(\bfL F, \bfL\ps)\circ (\bfL G, \bfL\ph):= (\bfL F \circ \bfL G, \bfL\ps\circ\bfL\ph),
$$
where
$\bfL F \circ \bfL G:=(\bfL F(i)\circ \bfL G(i))_{i\in I_0}$ and for each $a\colon i \to j$ in $I$,
$
(\bfL\ps \circ \bfL\ph)(a):= (\bfL F(j) \circ (\bfL\ph)(a))\bullet ((\bfL\ps)(a) \circ \bfL G(i))
$
is the pasting of the diagram
\begin{equation}
\label{eq:pasting-Lps-Lph}
\begin{tikzcd}[column sep=70pt]
\calD(X'(i))&\calD(T(i))&\calD(X(i)) \\
\calD(X'(j))&\calD(T(j)) &\calD(X(j))
\Ar{1-1}{2-1}{"\calD(X'(a))"'}
\Ar{1-2}{2-2}{"\calD(T(a))"}
\Ar{1-1}{1-2}{"\bfL G(i)"}
\Ar{2-1}{2-2}{"\bfL G(j)"'}
\Ar{1-2}{2-1}{"(\bfL\ph)(a)"', Rightarrow}
\Ar{1-2}{1-3}{"\bfL F(i)"}
\Ar{2-2}{2-3}{"\bfL F(j)"'}
\Ar{1-3}{2-3}{"\calD(X(a))"}
\Ar{1-3}{2-2}{"(\bfL\ps)(a)", Rightarrow}
\end{tikzcd}.
\end{equation}
Therefore,
\[
\begin{aligned}
\bfL F(j)(\bfL\ph)(a)\colon\bfL F(j)\circ \calD(T(a))\circ\bfL G(i)
\overset{\iso}{\To} \bfL F(j)\circ\bfL G(j)\circ\calD(X'(a))\\
(\bfL\ps)(a)\bfL G(i)\colon\calD(X(a))\circ\bfL F(i)\circ\bfL G(i)
\overset{\iso}{\To} \bfL F(j)\circ\calD(T(a))\circ\bfL G(i)\\
\end{aligned}.
\]
Consequently, 
\[
\begin{aligned}
(\bfL\ps\circ\bfL\ph)(a)&:=(\bfL F(j)\circ\bfL\ph(a))\bullet (\bfL\ps(a)\circ\bfL G(i))\colon\\
\calD(X(a))\circ\bfL F(i)\circ\bfL G(i)
&\overset{\iso}{\To} \bfL F(j)\circ\bfL G(j)\circ\calD(X'(a))
 \end{aligned}
\]
In the following, We will show that $(\blank\Lox Z)(a)$ is equivalent to $(\bfL\ps\circ\bfL\ph)(a)$.

From the definition of $Z(a)$, we have the following diagram:
{\small
\begin{equation}
\label{eq:derasso}
\begin{tikzcd}[column sep=30pt]
\blank\Lox _{X'(i)}Z(i)\Lox_{X(i)}\ovl{X(a)}_{X(j)} & \blank\Lox\ovl{X'(a)} \Lox_{X'(j)}Z(j)\\
\blank\Lox( _{X'(i)}M(i)\ox_{T(i)}\bfp N(i))\Lox_{X(i)}\ovl{X(a)} & \blank\Lox\ovl{X'(a)} \Lox_{X'(j)}(  M(j) \ox_{T(j)}\bfp N(j))\\
(\blank\Lox _{X'(i)}M(i)\Lox_{T(i)}N(i))\Lox_{X(i)}\ovl{X(a)} & \blank\Lox\ovl{X'(a)} \Lox_{X'(j)}(  M(j) \Lox_{T(j)}N(j))\\
\blank\Lox M(i) \Lox_{T(i)}( N(i)\Lox_{X(i)}\ovl{X(a)}) & (\ovl{X'(a)} \Lox_{X'(j)} M(j)) \Lox_{T(j)} N(j)\\
\blank\Lox  M(i) \Lox_{T(i)}(\ovl{T(a)} \Lox_{T(j)} N(j))& (\blank\Lox M(i) \Lox_{T(i)}\ovl{T(a)}) \Lox_{T(j)} N(j)
\Ar{1-1}{1-2}{"(\blank\Lox Z)(a)", dashed}
\Ar{1-1}{2-1}{equal}
\Ar{1-2}{2-2}{equal}
\Ar{2-1}{3-1}{"\iso", Rightarrow}
\Ar{3-2}{2-2}{"\iso"', Rightarrow}
\Ar{3-1}{4-1}{"{\bfa^\bfL_{M(i),  N(i),\ovl{X(a)}}}", Rightarrow}
\Ar{4-1}{5-1}{"{ M(i) \Lox_{T(i)}N(a)}", dashed, Rightarrow}
\Ar{5-1}{5-2}{"(\bfa^{\bfL}_{ M(i),\ovl{T(a)}, N(j)})\inv" ', Rightarrow}
\Ar{5-2}{4-2}{" M(a) \Lox_{T(j)}  N(j)"', Rightarrow}
\Ar{4-2}{3-2}{"{\bfa^\bfL_{\ovl{X'(a)},M(j),N(j)}}"', Rightarrow},
\end{tikzcd}
\end{equation}
}
where $(\blank\Lox N)(a)$ and $(\blank\Lox M)(a)$ are 2-isomorphisms if $(\bfL\ph)(a)$ and $(\bfL\ps)(a)$ are $2$-isomorphsims by Lemma \ref{lem:bimodule}, $\bfa^\bfL_{M(i),  N(i),\ovl{X(a)}}, (\bfa^{\bfL}_{ M(i),\ovl{T(a)}, N(j)})\inv$ and $\bfa^\bfL_{\ovl{X'(a)},M(j),N(j)}$ are derived associators by Proposition \ref{prp:derived-associator} because both $X'(i)$ and $X(i)$ are $\k$-flat.
Note that for each $i \in I_0$,
we have a $2$-isomorphism
\[
\xi_i \colon \bfL F(i)\circ\bfL G(i)
=(\blank \Lox N(i))\circ (\blank\Lox M(i))
\overset{\iso}{\To} \blank\Lox( M(i)\ox \bfp N(i))
= \blank \Lox Z(i),
\]
we get the following commutative diagram
\begin{equation}
\label{eq:Lps-Lph-Z}
\begin{tikzcd}[column sep=50pt]
(\blank\Lox X(a))\circ\bfL F(i)\circ\bfL G(i) & 
\bfL F(j)\circ\bfL G(j)\circ(\blank\Lox\ovl{X'(a)})
\\
\blank\Lox Z(i)\Lox_{X(i)}\ovl{X(a)} & \blank\Lox_{X'((i)}\ovl{X'(a)} \Lox_{X'(j)}Z(j)
\Ar{1-1}{1-2}{"(\bfL\ps\circ\bfL\ph)(a)"}
\Ar{1-1}{2-1}{"\iso", "\xi_i \Lox_{X(i)}\ovl{X(a)}" '}
\Ar{1-2}{2-2}{"\iso" ,"\blank\Lox_{X'(i)}\ovl{X'a)} \Lox_{X'(j)}\xi_i" '}
\Ar{2-1}{2-2}{"(\blank\Lox Z)(a)"}
\end{tikzcd}.
\end{equation}
Therefore, we have
\[
\begin{aligned}
&(\bfL F, \bfL\ps)\circ (\bfL G, \bfL\ph):= (\bfL F \circ \bfL G, \bfL\ps\circ\bfL\ph)\\
=& ((\bfL F \circ \bfL G)(i), (\bfL\ps\circ\bfL\ph)(a))_{i\in I_0, a\in I_1}\\
=& (\bfL G(i) \circ \bfL F(i), (\bfL F(j) \circ \bfL\ph(a))\bullet (\bfL\ps(a) \circ \bfL G(i)))_{i\in I_0, a \in I_1}
\quad \text{by \eqref{eq:pasting-Lps-Lph}}\\
\iso& (\blank\Lox_{X'(i)}Z(i), (\blank\Lox Z)(a))_{i\in I_0, a\in I_1}
\quad \text{by \eqref{eq:Lps-Lph-Z}}\\
=& \blank\Lox_{X'} Z.
 \end{aligned}
\]
Then $\blank\Lox_{X'}Z \colon \calD(X') \to \calD(X)$ is an equivalence in
$\Colax(I, \kuTRI^2)$.
\end{proof}

\begin{dfn}
\label{dfn:std-der-eq}
Let $X, X' \in \Colax(I, \kdgCat)$.
The equivalence form of Statement (1) in Theorem \ref{thm:characterization-1} above
(namely, of the form
$\blank\Lox_{X'}Z \colon \calD(X') \to \calD(X)$ in
$\Colax(I, \kuTRI^2)$ for an $X'$-$X$-bimodule $Z$)
is called a {\em standard derived equivalence} from $X'$ to $X$, and
if it exists, then
we say that $X'$ {\em is standardly derived equivalent to} $X$.
We denote this situation by $X' \stder X$.
At present we do not know whether this relation is symmetric or not.
See Problem \ref{prb:sd-1}.
\end{dfn}

\begin{prb}
\label{prb:sd-1}
Let $X, X' \in \Colax(I, \kdgCat)$.
Under which condition, does
$X' \stder X$ imply $X \stder X'$?
\end{prb}

\begin{rmk}
For dg categories $\calA, \calB$, the relation that $\calB$ is standardly derived equivalent to $\calA$
(see Definition \ref{dfn:dg-std-der-eq}), is known to
be symmetric (if $Z$ is a $\calB$-$\calA$-bimodule such that $\calB$ is standardly derived equivalent to $\calA$, take $Z^T:=\DGMod(\calA)(\bi{Z}{\calB}{\calA}, \bi{\calA}{\calA}{\calA})$ such that $\bi{Z}{\calB}{\calA}$ is homotopically projective $\calA$-module, then $Z^T$ is a $\calA$-$\calB$-bimodule such that $\calA$ is standardly derived equivalent to $\calB$, see \cite[Lemma 6.2]{Ke1}).
\end{rmk}

\section{Derived equivalences of Grothendieck constructions}

In this section,
we give our second main result stating that
for $X, X' \in \Colax(I, \kdgCat)$,
if $X'$ is standardly derived equivalent to $X$,
then their Grothendieck constructions are derived equivalent.

\subsection{Quasi-equivalence 1-morphisms}

In this subsection, we use the contents in Appendix \ref{sec:qeq-1mor}.

\begin{prp}\label{prp:qis-qis-1mor}
Let $X, X' \in \Colax(I, \kdgCat)$.
Assume that $(F, \ps) \colon  X' \to  X$ in $\Colax(I, \kdgCat)$ is a quasi-equivalence 1-morphism. Then $\Gr_I(F, \ps)\colon  \Gr_I X' \to  \Gr_I X$ is a quasi-equivalence.
\end{prp}

\begin{proof}

Recall that a 1-morphism
$$
\Gr_I(F, \ps) \colon \Gr_I X' \to \Gr_I X
$$
in $\DGkCat$ is defined by
\begin{itemize}
\item
for each ${}_ix \in (\Gr_I X)_0$, $\Gr_I(F, \ps)({}_ix'):={}_i(F(i)x)$, and
\item
for each ${}_ix, {}_jy \in (\Gr_I X')_0$ and
each $f=(f_a)_{a\in I(i,j)} \in (\Gr_I X')({}_ix, {}_jy)$,
$\Gr_I(F,\ps)(f):= (F(j)f_a\circ \ps(a)x)_{a\in I(i,j)}$, where each entry is the composite of
$$
X'(a)F(i)x \xrightarrow{\ps(a)x} F(j)X(a)x \xrightarrow{F(j)f_a} F(j)y.
$$
\end{itemize}
Then we have the following
\begin{equation}
\label{eq:Gr-str}
\begin{tikzcd}
(\Gr_I X')({}_ix, {}_jy) & (\Gr_I X)(\Gr_I(F, \ps)({}_ix), \Gr_I(F, \ps)({}_jy))\\
\dbigoplus_{a\in I(i,j)} X'(j)(X'(a)x, y) & \dbigoplus_{a\in I(i,j)}X(j)(X(a)F(i)x, F(j)y)
\Ar{1-1}{1-2}{"{\Gr_I(F, \ps)}"}
\Ar{1-1}{2-1}{equal}
\Ar{1-2}{2-2}{equal}
\Ar{2-1}{2-2}{"{\Gr_I(F, \ps)}"'}
\end{tikzcd}
\end{equation}
Assume that $(F, \ps): X'\to X$ is a quasi-equivalence, that is
\begin{enumerate}
\item
For each $i \in I_0$, $F(i): X'(i) \to  X(i)$
is a quasi-equivalence; and
\item
For each $a \in I_1$, $\ps(a)$ is a 2-quasi-isomorphism.
\end{enumerate}

\setcounter{clm}{0}
\begin{clm}
Let ${}_ix, {}_jy \in (\Gr_I X')_0$.
Then the restriction
$$
\Gr_I(F, \ps)_{{{}_ix},{{}_jy}}\colon
(\Gr_I X')({}_ix, {}_jy) \to  (\Gr_I X)(\Gr_I(F, \ps)({}_ix), \Gr_I(F, \ps)({}_jy))
$$
of $\Gr_I(F, \ps)$ to
$(\Gr_I X')({}_ix, {}_jy)$ is a quasi-isomorphism.
\end{clm}
Indeed,
we have to show that for each $k\in\mathbb{Z}$, the following is an isomorphism:
$$
H^k(\Gr_I X')({}_ix, {}_jy) \ya{H^k(\Gr_I(F, \ps)_{{}_ix,{{}_jy}})} H^k(\Gr_I X)(\Gr_I(F, \ps)({}_ix), \Gr_I(F, \ps)({}_jy)).
$$
By \eqref{eq:Gr-str}, it is decomposed as follows:

\begin{equation}\label{eq:H_kGr}\small
\begin{tikzcd}[column sep=-50pt, row sep=40pt]
H^k((\Gr_I X')({}_ix, {}_jy)) &&  H^k((\Gr_I X)(\Gr_I(F, \ps)({}_ix), \Gr_I(F, \ps)({}_jy)))\\
\dbigoplus_{a\in I(i,j)} H^k(X'(j)(X'(a)x, y)) && \dbigoplus_{a\in I(i,j)} H^k(X(j)(X(a)F(i)x, F(j)y))\\
&\dbigoplus_{a\in I(i,j)}H^k (X(j)(F(j)X'(a)x, F(j)y)
\Ar{1-1}{1-3}{"{H^k(\Gr_I(F, \ps)_{{{}_ix},{{}_jy}})}"}
\Ar{1-1}{2-1}{equal}
\Ar{1-3}{2-3}{equal}
\Ar{2-1}{3-2}{"{\bigoplus_{a\in I(i,j)}H^k(F(j)_{X'(a)x,y})}"'}
\Ar{3-2}{2-3}{"{\bigoplus_{a\in I(i,j)}H^k (X(j)(\ps(a)_x,F(j)y))}"'description}
\end{tikzcd}.
\end{equation}
By assumption, $H^k(F(j)_{X'(a)x,y})$
is an isomorphism for all $a \in I(i,j)$, and hence
so is $\bigoplus_{a\in I(i,j)}H^k(F(j))$.
Therefore, it remains to show that
\[
\bigoplus_{a\in I(i,j)}H^k (X(j)(\ps(a)_x,F(j)y))
\]
is an isomorphism.
Let $a \colon i \to j$ be a morphism in $I$.
Then since
\[
\ps(a) \colon X(a) F(i) \To F(j)X'(a)
\] 
is a 2-quasi-isomorphism, by definition, we have a 2-isomorphism  
\[
\blank\Lox_{X(i)}\ovl{\ps(a)}
\colon \blank\Lox_{X(i)}\ovl{X(a)F(i)} \To \blank\Lox_{X(i)}\ovl{F(j)X'(a)}.
\]
By specializing at $x^\wedge \in \calD(X(i))_0$, this yields an isomorphism
\[
x^\wedge\Lox_{X(i)}\ovl{\ps(a)} \colon
x^\wedge\Lox_{X(i)}\ovl{X(a)F(i)} \to x^\wedge\Lox_{X(i)}\ovl{F(j)X'(a)},
\]
i.e., an isomorphism
\[
 (\ps(a)_x)^\wedge \colon (X(a)F(i)(x))^\wedge \isoto (F(j)X'(a)(x))^\wedge
\]
in $\calD(X(j))$.
Since we have a commutative diagram
\[
{\small
\xymatrix@C=-30pt@R=30pt{
 \calD(X(j))((F(j)X'(a)(x))^\wedge,(F(j)(y))^\wedge[k])\\
&\calD(X(j))((X(a)F(i)(x))^\wedge,(F(j)(y))^\wedge[k])\\
H^kX(j)(F(j)X'(a)(x), F(j)(y))\\
& H^k(X(j)(X(a)F(i)(x), F(j)y) 
\ar"1,1";"2,2"|(0.5){\calD(X(j))((\ps(a)_x)^\wedge,(F(j)(y))^\wedge[k])}
\ar"3,1";"4,2"|{H^k (X(j)(\ps(a)_x,F(j)y))}
\ar"1,1";"3,1"_{\iso}
\ar"2,2";"4,2"^{\iso}
}}
\]
with the vertical canonical isomorphisms,
we see that
\[
\begin{split}
H^k (X(j)(\ps(a)_x,F(j)y)) \colon H^k (X(j)(F(j)X'(a)(x), F(j)(y))\\
\to
H^k (X(j)(X(a)F(i)(x), F(j)y)
\end{split}
\]
is an isomorphism, and hence so is
$\Ds_{a \in I(i,j)}H^k (X(j)(\ps(a)_x,F(j)y))$, as desired.
Therefore, we conclude that
$H^k(\Gr_I(F, \ps)_{{}_ix,{{}_jy}})$ is an isomorphism
by the commutative diagram \eqref{eq:H_kGr}.
Hence it follows that
$\Gr_I(F, \ps)_{{}_ix, {}_jy}$ is a quasi-isomorphism for all ${}_ix$ and ${}_jy$.

Next we show the following:

\begin{clm}
$H^0(\Gr_I X')\ya{H^0((\Gr_I(F, \ps))}H^0(\Gr_I X)$
is an equivalence.
\end{clm}

By Claim 1 for $k=0$, we have that
$$\begin{aligned}
\bigoplus_{a\in I(i,j)}H^0(X'(j)(X'(a)x, y)) \ya{H^0(\Gr_I(F, \ps)_{{}_ix, {}_jy})} \bigoplus_{a\in I(i,j)}H^0((X(j)(X(a)F(i)x, F(j)y))
\end{aligned}$$
is bijective for all ${}_ix$ and ${}_jy$. Thus,
$$
H^0(\Gr_I(F, \ps)) \colon H^0(\Gr_I X') \to H^0(\Gr_I X)
$$
is fully faithful.
It only remains to show that it is dense.
By the definition of Grothendieck construction, we have
$$
H^0(\Gr_I X)_0=H^0( \bigsqcup_{i\in I_0} X(i)_0)=\bigsqcup_{i\in I_0}H^0(X(i))_0=\bigsqcup_{i\in I_0}X(i)_0.
$$
For any $_ix \in \bigsqcup_{i\in I_0}X(i)_0$ with $i\in I_0$ and $x\in X(i)_0$,
note that
$$
H^0(X'(i))\ya{H^0(F(i))} H^0(X(i))
$$
is dense by (1) above. Thus there exists $x'\in X'(i)_0$ such that $y:=F(i)(x')=H^0(F(i)(x'))\iso x$ in $H^0(X(i))$.
Thus there exists $f:x\isoto y$ in $H^0(X(i))$. 
Since
$$
H^0(\Gr_I(F, \ps))(_ix')=\Gr_I(F, \ps)(_ix')={}_iF(i)(x')={}_iy,
$$
it suffices to show that $_iy\iso {}_ix$ in $H^0(\Gr_I X')$.
Noting that
$$
\begin{aligned}
H^0(\Gr_I X)({}_ix, {}_iy) &= H^0(\Gr_I X)({}_ix, {}_iy)) = H^0(\bigoplus_{a\in I(i,i)}(X(i)(X(a)x, y))
\\
&= \bigoplus_{a\in I(i,i)}H^0((X(i)(X(a)x, y)), \text{ and}\\
H^0(\Gr_I X)({}_iy,{}_ix) &= H^0(\Gr_I X)({}_iy,{}_ix)) = H^0(\bigoplus_{a\in I(i,i)}(X(i)(X(a)y,x))
\\
&= \bigoplus_{a\in I(i,i)}H^0((X(i)(X(a)y,x)),
\end{aligned}
$$
we can take elements
\[
\begin{aligned}
(\de_{b,\id_i}f^{-1}\circ X_i(y))_{b\in I(i,i)}&\in \bigoplus_{a\in I(i,i)}H^0((X(i)(X(a)y,x)),
\text{and}\\
(\de_{a,\id_i}f\circ X_i(x))_{a\in I(i,i)}&\in \bigoplus_{a\in I(i,i)}H^0((X(i)(X(a)x, y)),
\end{aligned}
\]
where entries are of the following forms
$$\begin{aligned}
X(\id_i)y\ya{X_i(y)} y\ya{f^{-1}} x,\quad X(\id_i)x\ya{X_i(x)} x\ya{f} y,
\end{aligned}$$
respectively.
A direct calculation shows that
$$\begin{aligned}
(\de_{b,\id_i}f^{-1}\circ X_i(y))_{b\in I(i,i)}\circ(\de_{a,\id_i}f\circ X_i(x))_{a\in I(i,i)}=\id_{{}_ix},\\
(\de_{a,\id_i}f\circ X_i(x))_{a\in I(i,i)}\circ(\de_{b,\id_i}f^{-1}\circ X_i(y))_{b\in I(i,i)}=\id_{{}_iy}
\end{aligned}$$
Then we have ${}_iy\iso {}_ix$ in $H^0(\Gr_I X)$.
Therefore $H^0(\Gr_I(F, \ps))$ is dense.
\end{proof}

The following is the main result in this subsection.

\begin{thm}
\label{mainthm2}
Let $X, X' \in \Colax(I, \kdgCat)$.
Assume that $X$ is $\k$-flat. If there exists a tilting colax functor $T$ for $X$
and there exists a quasi-equivalence $1$-morphism from $X'$ to $T$,
then $\Gr_I X'$ is derived equivalent to $\Gr_I X$.
\end{thm}

\begin{proof}
Note that $\Gr_I X$ is also $\k$-flat by definition of $\Gr_I X$.
Let $T$ be a tilting colax subfunctor of $\DGMod(X)$
with an $I$-equivariant inclusion
$(\si, \ro)\colon T \incl \DGMod(X)$.
Put $(P, \ph):= (P_X, \ph_X)$ for short.
Let $T'$ be the full dg subcategory of $\DGMod(\Gr_I X)$
consisting of the objects $\perf(P(i))(U)$\ ($\in \perf(\Gr_I X)$) 
with $i \in I_0$ and $U \in T(i)_0$, which is called the
gluing of $T(i)$.

We now show that $T'$ is a tilting dg subcategory for $\Gr_I X$. By Proposition \ref{prp:hprj}, the objects in $T'(i)$ are homotopically projective for all $i\in I_0$.
For each $i \in I_0$ and $x \in X(i)$,
we have
\[
\begin{aligned}
\perf(P(i))(X(i)(\blank, x)) &\iso X(i)(\blank, x)) \ox_{X(i)}\ovl{P(i)}\\
&= X(i)(\blank, x)) \ox_{X(i)}(\Gr_I X)(\blank, P(i)(?))\\
&\iso
(\Gr_I X) (\blank, P(i)(x))= (\Gr_I X)(\blank, {}_ix).
\end{aligned}
\]
Thus
$$
\begin{aligned}
(\Gr_I X)(\blank, {}_ix) &\iso
\perf(P(i))(X(i)(\blank, x))\\
&\in \perf(P(i))(\thick T(i))\\
&\subseteq \thick\{\perf( P(i))(U) \mid U \in T(i)\}\\
&\subseteq \thick T'.
\end{aligned}
$$
Therefore, $\thick T' = \perf(\Gr_I X)$, and hence
$T'$ is a tilting dg subcategory for $\Gr_I X$, as desired.
In particular, we see that 
$\Gr_I X$ is derived equivalent to $T'$.
Let $(F, \ps)$ be the restriction of $\perf((P, \ph))$ to $T$.
Then by construction $(F, \ps) \colon T \to \De(T')$
is a dense functor, and it is an $I$-precovering because
so is
$$
\perf((P, \ph)) \colon \perf(X) \to \De(\perf(\Gr_I X))
$$
by Proposition \ref{prp:precovering-preserved}.
Thus $(F, \ps)$ is an $I$-covering, which shows that
$T' \simeq \Gr_I T$ by Corollary \ref{covering-Gr}.
Thus we have
\begin{equation}
\label{eq:der-eq-tilt-colax}
\Gr_I X \dereq \Gr_I T.
\end{equation}
Since there exists a quasi-equivalence from $X'$ to $T$ 
in $\Colax(I, \DGkCat)$, we have a quasi-equivalence from
$\Gr_I X'$ to 
$\Gr_I T\ (\simeq T')$ in $\DGkCat$ by Proposition \ref{prp:qis-qis-1mor}, and hence $\Gr_I X'$ and $\Gr_I T$ are derived equivalent by
Theorem \ref{thm:q-eq-der-eq}. 
As a consequence, $\Gr_I X'$ is derived equivalent to $\Gr_I X$.
\end{proof}

Let $X, X' \in \Colax(I, \kdgCat)$.
Since the relation that $X'$ is quasi-equivalent to $X$ is
not symmetric with respect to $X'$ and $X$,
we consider a zigzag chain of quasi-equivalences between them
defined as follows.

\begin{dfn}
Let $X,X'\in \Colax(I,\kdgCat)$.
Then a {\em zigzag chain of quasi-equivalence $1$-morphisms} between $X$ and $X'$ is a chain
of 1-morphisms  of the form
\[
X =: X_0 \xleftarrow{(F_1,\ps_1)} X_1 \ya{(F_2,\ps_2)}  \cdots \xleftarrow{(F_{n-1},\ps_{n-1})} X_{n-1}\ya{(F_{n},\ps_{n})} X_n:= X'
\]
in $\Colax(I,\DGkCat)$ with $n$ even $\ge 2$, where $(F_i,\ps_i)$ are quasi-equivalence $1$-morphisms for all $i=1,\dots, n$.
Note that a quasi-equivalence $1$-morphism $X \ya{(F_2,\ps_2)} X'$ is also regarded
as a zigzag chain of quasi-equivalence $1$-morphism
by setting $n=2$ and $(F_1, \ps_1)$ to be the identity 1-morphism.
\end{dfn}

The following is immediate from 
Proposition \ref{prp:qis-qis-bimod},
Theorems \ref{thm:q-eq-der-eq} and \ref{mainthm2}.

\begin{cor}
\label{der-zigzag}
Let $X, X' \in \Colax(I, \DGkCat)$.
Assume that $X$ is $\k$-flat and that there exists a tilting colax functor $T$ for $X$
such that
there exists a zigzag chain of quasi-equivalence 1-morphisms
between $X'$ and $T$ in $\Colax(I, \DGkCat)$.
Then $\Gr_I X'$ and $\Gr_I X$ are derived equivalent.
\end{cor}

\begin{proof}
By assumption, we have a zigzag chain of quasi-equivalences between $X'$ and $T$ of the form
\[
T =: X_0 \xleftarrow{(F_1,\ps_1)} X_1 \ya{(F_2,\ps_2)}  \cdots \xleftarrow{(F_{n-1},\ps_{n-1})} X_{n-1}\ya{(F_{n},\ps_{n})} X_n:= X'
\]
in $\Colax(I, \DGkCat)$ with $n$ even $\ge 2$, where $(F_i,\ps_i)$ are quasi-equivalences for all $i=1,\dots, n$.
Then by Proposition \ref{prp:qis-qis-1mor}, we have a zigzag chain
\[
\Gr_I T =: \Gr_I X_0 \xleftarrow{} \Gr_I X_1 \ya{}  \cdots \xleftarrow{} \Gr_I X_{n-1}\ya{} \Gr_I X_n:= \Gr_I X'
\]
of quasi-equivalences of dg categories, and hence a sequence of derived equivalences
\[
\Gr_I T =: \Gr_I X_0 \dereq \Gr_I X_1 \dereq  \cdots \dereq \Gr_I X_{n-1}\dereq \Gr_I X_n:= \Gr_I X'
\]
by Theorem \ref{thm:q-eq-der-eq}.
Since to be derived equivalent is an equivalence relation,
we have $\Gr_I X' \dereq  \Gr_I T$.
It follows from (the proof of) Theorem \ref{mainthm2}
(see \eqref{eq:der-eq-tilt-colax-2}) that $ \Gr_I X\dereq  \Gr_I T$.
Therefore, $\Gr_I X'$ and $\Gr_I X$ are derived equivalent.
\end{proof}

\subsection{Quasi-equivalence bimodules}

\begin{prp}\label{prp:qis-qis-bimod}
Let $X, X' \in \Colax(I, \kdgCat)$.
Assume that there exists a quasi-equivalence $X'$-$X$-bimodule
$E$. Then we have a quasi-equivalence $\Gr_I X'$-$\Gr_I X$-bimodule
$\Gr_I E$.
\end{prp}

\begin{proof}
We construct a $\Gr_I X'$-$\Gr_I X$-bimodule $\Gr_I E$, namely a dg functor
\[
\Gr_I E \colon (\Gr_I X)\op  \ox_\k (\Gr_I X') \to \Cdg(\k).
\]
as follows.

{\bf On objects.}
Let $({}_ix, {}_{i'}x') \in (\Gr_I X)_0\op \times (\Gr_I X')_0$, where
$i, i' \in I_0,\ x \in X(i)_0,\ x' \in X'(i')_0$.
Then we set
\[
(\Gr_I E)({}_ix, {}_{i'}x'):= \Ds_{a \in I(i,i')}E(i')(X(a)x,x') \in \Cdg(\k)_0.
\]

{\bf On morphisms.}
Let $f \in (\Gr_I X)\op({}_ix, {}_{j}y) = (\Gr_I X)({}_{j}y, {}_ix) =
\bigoplus_{c\in I(j,i)} X(i)(X(c)y, x)$
and $g \in (\Gr_I X')({}_{i'}x', {}_{j'}y') = \bigoplus_{d\in I(i',j')} X(j')(X'(d)x', y')$.
Then by noting that $f_c \in X(j)(X(c)y, x)$ and $g_d \in X(j')(X'(d)x', y')$
for all $c \in I(j, i), d \in I(i',j')$,
we define $(\Gr_I E)(f \ox g)$ as follows (see the diagram below).
\[
\begin{tikzcd}
\Nname{1} (\Gr_I E)({}_ix, {}_{i'}x') & \Nname{2}(\Gr_I E)({}_jy, {}_{j'}y')\\
\Nname{3}\Ds_{a \in I(i,i')}E(i')(X(a)x,x') & \Nname{4}\Ds_{b \in I(j,j')}E(j')(X(b)y,y')
\Ar{1}{2}{"(\Gr_I E)(f \ox g)"}
\Ar{1}{3}{equal}
\Ar{2}{4}{equal}
\Ar{3}{4}{dashed}
\end{tikzcd}
\]
Let $(v_a)_{a \in I(i,i')} \in \Ds_{a \in I(i,i')}E(i')(X(a)x,x')$,
where for each $a \in I(i,i')$, we have $v_a \in E(i')(X(a)x,x')$.
Here in $I$, we have the following setting:
\[
\begin{tikzcd}
\Nname{i}i & \Nname{j}j\\
\Nname{i'}i' & \Nname{j'}j'
\Ar{j}{i}{"c" '}
\Ar{i}{i'}{"a" '}
\Ar{i'}{j'}{"d" '}
\Ar{j}{j'}{"b"}
\end{tikzcd}
\]
We then set
\[
\begin{aligned}
&(\Gr_I E)(f \ox g)((v_a)_{a \in I(i,i')})\\
&:= \left(\sum_{b=dac}((X(d)\circ X_{ac})\bullet X_{d,ac})_y\cdot g_d \cdot E(d)(v_a) \cdot X(d)X(a)(f_c)\right)_{b \in I(j,j')},
\end{aligned}
\]
where the sum is taken over all triples $(c,a,d) \in I(j,i) \times I(i,i') \times I(i',j')$
such that $b = dac$.
The right half of the terms in the summation can be visualized
in the following diagram, where the dashed arrows stand for ``elements'' of $E$:
\[
\begin{tikzcd}[column sep=50pt]
\Nname{y}X(a)X(c)y & \Nname{Xax}X(a)x & \Nname{x'}x'\\
\Nname{1}X(d)X(a)X(c)y & \Nname{2}X(d)X(a)x & \Nname{3}X'(d)x' &\Nname{4}y'
\Ar{y}{Xax}{"X(a)(f_c)", "" 'name=a}
\Ar{Xax}{x'}{"v_a","" 'name=b, dashed}
\Ar{1}{2}{""name=c,"X(d)X(a)(f_c)" '}
\Ar{2}{3}{""name=e, "E(d)(v_a)" ', dashed}
\Ar{3}{4}{"g(d)" '}
\Ar{y}{1}{"X(d)", mapsto}
\Ar{Xax}{2}{"X(d)", mapsto}
\Ar{x'}{3}{"X'(d)", mapsto}
\Ar{a}{c}{"X(d)", mapsto}
\Ar{b}{e}{"E(d)", mapsto}
\end{tikzcd},
\]
and the left half in the following:
\[
X(b) = X(dac) \ya{(X_{d,ac})_y} X(d)X(ac) \ya{X(d)\circ (X_{ac})_y} X(d)X(a)X(c).
\]
Note that $E(d) \colon E(i') \to E(j')$ is an $X'(i')$-$X(i')$-bimodule morphism,
where the $X'(i')$-$X(i')$-bimodule structure of $E(j')$ is defined from 
its $X'(j')$-$X(j')$-bimodule structure
by using the functors $X(d) \colon X(i') \to X(j')$
and $X'(d) \colon X'(i') \linebreak[3] \to X'(j')$, which explains the vertical correspondence in
the diagram above.

We next show that $\Gr_I E$ is a quasi-equivalence bimodule.
We have to show the conditions (a) and (b) in Definition
\ref{dfn:qeq-dg-bimod}.
We show the condition (a) by using Lemma \ref{lem:eq-bi-der-eq}.
Consider the canonical morphism
$(P, \ph)\colon X \to \De(\Gr_I X)$.
Then
for each $i \in I_0$ and $x \in X(i)$, we have
\[
\begin{aligned}
\perf(P(i))(X(i)(\blank, x)) &\iso X(i)(\blank, x) \ox_{X(i)}\ovl{P(i)}\\
&= X(i)(\blank, x) \ox_{X(i)}(\Gr_I X)(\blank, P(i)(?))\\
&\iso
(\Gr_I X)(\blank, P(i)(x))= (\Gr_I X)(\blank, {}_ix).
\end{aligned}
\]
For each $i \in I_0$ and $x' \in X'(i)$,
there exists some $x \in X(i)_0$ such that 
$E(i)(\blank, x') \iso X(i)(\blank, x)$.
\[
\begin{aligned}
\perf(P(i))(E(i)(\blank, x')) &\iso E(i)(\blank, x') \ox_{X(i)}\ovl{P(i)}\\
&= E(i)(\blank, x') \ox_{X(i)}(\Gr_I X)(\blank, P(i)(?))\\
&\iso X(i)(\blank, x) \ox_{X(i)}(\Gr_I X)(\blank, P(i)(?))\\
&\iso
(\Gr_I X)(\blank, P(i)(x))= (\Gr_I X)(\blank, {}_ix).
\end{aligned}
\]
\begin{equation}
\label{eq:intE-repres}
\begin{aligned}
(\Gr_I E)({}_?\blank, {}_{i}x')&:= \Ds_{a \in I(?,i)}E(i)(X(a)(\blank),x')\\
&\iso \Ds_{a \in I(?,i)}X(i)(X(a)(\blank),x) = (\Gr_I X)({}_?\blank, {}_ix).
\end{aligned}
\end{equation}
Thus
$$
\begin{aligned}
(\Gr_I X)(\blank, {}_ix) &\iso
\perf(P(i))(X(i)(\blank, x))\\
&\in \perf(P(i))(\thick E(i)(\blank, x'))\\
&\subseteq \thick\{\perf( P(i))E(i)(\blank, x')\}\\
&\subseteq \thick (\Gr_I E)(\blank, {}_ix').
\end{aligned}
$$
Therefore, $\thick (\Gr_I E) = \perf(\Gr_I X)$.
Namely, the conditions (1) and (3) in Lemma \ref{lem:eq-bi-der-eq} hold.
It remains to show the condition (2) of this lemma,
that is, to show that
 \[
 (\Gr_I X')({}_ix', {}_jy')\to \RHom_{\Gr_I X}((\Gr_I E)(\blank, {}_ix'),(\Gr_I E)(\blank, {}_jy'))
 \]
is a quasi-isomorphism.
\[
\begin{aligned}
&\RHom_{\Gr_I X}((\Gr_I E)(\blank, {}_ix'),(\Gr_I E)(\blank, {}_jy'))
\iso \Cdg(\Gr_I X)((\Gr_I X)(\blank, {}_ix),(\Gr_I X)(\blank, {}_jy))\\
&\iso (\Gr_I X)({}_ix, {}_jy)=
\Ds_{a \in I(i,j)}X(j)(X(a)x,y)\\
&\iso \Ds_{a \in I(i,j)}\Cdg(X(j))(X(j)(\blank, X(a)x),X(j)(\blank, y)))\\
&\iso \Ds_{a \in I(i,j)}\Cdg(X(j))(X(i)(\blank,x)\otimes_{X(i)}\ovl{X(a)} ,X(j)(\blank, y)))\\
&\iso \Ds_{a \in I(i,j)}\Cdg(X(j))((E(i)(\blank,x')\otimes_{X(i)}\ovl{X(a)} ,E(j)(\blank, y')))\\
&\iso \Ds_{a \in I(i,j)}\Cdg(X(j))(X'(i)(\blank,x')\otimes_{X'(i)}\ovl{X'(a)}\ox_{X'(j)} E(j), E(j)(\blank, y')))\\
&\iso \Ds_{a \in I(i,j)}\Cdg(X(j))(X'(j)(\blank,X'(a)x')\ox_{X'(j)} E(j), E(j)(\blank, y')))\\
&\iso \Ds_{a \in I(i,j)}\Cdg(X(j))((E(j)(\blank,X'(a)x'), E(j)(\blank, y'))).
\end{aligned}
\]
See the diagram for the last three lines:
\[
\begin{tikzcd}[row sep=scriptsize, column sep=scriptsize]
& \calD(X'(i)) \arrow[dl,"\calD X'(a)"'] \arrow[rr,"\blank\Lox_{X'(i)}E(i)"]  & & \calD(X(i)) \arrow[dl,"\calD X(a)"]  \\
\calD(X'(j))\arrow[rr, crossing over
]  & & \calD(X(j)) \\
&\udl{X'}(i) \arrow[dl] \arrow[rr] \arrow["\si'(i)" near start, uu,hookrightarrow] & & \udl{X}(i) \arrow[dl]\arrow[uu,hookrightarrow]\\
\udl{X'}(j) \arrow[rr]\arrow[uu,hookrightarrow] & & \udl{X}(j)
\Ar{1-4}{2-1}{"\blank\Lox_{X'(i)}E(a)"' description, Rightarrow,crossing over}
\Ar{3-4}{4-1}{"\blank\Lox_{X'(i)}E(a)"' description, Rightarrow}
\Ar{4-3}{2-3}{"\si(j)"' near end, crossing over,hookrightarrow}
\end{tikzcd}.
\]
Since 
\[
X'(j)(X'(a)x',y')\ya{\mathrm{qis}} \Cdg(X(j))((E(j)(\blank,X'(a)x'), E(j)(\blank, y'))
\]
is a quasi-isomorphism induced by the quasi-equivalence bimodule $E$,
it follows that 
\[
\begin{aligned}
(\Gr_I X')({}_ix', {}_jy')&=
\Ds_{a \in I(i,j)}X'(j)(X'(a)x',y')\\
&\ya{\mathrm{qis}} \Ds_{a \in I(i,j)}\Cdg(X(j))((E(j)(\blank,X'(a)x'), E(j)(\blank, y')))\\
&\iso\RHom_{\Gr_I X}((\Gr_I E)(\blank, {}_ix'),(\Gr_I E)(\blank, {}_jy'))
\end{aligned}
\]
is a quasi-isomorphism.
By Lemma \ref{lem:eq-bi-der-eq}, $\Gr_I E$ induces a derived equivalence between $\Gr_I X'$ and $\Gr_I X$.
Note that the condition (b) is trivially satisfied by \eqref{eq:intE-repres}.
Hence $\Gr_I E$ is a quasi-equivalence bimodule.
\end{proof}

The following is our second main result in this paper.

\begin{thm}
\label{mainthm2-bimod}
Let $X, X' \in \Colax(I, \kdgCat)$.
Assume that $X$ is $\k$-flat.
If there exist  a
tilting colax functor $T$ for $X$ and a quasi-equivalence
$X'$-$T$-bimodule (namely, if $X'$ is standardly derived 
equivalent to $X$ by Theorem \ref{thm:characterization-1}),
then $\Gr_I X'$ is derived equivalent to $\Gr_I X$.
\end{thm}

\begin{proof}
Note that $\Gr_I X$ is also $\k$-flat by definition of $\Gr_I X$.
Let $T$ be a tilting colax subfunctor of $\DGMod(X)$
with an $I$-equivariant inclusion
$(\si, \ro)\colon T \incl \DGMod(X)$.
Put $(P, \ph):= (P_X, \ph_X)$ for short.
Let $T'$ be the full dg subcategory of $\DGMod(\Gr_I X)$
consisting of the objects $\perf(P(i))(U)$\ ($\in \perf(\Gr_I X)$) 
with $i \in I_0$ and $U \in T(i)_0$, which is called the
\emph{gluing} of $T(i)$.

We now show that $T'$ is a tilting dg subcategory for $\Gr_I X$. By Proposition \ref{prp:hprj}, the objects in $T'(i)$ are homotopically projective for all $i\in I_0$.
For each $i \in I_0$ and $x \in X(i)$,
we have
\[
\begin{aligned}
\perf(P(i))(X(i)(\blank, x)) &\iso X(i)(\blank, x) \ox_{X(i)}\ovl{P(i)}\\
&= X(i)(\blank, x) \ox_{X(i)}(\Gr_I X)(\blank, P(i)(?))\\
&\iso
(\Gr_I X)(\blank, P(i)(x))= (\Gr_I X)(\blank, {}_ix).
\end{aligned}
\]
Thus
$$
\begin{aligned}
(\Gr_I X)(\blank, {}_ix) &\iso
\perf(P(i))(X(i)(\blank, x))\\
&\in \perf(P(i))(\thick T(i))\\
&\subseteq \thick\{\perf( P(i))(U) \mid U \in T(i)\}\\
&\subseteq \thick T'.
\end{aligned}
$$
Therefore, $\thick T' = \perf(\Gr_I X)$, and hence
$T'$ is a tilting dg subcategory for $\Gr_I X$, as desired.
In particular, we see that 
$\Gr_I X$ is derived equivalent to $T'$.
Let $(F, \ps)$ be the restriction of $\perf((P, \ph))$ to $T$.
Then by construction $(F, \ps) \colon T \to \De(T')$
is a dense functor, and it is an $I$-precovering because
so is
$$
\perf((P, \ph)) \colon \perf(X) \to \De(\perf(\Gr_I X))
$$
by Proposition \ref{prp:precovering-preserved}.
Thus $(F, \ps)$ is an $I$-covering, which shows that
$T' \simeq \Gr_I T$ by Corollary \ref{covering-Gr}.
Thus we have
\begin{equation}
\label{eq:der-eq-tilt-colax-2}
\Gr_I X \dereq \Gr_I T.
\end{equation}
Since there exists a quasi-equivalence 
$X'$-$T$-bimodule $E$, we have a quasi-equivalence 
$\Gr_I X'$-$\Gr_I T$-bimodule
$\Gr_I E$ by Proposition \ref{prp:qis-qis-bimod}, and hence $\Gr_I X'$ and $\Gr_I T$ are derived equivalent
(see Definition \ref{dfn:qe-bimod}) by Theorem \ref{thm:Keller}, since  $\Gr_I X$ is $\k$-flat. 
As a consequence, $\Gr_I X'$ is derived equivalent to $\Gr_I X$.
\end{proof}

\subsection{Group actions}

In the special case that $I = G$ is a group, which has a unique object $*$, Theorems \ref{mainthm2} and \ref{mainthm2-bimod} have the forms below. 

\begin{dfn}
Let $\calA$ and $\calB$ be dg categories with $G$-actions.
\begin{enumerate}
\item 
A tilting dg subcategory $T$ for $\calA$
is called $G$-{\em equivariant}  if
there exists a $G$-equivariant inclusion
$(\si, \ro) \colon T \to \DGMod(\calA)$.

\item An $\calA$-$\calB$-bimodule $Z$ is called a {\em quasi-equivalence} if (a) the functor
\[\blank\Lox_{\calA}Z \colon \calD(\calA) \to \calD(\calB)
\] is an equivalence in
$\Colax(G, \kuTRI^2)$, and (b) it gives rise to a equivalence $\udl{\calA}\to \udl{\calB}$.

\end{enumerate} 
\end{dfn}

\begin{cor}\label{cor:group-case-1}
Let $\calA$ and $\calB$ be dg categories with $G$-actions. Assume that $\calA$ is $\k$-flat. If there exist 
a $G$-equivariant tilting dg subcategory $T$ for $\calA$, and quasi-equivalence $1$-morphism from $\calB\to T$,
then the orbit categories $\calA/G$ and $\calB/G$ are derived equivalent.
\end{cor}

This corollary will be applied in Example \ref{exm:4-triangles}.

\begin{cor}\label{cor:group-case-2}
Let $\calA$ and $\calB$ be dg categories with $G$-actions. Assume that $\calA$ is $\k$-flat. If there exist 
a $G$-equivariant tilting dg subcategory $T$ for $\calA$, and a quasi-equivalence $\calB$-$T$-bimodule,
then the orbit categories
$\calA/G$ and $\calB/G$ are derived equivalent.
\end{cor}

\begin{rmk}
Remark \ref{rmk:q-eq-1-mor-bimod} also proves
Corollary \ref{cor:group-case-1} by Corollary \ref{cor:group-case-2}.
\end{rmk}

\subsection{Diagonal functors}

The following is easy to verify.

\begin{lem}
\label{lem:diagonal-dereq}
Let $\calC, \calC'$ be in $\kdgCat$.
If $\calC$ and $\calC'$ are {\em standardly} derived equivalent,
then so are $\De(\calC)$ and $\De(\calC')$.
\end{lem}

\begin{proof}
Since $\calC'$ is standardly derived equivalent to $\calC$,
there exists a dg functor $H: \DGMod(\calC') \to \DGMod(\calC)$ such that $\bfL H: \calD(\calC')\to \calD(\calC)$ is an equivalence.
Then 
$\De(\bfL H): \De(\calD(\calC'))\to \De(\calD(\calC))$ is an equivalence,
and it yields an equivalence 
$\bfL \De(H): \calD(\De(\calC'))\to \calD(\De(\calC))$, where
$\De(H): \DGMod(\De(\calC')) \to \DGMod(\De(\calC))$
is a dg functor.
\end{proof}

Theorem \ref{mainthm2} together with the lemma above
and Example \ref{exm:Gr} gives us
a unified proof of the following fact.

\begin{thm}\label{thm:unified-proof}
Assume that $\k$ is a field and that dg $\k$-algebras $A$ and $A'$ are derived equivalent.
Then the following pairs are derived equivalent as well:
\begin{enumerate}
\item
dg path categories $AQ$ and $A'Q$ for any quiver $Q$;
\item
incidence dg categories $AS$ and $A'S$ for any poset $S$; and
\item
monoid dg algebras $AG$ and $A'G$ for any monoid $G$.
\end{enumerate}
\end{thm}

\begin{proof}
Assume that $A$ is derived equivalent to $A'$.
Then $A'$ is standardly derived equivalent to $A$.
By Lemma \ref{lem:diagonal-dereq}, $\De(A')$ is standardly derived equivalent
to $\De(A)$.
Hence by Theorem \ref{mainthm2}, $\De(A')$ and $\De(A)$ are derived equivalent.
\end{proof}

\begin{rmk}
We remark that for a colax functor $X \colon I \to \kdgCat$,
the following does not hold in general:

(*) $\calD(\Gr_I X)$ is equivalent to $\Gr_I \calD X$
as triangulated categories.

\newcommand{\Mor}{\operatorname{Mor}}

If this would be true, then 
it immediately follows that 
if $\calD(X)$ and $\calD(X')$ are equivalent, then
$\calD(\Gr_I X)$ and $\calD(\Gr_I X')$ are equivalent.
Thus, our main theorem becomes trivial.

However, in many cases, (*) even does not make sense.
Since $\Gr_I X \in \kdgCat_0$, the left hand side $\calD(\Gr_I X)$ is a triangulated category.
On the other hand,
since $\calD X \in \Colax(I, \kuTRI^2)$,
we do not know how to define the right hand side $\Gr_I \calD X$
as a triangulated category.
For example,
let $I$ be a free category defined by the quiver $1 \to 2$,
$A$ a dg algebra regared as a dg category with only one object, and
$\De(A) \colon I \to \kdgCat$ the diagonal functor of $A$.
Then by Example \ref{exm:Gr}, $\Gr_I \Delta(A)$ is isomorphic to
the lower triagular matrix dg algebra of $A$, or equivalently, the morphism category $\Mor(A)$ of the category $A$.
If (*) would be true, then the following would hold:

(\#) $\calD(\Mor(A))$ is equivalent to $\Mor(\calD(A))$
as categories.

Indeed,
since  $\calD(\Delta(A)) \cong \Delta(\calD(A))$
(i.e., the derived tensor product by the $A$-$A$-bimodule $A$
is isomorphic to the identity of $\calD(A)$),
we heve
$\calD(\Mor(A)) \cong \calD(\Gr_I(\Delta(A)))$ is equivalent (by (*)) to
$\Gr_I \calD(\De(A)) \cong \Gr_I \De(\calD(A)) \cong \Mor \calD(A)$.

However, (\#) does not hold in general, because for a triangulated category $\calC$, it is impossible to give a triangulation on $\Mor \calC$, unless $\calC$ is semisimple.
\end{rmk}

\section{Examples}
\newcommand{\cGz}{\widehat{\Ga}}
In this section,
we give two examples that illustrate our main theorem.

\begin{rmk}\label{rmk:Gr-skew}
Let $G$ be a group, which we regard as a groupoid with only one object $*$.
Let $(Q,W)$ be a quiver with potential.
Regard the complete Ginzburg dg algebra $\cGz(Q,W)$ as a dg category with only one object,
and a $G$-action on it as a functor $X_{Q,W}\colon G \to \kdgCat$ with
$X_{Q,W}(*) = \cGz(Q,W)$.
If $G$ is a finite group, as in \cite{LM} (see also \cite{GPP} for the finite abelian case), then
$\Gr_G(X_{Q,W})$ is nothing but the orbit category $\cGz(Q,W)/G$, which is
also equivalent to the skew group dg algebra $\cGz(Q,W) * G$, and is calculated as $\cGz(Q_G,W_G)$ up to Morita equivalece.
Therefore in this case note that $\Gr_G(X_{Q,W})$ is calculated as $\cGz(Q_G,W_G)$
up to Morita equivalence. 
\end{rmk}

\subsection{Mutations, the complete Ginzburg dg algebras and derived equivalences} 

In the example below we will use the constructions of mutations and the Ginzburg dg algebras,
and a ``tilting'' bimodule given by Keller--Yang.
To make it easy to understand these examples,
we recall these constructions and fix our notations.

\subsubsection{Mutations}
Let $Q$ be a quiver.
A path in $Q$ is said to be {\em cyclic} if its source and target coincide.
A potential on $Q$ is an element of the closure $\Pot(\k Q)$ of the subspace
of $\k Q$ generated by all non-trivial cyclic paths in $Q$. 
 We say that two potentials are {\em cyclically equivalent} if their difference is in the closure of the subspace generated by the differences $a_1\cdots a_s - a_2\cdots a_sa_1$ for all cycles $a_1\cdots a_s$ in $Q$.

The complete path algebra $\widehat{\k Q}$ is the completion of the path algebra $\k Q$ with respect to the ideal generated by the arrows of $Q$. Let $\mathfrak{m}$ be the ideal of $\widehat{\k Q}$  generated by the arrows of $Q$. For a path $p$ of $Q$, we define $\partial_P: \Pot(\k Q)\to \widehat{\k Q}$ as continuous linear map which takes a cycle $c$ to the sum $\Sigma_{c=upv}vu$ taken over all decompositions of the cycle $c$ (where $u$ and
 $v$ are possibly lazy paths). 

A {\em quiver with potential} is a pair $(Q, W)$ of a quiver $Q$ and a potential $W$ of $Q$ such that $W$ is in $\mathfrak{m}^2$ and no two cyclically equivalent cyclic paths appear in the decomposition of $W$. The {\em Jacobian algebra} of a quiver with potential $(Q, W)$ is the quotient of complete path algebra $\widehat{\k Q}$ by the closure of the ideal generated by $\partial_a$, where $a$ runs over all arrows of $Q$. A quiver with potential is called {\em trivial} if its potential is a linear combination of cyclic paths of length $2$ and its Jacobian algebra is the product of copies of the base field $\k$. A quiver with potential is called {\em reduced} if $\partial_a W$ is contained in $\mathfrak{m}^2$ for all arrows $a$ of $Q$.

Let $(Q',W')$ and $(Q'', W'')$ be two quivers with potential such that $Q'$ and $Q''$ have the same set of vertices. Their direct sum, denoted by $(Q', W')\ds(Q'', W'')$, is the new quiver with potential $(Q,W)$, where $Q$ is the quiver whose vertex set is the same as the vertex set of $Q'$ (and $Q''$) and whose arrow set is the disjoint union of the arrow set of $Q'$ and the arrow set of $Q''$, and $W = W'+W''$.

 Two quivers with potential $(Q,W)$ and $(Q', W')$ are {\em right-equivalent} if $Q$ and $Q'$ have the same set of vertices and there exists an algebra isomorphism $\phi: \k Q \to \k Q'$ whose restriction on vertices is the identity map and $\phi(W)$ and $W'$ are cyclically equivalent. Such an isomorphism $\phi$ is called a {\em right-equivalence}.

For any quiver with potential $(Q,W)$, there exist a trivial quiver
with potential $(Q_{\tri}, W_{\tri})$ and a reduced quiver with potential $(Q_{\red}, W_{\red})$ such that
$(Q,W)$ is right-equivalent to the direct sum $(Q_{\tri}, W_{\tri})\ds (Q_{\red}, W_{\red})$. Furthermore, the right-equivalence class of each of $(Q_{\tri}, W_{\tri})$ and $(Q_{\red}, W_{\red})$ is uniquely determined by the right equivalence class of $(Q,W)$. We call $(Q_{\tri}, W_{\tri})$ and $(Q_{\red}, W_{\red})$ the {\em trivial part} and the {\em reduced part} of $(Q,W)$, respectively.

\begin{dfn}
\label{dfn:qvr-pot}
Let $(Q,W)$ be a quiver with potential, and $i$ a vertex of $Q$. Assume the following conditions:

\begin{enumerate}
\item
the quiver $Q$ has no loops;
\item
the quiver $Q$ does not have $2$-cycles at $i$;
\item
no cyclic path occurring in the expansion of $W$ starts and ends at $i$.
\end{enumerate}

Note that under the condition (1), any potential is cyclically equivalent to a potential satisfying (3). We define a new quiver with potential $\widetilde{\mu_i}(Q,W)=(Q', W')$ as follows. The new
quiver $Q'$ is obtained from $Q$ by the following procedure:

\begin{description}
\item[Step 1]
 For each arrow $\be$ with target $i$ and each arrow $\al$ with source $i$, add a new arrow $[\al\be]$ from the source of $\be$  to the target of $\al$ .
 \item[Step 2]
 Replace each arrow $\al$ with source or target $i$ with an arrow $\al^*$ in the opposite direction.
\end{description}

The new potential $ W'$ is the sum of two potentials $ W_1'$ and $ W'_2$, where the potential $ W_1'$ is obtained from $W$ by replacing each composition $\al\be$ by $[\al\be]$, where $\be$ is an arrow with target $i$, and the potential $W'_2$ is given by
$$
W'_2=\sum_{\al,\be\in Q_1} [\al\be]\be^*\al^*,
$$
where the sum ranges over all pairs of arrows  $\al$ and  $\be$ such that $\be$ ends at $i$ and $\al$ starts at $i$. It is easy to see that $\widetilde{\mu_i}(Q,W)$ satisfies (1), (2) and (3).
We define $\mu_i(Q,W)$ as the reduced part of $\widetilde{\mu_i}(Q,W)$, and call $\mu_i$ the {\em mutation} at the vertex $i$.
\end{dfn}

\subsubsection{The complete Ginzburg dg algebras}
\begin{dfn}
Let $(Q,W)$ be a quiver with potential. The {\it complete Ginzburg dg algebra} $\widehat{\Gamma}(Q,W)$ is constructed as follows \cite{Ginz}: 
Let $\widetilde{Q}$ be the graded quiver with the same vertices as $Q$ and whose arrows are
\begin{itemize}
\item the arrows of $Q$ (they all have degree $0$),
\item an arrow $\ovl{\al}: j\to i$ of degree $-1$ for each arrow $\al: i\to j$ of $Q$,
\item a loop $t_i: i\to i$ of degree $-2$ for each vertex $i$ of $Q$.
\end{itemize}
The underlying graded algebra of $\widehat{\Gamma}(Q,W)$ is the completion of the graded path algebra $k\widetilde{Q}$ in the category of graded vector spaces with respect to the ideal generated by the arrows of $\widetilde{Q}$. Thus,
the $n$-th component of $\widehat{\Gamma}(Q,W)$ consists of elements of the form $\sum_{p}\lambda_p p$ with $\la_p \in \Bbbk$, where $p$ runs over all paths of degree $n$. The differential of $\widehat{\Gamma}(Q,W)$ is the unique continuous linear endomorphism homogeneous of degree $1$ which satisfies the Leibniz rule
$$
d(u v)=d(u)v+(-1)^p u d(v),
$$
for all homogeneous $u$ of degree $p$ and all $v$, and takes the following values on the arrows of $\widetilde{Q}$:
\begin{itemize}
\item $da=0$ for each arrow $a$ of $Q$,
\item $d(\ovl{a})=\partial_a W$ for each arrow $\ovl{a}$ of $Q$,
\item $d(t_i)=e_i(\sum_a[a,\ovl{a}])e_i$ for each vertex $i$ of $Q$, where $e_i$ is the trivial path at $i$ and the sum is taken over the set of arrows of $Q$.
\end{itemize}
\end{dfn}

\begin{rmk}
\label{rmk:Ginz-dg-cat}
We regard the complete Ginzburg dg algebra
$\Ginz(Q,W)$ as a dg category up to equivalence of dg categories as follows.
\begin{itemize}
\item
The objects are the vertices of $\widetilde{Q}$ (namely the vertices of $Q$).
\item
$\Ginz(Q,W)(i,j):= e_j\Ginz(Q,W)e_i$ for all objects $i,j$.
\item
The composition is given by the multiplication of $\Ginz(Q,W)$.
\item
The grading and the differential are naturally defined from those of
the dg algebra structure.
\end{itemize}
\end{rmk}

The following lemma is an easy consequence of the definition
(cf.\ \cite[Lemma 2.8]{KeY}).

\begin{lem} Let $(Q,W)$ be a quiver with potential. Then the Jacobian algebra $\Jac(Q,W)$ is the $0$-th cohomology of the complete Ginzburg dg algebra $\widehat{\Gamma}(Q,W)$, i.e.
\[
\Jac(Q,W)=H^0(\widehat{\Gamma}(Q,W)).
\]
\end{lem}

\subsubsection{Derived equivalences}

Let $(Q,W)$ be a quiver with potential and $i$ a fixed vertex of $Q$. We assume $(1), (2)$ and $(3)$
as above. Write $\widetilde{\mu_i}(Q,W)=(Q', W')$. Let $\Gamma=\widehat{\Gamma}(Q,W)$ and $\Gamma'=\widehat{\Gamma}(Q',W')$ be the
complete Ginzburg dg algebras associated to $(Q,W)$ and $(Q', W')$, respectively.
We set $P_j=e_j\Ga$ and $P'_j=e_j\Ga'$ for all vertices $j$ of $Q$.

We cite the following from \cite[Theorem 3.2]{KeY} without a proof.

\begin{thm}\label{thm:KeY3.2}
There is a triangle equivalence 
\[
F:\calD(\Ga')\to \calD(\Ga)
\]
which sends the $P'_j$ to $P_j$ for $j\neq i$, and sends $P'_i$
to the cone $T_i$ over the morphism
$$
 \begin{alignedat}{3}
P_i\to \bigoplus_{\alpha\in Q_1, s(\alpha)=i} P_{t(\alpha)}\\
a \mapsto \sum_{\alpha\in Q_1, s(\alpha)=i} e_{t(\alpha)}\alpha a,
 \end{alignedat}
$$
The functor $F$ restricts
to triangle equivalences from $\perf(\Ga')$ to $\perf(\Ga)$ and from $\calD_{fd}(\Ga')$ to $\calD_{fd}(\Ga)$.
\end{thm}

The proof is based on a construction of a $\Ga'$-$\Ga$-bimodule $T$,
and $F$ is defined by
$F:= \blank\Lox_{\Ga}T \colon \calD(\Ga') \to \calD(\Ga)$.
We recall the construction of $T$ by Keller-Yang below.
As a right $\Ga$-module, let $T$ be the direct sum of $T_i$ and $P_j$ for all $j \in Q_0$ with $j \ne i$.
A left $\Ga'$-module structure on $T$ will be defined in the next proposition.
To this end, we define a map
$f \colon \{e_j \mid j \in Q_0\} \cup (\widetilde{Q'})_1 \to \End_{\Ga}(T)$
as follows.
First, we set $f(e_j):= f_j:T_j\to T_j$ to be the identity map for all $j\in Q_0$.

We denote by $\la_a$ the left multiplication $x \mapsto ax$ by $a$ below when this makes sense, and
by $e_{\Si i}$ the unique idempotent in $\Ga$ such that
$e_{\Si i}\Ga = \Si P_i=P_i[1]$, the shift of $P_i$, for all $i \in Q_0$.

Let $\alpha \in Q_1$ with $s(\alpha)=i$.
Then define
$f_{\alpha^{\ast}}:T_{t(\alpha)}\to T_i$ of degree $0$ as
the cannonical embedding
$T_{t(\alpha)}= P_{t(\al)}\hookrightarrow T_i$, that is,
$$
f_{\al^{\ast}}:= \la_{e_t(\al)}:T_{t(\al)}\to T_i,\quad
a \mapsto  e_{t(\al)} a.
$$
Define also the morphsim $f_{\ovl{\alpha^{\ast}}}: T_i\to T_{t(\alpha)}$ of degree $-1$ by
$$
f_{\ovl{\al^\ast}}((e_{\Sigma i})a_i+\sum_{\rho\in Q_1, s(\rho)=i} e_{t(\rho)} a_{\rho})=-\alpha t_i a_i-\sum_{\rho\in Q_1, s(\rho)=i}\alpha \ovl{\rho}a_{\rho}
$$

Let $\beta\in Q_1$ with $t(\be)=i$.
Then define the morphism $f_{\be^\ast}: T_i\to T_{s(\beta)}$ of degree $0$ by 
$$
f_{\be^\ast}((e_{\Si i})a_i+\sum_{\rho\in Q_1, s(\ro)=i} e_{t(\ro)} a_{\ro})=-\ovl{\be} a_i-\sum_{\ro\in Q_1, s(\ro)=i}(\partial_{\ro\be}W)  a_{\ro}.
$$
Define also the morphism $f_{\ovl{\be^\ast}}: T_{s(\be)}\to T_i$ of degree $-1$ as the composite of the morphism $\la_{e_{\Si i}\be} \colon T_{s(\be)}\to \Si P_{ i}$ 
and the canonical embedding $ \Si P_i\hookrightarrow T_i$, that is,
$$
f_{\ovl{\be^{\ast}}}:=\la_{e_{\Si i}\be} : T_{s(\be)}\to T_i,\quad
a\mapsto e_{\Sigma i}\beta a.
$$

Let $\alpha,\beta\in Q_1$ with $s(\alpha)=i, t(\beta)=i$.
Then define 
$$
f_{[\alpha\beta]}:= \la_{\al\be}: T_{s(\beta)}\to T_{t(\alpha)}, \quad a\mapsto \alpha\beta a.
$$
and 
$$
f_{\ovl{[\alpha\beta]}}:=0:  T_{t(\alpha)}\to T_{s(\beta)}.
$$

Let $\ga\in Q_1$ be an arrow not incident to $i$.
Then define
$$
 \begin{aligned}
f_{\ga}&:= \la_{\ga}: T_{s(\ga)}\to T_{t(\ga)}, \quad a\mapsto\ga a,\\
f_{\ovl{\ga}}&:= \la_{\ovl{\ga}}: T_{t(\ga)}\to T_{s(\ga)}, \quad a\mapsto\ovl{\ga} a.
\end{aligned}
$$

Let $j\in Q_0$ with $j\neq i$.
Then define
$$
f_{t'_j}:= \la_{t_j}: T_j \to T_j, \quad a\mapsto t_j a.
$$
It is a morphism of degree $-2$. 
Finally,
define $f_{t'_i}$ as the linear morphism of degree $-2$ from $T_i$ to itself given by 
$$
f_{t'_i}((e_{\Sigma i})a_i+\sum_{\rho\in Q_1, s(\rho)=i} e_{\rho} a_{\rho})=-e_{\Sigma i}(t_i a_i+\sum_{\rho\in Q_1, s(\rho)=i} \ovl{\rho}a_{\rho}).
$$

By \cite[Proposition 3.5]{KeY} we have the following.

\begin{prp}
\label{prp:KeY3.5}
The map $f \colon \{e_j \mid j \in Q_0\} \cup (\widetilde{Q'})_1 \to \End_{\Ga}(T)$
defined above extends to a homomorphism of dg algebras from $\Ga'$ to $\End_{\Ga}(T)$. In this way, $T$ becomes a left dg $\Ga'$-module,
and also a dg $\Ga'$-$\Ga$-bimodule.
\end{prp}

\subsection{Examples}

\begin{exm}
\label{exm:4-triangles}
Let $(Q,W)$ be the quiver with potential given as follows:
$$
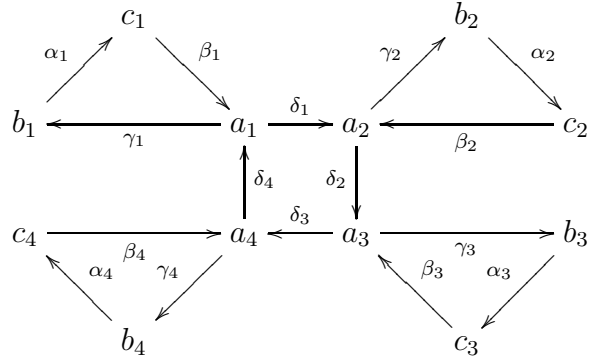
\xymatrix{
&c_1&&&b_2\\
b_1&&a_1& a_2&&c_2\\
c_4&&a_4&a_3& &b_3\\
&b_4&&&c_3
\ar^{\gamma_1}"2,3";"2,1"
\ar^{\alpha_1}"2,1";"1,2"
\ar^{\beta_1}"1,2";"2,3"
\ar^{\gamma_2}"2,4";"1,5"
\ar^{\alpha_2}"1,5";"2,6"
\ar^{\beta_2}"2,6";"2,4"
\ar^{\delta_1}"2,3";"2,4"
\ar_{\delta_4}"3,3";"2,3"
\ar_{\delta_2}"2,4";"3,4"
\ar_{\delta_3}"3,4";"3,3"
\ar_{\gamma_3}"3,4";"3,6"
\ar_{\alpha_3}"3,6";"4,5"
\ar_{\beta_3}"4,5";"3,4"
\ar_{\gamma_4}"3,3";"4,2"
\ar_{\alpha_4}"4,2";"3,1"
\ar_{\beta_4}"3,1";"3,3"
}
$$
$W=\delta_4\delta_3\delta_2\delta_1+\sum_{i=1}^3\gamma_i\beta_i\alpha_i$.
If we do mutations at $c_1$ and $c_3$ for $(Q,W)$, we get the following quiver with potential $(Q',W')$

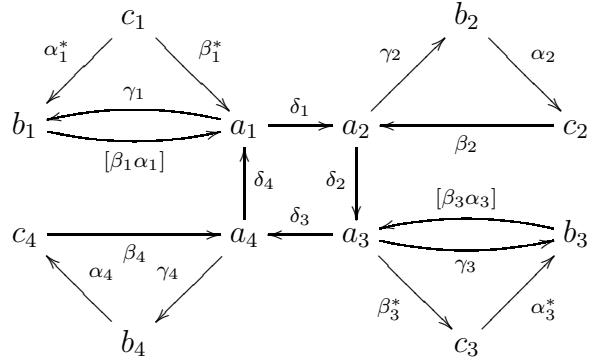
$$
\xymatrix{
&c_1&&&b_2\\
b_1&&a_1& a_2&&c_2\\
c_4&&a_4&a_3& &b_3\\
&b_4&&&c_3
\ar@/_/_{\gamma_1}"2,3";"2,1"
\ar@/_/_{[\beta_1\alpha_1]}"2,1";"2,3"
\ar_{\alpha^*_1}"1,2";"2,1"
\ar^{\beta^*_1}"1,2";"2,3"
\ar^{\gamma_2}"2,4";"1,5"
\ar^{\alpha_2}"1,5";"2,6"
\ar^{\beta_2}"2,6";"2,4"
\ar^{\delta_1}"2,3";"2,4"
\ar_{\delta_4}"3,3";"2,3"
\ar_{\delta_2}"2,4";"3,4"
\ar_{\delta_3}"3,4";"3,3"
\ar@/_/_{\gamma_3}"3,4";"3,6"
\ar@/_/_{[\beta_3\alpha_3]}"3,6";"3,4"
\ar_{\alpha^*_3}"4,5";"3,6"
\ar_{\beta^*_3}"3,4";"4,5"
\ar_{\gamma_4}"3,3";"4,2"
\ar_{\alpha_4}"4,2";"3,1"
\ar_{\beta_4}"3,1";"3,3"
}
$$
$W'=\delta_4\delta_3\delta_2\delta_1+\gamma_1[\beta_1\alpha_1]+\gamma_3[\beta_3\alpha_3]+\gamma_2\beta_2\alpha_2+\gamma_4\beta_4\alpha_4+[\beta_1\alpha_1]\alpha^*_1\beta^*_1+[\beta_3\alpha_3]\alpha^*_3\beta^*_3$.

The reduced part $(Q'_\red,W'_\red)$ of $(Q',W')$ is given as follows:
$$
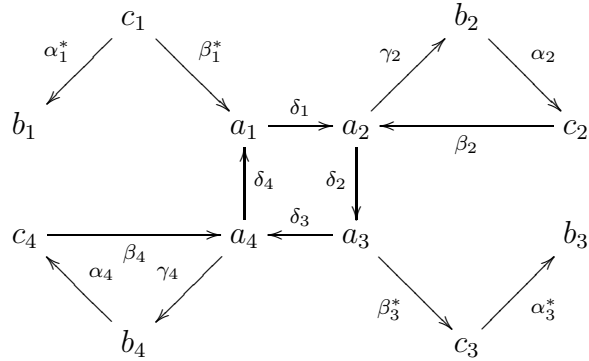
\xymatrix{
&c_1&&&b_2\\
b_1&&a_1& a_2&&c_2\\
c_4&&a_4&a_3& &b_3\\
&b_4&&&c_3
\ar_{\alpha^*_1}"1,2";"2,1"
\ar^{\beta^*_1}"1,2";"2,3"
\ar^{\gamma_2}"2,4";"1,5"
\ar^{\alpha_2}"1,5";"2,6"
\ar^{\beta_2}"2,6";"2,4"
\ar^{\delta_1}"2,3";"2,4"
\ar_{\delta_4}"3,3";"2,3"
\ar_{\delta_2}"2,4";"3,4"
\ar_{\delta_3}"3,4";"3,3"
\ar_{\alpha^*_3}"4,5";"3,6"
\ar_{\beta^*_3}"3,4";"4,5"
\ar_{\gamma_4}"3,3";"4,2"
\ar_{\alpha_4}"4,2";"3,1"
\ar_{\beta_4}"3,1";"3,3"
}
$$
$W'_\red =\delta_4\delta_3\delta_2\delta_1+\gamma_2\beta_2\alpha_2+\gamma_4\beta_4\alpha_4$.

Consider the cyclic group $G$ of order $2$ with generator $g$,
and define a $G$-action on $(Q,W)$ as a unique quiver automorphism
induced by the permutation of indexes $i=1,2,3,4$:
\begin{equation}
\label{eq:G-action2}
i\mapsto i-2 \quad (\mod \ 4).
\end{equation}
Then the quiver with potential $(Q_G,W_G)$ is given as follows:
$$
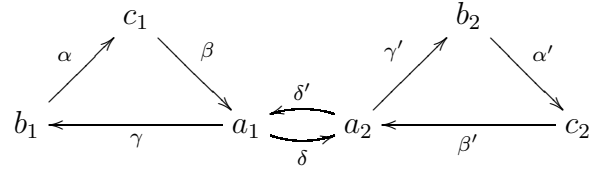
\xymatrix{
&c_1&&&b_2\\
b_1&&a_1& a_2&&c_2
\ar^{\gamma}"2,3";"2,1"
\ar^{\alpha}"2,1";"1,2"
\ar^{\beta}"1,2";"2,3"
\ar^{\gamma'}"2,4";"1,5"
\ar^{\alpha'}"1,5";"2,6"
\ar^{\beta'}"2,6";"2,4"
\ar@/_/_{\delta}"2,3";"2,4"
\ar@/_/_{\delta'}"2,4";"2,3"
}
$$
$W_G=(\delta'\delta)^2+2\gamma\beta\alpha+2\gamma'\beta'\alpha'$.
Define also a $G$-action on $(Q'_\red,W'_\red)$ by the same permutation of indexes as \eqref{eq:G-action2}.
Then the quiver with potential $((Q'_\red)_G,(W'_\red)_G)$ is given as follows:
$$
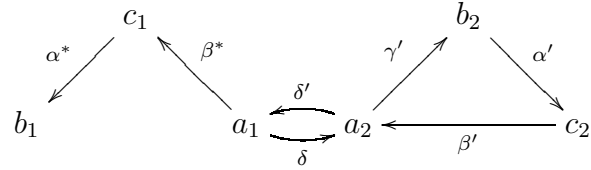
\xymatrix{
&c_1&&&b_2\\
b_1&&a_1& a_2&&c_2
\ar_{\alpha^*}"1,2";"2,1"
\ar_{\beta^*}"2,3";"1,2"
\ar^{\gamma'}"2,4";"1,5"
\ar^{\alpha'}"1,5";"2,6"
\ar^{\beta'}"2,6";"2,4"
\ar@/_/_{\delta}"2,3";"2,4"
\ar@/_/_{\delta'}"2,4";"2,3"
}
$$
$(W'_\red)_G = (\delta'\delta)^2+2\gamma'\beta'\alpha'$.

If we do mutations at $c_1$ and $c_3$ for $(Q,W)$, then we do mutation at $c_1$ for $(Q_G,W_G)$.
Then the reduced part of $\mu_{c_1}(Q_G,W_G)$
coincides with $((Q'_\red)_G,(W'_\red)_G)$.
Indeed,
the quiver with potential $\mu_{c_1}(Q_G,W_G)$ is the following
$$
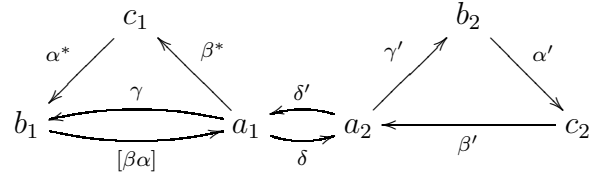
\xymatrix{
&c_1&&&b_2\\
b_1&&a_1& a_2&&c_2
\ar@/_/_{\gamma}"2,3";"2,1"
\ar@/_/_{[\beta\alpha]}"2,1";"2,3"
\ar_{\alpha^*}"1,2";"2,1"
\ar_{\beta^*}"2,3";"1,2"
\ar^{\gamma'}"2,4";"1,5"
\ar^{\alpha'}"1,5";"2,6"
\ar^{\beta'}"2,6";"2,4"
\ar@/_/_{\delta}"2,3";"2,4"
\ar@/_/_{\delta'}"2,4";"2,3"
}
$$
$\mu_{c_1}(W_G)=(\delta'\delta)^2+2\gamma[\beta\alpha]+2\gamma'\beta'\alpha'+2[\beta\alpha]\alpha^*\beta^*$. The potential is not reduced, so we have the following quiver with potential
$$
\xymatrix{
&c_1&&&b_2\\
b_1&&a_1& a_2&&c_2
\ar_{\alpha^*}"1,2";"2,1"
\ar_{\beta^*}"2,3";"1,2"
\ar^{\gamma'}"2,4";"1,5"
\ar^{\alpha'}"1,5";"2,6"
\ar^{\beta'}"2,6";"2,4"
\ar@/_/_{\delta}"2,3";"2,4"
\ar@/_/_{\delta'}"2,4";"2,3"
}
$$
$\mu_{c_1}(W_G)=(\delta'\delta)^2+2\gamma'\beta'\alpha'$.
Hence by Theorem \ref{thm:KeY3.2}
the Ginzburg dg algebras of
$(Q_G,W_G)$ and $((Q'_\red)_G,(W'_\red)_G)$ are derived equivalent.
On the other hand, by Remark \ref{rmk:Gr-skew}
we know that $\Gr_G X_{Q,W}$  is Morita equivalent to $\widehat{\Ga}(Q_G,W_G)$, and $\Gr_G X_{Q',W'}$
is Morita equivalent to $\cGz(Q'_G,W'_G)$, and which
is isomorphic to
$\widehat{\Ga}((Q'_\red)_G,(W'_\red)_G)$
by Keller--Yang \cite[Lemma 2.9]{KeY}
because $(Q',W')$ and $(Q'_\red, W'_\red)$ are
right-equivalent.
As a consequence,
$\Gr_G X_{Q,W}$ and $\Gr_G X_{Q',W'}$ are derived equivalent.
The same conclusion can be obtained from our result
Corollary \ref{cor:group-case-1}
as in the next example.
\end{exm}

\begin{exm}
\label{exm:Mizuno}
Let $(Q,W)$ be the quiver with potential given as follows:
$$
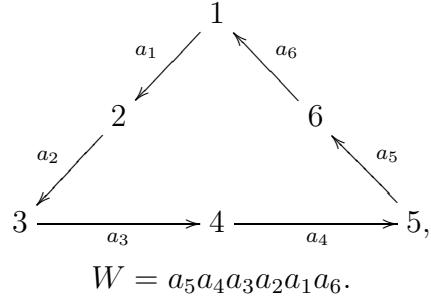
\begin{gathered}
\xymatrix{
&& 1\\
&2 && 6\\
3&&4 && 5,
\ar_{a_1}"1,3";"2,2"
\ar_{a_2}"2,2";"3,1"
\ar_{a_3}"3,1";"3,3"
\ar_{a_4}"3,3";"3,5"
\ar_{a_5}"3,5";"2,4"
\ar_{a_6}"2,4";"1,3"
}\\
W=a_5a_4a_3a_2a_1a_6.
\end{gathered}
$$
Let $I=\{1,3,5\}$. Mizuno \cite{Mu} defined successive mutation $\mu_I(Q,W)=\mu_5\circ\mu_3\circ\mu_1(Q,W)=(Q',W')$ given by the quiver with potential as follows:
$$
\begin{gathered}
\xymatrix{
&& 1\\
&2 && 6\\
3&&4 && 5,
\ar^{a^*_1}"2,2";"1,3"
\ar^{a^*_2}"3,1";"2,2"
\ar^{a^*_3}"3,3";"3,1"
\ar^{a^*_4}"3,5";"3,3"
\ar^{a^*_5}"2,4";"3,5"
\ar^{a^*_6}"1,3";"2,4"
\ar^{[a_1a_6]}"2,4";"2,2"
\ar_{[a_3a_2]}"2,2";"3,3"
\ar_{[a_5a_4]}"3,3";"2,4"
}
\\
W'=[a_1a_6]a^*_6a^*_1+[a_3a_2]a^*_2a^*_3+[a_5a_4]a^*_4a^*_5+[a_5a_4][a_3a_2][a_1a_6].
\end{gathered}
$$
By \cite[Theorem 1.1]{Mu}, the Jacobian algebras $\Jac(Q,W)$ and $\Jac(Q',W')$ are derived equivalent.

(1) Consider the cyclic group $G$ of order $3$ with generator $g$, and define the action of $g$ on $(Q, W)$
by $i\mapsto i-2$ and $a_i \mapsto a_{i-2}$ (modulo $6$).
Therefore, we have 
$$
Ga_1=\{a_1,a_5,a_3\},Ga_2=\{a_2,a_6,a_4\}.
$$
In this case $(Q_G,W_G)$ is the quiver with potential given as follows:
$$
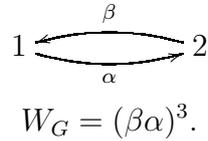
\begin{gathered}
\xymatrix{
1 && 2
\ar@/_/_{\alpha}"1,1";"1,3"
\ar@/_/_{\beta}"1,3";"1,1"
}\\
W_G=(\beta\alpha)^3.
\end{gathered}
$$

(2) Next we define the action of $g$ on $(Q',W')$ by
\[
i\mapsto i-2,\  a^*_i \mapsto a^*_{i-2},\  \text{and}\ [a_i a_{i+5}] \mapsto [a_{i-2} a_{i+3}] \ (\mod 6)
\]
for all $i =1, \dots, 6$.

Therefore, we have 
$$
Ga^*_1=\{a^*_1,a^*_5,a^*_3\},Ga^*_2=\{a^*_2,a^*_6,a^*_4\}.
$$
In this case $(Q'_G,W'_G)$ is the quiver with potential given as follows:
$$
\begin{gathered}
\xymatrix{
G1 && G2
\ar@/_/_{Ga^*_2}"1,1";"1,3"
\ar@/_/_{Ga^*_1}"1,3";"1,1"
\ar@(ur,dr)^{G[a_6a_1]}"1,3";"1,3"
}\\
W'_G=3G[a_6a_1]G(a^*_1)G(a^*_6)+(G[a_6a_1])^3.
\end{gathered}
$$
Here the Jacobian algebras
$\Jac(Q_G,W_G)$ and $\Jac(Q'_G,W'_G)$ are representation-finite,
selfinjective algebras, and by the main theorem in \cite{Asa99}, they are derived equivalent
because their derived equivalence types are the same. By Theorem \ref{thm:KeY3.2},
the complete Ginzburg dg algebras $\widehat{\Gamma}(Q,W)$ and $\widehat{\Gamma}(Q',W')$ are derived equivalent as dg algebras. 
By using Corollary \ref{cor:group-case-1}, we will show that 
$\widehat{\Gamma}(Q,W)/G$ and $\widehat{\Gamma}(Q',W')/G$ are derived equivalent as dg algebras.
Therefore the complete Ginzburg dg algebras $\widehat{\Ga}(Q_G,W_G)$ and $\widehat{\Ga}(Q'_G,W'_G)$ are derived equivalent as dg algebras by Remark \ref{rmk:Gr-skew}.
We set
$\Ga^{(1)}:= \widehat{\Ga}(\mu_1(Q,W)), \Ga^{(2)}:= \widehat{\Ga}(\mu_3\circ\mu_1(Q,W)), \Ga':= \widehat{\Ga}(\mu_5\circ\mu_3\circ\mu_1(Q,W)) = \widehat{\Ga}(Q',W')$.
Then \cite[Proposition 3.13]{KeY} gives us the following derived equivalences $F_3, F_2, F_1$ 
defined as $(\hyph)\Lox_{\Ga'}T^{(3)}$, $(\hyph)\Lox_{\Ga^{(2)}}T^{(2)}$, $(\hyph)\Lox_{\Ga^{(1)}}T^{(1)}$
using the dg bimodules $T^{(3)}, T^{(2)},T^{(1)}$ constructed as in
Proposition \ref{prp:KeY3.5}, respectively. 
These functors send objects as follows:
 \newcommand{\xar}{\xrightarrow}
 $$
 \begin{alignedat}{3}
 \calD(\Ga') &\xar{F_3}\quad  \calD(\Ga^{(2)})& &\xar{F_2}\quad  \calD(\Ga^{(1)})& &\xar{F_1} \quad \calD(\Ga)\\
P'_5 &\mapsto \quad (P^{(2)}_5\to P^{(2)}_6)& &\mapsto \quad  (P^{(1)}_5\to P^{(1)}_6)& & \mapsto\quad (P_5\to P_6)=:T(5) \\
P'_3&\mapsto \quad  P^{(2)}_3& &\mapsto \quad  (P^{(1)}_3\to P^{(1)}_4)& & \mapsto\quad (P_3\to P_4) =: T(3)\\
P'_1&\mapsto \quad  P^{(2)}_1& &\mapsto \quad  P^{(1)}_1& & \mapsto\quad (P_1\to P_2) =: T(1)\\
 P'_i&\mapsto \quad  P^{(2)}_i& &\mapsto \quad  P^{(1)}_i & & \mapsto\quad P_i =: T(i), (i=2,4,6)\
 \end{alignedat}
$$
where $P'_i=e_i\Ga',P^{(2)}_i=e_i\Ga^{(2)},P^{(1)}_i=e_i\Ga^{(1)}$
for all $i \in Q_0$. 
Then $F:= F_1 \circ F_2 \circ F_3 = \blank\Lox_{\Ga'} T^{(3)} \Lox_{\Ga^{(2)}}T^{(2)}\Lox_{\Ga^{(1)}}T^{(1)}$ is an equivalence from $\calD(\Ga')$ to $\calD(\Ga)$.
Here $T^{(3)} \Lox_{\Ga^{(2)}}T^{(2)}\Lox_{\Ga^{(1)}}T^{(1)}$
is a dg $\Ga'$-$\Ga$-bimodule and
is isomorphic to the direct sum $T$ of the indecomposable objects
$T(i), (i=1,\dots, 6)$ as a dg right $\Ga$-module, by which we identify
these and regard $T$ as a dg $\Ga'$-$\Ga$-bimodule.
Let $T$ be the full subcategory of $\DGMod(\Ga)$ consisting of
$T(1), T(2), \cdots, T(6)$.
We show that $T$ is the desired tilting subcategory for $\Ga$.

Now since $g$ acts on $P_i$ by ${}^gP_i = P_{i-2}, (i = 1,\dots, 6)$ by the $G$-action in (1) above, we have
${}^gT(i) = T(i-2), (i = 1,\dots, 6)$.
On the other hand by the $G$-action in (2), $g$ acts on $P'_i$ by ${}^gP'_i = P'_{i-2}, (i = 1,\dots, 6)$. 

We construct a 1-morphism $(F', \ph) \colon \Ga' \to T$ that is a $G$-quasi-equivalence.
To this end we have to construct
a quasi-equivalence $F'\colon \Ga' \to T$ and a 2-quasi-isomorphism
$\ph(a) \colon T(a) \circ F' \To F' \circ a$
in $\DGkCat$
for each $a \in G$ (see Definition \ref{dfn:der-eq-criterion}):
\[
\begin{tikzcd}
\Ga'&T &\DGMod(\Ga) \\
\Ga'&T &\DGMod(\Ga)
\Ar{1-1}{2-1}{"a"'}
\Ar{1-2}{2-2}{"T(a)= {}^a(\hyph)"}
\Ar{1-1}{1-2}{"F'"}
\Ar{2-1}{2-2}{"F'"'}
\Ar{1-2}{1-3}{hook}
\Ar{2-2}{2-3}{hook}
\Ar{1-3}{2-3}{"^a(\hyph)"}
\Ar{1-2}{2-1}{"\ph(a)", Rightarrow}
\end{tikzcd}
\]
(It is trivial that the right square is strictly commutative).
We now define $F'$ as follows:
First recall the Yoneda embedding
$Y \colon \Ga' \to \DGMod(\Ga')$ is defined by
$Y(i):= \Ga'(\hyph, i) = e_i\Ga'$ for all $i \in \Ga'_0$, and
$Y(\mu):= \Ga'(\hyph, \mu)$ for all $\mu \in \Ga'_1$.
Let $\al_{M} \colon \Ga' \Lox_{\Ga'} M \to M$ be the usual natural isomorphism for all $\Ga'$-$\Ga$-bimodules $M$.
This yields the isomorphism
$e_i\al_M \colon e_i\Ga' \Lox_{\Ga'} M \to e_iM$ for each
$i \in \Ga'_0$, which is natural in $i$ and in $M$.
Note that the naturality in $i$ means that for each $f \colon i \to j$ in $\Ga'$,
we have a commutative diagram
\[
\xymatrix{
e_i\Ga' \Lox_{\Ga'} M & e_iM\\
e_j\Ga' \Lox_{\Ga'} M & e_jM.
\ar"1,1";"1,2"^(0.6){e_i\al_M}
\ar"1,1";"2,1"_{\Ga'(\hyph, f)\Lox_{\Ga'}M}
\ar"1,2";"2,2"^{M(\hyph, f)}
\ar"2,1";"2,2"_(0.6){e_j\al_M}
}
\]
We then define $F':= F\circ Y \colon \Ga' \to \perf(\Ga) \subseteq \calD(\Ga)$,
thus $F'(i) = e_i\Ga' \Lox_{\Ga'} T \ya{e_i\al_{T}} T(i)$ for all $i \in Q_0$, and
$F'(\mu)= \Ga'(\hyph,\mu)\Lox_{\Ga'}T \iso \la_{\mu} \colon T(i) \to T(j)$
for all $\mu \in \Ga'_1(i, j)$ with $i, j \in \Ga'_0$.
Thus we have a commutative diagram
\[
\xymatrix{
F'(i) & T(i)\\
F'(j) & T(j).
\ar"1,1";"1,2"^(0.6){e_i\al_T}
\ar"1,1";"2,1"_{F'(\mu)}
\ar"1,2";"2,2"^{\la_{\mu}}
\ar"2,1";"2,2"_(0.6){e_j\al_T}
}
\]
Next we define a 2-quasi-isomorphism
$\ph(a) \colon {}^{a}F' \To F'a$ for each $a \in G$.
Let $i \in \Ga'_0$, and $a \in G$.  Then
the isomorphism $e_i\al_T \colon F'(i) \to T(i)$ yields isomomrphisms
${}^a(F'(i)) \ya{^a(e_i\al_T)} {}^aT(i) = T(ai)$,
and
$F'(ai) \ya{e_{ai}\al_T} T(ai)$.
Thus we have an isomorphism
\[
\ph_i(a):=(e_{ai}\al_T)\inv \circ {}^a(e_i\al_T) \colon {}^a(F'(i)) \to F'(ai).
\]
We then define $\ph(a):= (\ph_i(a))_{i\in \Ga'_0}
\colon {}^aF' \To F'a$ for all $a \in G$ and
$\ph:= (\ph(a))_{a\in G}$.

\setcounter{clm}{0}
\begin{clm}
The pair $(F', \ph)$ is a $1$-morphism
$\Ga' \to T$ in $\DGkCat$ {\rm (see Definition \ref{dfn:1-mor-Colax(I,C)})}.
\end{clm}
Indeed, because $F'(i)$ is clearly a dg-functor, it suffices to show that
$\ph(a)$ is a 2-morphism in $\DGkCat$ for each $a \in G$.
Namely, we have to show the commutativity of the diagram
\[
\xymatrix{
{}^aF'(u) & F'(au)\\
{}^aF'(v) & F'(av)
\ar"1,1";"1,2"^{\ph_u(a)}
\ar"1,1";"2,1"_{{}^aF'(\mu)}
\ar"2,1";"2,2"_{\ph_v(a)}
\ar"1,2";"2,2"^{F'(a\mu)}
}
\]
for all $\mu \colon u \to v$ in $\Ga'_1$ and $a \in G$.
It suffices to show the commutativity of this
only for $a = g$ and for all $\mu \in \widetilde{Q'}_1$.
Therefore finally we have only to show the commutativity of the diagram
\begin{equation}
\label{eq:naturality}
\xymatrix{
{}^gF'(u) & {}^gT(u) & T(gu) & F'(gu)\\
{}^gF'(v) & {}^gT(v) & T(gv) & F'(gv)
\ar"1,1";"1,2"^{{}^g(e_u\al_T)}
\ar@{=}"1,2";"1,3"
\ar"1,4";"1,3"_{e_{a_u}\al_T}
\ar"1,1";"2,1"_{{}^gF'(\mu)}
\ar"2,1";"2,2"_{e_{a_v}\al_T}
\ar@{=}"2,2";"2,3"
\ar"2,4";"2,3"^{e_{a_v}\al_T}
\ar"1,4";"2,4"^{F'(g\mu)}
\ar"1,2";"2,2"^{{}^gT(\mu)}
\ar"1,3";"2,3"^{T(g\mu)}
}
\end{equation}
for all $\mu \in \widetilde{Q'}_1$.
We check this only for three cases below.
The remaining cases are checked similarly, and is left to the reader.

Now the quivers of $\Ga', \Ga^{(2)}, \Ga^{(1)}, \Ga$ are given as follows:
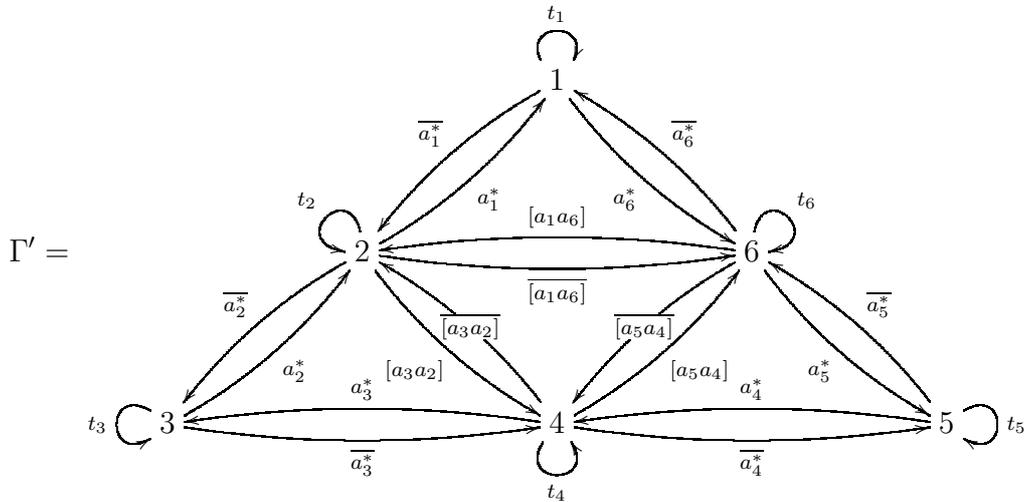
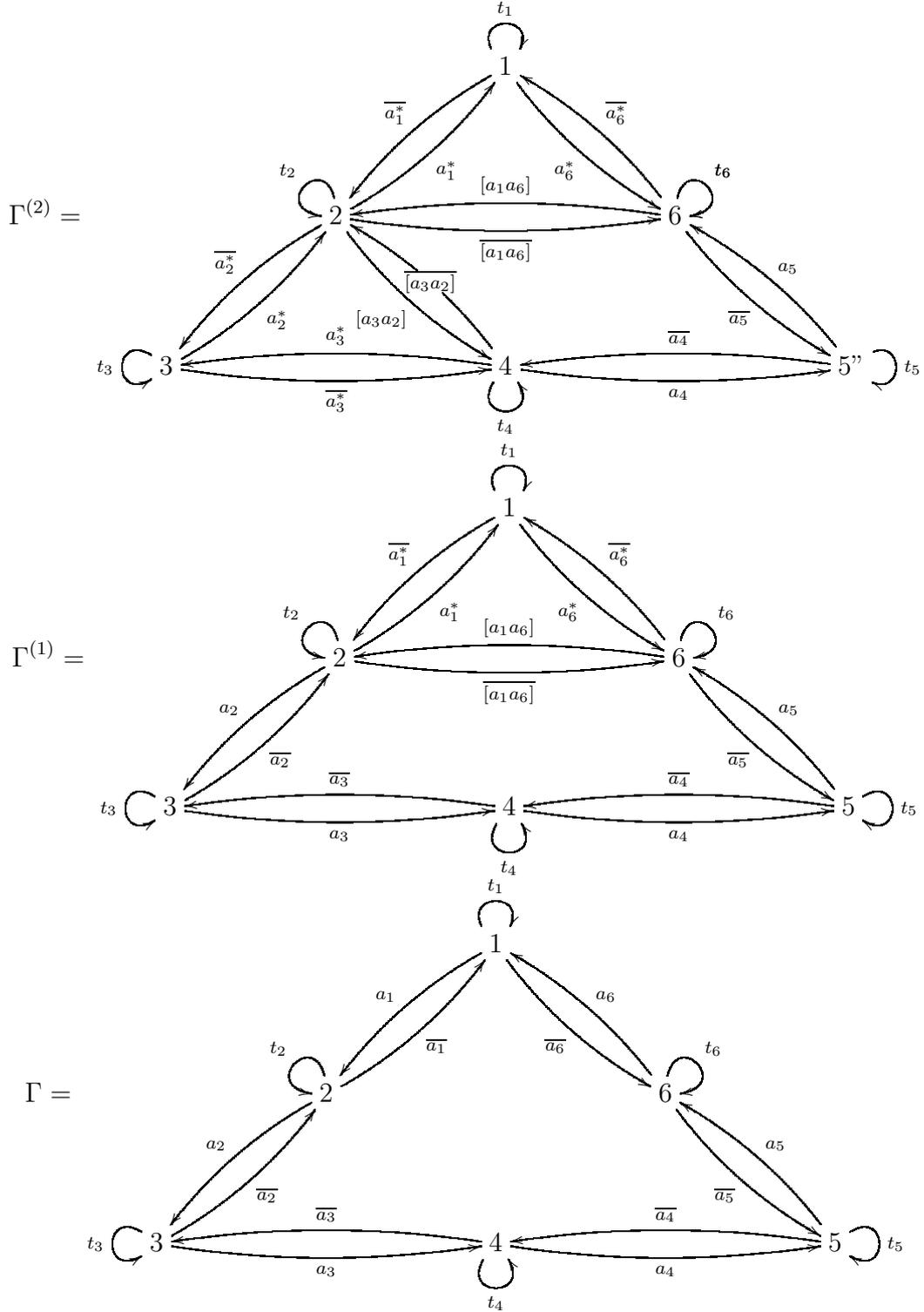
\[
\Ga' = \vcenter{
\xymatrix@C=60pt@R=50pt{
&& 1\\
&2 && 6\\
3&&4 && 5
\ar@/_/_{a^*_1}"2,2";"1,3"
\ar@/_/_{\ovl{a^*_1}}"1,3";"2,2"
\ar@/_/_{a^*_2}"3,1";"2,2"
\ar@/_/_{\ovl{a^*_2}}"2,2";"3,1"
\ar@/_/_{a^*_3}"3,3";"3,1"
\ar@/_/_{\ovl{a^*_3}}"3,1";"3,3"
\ar@/_/_{a^*_4}"3,5";"3,3"
\ar@/_/_{\ovl{a^*_4}}"3,3";"3,5"
\ar@/_/_{a^*_5}"2,4";"3,5"
\ar@/_/_{\ovl{a^*_5}}"3,5";"2,4"
\ar@/_/_{a^*_6}"1,3";"2,4"
\ar@/_/_{\ovl{a^*_6}}"2,4";"1,3"
\ar@/_/_{[a_1a_6]}"2,4";"2,2"
\ar@/_/_{\ovl{[a_1a_6]}}"2,2";"2,4"
\ar@/_/_{[a_3a_2]}"2,2";"3,3"
\ar@/_/|{\ovl{[a_3a_2]}}"3,3";"2,2"
\ar@/_/_{[a_5a_4]}"3,3";"2,4"
\ar@/_/|{\ovl{[a_5a_4]}}"2,4";"3,3"
\ar@(ul,ur)^{t_1}"1,3";"1,3"
\ar@(u,l)_{t_2}"2,2";"2,2"
\ar@(ul,dl)_{t_3}"3,1";"3,1"
\ar@(dl,dr)_{t_4}"3,3";"3,3"
\ar@(ur,dr)^{t_5}"3,5";"3,5"
\ar@(u,r)^{t_6}"2,4";"2,4"
}
}
\]

\[
\Ga^{(2)} = \vcenter{
\xymatrix@C=60pt@R=50pt{
&& 1\\
&2 && 6\\
3&&4 && 5
\ar@/_/_{a^*_1}"2,2";"1,3"
\ar@/_/_{\ovl{a^*_1}}"1,3";"2,2"
\ar@/_/_{a^*_2}"3,1";"2,2"
\ar@/_/_{\ovl{a^*_2}}"2,2";"3,1"
\ar@/_/_{a^*_3}"3,3";"3,1"
\ar@/_/_{\ovl{a^*_3}}"3,1";"3,3"
\ar@/_/_{\ovl{a_4}}"3,5";"3,3"
\ar@/_/_{a_4}"3,3";"3,5"
\ar@/_/_{\ovl{a_5}}"2,4";"3,5"
\ar@/_/_{a_5}"3,5";"2,4"
\ar@/_/_{a^*_6}"1,3";"2,4"
\ar@/_/_{\ovl{a^*_6}}"2,4";"1,3"
\ar@/_/_{[a_1a_6]}"2,4";"2,2"
\ar@/_/_{\ovl{[a_1a_6]}}"2,2";"2,4"
\ar@/_/_{[a_3a_2]}"2,2";"3,3"
\ar@/_/|{\ovl{[a_3a_2]}}"3,3";"2,2"
\ar@(ul,ur)^{t_1}"1,3";"1,3"
\ar@(u,l)_{t_2}"2,2";"2,2"
\ar@(ul,dl)_{t_3}"3,1";"3,1"
\ar@(dl,dr)_{t_4}"3,3";"3,3"
\ar@(ur,dr)^{t_5}"3,5";"3,5"
\ar@(u,r)^{t_6}"2,4";"2,4""
\ar@(u,r)^{t_6}"2,4";"2,4"
}
}
\]
\vspace{-10pt}
\[
\Ga^{(1)} = \vcenter{
\xymatrix@C=60pt@R=50pt{
&& 1\\
&2 && 6\\
3&&4 && 5
\ar@/_/_{a^*_1}"2,2";"1,3"
\ar@/_/_{\ovl{a^*_1}}"1,3";"2,2"
\ar@/_/_{\ovl{a_2}}"3,1";"2,2"
\ar@/_/_{a_2}"2,2";"3,1"
\ar@/_/_{\ovl{a_3}}"3,3";"3,1"
\ar@/_/_{a_3}"3,1";"3,3"
\ar@/_/_{\ovl{a_4}}"3,5";"3,3"
\ar@/_/_{a_4}"3,3";"3,5"
\ar@/_/_{\ovl{a_5}}"2,4";"3,5"
\ar@/_/_{a_5}"3,5";"2,4"
\ar@/_/_{a^*_6}"1,3";"2,4"
\ar@/_/_{\ovl{a^*_6}}"2,4";"1,3"
\ar@/_/_{[a_1a_6]}"2,4";"2,2"
\ar@/_/_{\ovl{[a_1a_6]}}"2,2";"2,4"
\ar@(ul,ur)^{t_1}"1,3";"1,3"
\ar@(u,l)_{t_2}"2,2";"2,2"
\ar@(ul,dl)_{t_3}"3,1";"3,1"
\ar@(dl,dr)_{t_4}"3,3";"3,3"
\ar@(ur,dr)^{t_5}"3,5";"3,5"
\ar@(u,r)^{t_6}"2,4";"2,4"
}
}
\]
\vspace{-10pt}
\[
\Ga = \vcenter{
\xymatrix@C=60pt@R=50pt{
&& 1\\
&2 && 6\\
3&&4 && 5
\ar@/_/_{\ovl{a_1}}"2,2";"1,3"
\ar@/_/_{a_1}"1,3";"2,2"
\ar@/_/_{\ovl{a_2}}"3,1";"2,2"
\ar@/_/_{a_2}"2,2";"3,1"
\ar@/_/_{\ovl{a_3}}"3,3";"3,1"
\ar@/_/_{a_3}"3,1";"3,3"
\ar@/_/_{\ovl{a_4}}"3,5";"3,3"
\ar@/_/_{a_4}"3,3";"3,5"
\ar@/_/_{\ovl{a_5}}"2,4";"3,5"
\ar@/_/_{a_5}"3,5";"2,4"
\ar@/_/_{\ovl{a_6}}"1,3";"2,4"
\ar@/_/_{a_6}"2,4";"1,3"
\ar@(ul,ur)^{t_1}"1,3";"1,3"
\ar@(u,l)_{t_2}"2,2";"2,2"
\ar@(ul,dl)_{t_3}"3,1";"3,1"
\ar@(dl,dr)_{t_4}"3,3";"3,3"
\ar@(ur,dr)^{t_5}"3,5";"3,5"
\ar@(u,r)^{t_6}"2,4";"2,4"
}
}
\]

{\bf Case 1.}
$\mu = a^*_i\in \widetilde{Q'}$ for some $i = 1,\dots, 6$, say $i=1$.
Then up to Yoneda embeddings (for the first three correspondences) we have
$a^*_1 \overset{F_3}{\mapsto}a^*_1
\overset{F_2}{\mapsto}a^*_1\overset{F_1}{\mapsto}f_{a^*_1} \overset{{}^g(\hyph)}{\mapsto} f_{a^*_5}$.
Since we have commutative diagrams
\[
\vcenter{
\xymatrix{
F'(2) & T(2) \\
F'(1) & T(1)
\ar"1,1";"1,2"^{e_2\al_T}
\ar"1,1";"2,1"_{F'(a^*_1)}
\ar"2,1";"2,2"_{e_1\al_T}
\ar"1,2";"2,2"^{f_{a^*_1}}
}}
\ \text{and}\ 
\vcenter{
\xymatrix{
T(6) & F'(6)\\
T(5) & F'(5)
\ar"1,1";"2,1"^{f_{a^*_5}}
\ar"1,2";"2,2"^{F'(a^*_5)}
\ar"1,2";"1,1"_{e_{6}\al_T}
\ar"2,2";"2,1"^{e_{5}\al_T}
}},
\]
we have a commutative diagram:
\[
\xymatrix{
{}^gF'(2) & {}^gT(2) & T(6) & F'(g2)\\
{}^gF'(1) & {}^gT(1) & T(5) & F'(g1),
\ar"1,1";"1,2"^{{}^g(e_2\al_T)}
\ar@{=}"1,2";"1,3"
\ar"1,4";"1,3"_{e_{g2}\al_T}
\ar"1,1";"2,1"_{{}^gF'(a^*_1)}
\ar"2,1";"2,2"_{{}^g(e_1\al_T)}
\ar@{=}"2,2";"2,3"
\ar"2,4";"2,3"^{e_{g1}\al_T}
\ar"1,4";"2,4"^{F'(ga^*_1)}
\ar"1,2";"2,2"^{{}^gf_{a^*_1}}
\ar"1,3";"2,3"^{f_{ga^*_1}}
}
\]
and hence \eqref{eq:naturality} is verified in this case.

{\bf Case 2.}
$\mu = \ovl{a^*_i}\in \widetilde{Q'}$ for some $i = 1,\dots, 6$, say $i=1$.
Then up to Yoneda embeddings (for the first three correspondences) we have
$\ovl{a^*_1} \overset{F_3}{\mapsto}\ovl{a^*_1}
\overset{F_2}{\mapsto}\ovl{a^*_1}\overset{F_1}{\mapsto}f_{\ovl{a^*_1}} \overset{{}^g(\hyph)}{\mapsto} f_{\ovl{a^*_5}}$.
Since we have commutative diagrams
\[
\vcenter{
\xymatrix{
F'(1) & T(1) \\
F'(2) & T(2)
\ar"1,1";"1,2"^{e_1\al_T}
\ar"1,1";"2,1"_{F'(\ovl{a^*_1})}
\ar"2,1";"2,2"_{e_2\al_T}
\ar"1,2";"2,2"^{f_{\ovl{a^*_1}}}
}}
\ \text{and}\ 
\vcenter{
\xymatrix{
T(5) & F'(5)\\
T(6) & F'(6)
\ar"1,1";"2,1"^{f_{\ovl{a^*_5}}}
\ar"1,2";"2,2"^{F'(\ovl{a^*_5})}
\ar"1,2";"1,1"_{e_{5}\al_T}
\ar"2,2";"2,1"^{e_{6}\al_T}
}},
\]
we have a commutative diagram:
\[
\xymatrix{
{}^gF'(1) & {}^gT(1) & T(5) & F'(g1)\\
{}^gF'(2) & {}^gT(2) & T(6) & F'(g2),
\ar"1,1";"1,2"^{{}^g(e_1\al_T)}
\ar@{=}"1,2";"1,3"
\ar"1,4";"1,3"_{e_{g1}\al_T}
\ar"1,1";"2,1"_{{}^gF'(\ovl{a^*_1})}
\ar"2,1";"2,2"_{{}^g(e_1\al_T)}
\ar@{=}"2,2";"2,3"
\ar"2,4";"2,3"^{e_{g2}\al_T}
\ar"1,4";"2,4"^{F'(g\ovl{a^*_1})}
\ar"1,2";"2,2"^{{}^gf_{\ovl{a^*_1}}}
\ar"1,3";"2,3"^{f_{g\ovl{a^*_1}}}
}
\]
and hence \eqref{eq:naturality} is verified in this case.

{\bf Case 3.}
$\mu = t_i\in \widetilde{Q'}$ for some $i = 1,\dots, 6$, say $i=1$.
Then up to Yoneda embeddings (for the first three correspondences) we have
$t_1 \overset{F_3}{\mapsto}t_1
\overset{F_2}{\mapsto}t_1\overset{F_1}{\mapsto}f_{t'_1} \overset{{}^g(\hyph)}{\mapsto} f_{t'_5}$.
Therefore we have ${}^gF' (t_1) = f_{t'_5} = F'(t_5) = F'(gt_1)$.
Hence ${}^g(F'(t_1)) = F'(gt_1)$.

Since we have commutative diagrams
\[
\vcenter{
\xymatrix{
F'(1) & T(1) \\
F'(1) & T(1)
\ar"1,1";"1,2"^{e_1\al_T}
\ar"1,1";"2,1"_{F'(t_1)}
\ar"2,1";"2,2"_{e_1\al_T}
\ar"1,2";"2,2"^{f_{t'_1}}
}}
\ \text{and}\ 
\vcenter{
\xymatrix{
T(5) & F'(5)\\
T(5) & F'(5)
\ar"1,1";"2,1"^{f_{t'_5}}
\ar"1,2";"2,2"^{F'(t_5)}
\ar"1,2";"1,1"_{e_{5}\al_T}
\ar"2,2";"2,1"^{e_{5}\al_T}
}},
\]
we have a commutative diagram:
\[
\xymatrix{
{}^gF'(1) & {}^gT(1) & T(5) & F'(g1)\\
{}^gF'(1) & {}^gT(1) & T(5) & F'(g1),
\ar"1,1";"1,2"^{{}^g(e_1\al_T)}
\ar@{=}"1,2";"1,3"
\ar"1,4";"1,3"_{e_{g1}\al_T}
\ar"1,1";"2,1"_{{}^gF'(t_1)}
\ar"2,1";"2,2"_{e_{a_1}\al_T}
\ar@{=}"2,2";"2,3"
\ar"2,4";"2,3"^{e_{g1}\al_T}
\ar"1,4";"2,4"^{F'(gt_1)}
\ar"1,2";"2,2"^{{}^gf_{t'_1}}
\ar"1,3";"2,3"^{f_{gt'_1}}
}
\]
and hence \eqref{eq:naturality} is verified in this case.
We check the conditions (a) and (b) in Definition \ref{dfn:1-mor-Colax(I,C)}.

\medskip
{\bf Verifications of (a):}
This is equivalent to the equation that $\ph(1) = \id_{F'}$, which
follows from the construction of $\ph$ and the fact that both $\Ga'$ and $T$
have strict $G$-actions.

\medskip
{\bf Verification of (b):} 
This condition is equivalent to saying that the following diagram is commutative:

\begin{equation}\label{eq:ver-b}
\vcenter{
\xymatrix{
{}^b({}^a(F'(i)) ) & {}^b((F'(ai)) )\\
&  F'(bai)\\
\ar"1,1";"1,2"^{{}^b(\ph_i(a))}
\ar"1,2";"2,2"^{\ph_{(ai)}(b)}
\ar"1,1";"2,2"_{\ph_i(ba)}
}}
\end{equation}
for all $a,b\in G$ and $i\in \Ga'_0$.
By definition of $\ph_i(a)$, the following diagram is commutative:

\[
\xymatrix{
{}^a(F'(i))& {}^aT(i)\\
F'(ai) & T(ai). \\
\ar"1,1";"1,2"^{^a(e_i\al_T)}
\ar"2,1";"2,2"_{e_{ai}\al_T}
\ar"1,1";"2,1"_{\ph_i(a)}
\ar@{=}"1,2";"2,2"
}
\]
This yields the following commutative diagram:

\[
\xymatrix@C=40pt{
{}^b({}^a(F'(i))) & {}^b({}^a(T(i))) &  {}^{ba}(F'(i))\\
 {}^b(F'(ai)) & {}^b(T(ai))  & \\
 F'(bai) & T(b(ai))&  F'(bai),\\
\ar"1,1";"1,2"_{^b(^a(e_i\al_T))}
\ar"2,1";"2,2"_{^b(e_{ai}\al_T)}
\ar"1,1";"2,1"_{{}^b(\ph_i(a))}
\ar@{=}"1,2";"2,2"
\ar@{=}"2,2";"3,2"
\ar"2,1";"3,1"_{ \ph_{ai}(b)}
\ar"3,1";"3,2"^{e_{b(ai)}\al_T}
\ar"1,3";"1,2"^{^{ba}(e_i\al_T)}
\ar"3,3";"3,2"_{e_{bai}\al_T}
\ar"1,3";"3,3"^{\ph_i(ba)}
\ar@{=}@/^1.5pc/"1,1";"1,3"
\ar@{=}@/_1.5pc/"3,1";"3,3"
}
\]
which shows the commutativity of the diagram \eqref{eq:ver-b}.

It remains to show that $(F', \ph)$ is a quasi-equivalence.
Namely we have to show the following claims:

\begin{clm}
$F'$ is an isomorphism, and hence a quasi-equivalence.
\end{clm}

Indeed, we regard $\Ga'$ as a dg category following Remark \ref{rmk:Ginz-dg-cat}.
For each $i \in Q_0$, we have $F'(i) = T(i)$.
Hence $F'$ is bijective on objects.
Moreover, for each $i, j \in Q_0$, we have
a commutative diagram
\[
\xymatrix{
\calD(\Ga')(\Ga'(\blank,i),\Ga'(\blank,j))  & \calD(\Ga)(F(i),F(j))\\
\Ga'(i,j) & T(F'(i),F'(j)),
\ar"1,1";"1,2"^F
\ar"2,1";"1,1"^Y
\ar"2,1";"2,2"_{F'}
\ar@{=}"2,2";"1,2"
}
\]
where $Y$ and $F$ above are bijective.
Hence $F'$ above is an isomorphism.

\begin{clm}
$\ph(a)$ is a 2-quasi-isomorphism for all $a \in G$, i.e.,
$T(\text{-},\ph_i(a))\colon T(\text{-}, {}^aF'(i)) \to T( \text{-}, F'(ai))$ is a quasi-isomorphism in $\ChMod(T)$
for all $a \in G$ and $i \in \Ga'_0$.
\end{clm}
Indeed,
by construction $\ph_i(a) \colon {}^aF'(i) \to F'(ai)$ is an isomorphism in $T$.
Therefore $T(\text{-},\ph_i(a))$ is an isomorpism in $\ChMod(T)$, and thus
it is a quasi-isomorphism.

As a consequence, 
$\widehat{\Gamma}(Q_G,W_G)$ and $\widehat{\Gamma}(Q'_G,W'_G)$ are derived equivalent.
Note that the quivers with potentials $(Q_G,W_G)$ and $(Q'_G,W'_G)$
are not mutated from each other in this case.
Therefore we cannot apply \cite[Theorem 3.2]{KeY} by
Keller-Yang
to obtain this derived equivalence.
\end{exm}

To give an example of the case that the category $I$ is not a group, we need to determine how to compute the Grothendieck construction of a functor
$X \colon I \to \DGkCat$ in general.
This will be done in the forthcoming paper, which will include such an example.

\setcounter{section}{0}
\renewcommand{\thesection}{\Alph{section}}
\renewcommand{\sectionname}{Appendix}

\section{Quasi-equivalence morphisms and derived equivalences}
\label{sec:qeq-1mor}

The following statement is stated in \cite{Ke3} without a proof
in a remark after \cite[Lemma 3.10]{Ke3}.
For completeness, we give a proof of it in this appendix.

\begin{thm}\label{thm:q-eq-der-eq}
Let $E: \calA\to\calB$ be a quasi-equivalence between dg categories $\calA$ and $\calB$. Then
$\blank\Ltimes_{\calA}\ovl{E} \colon \calD( \calA) \to 
\calD(\calB)$
is an equivalence of triangulated categories,
where $\ovl{E}$ is the $\calA$-$\calB$-bimodule
$\ovl{E}:= {}_E\calB$.
In particular, $\calA$ and $\calB$ are derived equivalent.
\end{thm}

For the proof we prepare the following three lemmas.
The first one is the following, which seems to be well-known (e.g., see \cite[Lemma 13.7.2]{Sta}), but we give a proof of it for convenience of the reader.

\begin{lem}\label{equ}
Let $\calD$ and $\calD'$ be triangulated categories, and
$F: \calD\to \calD'$ and $G: \calD'\to \calD$ triangle functors.
Assume that the following conditions are satisfied
\begin{enumerate}
\item
$F$ is fully faithful,
\item
$G$ is a right adjoint to $F$, and
\item
$G(X)= 0$ implies $X = 0$ for all objects $X$ of $\calD'$.
\end{enumerate}
Then $F$ is an equivalence.
\end{lem}

\begin{proof}
We denote the unit and the counit of the adjoint by $\eta: \id_\calD \To G \circ F$ and by $\ep : F \circ G \To \id_{\calD'}$, respectively. 
Let $D\in\calD'$, and take a distinguished triangle
$$
FG(D)\ya{\ep_D} D\to Y\to FG(D)[1]
$$
in $\calD'$.
Apply the functor $G$ to get
$$
GFG(D)\ya{G(\ep_D)} G(D)\to G(Y)\to G(D)[1].
$$
Since $F$ is fully faithful, $\eta: \id \Rightarrow G \circ F$ is an isomorphism.
In particular, $\et_{G(D)}$ is an isomorphism.
Then the equality $G(\ep_D) \et_{G(D)} = \id_{G(D)}$ yields
a commutative diagram with triangle rows:
\[
\begin{tikzcd}
GFG(D) & G(D) & G(Y) & G(D)[1]\\
G(D)  & G(D) & G(Y) & G(D)[1]
\Ar{1-1}{1-2}{"G(\ep_D)"}
\Ar{1-2}{1-3}{}
\Ar{1-3}{1-4}{}
\Ar{2-1}{2-2}{"\id_{G(D)}" '}
\Ar{2-2}{2-3}{}
\Ar{2-3}{2-4}{}
\Ar{1-1}{2-1}{"\et_{G(D)}\inv" '}
\Ar{1-2}{2-2}{equal}
\Ar{1-3}{2-3}{}
\Ar{1-4}{2-4}{}
\end{tikzcd}.
\]
Thus $G(Y) = 0$. Therefore, $Y = 0$ and $FG(D) \iso D$. Consequently, $F$ is essentianlly surjective.
Then $F$ is an equivalence.
\end{proof}

\begin{lem}\label{ff} Let $\calA$ and $\calB$ be dg categories,
and $N$ a dg $\calA$-$\calB$-bimodule.
Assume that
\begin{enumerate}
\item
the dg module ${}_AN$ is compact in $\calD(\calB)$ for all $A\in\calA$,
\item
The canonical morphism $\al_{Y,Z,k} \colon H^k(\calA(Y,Z))\to\Hom_{\calD(\calB)}({}_YN,{}_ZN[k])$ is an isomorphism for all $Y,Z\in\calA$ and for all $k\in\bbZ$.
\end{enumerate}
Then $\blank\Lox_{\calA}N$ is fully faithful.
\end{lem}

\begin{proof} We know that $(\blank\Lox_{\calA}N,\RHom_{\calB}(N,\blank))$ is an adjoint pair. Let $\et$ denote the unit of this adjunction. Therefore to show that $\blank\Lox_{\calA}N$ is fully faithful, it suffices to show the following.
\begin{clm-nn}
For each $M\in \calD(\calA)$,
$
\et_M: M\to \RHom_{\calB}(N,M\Lox_{\calA}N)
$
is an isomorphism in $\calD(\calA)$.
\end{clm-nn}
The proof proceeds in a usual way (e.g., see \cite[Lemma 4.2]{Ke1} or \cite[Proposition 3.10]{Sch}).
Namely, by setting $\calC$ to be the full subcategory of $\calD(\calA)$ formed by those objects $M$ such that $\eta_M$ is an isomorphism,
it is enough to show that $\calC = \calD(\calA)$.
As is easily seen $\calC$ is a triangulated subcategory of $\calD(\calA)$.
Therefore it suffices to show the following two facts:
\begin{enumerate}
\item[(i)]
${}_A\calA \in\calC$ for all $A\in\calA$; and
\item[(ii)]
$\calC$ is closed under small coproducts.
\end{enumerate}
(i) Let $A\in\calA$. We show that ${}_A\calA\in\calC$, namely that 
$$
\et_{{}_A\calA}: {}_A\calA\to \RHom_{\calB}(N,{}_A\calA\Lox_{\calA}N)\iso \RHom_{\calB}(N,{}_AN)
$$
is an isomorphism in $\calD(\calA)$.
It suffices to show that
$$
\eta_{{}_A\calA}: {}_A\calA\to  \RHom_{\calB}(N,{}_AN)
$$
is a quasi-isomorphism.
For each $A'\in\calA$ and $k \in \bbZ$ we have the following commutative diagram:
\[
\begin{tikzcd}[column sep=1em]
H^k(\calA(A',A)) && H^k(\RHom_{\calB}({}_{A'}N, {}_AN))\\
& \Hom_{\calD(\calB)}({}_{A'}N, {}_AN[k])
\Ar{1-1}{1-3}{"{H^k(\eta_{\calA(A',A)})}"}
\Ar{1-1}{2-2}{"\al_{A',A,k}" '}
\Ar{1-3}{2-2}{"\be_{A',A,k}"}
\end{tikzcd},
\]
where $\be_{A',A,k}$ is the canonical isomorphism.
Since $\al_{A',A,k}$ is an isomorphism by the assumption (2),
$H^k(\eta_{\calA(A',A)})$ turns out to be an isomorphism, which shows (i).

(ii)
Let $I$ be a small set and let $M_i \in \calC$ for all $i \in I$.
We have the following commutative diagram
with canonical morphisms in $\calD(\calA)$:
\[
\begin{tikzcd}
\Ds_{i\in I} M_i & \RHom_{\calB}(N,(\Ds_{i\in I} M_i)\Lox_{\calA}N)\\
 & \RHom_{\calB}(N,\Ds_{i\in I} (M_i\Lox_{\calA}N)\\
\Ds_{i\in I} M_i & \Ds_{i\in I}\RHom_{\calB}(N, M_i\Lox_{\calA}N)
\Ar{1-1}{1-2}{"\et_{\Ds_{i\in I} M_i}"}
\Ar{3-1}{3-2}{"\Ds_{i\in I}\et_{M_i}" ', "\sim"}
\Ar{1-1}{3-1}{equal}
\Ar{2-2}{1-2}{"\wr","(\mathrm{a})" '}
\Ar{3-2}{2-2}{"\wr","(\mathrm{b})" '}
\end{tikzcd},
\]
\\
where (a) is an isomorphism because $\blank\Lox_\calA N$ is a left adjoint
and preserves small coproducts, and (b) is an isomorphism by the assumption (1).
Thus 
\[
\et_{\Ds_{i\in I} M_i}: \Ds_{i\in I} M_i\to \RHom_{\calB}(N,\Ds_{i\in I} M_i\Lox_{\calA}N)
\]
is an isomorphism,
and hence we have
$\Ds_{i\in I} M_i\in\calC$.
As a consequence,  $\calC$ is closed under small coproducts.
\end{proof}

\begin{lem}
\label{ker}
Let $\calA$ and $\calB$ be dg categories and
$E \colon \calA \to \calB$ a quasi-equivalence.
Then
for each right $\calB$-module $M$ the following holds:
\[
\RHom_{\calB}({}_E\calB,M)=0
\text{ implies }
M = 0.
\]
\end{lem}

\begin{proof}
Let $M$ be a $\calB$-module, and assume that $\RHom_{\calB}({}_E\calB,M)=0$.
Take any $B \in \calB$.
It is enough to show that $M(B) = 0$.
Now, since $H^0(E): H^0(\calA)\to H^0(\calB)$ is an equivalence (the condition (2) in Definition \ref{dfn:q-eq}), there exists an object $A\in\calA$, such that $E(A) = H^0(E)(A)\iso B$ in $H^0(\calB)$.
Then by the functor
$H^0(\calB) \to \calD(\calB), X \mapsto {}_X\calB$
we have 
${}_{E(A)}\calB\iso {}_B\calB$ in $\calD(\calB)$.
Hence by the dg Yoneda lemma we have
\[
M(B) \iso
\RHom_{\calB}({}_B\calB,M)
\iso
\RHom_{\calB}({}_{E(A)}\calB,M) =0,
\]
as required.
\end{proof}

\begin{proof}[Proof of Theorem \ref{thm:q-eq-der-eq}]
Define a dg $\calA$-$\calB$-bimodule $N$ by $N:= {}_E\calB$.
Then $N$ satisfies the condition (1) in Lemma \ref{ff},
and by the assumption (in particular, by the condition (1) in Definition \ref{dfn:q-eq}) $N$ also satisfies the condition (2) in Lemma \ref{ff}.
Therefore $F:=\blank\Lox_{\calA}N \colon \calD(\calA) \to \calD(\calB)$ is fully faithful by Lemma \ref{ff}.
Moreover $G:= \RHom_{\calB}(N,\blank)$ is a right adjoint to $F$ and satisfies the condition (3) in Lemma \ref{equ}
by the assumption and Lemma \ref{ker}. 
Hence $F$ is an equivalence between $\calD(\calA)$ and $\calD(\calB)$ by Lemma \ref{equ}.
\end{proof}

\section{Flat targets and dg bimodules}
\label{sec:flat-target}

Recall that a dg category $\calA$ is \emph{$\k$-projective}, if for any $x,y\in\calA$, the complex $\calA(x,y)$ is homotopically projective in $\calH(\k)$, that is,
$\Hom_\k(\calA(x,y),\blank)$ sends acyclic complexes to acyclic ones.
A dg category $\calA$ is \emph{$\k$-flat}, if for any $x,y\in\calA$, the complex $\Hom(x,y)$ is homotopically flat in $\calH(\k)$, that is, $\blank\ox_{\k}\calA(x,y)\colon \calH(\k)\to \calH(\k)$ sends acyclic complexes to acyclic ones.

For a $\calB$-$\calA$-bimodule $U$, we set
$\bfp_{\calB\hyph\calA} U:= \bfp_{\calB\ox_\k\calA\op} U$.

\medskip

\begin{dfn}
\label{dfn:Lox-bimodules}
Let $\calA, \calB, \calC$ be small dg categories,
$M$ a $\calC$-$\calB$-bimodule, and $N$ a $\calB$-$\calA$-bimodule.
Then we set
$$
M \Lox_\calB N:= \bfp_{\calC\hyph\calB}M \ox_\calB \bfp_{\calB\hyph\calA}N,
$$
which defines the functor
$$
\blank \Lox_\calB ? \colon
\calD(\calC \ox_\k \calB\op) \times \calD(\calB \ox_\k \calA\op)
\to \calD(\calC \ox_\k \calA\op)
$$
by $(M, N) \mapsto M \Lox_\calB N$.
\end{dfn}

\begin{lem}\cite[Lemma 6.3]{Ke1}\label{lem:flat}
Let $\calA,\calB$ and $\calC$ be dg categories and let $_{\calB}U_{\calA}$ and $_{\calC}V_{\calB}$ be dg bimodules. We assume that $\calA$ is $\k$-flat. Then the following hold:
\begin{enumerate}    
\item  $N\Lox_{\calB} U\iso N\ox_{\calB}\bfp_{\calB\mbox{-}\calA}U$
in $\calD(\calA)$ for all $N\in\calD(\calB)$.

\item 
$(\blank\Lox_{\calC}V)\Lox_{\calB} U
= (\blank \Lox_{\calB} U) \circ (\blank\Lox_{\calC}V) \iso\blank \Lox_{\calB} (V\ox_\calB\bfp_{\calB\hyph\calA} U)$ as functors $\calD(\calC) \to \calD(\calA)$.
\end{enumerate}
\end{lem}

\begin{dfn}
Let $\calA$ be a small dg category.
A dg $\calA$-module $M_{\calA}$ is called {\em homotopically flat}
(\cite[3.3]{Dr2004}) 
if for all acyclic left dg $\calA$-modules $N$, the tensor product 
$M\otimes_{\calA} N$ is also acyclic.
\end{dfn}

The following gives examples of homotopically flat dg modules.

\begin{lem}\label{lem:homo-flat}
Let $\calA$ be a small dg category.
Then any homotopically projective dg $\calA$-module is homotopically flat.
\qed
\end{lem}

\begin{proof}
Let $P$ be a homotopically projective right dg $\calA$-module, $N$ an acyclic
left dg $\calA$-module, and $E$ an injective cogenerator in $\Mod \k$.
Then for each $i \in \bbZ$,
$$
\begin{aligned}
\Hom_\k(H^i(P\ox_\calA N), E)
&\iso H^{-i}\Hom_\k(P\ox_\calA N, E)\\
&\iso H^{-i}(\Cdg(\k)(P \ox_\calA N, E))\\
&\iso H^{-i}(\Cdg(\calA)(P, \Hom_\k(N, E))\\
&\iso \calH(\calA)(P, \Hom_\k(N, E)[-i]) = 0
\end{aligned}
$$
because $\Hom_\k(N, E)$ is an acyclic right dg $\calA$-module.
Then $H^i(P\ox_\calA N) = 0$, and $P\ox_\calA N$ is acyclic. Hence $P$ is a homotopically flat right dg $\calA$-module.
\end{proof}

\begin{lem}\cite[Lemma 3.3]{Im}\label{homo-top}
 Let $\calA, \calB$ and $\calC$ be small dg categories, and
${}_\calC V_\calB, {}_\calB U_\calA$ bimodules. Then 
\begin{enumerate}
\item If $\calA$ is $\k$-flat and a dg bimodule ${}_\calB U_\calA$ is homotopically flat, then $U$ is homotopically flat left $\calB$-module. 
\item If $\calC$ is $\k$-flat and a dg bimodule ${}_\calC V_\calB$ is homotopically flat, then $V$ is homotopically flat right $\calB$-module.
\end{enumerate}
\end{lem}

\begin{lem} \cite[Lemma 3.2]{Im}\label{lem:perv-quasi}
 Let $\calA, \calB$ and $\calC$ be small dg categories, and
${}_\calB U_\calA,{}_\calC V_\calB$ bimodules. Then 
\begin{enumerate}    
\item  If ${}_\calB U_\calA$ is homotopically flat as $\calB$-module, then 
$$
\blank\otimes{}_\calB U_\calA\colon \calC(\calB^{\op}\otimes\calC)\to \calC(\calA^{\op}\otimes\calC)
$$ 
preserves quasi-isomorphisms.
\item If ${}_\calC V_\calB$ is homotopically flat as $\calB$-module, then 
$$
{}_\calC V_\calB\otimes\blank\colon \calC(\calB\otimes\calA^{\op})\to \calC(\calC\otimes\calA^{\op})
$$
\end{enumerate}
preserves quasi-isomorphisms.
\end{lem}

\begin{proof}
It suffices to show that $N\otimes{}_\calB U_\calA$ is acyclic if $N\in \DGMod(\calB^{\op}\otimes\calC)$. Therefore, for any $A\in\calA, C\in\calC$,
$$
(N\otimes{}_\calB U_\calA)(C,A)=N(\blank,C)\otimes_{\calB}U(A,\blank).
$$
By the assumption that ${}_\calB U_\calA$ is homotopically flat as $\calB$-module, and $N(\blank,C)$ is an acyclic dg $\calB$-module. Then 
$N(\blank,C)\otimes_{\calB}U(A,\blank)$ is acyclic, so is $N\otimes{}_\calB U$.
\end{proof}

Although the derived tensor product does not have associativity in general,
the following is known. 
 \begin{prp}\cite[Proposition 3.7]{Im}
 \label{prp:derived-associator}
 Let $\calA, \calB, \calC$ and $\calD$ be small dg categories, and
${}_\calD W_\calC, {}_\calC V_\calB$, ${}_\calB U_\calA$ bimodules.
 If $\calA$ and $\calD$ are $\k$-flat, then there is a natural isomorphism
\[
\bfa^\bfL = \bfa^\bfL_{W,V,U} \colon ({}_\calD W\Lox_\calC V) \Lox_\calB U \isoto {}_\calD W \Lox_\calC (V \Lox_\calB U),
\]
\end{prp}
which means an associativity of the derived tensor products.
This isomorphism is called the {\em derived associator} of derived tensor products.
\begin{proof}

By Lemma \ref{lem:homo-flat}, $\bfp_{\calD\hyph\calC} W$ and $\bfp_{\calB\hyph\calA} U$ are homotopically flat $\calD$-$\calC$-bimodule and $\calB$-$\calA$-bimodule, respectively. 
Since $\calA$ and $\calD$ are $\k$-flat, 
we have the following by Lemma \ref{homo-top}:
\begin{equation}
\label{eq:h-flat}
\left\{\begin{aligned}
&\text{$\bfp_{\calD\hyph\calC} W$ is a homotopically flat right $\calC$-module, and}\\
&\text{$\bfp_{\calB\hyph\calA} U$ is a homotopically flat left $\calB$-module.}
\end{aligned}\right.
\end{equation}
Then we have the following isomorphisms
$$
\begin{aligned}
(W\Lox_\calC V) \Lox_\calB U 
&\overset{(\rm a)}{\to}
\bfp_{\calD\hyph\calB}(W\Lox_\calC V) \ox_\calB \bfp_{\calB\hyph\calA} U\\
&\overset{(\rm b)}{\to}
(W\Lox_\calC V) \ox_\calB \bfp_{\calB\hyph\calA} U\\
&\overset{(\rm c)}{\to}
(\bfp_{\calD\hyph\calC} W\ox_\calC \bfp_{\calC\hyph\calB} V) \ox_\calB \bfp_{\calB\hyph\calA} U\\
&\overset{(\rm d)}{\to}
\bfp_{\calD\hyph\calC} W \ox_\calC (\bfp_{\calC\hyph\calB} V \ox_\calB \bfp_{\calB\hyph\calA} U)\\
&\overset{(\rm e)}{\to}
\bfp_{\calD\hyph\calC} W \ox_\calC (V \Lox_\calB U)\\
&\overset{(\rm f)}{\to}
\bfp_{\calD\hyph\calC} W \ox_\calC \bfp_{\calC\hyph\calB}(V \Lox_\calB U),\\
&\overset{(\rm g)}{\to}
W \Lox_\calC (V\Lox_\calB U).
\end{aligned}
$$
where the isomorphisms (a), (c), (e) and (g) are given by Definition \ref{dfn:Lox-bimodules},
the isomorphisms (b) and (f) are obtained by \eqref{eq:h-flat};
the isomorphism (d) is just the associator of tensor product.
\end{proof}

\section{Dg lifts}
\label{lifts}
In this section, we cite necessary terminologies from the paper \cite{Ke1} by Keller.
Let $\calA$ be a small dg category, and let $\mathscr{U}$ be a small full subcategory of $\calD(\calA)$ and $\bbZ\mathscr{U}$ the full subcategory of $\calD(\calA)$ whose objects are the $U[n]$ for all $U \in \calU, n \in \bbZ$.

\begin{dfn}
  A \emph{lift} of $\mathscr{U}$ is a dg category 
 $\calB$ together with an $\mathscr{B}$-$\mathscr{A}$-bimodule $M$ such that $\blank\Lox_{\calB} M$ gives rise to equivalences $\bbZ\udl{\calB} \stackrel{\sim}{\to} \bbZ\mathscr{U}$ and $\udl{\calB}\stackrel{\sim}{\to} \mathscr{U}$, where $\udl{\calB}$ is the full subcategory of $\calD(\calB)$ with the
 object set $\{N \in \calD(\calB) \mid N \iso B^\wedge, \exists B\in\calB\}$.
\[
\begin{tikzcd}
\bbZ\udl{\calB}  & \bbZ\mathscr{U} \\
\calD(\calB) & \calD(\calA)
\Ar{1-1}{1-2}{"\sim"}
\Ar{1-1}{2-1}{}
\Ar{1-2}{2-2}{}
\Ar{2-1}{2-2}{"\blank\Lox_{\calA} M"}
\end{tikzcd}
\]
\end{dfn}

\begin{dfn}
A \emph{standard lift} of $\mathscr{U}$ is defined
by taking $\calB$ to be the full subcategory of $\Cdg(\calA)$ formed by chosen objects $ \bfp U, U \in\mathscr{U}$, and $M$ to be the bimodule
\[
\bi{M}{\bfp U}{A}:= (\bfp U)(A), \quad \bfp U \in \calB, A \in \calA.
\]
\end{dfn}

Now let $(\calB, M)$ be any lift of $\mathscr{U}$ such that ${}_BM$ is homotopically projective for each $B \in \calB$
(for example, the standard lift satisfies this condition).
Let $\calC$ be a dg category and $\bfF\colon \Cdg(\calC)\to \Cdg(\calA)$ a dg functor such that $\bfL\bfF\colon \calD(\calC)\to\calD(\calA)$ induces a functor $\udl{\calC}\to\mathscr{U}$, then we have the following commutative diagram
\[
\begin{tikzcd}
&\calD(\calC)& \udl{\calC}  \\
\calD(\calB)  & \calD(\calA)& \mathscr{U}\\
\udl{\calB}  & \mathscr{U} 
\Ar{1-3}{1-2}{""}
\Ar{1-2}{2-2}{"\mathbf{L}\bfF"}
\Ar{2-1}{2-2}{"\blank\Lox_{\calB} M"}
\Ar{3-1}{2-1}{}
\Ar{3-2}{2-2}{}
\Ar{3-1}{3-2}{}
\Ar{2-3}{2-2}{}
\Ar{1-3}{2-3}{}
\end{tikzcd}.
\]

The following lemma is well-known by Keller \cite[Lemma 7.3]{Ke1}.

\begin{lem}\label{lem:dg-lift} 
Define a $\calC$-$\calB$-bimodule $Y = Y(M, \bfF)$ by setting
$$\bi{Y}CB:= \Cdg(\calA)({}_B M, \bfF(C^\wedge))$$
for all $B\in \calB, C\in \calC$.

\begin{enumerate}
\item [(a)]
$\blank\Lox_{\calC}Y$ induces a functor $\udl{\calC} \to \udl{\calB}$, hence $Y$ is a quasi-functor. It is a quasi-equivalence if $\bfL \bfF$ induces an equivalence $\bbZ\udl{\calC} \stackrel{\sim}{\to} \bbZ\mathscr{U}$.

\item [(b)]
There is a canonical morphism
\[
(N\Lox_{\calC} Y) \Lox_{\calB} M \To \bfL \bfF N, \quad N\in \calD(\calC)
\]
which is invertible for all $N\in \hprj^b(\calC)$. It is invertible for arbitrary $N\in \calD(\calC)$ iff $\bfL \bfF$ commutes with direct sums.

\item[(c)]
If $\bfF$ has the form $\bfF = \blank\ox_{\calC} Z$ for 
some $Z$ such that $(\calC,Z)$ is a lift of $\mathscr{U}$, then
$\blank\Lox_{\calC}Y$ tuns out to be a quasi-equivalence $\udl{\calC} \to \udl{\calB}$,
and we have $ (\blank\Lox_{\calC} Y)\Lox_{\calB} M  \stackrel{\sim}{\to}\blank\Lox_{\calC} Z$. If moreover ${}_C Z$ is homotopically projective for each $C \in\calC$, then $ \RHom(Y, \blank)\circ\RHom(M, \blank)\stackrel{\sim}{\to}\RHom(Z, \blank)$.
\end{enumerate}
\end{lem}

\end{document}